\documentclass[a4paper,9pt, reqno]{amsart}
\usepackage[utf8]{inputenc}
\usepackage[english]{babel}
\usepackage{amsthm,amsmath,amsfonts,amssymb}
\usepackage{array,enumerate,float,moreverb,appendix,mathrsfs, booktabs, eucal, mathtools, upgreek, mathabx}
\usepackage[svgnames]{xcolor}
\usepackage[final]{graphicx}
\usepackage{algorithm, algorithmic}
\usepackage[colorlinks=true,citecolor=Chocolate,linkcolor=DarkMagenta]{hyperref}
\usepackage[lofdepth,lotdepth]{subfig}
\numberwithin{equation}{section}

\theoremstyle{plain}
\newtheorem{prop}{Proposition}
\newtheorem{lem}[prop]{Lemma}
\newtheorem{corol}[prop]{Corollary}
\newtheorem{theorem}[prop]{Theorem}
\newtheoremstyle{thm}
  {12pt}
  {7pt}
  { \slshape}
  {}
  {\bfseries }
  {. }
  { }
  {}
\theoremstyle{thm}

\newtheorem{defin}[prop]{Definition}
\newtheorem{defins}[prop]{Definitions}

\numberwithin{prop}{section}

\newtheoremstyle{rq}{}{}{}{}{\bfseries}{.}{.5em}{}
\theoremstyle{rq}
\newtheorem{remark}[prop]{Remark}
\newtheorem{exemple}[prop]{Example}

\usepackage{etoolbox}
\patchcmd{\section}{\scshape}{\scshape\large}{}{}
\patchcmd{\section}{.7}{1.4}{}{}
\patchcmd{\section}{.5}{1.2}{}{}

\newcommand{\EE}{{\mathbf{E}}}
\newcommand{\PP}{{\mathbf{P}}}

\newcommand{\Spec}{\mathrm{Sp}}
\newcommand{\COV}{{\mathrm{Cov}}}

\newcommand{\tr}{{\rm tr}}

\newcommand{\dP}{\mathbb{P}}
\newcommand{\dN}{\mathbb {N}}
\newcommand{\dR}{\mathbb {R}}

\newcommand{\cF}{\mathscr{F}}

\newcommand{\cW}{\mathcal {W}}

\newcommand{\defeq}{\vcentcolon=}

\newcommand{\BER}{{\mathsf{Ber}}}
\newcommand{\BIN}{{\mathsf{Bin}}}
\newcommand{\POI}{ \mathsf{Poi}}
\newcommand{\DTV}{{\mathrm{d_{TV}}}}
\newcommand{\SPAN}{ \mathrm{span}}
\newcommand{\ER}{\mathsf{ER}}

\newcommand{\KURT}{\mathsf{Kurt}}

\newcommand{\sign}{ \mathrm{sign}}

\newcommand{\erd}{Erd\H{o}s-R\'enyi }

\newcommand{\ANDalt}{\quad\hbox{and}\quad}
\newcommand{\ABS}[1]{{{\left| #1 \right|}}} 
\newcommand{\BRA}[1]{{{\left\{#1\right\}}}} 
\newcommand{\SBRA}[1]{{{\left[#1\right]}}} 
\newcommand{\PAR}[1]{{{\left(#1\right)}}} 

\newcommand{\uA}{{\underline A}}
\newcommand{\uB}{{\underline B}}
\newcommand{\uM}{{\underline M}}

\newcommand{\ic}{\mathrm{i}}

\newcommand{\1}{1\!\!{\sf I}}\newcommand{\IND}{\1}
\newcommand{\veps}{\varepsilon}

\newcommand{\BEAS}{\begin{eqnarray*}}
\newcommand{\EEAS}{\end{eqnarray*}}
\newcommand{\BEA}{\begin{eqnarray}}
\newcommand{\EEA}{\end{eqnarray}}
\newcommand{\BEQ}{\begin{equation}}
\newcommand{\EEQ}{\end{equation}}
\newcommand{\BIT}{\begin{itemize}}
\newcommand{\EIT}{\end{itemize}}
\newcommand{\BNUM}{\begin{enumerate}}
\newcommand{\ENUM}{\end{enumerate}}
\newcommand{\thresh}{\upvartheta}
\newcommand{\seed}{\mathcal{O}}

\newcommand{\proj}{\mathsf{P}}

\DeclareMathSymbol{I}{\mathalpha}{operators}{`I}
\DeclareMathSymbol{o}{\mathalpha}{operators}{`o}

\title[Very sparse matrix completion]{Detection thresholds in very sparse matrix completion
}
\author{Charles Bordenave, Simon Coste, Raj Rao Nadakuditi}
\date{\today}

\begin{document}

\begin{abstract}

Let $A$ be a rectangular matrix of size $m\times n$ and $A_1$ be the random matrix where each entry of $A$ is multiplied by an independent  $\{0,1\}$-Bernoulli random variable with parameter $1/2$. This paper is about when, how and why the non-Hermitian eigen-spectra of the matrices $A_1 (A - A_1)^*$ and $(A-A_1)^*A_1$ captures more of the relevant information about the principal component structure of $A$ than the eigen-spectra of $A A^*$ and $A^* A$. 

We illustrate the  application of this striking phenomenon on the matrix completion problem for the setting where the underlying matrix $P$ is low rank, with incoherent singular  vectors, and where the matrix $A$ is equal to the matrix $P$ on a (uniformly) random subset of entries of size $dn$ and all other entries of $A$ are equal to zero. We show that the eigenvalues of the asymmetric matrices $A_{1} (A - A_{1})^{*}$ and $(A-A_{1})^{*} A_{1}$ with modulus greater than a detection threshold are asymptotically equal to the eigenvalues of $PP^*$ and $P^*P$ and that the associated eigenvectors are aligned as well. The central surprise is that by intentionally inducing asymmetry and additional randomness via the $A_1$ matrix, we can extract more information than if we had worked with the singular value decomposition (SVD) of $A$!

The associated detection threshold is asymptotically exact and is  non-universal since it explicitly depends on the element-wise distribution of the underlying matrix $P$. We show that reliable, statistically optimal but not perfect matrix recovery, via a universal data-driven algorithm,  is possible above this detection threshold using the information extracted from the asymmetric eigen-decompositions. Averaging the left and right eigenvectors provably improves estimation accuracy but not the detection threshold. Our results encompass the very sparse regime where $d$ is of order $1$ where matrix completion via the SVD of $A$ fails or produces unreliable recovery.

We define another variant of this asymmetric principal component analysis procedure that bypasses the randomization step and has a detection threshold that is smaller by a constant factor but with a computational cost that is larger by a polynomial factor of  the number of observed entries. Both detection thresholds shatter the seeming barrier due to the well-known information theoretical limit $d \asymp \log n$ for matrix completion found in the literature.   

\vspace{-1.45cm}

\end{abstract}

\maketitle

\bibliographystyle{siam}

\tableofcontents


\section{Introduction}

\subsection{Setting and overview}

We start by describing the main mathematical model that will be studied in this paper. Let $m,n \geqslant 1$ large integers, and let $P=(P_{x,y})_{x \in [m],y \in [n]}\in \mathscr{M}_{m,n} (\dR)$ be a real matrix with singular value decomposition
\begin{equation}
\label{SVDofP}
P = \sum_{k= 1}^r \sigma_k \zeta_k \xi_k^*, 
\end{equation}
the positive numbers $\sigma_k$ are the singular values of $P$, and $(\zeta_1, \dotsc, \zeta_r),( \xi_1, \dotsc, \xi_r)$ are two orthonormal families of singular vectors.  Let $M \in \mathscr{M}_{m,n} (\dR)$ be a random matrix whose entries are independent Bernoulli with parameter $d/n$: for all $x,y$, we have 
\[
    \PP ( M_{x,y}  = 1) = 1 - \PP ( M_{x,y} = 0) = \frac{d}{n}. 
\]

The non-zeros entries of the matrix $M$ correspond to the entries of the matrix $P$ that are observed, the remaining ones are  hidden. The {\em observed matrix} is then defined as $$A = \PAR{ \frac{n}{d}} P \odot M, $$ where $\odot$ denotes the Hadamard pr entry-wise product of two matrices.  
The normalization is chosen so that $\EE[A]=P$, hence $A$ is an unbiased estimator of $P$. The matrix $A$ has an average of $d$ revealed revealed entries per row. From now on, we will concentrate on the asymptotic regime where $m$ and $n$ are large and have the same order, by  supposing that $\alpha = m/n$ is bounded away from $0$ and $\infty$. 

\bigskip

The matrix completion problem aims at answering the following general question: \emph{what parts of $P$ can be recovered from the observed entries?}  The literature around this problem is gigantic, see Section \ref{sec:related}. Roughly speaking, it is known that under natural assumptions on $P$ (low-rank with delocalized eigenvectors), we can recover exactly $P$ as soon as $d$ has order $\log n$. Below this threshold, there exists an estimator $\hat P$ whose mean square error, that is $\tr( \hat P  - P )^2 / n$, is of order $1 / d$.

In this paper, we go much beyond this last result and study the  \emph{detection problem of the spectrum of $P$}. Namely, for a given singular value $\sigma_k$ of $P$ with corresponding unit singular vectors $\zeta_k, \xi_k$, our goal is to address the following two questions: 
\begin{enumerate}[(i)]
\item
For which values of $d$ is it possible to design a consistent estimator of $\sigma_k$?  
\item 
For a given  $\veps >0$,  for which values of $d$ is it possible to design an estimator $\hat \zeta_k$ of $\zeta_k$ such that $|\langle \zeta_k , \hat \zeta_k \rangle| \geq \veps$ with high probability, and the same for $\xi$? 
\end{enumerate}

The above mentioned previous results on the mean square error imply that $(i)$ and $(ii)$ are feasible when $d$ goes to infinity. Our results prove that this is possible for $d$ of order $1$ simply by considering the $k$-th largest eigenvalue of an $m \times m$ carefully chosen matrix $X$ and its corresponding eigenvector. 

\subsection{Informal statement of the main result} Let $Z \in \mathscr{M}_{m,n}(\mathbb{R})$ be an auxiliary random matrix whose entries are independent Bernoulli with parameter $1/2$. We set $A_1 = Z \odot A$ and $A_2 = A-A_1$ and we define
\begin{align}
&X = A_1A_2^* &&Y = A_1^* A_2.
\end{align}
These are square matrices, with respective sizes $m \times m$ and $n \times n$. They are not Hermitian, their eigenvalues are complex numbers. Then, given any fixed $d$, there is a threshold $\thresh$, intrinsic to the matrix $P$ and to $d$, such that the following holds with high probability when $n$ is large.
\begin{enumerate}[(i)]
\item Each singular value $\sigma_i> \thresh$ gives rise to an \emph{eigenvalue} $\nu_i$ of $X$ close to $\sigma_i^2$. The rest of the eigenvalues of $X$ are quarantined in the disc $D(0, \thresh)$.  
\item  Moreover, if $\chi_i$ is a unit right-eigenvector of $X$ associated with the eigenvalue $\nu_i \sim \sigma_i^2$ above the threshold $\thresh$, then the scalar product between $\chi_i$ and the singular vector $\zeta_i$ has an explicit non-vanishing limit. 
\end{enumerate}
Similar results hold for $Y$ and $\xi_i$. All the theoretical quantities at stake can efficiently be estimated from the observation of $P \odot M$, leading to asymptotically efficient data-driven estimators even in the sparsest regime where $d$ is fixed. This threshold will be essentially the analog of the Kesten-Stigum.  When all the singular values of $P$ are above the threshold $\thresh$, we can thus use the eigenvalues and eigenvectors of $X$ and $Y$ to craft an estimator $\hat{P}$ of $P$ which is correlated with $P$.

\begin{figure}[t]
 
 \subfloat[Normally distributed singular vectors.]{\label{fig:normal}
    \includegraphics[width=0.95\textwidth]{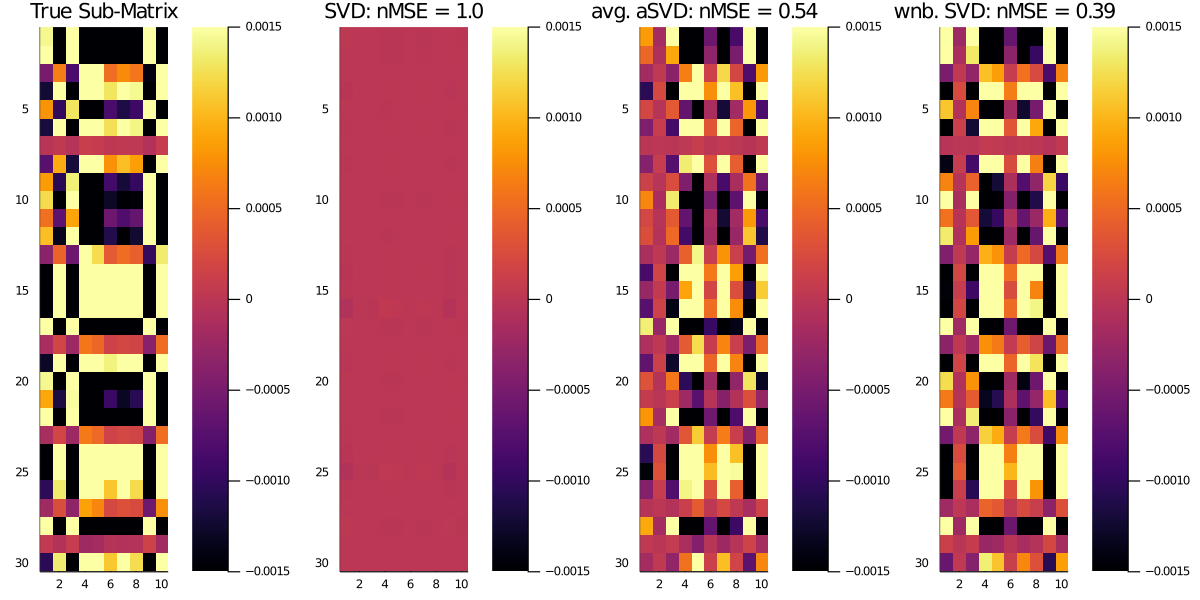} 

    }\\

     \subfloat[Hyperbolic secant distributed singular vectors.]{\label{fig:hyperbolic}
    \includegraphics[width=0.95\textwidth]{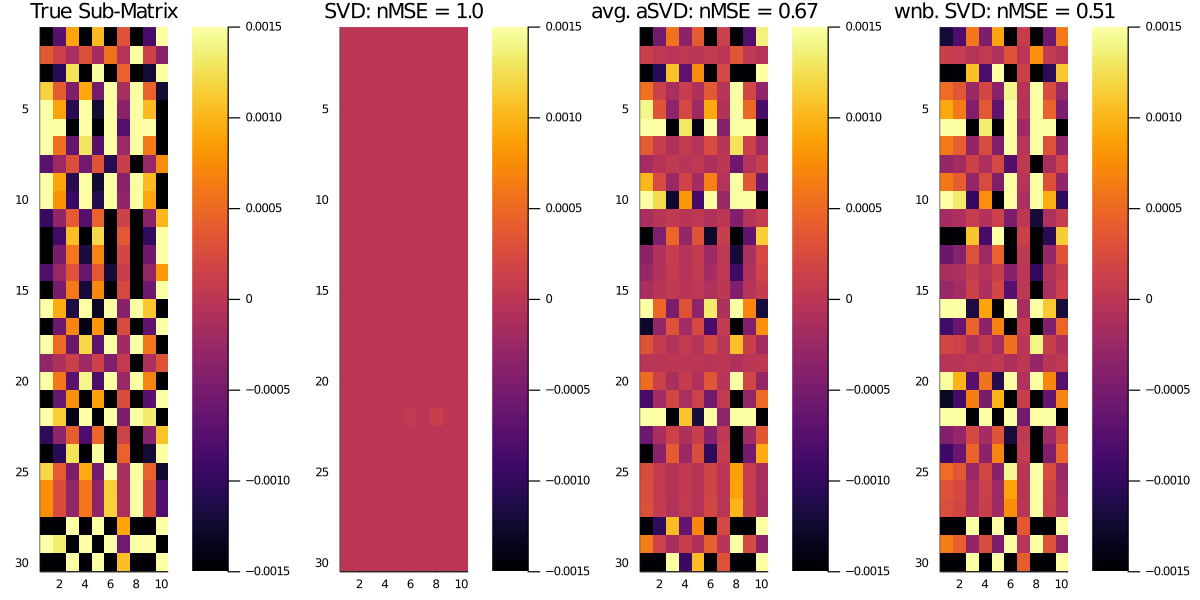} 
}

    \caption{Matrix completion using various methods for the setup described in Section \ref{sec:sneak peek}. Of interest is the fact that the asymmetric eigen-methods proposed in this paper (rightmost two columns) succeed in the regime where the SVD based method (second column from the left) fails.}
    \label{fig:numerical sim mat completion}
\end{figure}

This spectral method does not need complex manipulations, does not require trimming, uses all the available information, and only needs computing the top eigenvalues and eigenvectors of $m \times m$ or $n \times n$ matrices, hence is computationally efficient. On the other hand, spectral algorithms are generally known to be less robust than other methods, and the detection threshold is $1.44$ times higher than the so-called \emph{non-backtracking threshold} studied in our Subsection \ref{subsec:NB}, the latter gives better results but requires a higher computational cost. 

\subsection{Sneak peek: Improved matrix recovery using asymmetric eigen-methods}\label{sec:sneak peek}

Our results are part of a new and promising philosophy, namely that \textit{in many problems eigenvalues of non-symmetric matrices can perform better than eigenvalues of symmetric matrices}. We will highlight this statement with several numerical experiments in Section \ref{sec:numerics}. we would like to motivate the reader by beginning our exposition with a preview of the striking gains our methods  obtain. 

The first column of Figure \ref{fig:numerical sim mat completion} displays the $30 \times 10$ upper-left sub-matrices of $2000 \times 3000$ rank one matrices  modeled as in  (\ref{SVDofP}). In Figure \ref{fig:normal} the singular vectors are normally distributed whereas in Figure \ref{fig:hyperbolic} they are drawn from the hyperbolic secant distributions. This matrix was very sparsely sampled  ($d= 9.7$ and $d = 22.6$ for the Gaussian and Hyperbolic setting, respectively).  The second, third and fourth columns of Figure \ref{fig:numerical sim mat completion} show reconstructions obtained using the SVD (with the missing entries replaced by zeros), the eigen-spectra of the asymmetric matrix as described above (which averages  the left and right eigenvectors of the $X$ and $Y$ matrices to produce an estimate of the left and right singular vectors, respectively) and  using the (right) eigenvector of the  weighted non-backtracking matrix.

\begin{figure}[t]
    \centering
    \subfloat[Symmetric matrix completion via  asymmetrization  and (non-symmetric) eigendecomposition.]{
    \includegraphics[width=0.65\textwidth]{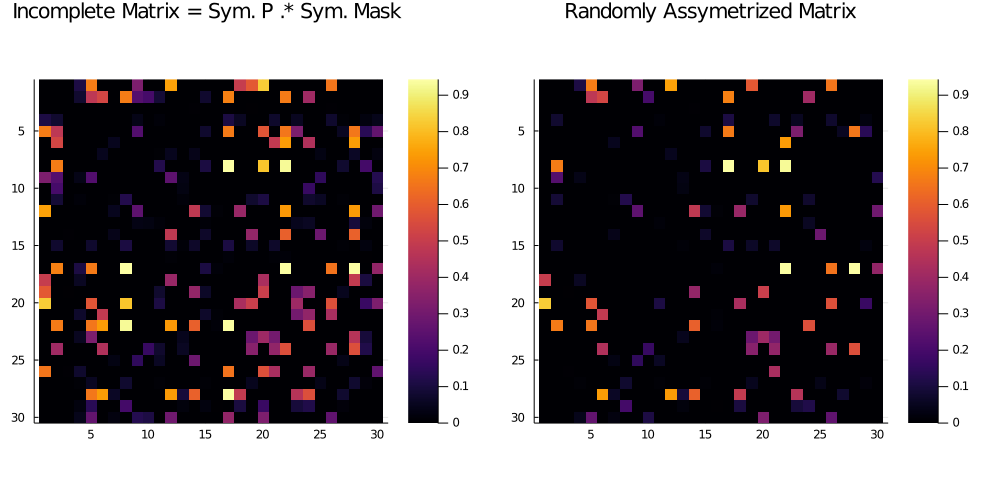}
    \label{fig:sym setup}}
    \\
     \subfloat[Rectangular matrix completion via  asymmetrization  and (non-symmetric) eigendecomposition.]{
    \includegraphics[width=0.95\textwidth]{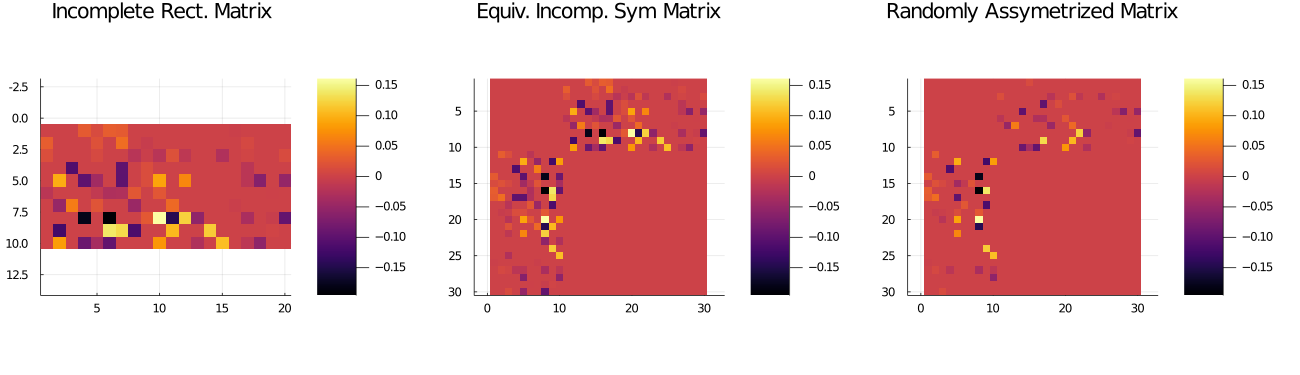}
    \label{fig:rect setup}}
    \label{fig:setup}
    \caption{How we construct asymmetric (square) matrices for extracting more/better eigen-information from sparsely observed symmetric and rectangular low-rank matrices, as described in  Section \ref{sec:roadmap}.} 
\end{figure}

The SVD reconstruction fails completely, whereas the two asymmetric methods methods described succeed and produce reliable, even if imperfect, estimates. To summarize the emergent philosophy: we succeed in finding  (symmetric) structure in a regime where a symmetric eigen-decomposition  fails by inducing, via randomization, asymmetry and viewing the symmetric problem through an asymmetric eigen-decomposition lens.

Intuitively this is happening because the SVD method is crippled by the localization of singular vector estimates in the very sparse regime due to the echo-chamber like effect of a few rows/columns having a larger number of observed entries  is bypassed when we induce asymmetricity and/or use the weighted non-backtracking matrix.  Hence the asymmetric methods detect and extract structure well below the threshold where the SVD can detect and extract it. The asymmetric randomization acts as an implicit spectral regularizer. See Figures \ref{fig:numerical sim 1} and \ref{fig:numerical sim back} for additional illustrations of this phenomenon, including subtler aspects pertaining to how the structure of the matrix we are trying to recover governs the improvement in performance we can (or cannot) expect. 

To summarize our findings: all other things being equal, practitioners can expect greater gains with these methods for recovering incoherent matrices whose elements have a larger kurtosis.

We hope that practitioners who encounter low-rank matrix completion problems in high-dimensional, very sparse settings, such as in reinforcement learning and computational game theory where reward/payoff matrices are often modeled this way \cite{stein2008separable,kannan2010games}, can utilize these methods to find structure in severely under-sampled regimes where the failure of SVD  based matrix completion methods might have been misconstrued as a by-product of a fundamental informational barrier. We shall release a numerical implementation of our methods to facilitate revised experimentation on such given-up-for-being-too-sparse data sets.


\subsection{Extensions}
The results in  this paper extend naturally to the setting where the matrix $P$ is modeled as the Kronecker product of two matrices so that  $P = P_A \otimes P_B$. In this setting, an application of the results of Van Loan and Pitsianis \cite[Section 2]{van1993approximation} shows that there is a rearrangement operator $\mathcal{R}(A)$ which rearranges the elements of the matrix $A$ such that $\mathcal{R}(A)$ is rank one. In other words, $\mathcal{R}(A) = \textrm{vec}(P_A) \textrm{vec}(P_B)$ where $\textrm{vec}(\cdot)$ stacks the columns of the matrix on top of each  other  and creates a single column vector.  We can use this to induce low rank matrices which fit into our framework. Kronecker product structured matrices are ubiquitous in many scientific applications (see \cite{van2000ubiquitous,tsiligkaridis2013covariance}) and spotting Kronecker structure in them and applying the re-arrangement trick can lead to improved matrix completion in Kronecker structured matrix completion problems. 

Tensor completion (see \cite{huang2015provable}) is a natural extension of our ideas in matrix completion for the problem of completing higher order arrays with missing data. Low-rank higher-dimensional tensors can, via a rearrangement and flattening operation (see \cite[Section 2]{grasedyck2013literature}) be expressed as low-rank matrices (with not necessarily orthogonal components). This allows us to connect the low-rank incoherent tensor completion problem with our framework.

Finally, the idea that random asymmetrization bypasses the echo-chamber effect that cripples the SVD (or the symmetric eigen-decomposition) can be employed as a non-parametric estimation technique wherever the underlying singular vectors we are trying to estimate are delocalized but where the SVD (or a symmetric eigen-decomposition) returns localized estimates. This trick will work right out of the box, for example, in the estimation of low-rank matrices with incoherent singular vectors contaminated with heavy tailed noise or just-sparse-enough-but-too-large outliers. Trimming techniques which precisely tune the threshold parameters based on precise structural information might have lower estimator errors than a non-parametric technique such as ours but the random asymmetrization trick will be more robust to errors to errors due to a mismatch in the structure model for the outliers or heavy tails. A hint that the random asymmetrization is a useful technique can spring from analytical or numerical simulation insights as in Figure \ref{fig:illustration_XY} whenever one gleans that the operator norm (or the largest singular value) might be asymptotically unbounded but that the spectral radius (or largest eigenvalue in magnitude) is not. Where else might this be trick be useful beyond our context?  An important extension of our framework is in settings where the missing entries are not sampled uniformly. One important scenario, that lends significant structure in the pattern of the observed entries,  corresponds to the setting where the entries are observed via (for example) a Poissonian process with an intensity  that is proportional to a (assumed, known) function of the (magnitude of the) underlying matrix entries we are trying to estimate, as in for example \cite{arora2016latent}.

We leave related explorations and excursions to follow-up work and  interested readers.

\subsection{Roadmap of the paper and auxiliary results}\label{sec:roadmap}

The main results summarized in the preceding paragraph will be precisely stated later, in Theorem \ref{thm:stats}. They will indeed follow from the simpler case where $P$ is a square, Hermitian matrix. In this case, we do not need to form the matrices $X$ and $Y$ above: the observed matrix, $A$, is itself a square matrix, hence we can directly show the aforementioned phenomenon directly on $A$.

The setup is depicted in  Figure \ref{fig:sym setup}. Informally we will show that when the underlying is symmetric and low-rank and we observe missing entries as in the left panel of Figure \ref{fig:sym setup}, then the randomly asymmetric matrix formed from the original matrix has (left and right) eigenvectors that are well-aligned with the eigenvectors of the underlying low-rank symmetric matrix. 

The detailed statement of this result is contained in Theorem \ref{thm:1}, whose proof runs from Section \ref{sec:perturbation} to Section \ref{sec:eigenwaves} --- it is the most voluminous part of the paper.

The statements characterize the eigenvalues of the randomly asymmeterized matrix and the accuracy, measured via an inner-product, of the left and right eigenvectors with respect to corresponding the ground truth latent eigenvector. What emerges from the results is the fact that averaging the left and right eigenvectors produces more accurate estimates of the underlying eigenvectors and that we can estimate the accuracy of the resulting improve eigenvector estimate directly from the point estimate of the inner product between the left and right eigenvector pairs. This paves the way for a statistically optimal, in a Frobenius norm error sense, estimator of the underlying low-rank matrix that accounts for the noisiness in the estimated eigenvectors \emph{à la} OptShrink \cite{raj}.

The symmetric setup underpins our extension to the matrix completion  because the results on rectangular matrices will then be obtained through a Hermitization trick, by considering the matrix 
\begin{equation}
\begin{pmatrix}
0 & P \\ P^* & 0
\end{pmatrix}
\end{equation}
and applying our theorem for square matrices as depicted in Figure \ref{fig:rect setup}.

This is all done in Section \ref{sec:rectangular}, where the reader will find a complete elucidation of the behaviour of the high eigenvalues and eigenvectors of the matrices $X$ and $Y$, which paves the way for our main interest, the problem of matrix completion in Section \ref{sec:MC}. There, we precisely describe our method and show a few theoretical guarantees for its performance. Numerical simulations are displayed in these first sections, but in Section \ref{sec:rankone} we focus on illustrating the case where the rank of $P$ is one, which has attracted considerable attention in the literature.

Section \ref{sec:related} shortly surveys the rich literature on sparse spectral graph theory, random matrices, principal component analysis and matrix completion. All the subsequent sections, starting with Section \ref{sec:perturbation} at page \pageref{sec:perturbation}, are the technical proofs of our results. 

We included in Section \ref{subsec:NB} several results on non-backtracking matrices (see \ref{subsec:NB}), when the underlying $P$ is square. The very recent paper \cite{stephan2020nonbacktracking} was build on a preliminary version of the present work and generalizes this portion to the  case of weighted \emph{inhomogeneous} graphs. 

A few technical results which were developed in the course of the proof might be of independent interest. Among them, we mention the perturbation results from Section \ref{sec:perturbation}, dealing with spectra of perturbations of non-normal matrices (Theorem \ref{thm:algebra}), and also the results of Section \ref{sec:functionals} which include new and powerful concentration inequalities for functionals on Erd\H{o}s-R\'enyi random graphs (see Proposition \ref{efron-stein-lemma}). 

\subsection{Acknowledgments}
CB was supported by ANR-16-CE40-0024-01. SC is supported by ERC NEMO, under the European Union’s Horizon
2020 research and innovation programme grant agreement
number 788851. RRN's work was supported by ONR grant N00014-15-1-2141, DARPA Young Faculty Award D14AP00086, and ARO MURI W911NF-11-1-039CB. CB and SC thank the University of Michigan for its hospitality in June 2016 and June 2018 where this work was initiated and continued. CB and RRN thank  Literati Coffee in Ann Arbor, MI for the stimulating environment in which the problem considered here was first brewed. 

\subsection{Notation and convention}
When $n$ is an integer, $[n]$ denotes the set $\{1, \dotsc, n\}$. The group of permutations of $[r]$ is noted $\mathfrak{S}_r$. We identify $\mathbb{R}^n$ with the set $\ell^2([n])$. Elements in $\mathbb{R}^n$ will be noted $u=(u(x))_{x \in [n]}$. We will note $|\cdot |_\infty, |\cdot |_p$ the usual norms on $\mathbb{R}^n$, namely
$$|u|_\infty = \max_{x \in [n]} |u(x)| \qquad |u|_p = \Big| \sum_{x \in [n]}|u(x)|^p \Big|^{1/p}.$$
The Euclidean norm ($p=2$) will simply be noted $|\cdot|$. The operator norm of the matrix $X$ is noted $\Vert X \Vert$; it is the greatest singular value of the matrix. The Frobenius norm is noted $\Vert X \Vert_F$ and is defined by $\Vert X \Vert_F = \sqrt{\tr(X^* X)}$. It is also the $L^2$-norm of the singular values.

The letter $c$ denotes a universal numerical constant. It might be used from line to line to denote different constants.

We will also make the following convention on the phase of eigenvectors. If $\psi$ and $\psi'$
are two right and left eigenvectors of a matrix  associated to the same simple eigenvalue, we will also assume that their phase is chosen so that 
\begin{equation}\label{eq:choicephase}
\langle \psi' , \psi \rangle \geq 0.
\end{equation}


\section{Detailed results: square matrices}\label{sec:results}

In this section, we restrict ourselves to the case where $P$ is a square $n \times n$ matrix. We write its spectral decomposition as \begin{equation}
\label{spectral_decomposition}
P = \sum_{k= 1}^n \mu_k \varphi_k \varphi_k^*, 
\end{equation}
the real numbers $\mu_k$ are the eigenvalues of $P$, and $\varphi_1, \dotsc, \varphi_n$ is an orthonormal basis of eigenvectors. The eigenvalues are ordered by decreasing modulus: 
\[|\mu_1|  \geqslant \dotsb \geqslant |\mu_n| \geq 0. \]
As above, $M$ is a matrix with i.i.d. Bernoulli entries with parameter $d/n$, and the observed matrix is 
\[A = \left(\frac{n}{d} \right) P \odot M. \]
Our goal is to describe the behaviour of the high eigenvalues of $A$.

\subsection{Main result}
Our result will hold uniformly over a wide class of matrices that match the usual hypothesis from the literature: low (stable) rank with incoherence conditions. The goal is is not really to restrict the range of applications, but to track the dependence of the error terms with respect to the parameters at stake (such as stable rank, measure of incoherence or spectral separations). We list these definitions which are central to this paper, and then we explain them in the subsequent remark. 

\begin{defins}[complexity parameters of $P$]
Let $P = \sum \mu_i \varphi_i \varphi_i^*$ be a square Hermitian matrix. The amplitude, stable rank and incoherence describe the complexity of the matrix $P$.
\begin{enumerate}
\item \textbf{Amplitude parameter} $L$: \begin{equation}\label{def:L}
L = n \max_{x,y}|P_{x,y} |.
\end{equation}
Equivalently, it is the scaled $L^1$ to $L^\infty$ norm of $P$.
\item \textbf{Stable numerical rank} $r$: $$
r = \frac{\| P \|_F^2 }{\| P \|^2} = \frac{ \sum_{k=1}^n \mu_k^2 }{\mu_1^2}.
$$
\item \textbf{Incoherence parameter} $b$: any scalar $b \geq 1$ such that for every $k $ in $ [n]$ with $\mu_k \ne 0$, we have
\begin{equation}\label{incoherence1}
\max_{x \in [n]} |\varphi_k(x)| \leq  \frac{b}{\sqrt{n}}.
\end{equation}
\end{enumerate}
\end{defins}

\begin{defins}[detection parameters of $P$ and $d$]Let $P = \sum \mu_i \varphi_i \varphi_i^*$ be a square Hermitian matrix and $d>1$ be a real number. The detection threshold, rank and gap describe what parts of $P$ can be detected and how easily. 
\begin{enumerate}
\item \textbf{Variance matrix} $Q$: \begin{equation}\label{def:Q}
Q_{x,y} = n |P_{x,y}|^2 \qquad \qquad  \rho = \Vert Q \Vert.
\end{equation}
\item \textbf{Detection threshold} $\thresh$: any number $\thresh$ such that 
 \begin{equation}\label{def:thresh}
\thresh \geq \max \{\thresh_1,  \thresh_2 \},
\end{equation}
where the `theta parameters' are defined by
\begin{equation*}
\thresh_2 = \sqrt{\frac{\rho}{d}} \quad \text{ and } \quad \thresh_1=\frac{L}{d}.
\end{equation*}
\item \textbf{Detection rank} $r_0$:  number of eigenvalues of $P$ which have modulus strictly larger than  $\thresh$, i.e.
\begin{equation}
|\mu_1|  \geqslant \cdots \geqslant |\mu_{r_0}|  > \thresh \geqslant   |\mu_{r_0+1}|  \geqslant \cdots \geqslant |\mu_n|.
\end{equation}
\item \textbf{Detection hardness or gap} $\tau_0$: \begin{equation}\label{def:tau}
\frac{\thresh}{|\mu_{r_0}|}  = \tau_0 \in (0,1).
\end{equation}
It is the gap between $\thresh$ and the smallest eigenvalue above $\thresh$.
\end{enumerate}
\end{defins}

We observe that our threshold $\thresh$ can vary above $\max \{\thresh_1,  \thresh_2 \}$. This is to allow an optimal application of our main theorem below. It is also interesting to note that the usual algebraic rank of $P$ is an upper bound on $r$. We can now state our main theorem.

\begin{theorem}\label{thm:1}
Let $P$ and $A$ be as above.
We define $D = \max(2d,1.01)$ and 
\begin{equation}\label{def:l} 
\ell =\left\lfloor (1/8) \log_{D} (n)\right\rfloor. 
\end{equation}
There exists a universal constant $c \geq 1$ such that if  the inequality 
\begin{equation}\label{eq:defC0}
C_0 \defeq c  r   r_0^{4} b^{44}  \ln(n)^{16}  \leq  \tau_0^{-\ell},
\end{equation}
holds true then  with probability greater than $1-cn^{-1/4}$, the following event occurs:
\bigskip

\emph{1) Eigenvalues}. There exists an ordering of the largest $r_0$ eigenvalues in modulus $\lambda_1, \ldots, \lambda_{r_0}$ of $A$ such that for all $i \in[r_0]$,
\begin{equation}\label{permutation}
 |\lambda_i - \mu_{i}|\leqslant C_0 \left| \frac{ \thresh }{\mu_{i}}\right|^\ell |\mu_{i}|,
\end{equation}
and all the other eigenvalues of $A$ have modulus smaller than $C_0^{1/\ell} \thresh$. 

\bigskip

\emph{2) Eigenvectors}. We denote by $\psi_i$ and $\psi_i'$ two unit right and left eigenvectors of $\lambda_i$ with positive scalar product.  The relative spectral gap ratio at $\mu_i$ is defined as
\begin{equation}
\label{hyp:spectral_sep}
\tau_{i,\ell} =  1 - \min_{j \in [n] \setminus \{i\}} | 1 - (\mu_j/\mu_i)^\ell |.
\end{equation}
Then, for every $i \in [r_0]$ and $j \in [r]$, one has
\begin{equation}\label{eigenvector_errorbound}
\left| |\langle \psi_i, \varphi_j \rangle | - \frac{\delta_{i,j}}{\sqrt{\gamma_i}} \right| \leqslant \frac{C_0 \tau_0^{\ell}}{1 - \tau_{i,\ell} }.
\end{equation}
where $\gamma_i \geq 1$ is the deterministic number only depending on $d$ and $P$ defined by
\begin{equation}\label{def:smallgamma}\gamma_i \defeq  \sum_{s=0}^{\ell} \frac{\langle \mathbf{1}, Q^s \varphi_i \odot \varphi_i \rangle}{(\mu_i^2 d)^s}. \end{equation}
The overlap between eigenvectors satisfy
\begin{equation}\label{eigenvector_overlap}
\left| |\langle \psi_i, \psi_j \rangle | - \frac{|\Gamma_{i,j}|}{\sqrt{\gamma_i\gamma_j}} \right| \leqslant \frac{C_0 \tau_0^{\ell}}{\sqrt{(1 - \tau_{i,\ell})(1 - \tau_{i,\ell}) }}.
\end{equation}
where $\Gamma_{i,j}$ are real numbers defined by 
\begin{equation}
\Gamma_{i,j} \defeq \sum_{s=0}^\ell \frac{\langle \mathbf{1}, Q \varphi_i \odot \varphi_j \rangle}{(\sigma_i \sigma_j d)^s}.
\end{equation}
Finally, the same bound \eqref{eigenvector_errorbound}-\eqref{eigenvector_overlap} also hold for the unit \emph{left} eigenvectors  $\psi'_i$ and 
\begin{equation}\label{eigenvector_errorboundLR}
\left| \langle \psi_i, \psi'_i \rangle  - \frac{\delta_{i,j}}{ \gamma_i } \right| \leqslant \frac{C_0 \tau_0^{\ell}}{1 - \tau_{i,\ell} }.
\end{equation}
\end{theorem}

We have stated this theorem for any matrix $P$ with parameters $b,r,r_0,\tau_0,d,\thresh$ without mentioning any dependence on $n$. In fact, all those parameters can indeed depend on $n$ since our result is non-asymptotic and quantitative. A numerical value for the universal constant $c \geq 1$ could be extracted from the proof, even if it would not be very informative: we have used various crude bounds to arrive at a tractable constant $C_0$ and a readable proof. There are however various ways to improve the value of $C_0$, notably by decreasing very substantially the factor $b^{44}$ or $\ln(n)^{16}$. We have postponed this technical discussion in the final Section \ref{sec:techdiscuss}.

In Theorem \ref{thm:1}, the value of the threshold $\thresh \geq \thresh_2 \vee \thresh_1$ is free, it determines the eigenvalues above the threshold and the gap $\tau_0$. Theorem \ref{thm:1} is not trivial if $C_0$ is smaller than $\tau_0^{-\ell}$. In the typical situation where the parameters $b,r,r_0$ are $O(\ln(n)^c)$, for some constant $c> 0$, it happens if $\log_{d} (1/\tau_0) \gg \ln(\ln(n)) / \ln(n)$. We note also that in this paper we only focus in the regime where $d$ is small, typically for $d = O( \sqrt{\ln(n)})$, where usual spectral methods  on the symmetric matrices are not working. We have thus made no effort in obtaining an interesting error bound when $d$ is larger.

The threshold $\thresh_2 = \sqrt{\rho / d}$ is the analog of the Kesten-Stigum bound in community detection, see \cite{MR3699594}, it is related to an intrinsic property on the existence of an eigenwave in a Galton-Watson tree with Poisson offspring distribution with parameter $d$, see Section \ref{sec:eigenwaves}. In most applications and simulations, $\thresh_2$ is bigger than $\thresh_1 = L/d$. There is  however a regime where $d$ is very small and as a consequence, the  actual threshold is $\thresh_1$. This is the same phenomenon as the one uncovered in \cite{coste2017} --- see the definition of $\tilde{\rho}$ in Theorem 1 of that paper, see also \cite{BQZ} for a similar phenomenon. In our setting, as it is defined, the threshold $\thresh_1 = L / d$ is often pessimistic. We have used it mostly for the readability of the proofs. There are ways to decrease the value of $L$ for most choices of matrices $P$. We have again postponed this technical discussion in Section \ref{sec:techdiscuss}.

\begin{figure}\centering
\includegraphics[width=0.45\textwidth]{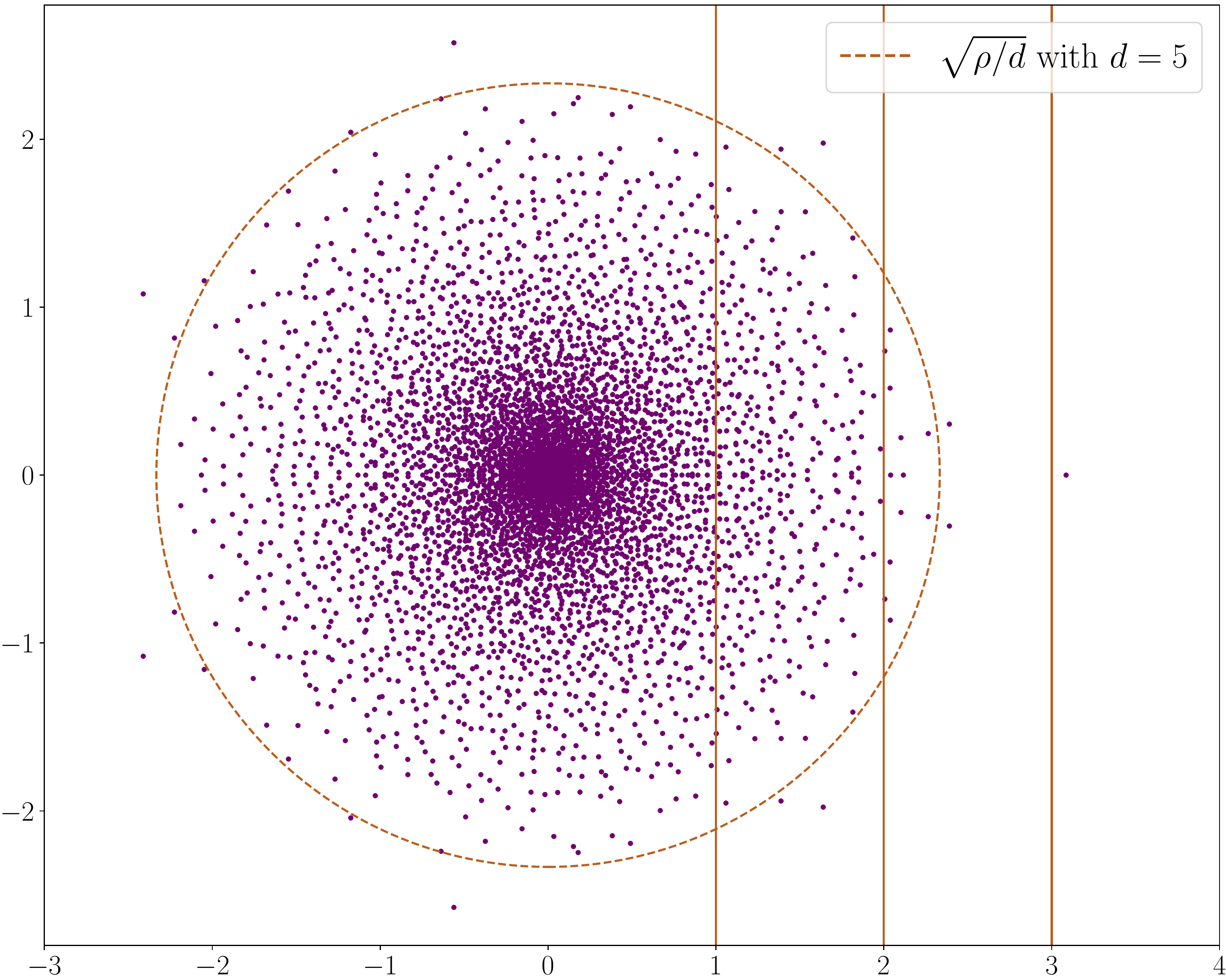}
\includegraphics[width=0.45\textwidth]{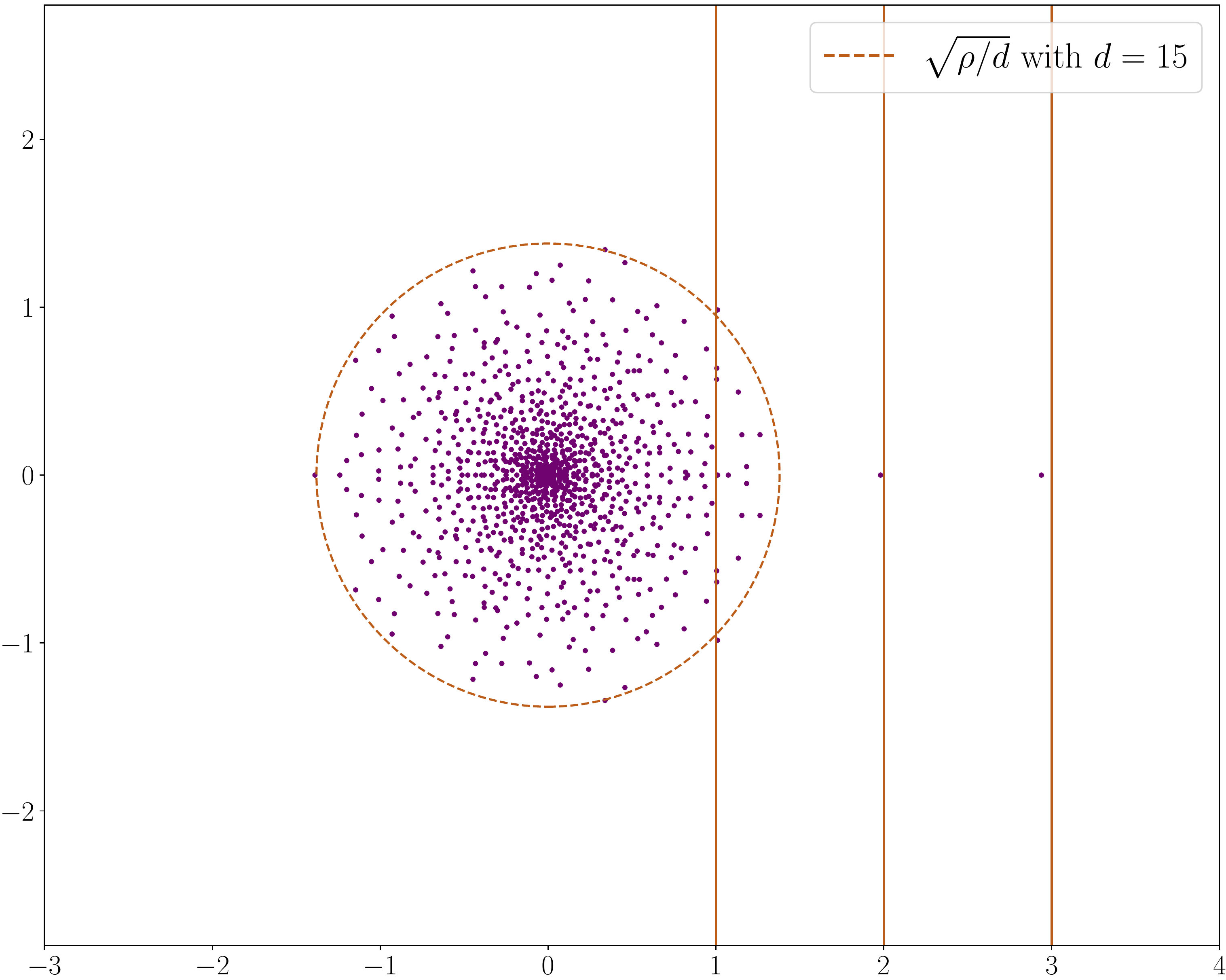}

\caption{An underlying $10000\times 10000$ symmetric matrix $P$ satisfying the above conditions, and with eigenvalues $\mu_3=1,\mu_2=2$ and $\mu_1=3$, has been fixed. Above, we see the eigenvalues of $A$ when the sparsity parameter is $d=5$. In this case, the threshold $\max\{\thresh, \thresh_0\}$ is approximately 2.44, hence only the eigenvalue $\mu_1=3$ gives rise to an outlier in the spectrum of $A$. On the right panel, we chose $d=15$, and in this case the threshold is close to 1.3, thus the two eigenvalues $\mu_1=3,\mu_2=2$ give rise to two outliers of $A$ close to $3$ and $2$.  }\label{fig:illustration1}
\end{figure}

\begin{remark}We note that it is immediate to check that as $d$ grows, $\gamma_i - 1 \sim C/d$ where $C$ depends on $P$. \end{remark}

\begin{remark}
The coefficients $\gamma_i$ have a simple expression if $P$ is such that $\sum_y Q_{xy} = n \sum_y P^2_{xy}$ does  not depend on $x$. Indeed, in this case, the vector $\mathbf{1}$ is the top eigenvector of $Q$ and we have $\sum_x Q_{xy} = \rho /n$. In particular,  $Q^s \mathbf{1} = \rho^s \mathbf{1}$. From \eqref{def:smallgamma}, for $i \in [r_0]$, we get 
\begin{equation}\label{eq:gammaiCte}
\gamma_i = \sum_{s=0}^{\ell} \PAR{\frac{\thresh_2}{\mu_i}}^{2s} = \frac{1 - (\thresh_2/\mu_i)^{2(\ell+1)}}{1 - (\thresh_2/\mu_i)^2}.
\end{equation}

\end{remark}

We conclude this subsection with an easy corollary of Theorem \ref{thm:1}. It asserts that, above the threshold $\thresh$, the average of the left and right unit eigenvectors of $A$ is a good estimate of the corresponding eigenvector of $P$.

\begin{corol}\label{cor:thm1}
With the notation of Theorem \ref{thm:1}, for $i \in [r_0]$, let $\hat \varphi_i = (\psi_i + \psi'_i ) / | \psi_i + \psi'_i|$ where $|\cdot|$ is the Euclidean norm of a vector. On the event of Theorem \ref{thm:1}, we have for all  $j \in [r]$,
\begin{equation}
\left| |\langle \hat \varphi_i, \varphi_j \rangle | - \frac{\sqrt 2\delta_{i,j}}{\sqrt{\gamma_i+1}} \right| \leqslant \frac{4 C_0 \tau_0^{\ell}}{1 - \tau_{i,\ell} }.
\end{equation}
\end{corol}

The above result states that \emph{weak recovery is feasible even in the regime where $d$ is of order $1$} provided that $\thresh$ itself is of order $1$. This is in very sharp contrast with what would happen if the revealed entries were symmetric. As known in the literature (a general survey is given in Subsection \ref{subsec:erd_bibli} at page \pageref{subsec:erd_bibli}), the top eigenvalues would then be aligned with the high-degree vertices, but also the top eigenvectors would be localized on those vertices, losing all the signal information.

\subsection{The rank one case and \erd graphs}\label{subsec:square_rank1} We illustrate Theorem \ref{thm:1} for rank one matrices which are already an interesting first example: $P = \varphi \varphi^*$. With the above notation $r=1$, $\mu_1 =1$ and $\varphi_1 = \varphi$.  In this case, from \eqref{eq:gammaiCte} it is easy to check that 
we have
\begin{equation}\label{eq:rankgamma1}
\thresh_2 = \sqrt{ \frac{n |\varphi |^4_4}{d}} \quad \hbox{ and } \quad \gamma =\frac{1 - \thresh_2^{2(\ell+1)}}{1 - \thresh_2^2}
\end{equation}
where $\gamma = \gamma_1$ is defined by \eqref{def:smallgamma} and $|\varphi|^4_4 = \sum_x |\varphi(x)|^4$.  For $d = O(\ln(n))$, Theorem \ref{thm:1} is an improvement of the results in \cite{chen2018asymmetry}. 

This result on rank-one matrices can be applied to the adjacency matrix of a directed \erd graph. It corresponds to a matrix $P$ whose entries are $P_{x,y} = 1/n$ for all $x,y$. Then the matrix $A' = dA$ is the adjacency matrix of a random graph where each directed edge $(x,y)$ (including loops $(x,x)$) is present independently with probability $d/n$. In the asymptotic regime $n \to \infty$ and $d >1$ is fixed, Theorem \ref{thm:1} implies that  1) $A'$ has one outlier eigenvalue close to $d$, all the other eigenvalues being smaller than $\sqrt{d}$ and 2) the unit eigenvector $\psi$ associated with the outlier eigenvalue  satisfies $|\langle \psi, \mathbf{1}/\sqrt{n}\rangle| = \sqrt{1-1/d} + o(1)$. These results are illustrated on Figure \ref{fig:ERa}, at page \pageref{fig:ERa}.

\begin{figure}\centering
\includegraphics[width=0.85\textwidth]{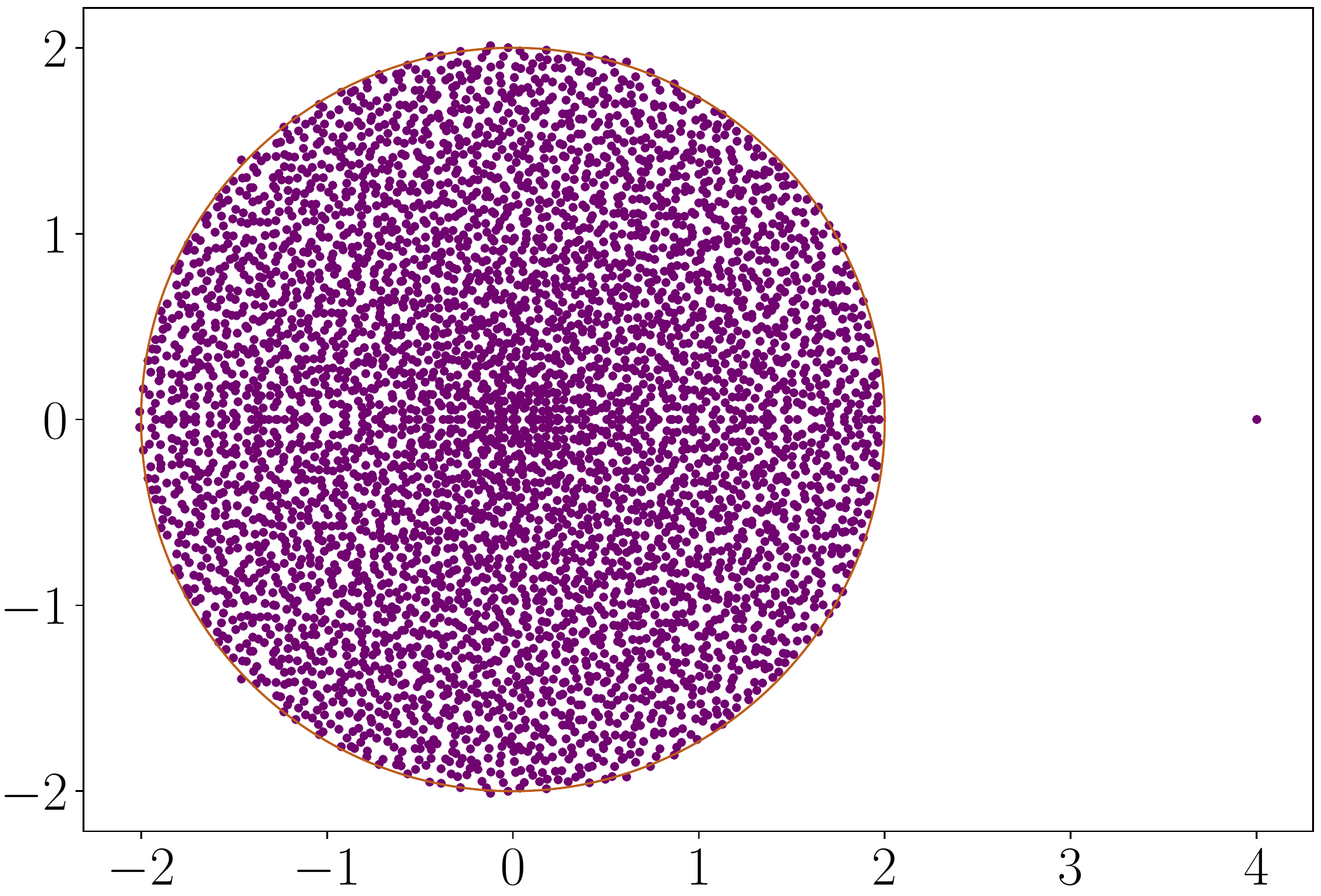}
\begin{tabular}{cc}
\includegraphics[width=0.4\textwidth]{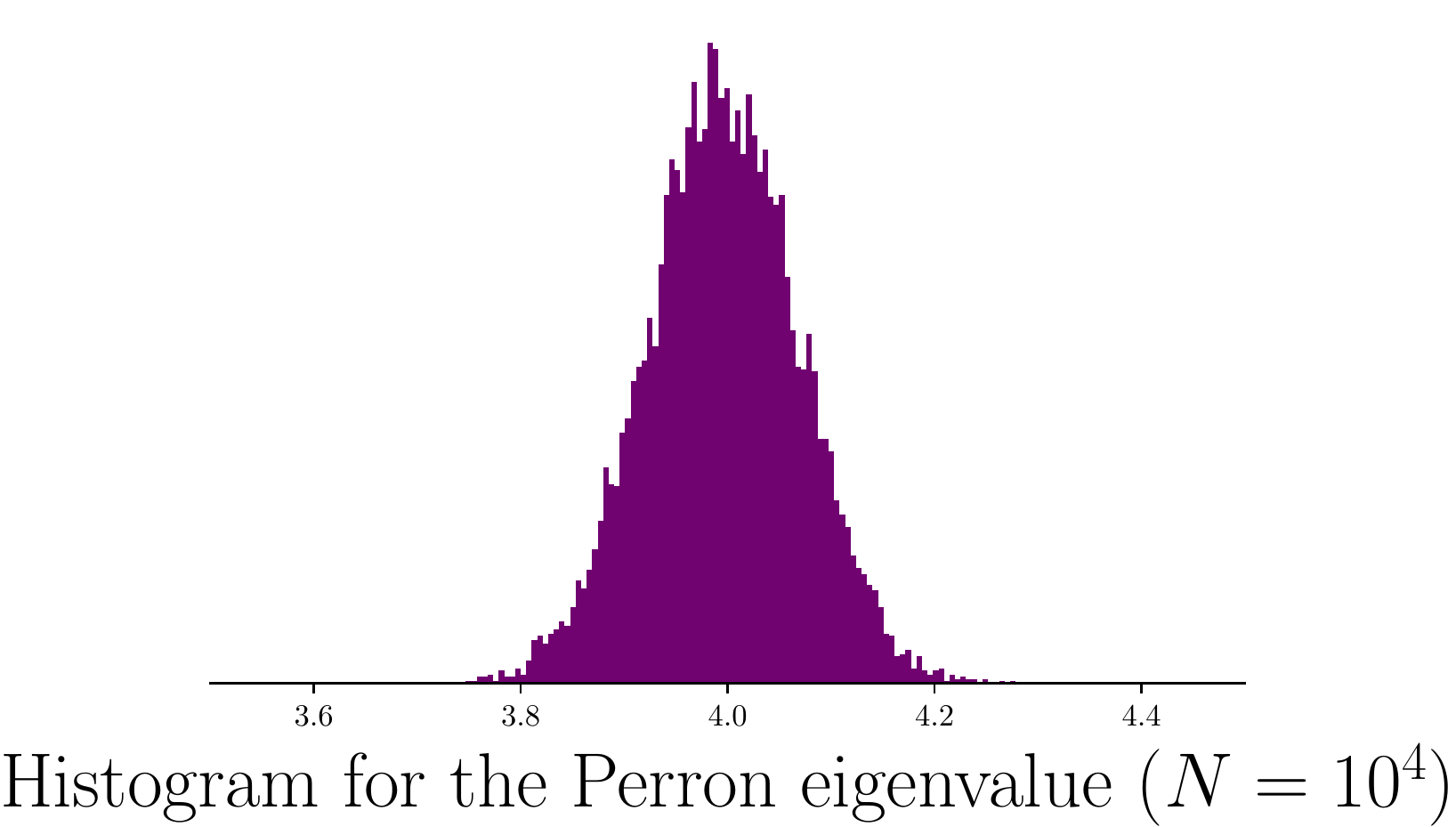}&\includegraphics[width=0.4\textwidth]{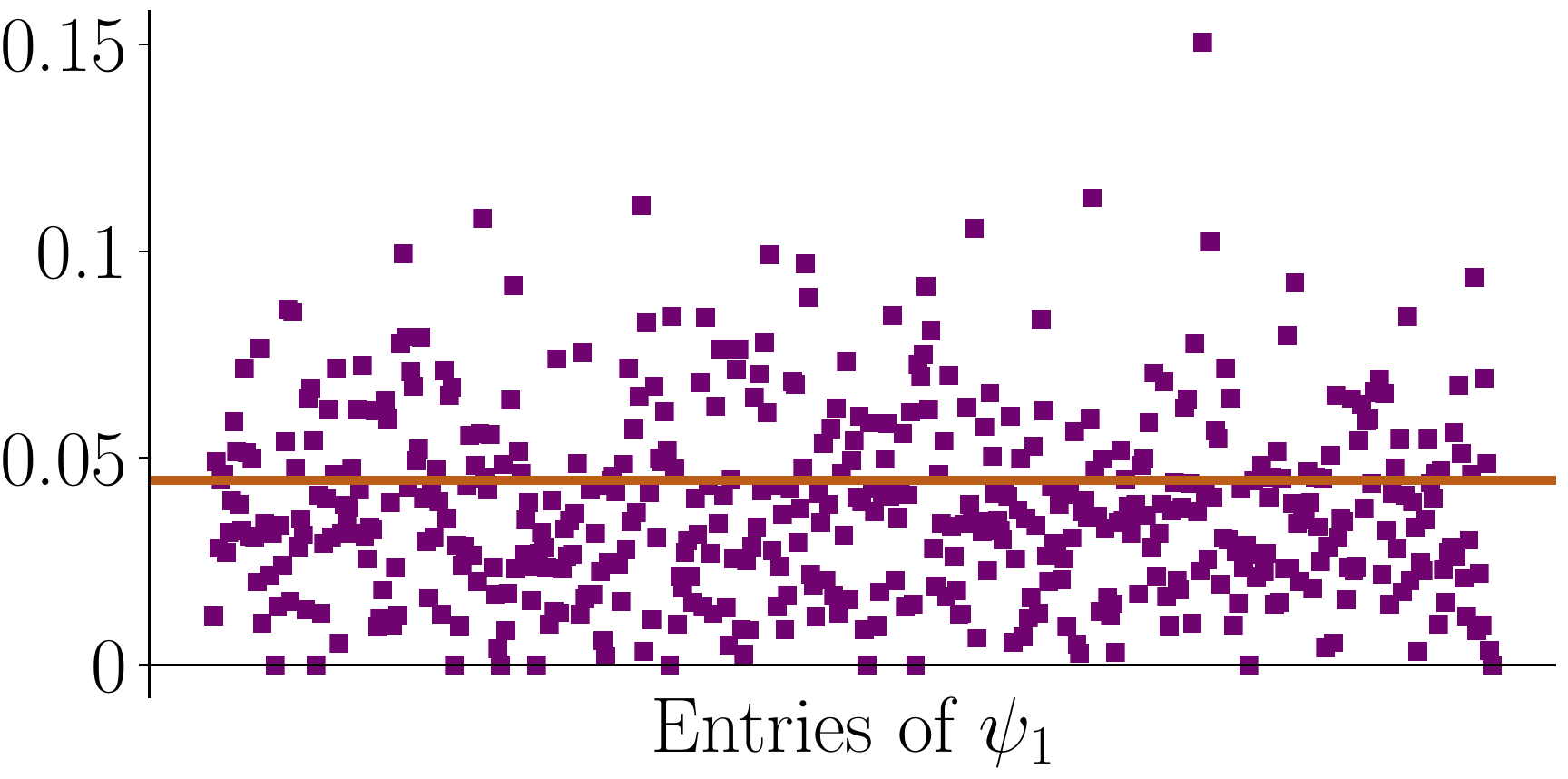}
\end{tabular}
\caption{On top, the spectrum of a directed \erd graph with $d=4$ an $n=10000$ vertices. The outlier $\lambda_1 \approx 4$ is clearly visible. Above (right), a plot of the entries of the eigenvector of $A$ associated with $\lambda_1\approx 4$. All entries are positive (Perron's theorem) and are stacked close to the real eigenvector $\varphi$ which is in orange; their scalar product is close to $\sqrt{1-1/4}$. Above (left) is a histogram of the values of $\lambda_1$, over $N=10000$ realizations of directed \erd graphs with parameters $n=10000$ and $d=4$. }\label{fig:ERa}
\end{figure}

This is quite a striking contrast with the \emph{undirected} sparse \erd graphs, where the high eigenvalues are aligned with the high-degree vertices, and the associated eigenvectors are localized on those vertices, see references below.

\section{Detailed results: non-backtracking matrix }\label{subsec:NB}

In this section, we work out what happens if instead of using the adjacency matrix of the problem, we use the non-backtracking matrix. 

\subsection{Setting: weighted non-backtracking matrix}

In the square symmetric case $P=P^*$, Theorem \ref{thm:1} and its Corollary \ref{cor:thm1} illustrate the striking accuracy of the spectrum of the non-symmetric matrix $A$ to estimate the symmetric matrix $P$. There is however some information which has been lost: the fact that $P = P^*$ has not been used in the definition of $A$ and the set of revealed entries could be almost doubled in principle. At some extra computational cost, a weighted variant of the non-backtracking matrix can cope with this issue.

Let us first define the probabilistic model.  Let $\bar d \geq 1$ and $\bar M \in \mathscr{M}_{n} (\dR)$  be a random \emph{symmetric} matrix where all entries above the diagonal are independent Bernoulli random variable with parameter $\bar d / n$: for all $x,y \in [n]$, $\bar M_{x,y} = \bar M_{y,x}$ and
$$
\PP ( \bar M_{x,y} = 1) =  1 - \PP ( \bar M_{x,y} = 0) = \frac{\bar d}{n}.
$$
Then, we define $E = \{ (x,y) : \bar M_{x,y}= 1 \}$, it is the set of revealed entries of $P$. With $\bar d = 2d - d/n$, this would correspond to the revealed entries of the matrix $(A+A^*)/2$ in our previous model up to a slight modification of the law of entries on the diagonal which is harmless in our setting.

We will need the following notation. A vector $\varphi \in \dR^n$ can be lifted as two vectors, $\varphi^+$, $\varphi^-$ in $\dR^{E}$ by setting
$$
\varphi^-((x,y)) = \frac{\varphi(x)}{\sqrt{\bar d}} \quad \hbox{ and } \quad \varphi^+((x,y)) = \frac{\varphi(y)}{\sqrt{\bar d}}.
$$
The scaling is chosen so that these lifts are isometries from $\dR^n$ to $L^2(\dR^E,\PP)$: $\EE [|\varphi^{\pm} |^2 ]= |\varphi|^2$, where  $|\cdot|$ is the Euclidean norm of a vector. The norm of $\varphi^{\pm}$ is tightly concentrated: if $|\varphi| = 1$ and $|\varphi|_{\infty} \leq b /\sqrt n$, then with probability at least $1 - 1/n$,
$$
\ABS{ |\varphi^+|^2 - 1 } \leq c b^2 d \ln(n)^{5/2} n^{-1/2}, 
$$
for some universal constant $c >0$ (it follows for example from the forthcoming Theorem \ref{thm:concentration}).

The {\em weighted non-backtracking matrix} $B \in \mathscr{M}_{E} (\dR)$ is the non-symmetric matrix indexed by $E$ with entries, for $e = (x,y) \in E$ and $f = (a,b)\in E$ (those are directed edges):
\begin{equation*}
B_{e , f}  =   \frac{n}{\bar d } \IND_{  a = y} \IND_{ x \ne b} P_{a,b}.
\end{equation*}

Exactly as for the matrix $A$, we can relate the top eigenvalues and eigenvectors of $B$ with those of $P$.
The weighted non-backtracking matrix $B$ is defined on the directed edges of the graph induced by the non-missing entries of the matrix  $P$ and hence, so are its  eigenvectors. At each vertex, we sum the elements of an eigenvector of $B$ over all its incoming edges and then normalize the resulting vector to have unit norm. We refer to these vectors as the {\em weighting non-backtracking eigenvectors} of the matrix $P$.

\subsection{Results}

The non-backtracking matrix allows to reduce the detection threshold, the cost being that the size of the matrix $B$ is typically larger by a factor $\bar d$ (see Remark \ref{rq:IB} below for possible ways to solve this issue). 

A version of Theorem \ref{thm:1} also holds for the matrix $B$, but before stating it we need some definitions. The first ones are simply adapted from the square case. We emphasize the difference between the original definitions and the non-backtracking ones by overlining the corresponding quantities.

\begin{defins}[complexity parameters of $P$]
Let $P = \sum \mu_i \varphi_i \varphi_i^*$ be a square Hermitian matrix. The amplitude, stable rank and incoherence describe the complexity of the matrix $P$.
\begin{enumerate}
\item \textbf{Amplitude parameter} $L$: \begin{equation}
L = n \max_{x,y}|P_{x,y} |.
\end{equation}
Equivalently, it is the scaled $L^1$ to $L^\infty$ norm of $P$.
\item \textbf{NB-Stable rank} $\bar{r}$: 
for technical reasons, we introduce a stronger notion of stable rank by setting
$$
\bar r := \frac{\sum_{j=1}^n |\mu_j | }{|\mu_1|}.
$$
Note that $r \leq \bar r \leq \mathrm{rank}(P)$. 
\item \textbf{Incoherence parameter} $b$: any scalar $b \geq 1$ such that for every $k $ in $ [n]$ with $\mu_k \ne 0$, we have
\begin{equation}
\max_{x \in [n]} |\varphi_k(x)| \leq  \frac{b}{\sqrt{n}}.
\end{equation}
\end{enumerate}
\end{defins}

\begin{defins}[NB-detection parameters of $P$ and $d$]Let $P = \sum \mu_i \varphi_i \varphi_i^*$ be a square Hermitian matrix and $d>1$ be a real number. The detection threshold, rank and gap describe what parts of $P$ can be detected and how easily. 
\begin{enumerate}
\item \textbf{Variance matrix} $Q$: \begin{equation}
Q_{x,y} = n |P_{x,y}|^2 \qquad \qquad  \rho = \Vert Q \Vert.
\end{equation}
\item \textbf{NB-Detection threshold} $\bar\thresh$: any number $\bar\thresh$ such that 
$$
\bar \thresh \geq \max (\bar \thresh_1, \bar \thresh_2),
$$
where 
\begin{equation*}
\bar \thresh_2 = \sqrt{\frac{\rho}{\bar d}} \quad \text{ and } \quad \bar \thresh_1=\frac{L}{\bar d}.
\end{equation*}
\item \textbf{Detection rank} $r_0$:  number of eigenvalues of $P$ which have modulus strictly larger than  $\thresh$, i.e.
\begin{equation}
|\mu_1|  \geqslant \cdots \geqslant |\mu_{r_0}|  > \bar \thresh \geqslant   |\mu_{r_0+1}|  \geqslant \cdots \geqslant |\mu_n|.
\end{equation}
\item \textbf{Detection hardness or gap} $\tau_0$: \begin{equation}
\frac{\bar\thresh}{|\mu_{r_0}|}  = \tau_0 \in (0,1).
\end{equation}
\end{enumerate}
\end{defins}

Before stating our main theorem for the non-backtracking matrix, we need to introduce a new parameter on the matrix $P$ which is really specific to the non-backtracking setting: this is due to the fact $B$ and $B^*$ are not equal in law. 

\begin{defin}If $r_0 \geq 1$, we define the matrix $C \in \mathscr{M}_{r_0} (\dR)$ by, for all $i,j \in [r_0]$,
$$
C_{i,j} = \frac{\langle \mathbf{1}, Q \varphi_i \odot \varphi_j \rangle}{\mu_i \mu_j} = \sum_{x,y} \frac{Q_{x,y}}{\mu_i \mu_j} \varphi_i (x) \varphi_j(x).
$$

\end{defin}

Note that $C$ is scale invariant. The matrix $C$ is the Gram matrix of the vectors $(\varphi_i/\mu_i )_{i \in [r_0]}$ associated to the  scalar product $\langle \psi,\phi \rangle_Q := \sum_{x} (\sum_y Q_{x,y}) \psi(x) \phi(x)$ on the vector space spanned by the coordinate vectors $(e_x)_{x \in V}$ where $V$ is the set of $x \in [n]$ such that $n\sum_y P_{x,y}^2 = \sum_y Q_{x,y} >0$. It is thus easy to check that $C$ is definite positive. We denote by $\sigma > 0 $ the smallest eigenvalue of $C$. For example, if $\sum_{y} Q_{x,y}$ does not depend on $x$ then $\sigma \geq 1$. In general, we have $\sigma \geq \min_{x \in V} \sum_{y} Q_{x,y}/ \mu_1^2$.

We are now ready to state a version of Theorem \ref{thm:1} for the non-backtracking matrix $B$.

\begin{theorem}\label{thm:1nb}
Let $P$ and $B$ be as above and assume $\bar d \geq 1$ and $\sigma \geq 0.01$. We define $D = \max(d,1.01)$ and 
\begin{equation}\label{def:lnb} 
\ell = \lfloor (1/8) \log_{D} (n)\rfloor. 
\end{equation}
There exists a universal constant $c \geq 1$ such that if  the inequality 
\begin{equation}\label{eq:defC0nb}
\bar C_0 = c  d \bar r   r_0^{4} b^{40}  \ln(n)^{14} \leq \tau_0^{-\ell},
\end{equation}
holds true then  with probability greater than $1-cn^{-1/4}$, the following event occurs:
\bigskip

1) Eigenvalues. There exists an ordering of the largest $r_0$ eigenvalues in modulus $\lambda_1, \ldots, \lambda_{r_0}$ of $B$ such that for all $i \in[r_0]$,
\begin{equation}\label{permutation_nb}
 |\lambda_i - \mu_{i}|\leqslant \bar C_0 \left| \frac{ \thresh }{\mu_{i}}\right|^\ell |\mu_{i}|,
\end{equation}
and all the other eigenvalues of $B$ have modulus smaller than $\bar C_0^{1/\ell} \thresh$. 

\bigskip

2) Eigenvectors. We denote by $\psi_i$ and $\psi_i'$ two unit right and left eigenvectors of $\lambda_i$ with positive scalar product.  For every $i \in [r_0]$ and $j \in [r]$, one has
\begin{equation}\label{eigenvector_errorboundnb}
\left| |\langle \psi_i, \varphi^+_j \rangle | - \frac{\delta_{i,j}}{\sqrt{\gamma_i}} \right| \leqslant \frac{\bar C_0 \tau_0^{\ell}}{1 - \tau_{i,\ell} } \quad \hbox{ and } \quad \left| |\langle \psi'_i, \varphi^+_j \rangle | - \frac{\delta_{i,j}}{\sqrt{\hat \gamma_i}} \right| \leqslant \frac{\bar C_0 \tau_0^{\ell}}{1 - \tau_{i,\ell} } .
\end{equation}
where $\tau_{i,\ell}$ is defined in \eqref{hyp:spectral_sep} and $\gamma_i $ is defined in \eqref{def:smallgamma} with $\bar d$ in place of $d$ and 
$$
\hat \gamma_i:= \bar d \sum_{s=1}^{\ell+1} \frac{\langle \mathbf{1}, Q^s \varphi_i \odot \varphi_j \rangle}{(\mu_i^2 \bar d)^s}.
$$
Finally, the same bound \eqref{eigenvector_errorbound} also holds for the unit left eigenvector  $\psi'_i$ and 
\begin{equation}\label{eigenvector_errorboundLRnb}
\left| \langle \psi_i, \psi'_i \rangle  - \frac{\delta_{i,j}}{ \sqrt{\gamma_i \hat \gamma_i }} \right| \leqslant \frac{C_0 \tau_0^{\ell}}{1 - \tau_{i,\ell} }.
\end{equation}
\end{theorem}

The assumptions $\sigma \geq 0.01$ is only to guarantee a bound which is uniform in $\sigma \geq 0.01$. In general, it could easily be extracted from the proof an expression of $\bar C_0$ which depends on $\sigma$.

The coefficient $\hat \gamma_i$ has a simple expression if $\sum_{y} Q_{xy}$ is constant. Arguing as in Equation \eqref{eq:gammaiCte}, we have 

\begin{equation}\label{gamchapNB}
\hat \gamma_i = \bar d \sum_{s=1}^{\ell+1}  \PAR{\frac{\thresh_2}{\mu_i}}^{2s} = \frac{\rho}{\mu_i^2} \gamma_i = \frac{\rho}{\mu_i^2} \frac{1 - (\thresh_2/\mu_i)^{2(\ell+1)}}{1 - (\thresh_2/\mu_i)^2}.
\end{equation}

We will check that $\rho \geq \mu_1^2$ (in forthcoming \eqref{lower_bound_on_rho}). In particular, we always have in this case that $\hat \gamma_i \geq \gamma_i$. It follows from \eqref{eigenvector_errorboundnb} that the right eigenvector $\psi_i$ is closer than the left eigenvector $\psi'_i$ to $\varphi^+_i$.

There is also an analog of Corollary \ref{cor:thm1} which allows to define a new sharp estimator of eigenvectors of $P$. To this end, we define the 'left divergence' of a vector $\psi \in \dR^{E}$ as the vector $\check \psi \in \dR^n$: for all $y \in [n]$, 
$$
\check  \psi  (y) =  \frac{1}{d}\sum_{x  : (x,y) \in E} \psi((x,y)).
$$
The 'right divergence' is defined as follows. Let $\mathrm{deg}(y)$ be the number of edges $E$ attached to $y$ (with loops, $\bar M_{yy} = 1$, counting twice). If $\mathrm{deg}(y) \leq 1$, we set $\hat \psi  (y) = 0$, otherwise, we set 
$$
\hat \psi  (y) =  \frac{1}{d(\mathrm{deg}(y)-1)} \sum_{x  : (x,y) \in E} \psi((x,y)).
$$
With the notation of Theorem \ref{thm:1nb}, we will check that, for symmetry reasons, we have for $i \in [r_0]$, the vectors $\hat \psi_i / |\hat \psi_i| $ and $ \check \psi'_i / | \check \psi'_i  |$ are very close to each other and well-defined.


\begin{corol}\label{cor:1nb}
With the notation of Theorem \ref{thm:1nb}, for $i \in [r_0]$, let $\hat \varphi_i = \hat \psi_i / |\hat \psi_i|$ and $\check \varphi_i = \check \psi'_i / | \check \psi'_i  |$. For some universal constant $c >0$, with probability at least $1- c n^{-1/4}$, we have for all  $j \in [r]$,
\begin{equation*}
\left| |\langle \hat \varphi_i, \varphi_j \rangle | -  \frac{\delta_{i,j}}{\sqrt{\gamma_i}}  \right| \leqslant \frac{4 \bar C_0 \tau_0^{\ell}}{1 - \tau_{i,\ell}}.
\end{equation*}
The same statement holds with $\check \varphi_i$ in place of $\hat \varphi_i$.
\end{corol}


\begin{remark}\label{rq:IB}
The spectrum of $B$ is related through the so-called Ihara-Bass formulas to the spectrum of Hermitian matrices of dimension $n$, see \cite{WAFU,NA17} for recent references. In the simplest case where the entries of $nP$ takes only two values, say  $0$ and $1$, then the spectrum of $B$ can be obtained from the spectrum of a matrix in $\mathscr{M}_{2n} (\dR)$, see \cite[Note 3.5]{AFH}.  These formulas have been used to design symmetric matrices in $\mathscr{M}_{n} (\dR)$ strongly connected to the spectrum of $B$, see notably \cite{10.5555/2968826.2968872,10.5555/2969239.2969380}.  

As an alternative, for an integer $\ell \geq 1$, we may consider the symmetric matrix $B_\ell \in \mathscr{M}_{n} (\dR)$ defined as 
$$
B_\ell =  \nabla^* \Delta B^{\ell-1} T \nabla, 
$$
where $\Delta \in \mathscr{M}_E(\dR)$ is the diagonal matrix  defined for $e = (a,b) \in E$ by $(\Delta)_{e,e} = n P_{a,b}$, $T \in \mathscr{M}_E(\dR)$ is the involution matrix defined for $e  = (a,b) \in E$ by $T \delta_{(a,b)} = \delta_{(b,a)}$ and the matrix $\nabla\in \mathscr{M}_{E,n}(\dR)$ defined for $e = (a,b) \in  E$ and $x \in [n]$ by $\nabla_{e,x} = \IND_{a =x} / \sqrt{\bar d}$ (so that $\varphi^- = \nabla \varphi$). The entry $(B_\ell)_{x,y}$ is equal to the weighted sum of non-backtracking paths of length $\ell$ between $x$ and $y$ on the random graph whose adjacency matrix is $M$ and  with edge weights $n P_{xy} / \bar d$.   Then with $\ell$ odd as in \eqref{def:lnb}, it can easily be checked from the proofs that a version of Theorem \ref{thm:1nb} holds for $B_\ell$ if we replace the eigenvalues $\lambda_i$ of $B$ by $\sign(\lambda_{i,\ell}) | \lambda_{i,\ell} |^{1/\ell}$ where the $\lambda_{i,\ell}$'s are the eigenvalues of $B_\ell$ (this result comes with better constants since we can rely on the spectral perturbation theory of symmetric matrices). In practice, this matrix $B_\ell$ is however less natural that the matrix $B$ since it has an extra parameter $\ell$ which is rather artificial.
\end{remark}

\subsection{A representative example}\label{example:nb}In this section we work out a small example, where we chose the very simple case where 
\[
P = \begin{pmatrix}
a & b \\ b & a
\end{pmatrix} \otimes E_{n/2} = \frac{2}{n}\begin{pmatrix}
a & \dots & a& b & \dots & b \\
\vdots & & \vdots & \vdots & & \vdots \\
a & \dots & a& b & \dots & b \\
b & \dots & b& a & \dots & a \\
\vdots & & \vdots & \vdots & & \vdots \\
b & \dots & b& a & \dots & a \\\end{pmatrix}
\]
with $E_k$ the matrix of size $k$ with $1$ everywhere. Its spectral decomposition is given by $P = (a+b) \mathbf{1}_1 \mathbf{1}_1^* + (a-b)\mathbf{1}_2\mathbf{1}_2^*$, where $\mathbf{1}_1=(1,\dotsc, 1, 0, \dotsc, 0)$ and $\mathbf{1}_2= \mathbf{1}-\mathbf{1}_1$.

Such a $P$ obviously satisfies the required hypothesis for our analysis (low-rank, incoherence). 
The eigenvalues of $P$ are $a-b, a+b$ with multiplicity one, and $0$ with multiplicity $n-2$. Consequently, the matrix $Q = n P \odot P$ is equal to 
\[Q =  \frac{4}{n}\begin{pmatrix}
a^2 & \dots & a^2& b^2 & \dots & b^2 \\
\vdots & & \vdots & \vdots & & \vdots \\
a^2 & \dots & a^2& b^2 & \dots & b^2 \\
b^2 & \dots & b^2& a^2 & \dots & a^2 \\
\vdots & & \vdots & \vdots & & \vdots \\
b^2 & \dots & b^2& a^2 & \dots & a^2 \\\end{pmatrix}.
\]
and we immediately infer $\rho = 2(a^2+b^2)$. When $a=4$ and $b=1$, the nonzero eigenvalues of $P$ are thus $5$ and $3$ and the detection threshold with the non-symmetric masked matrix is $\thresh= \sqrt{34/d}$, while the detection threshold with the non-backtracking matrix is $\bar{\thresh}=\sqrt{\rho/2d}= \sqrt{17/d}$. When $d=3$, these thresholds will be
\begin{align*}
\thresh = \sqrt{34/3}\approx 3.317 &&\bar{\thresh}=\sqrt{17/3}\approx 2.23.
\end{align*}
While the adjacency matrix will only have  one outlier close to $5$, the non-backtracking matrix will have two outliers, thus reflecting the whole structure of $P$ and capturing more information on $P$ --- at a higher computational cost though, because the average size of $B$ is $2dn$. The phenomenon is illustrated at Figure \ref{fig:illustration_nb}.

\begin{figure}\centering

\includegraphics[width=0.9\textwidth]{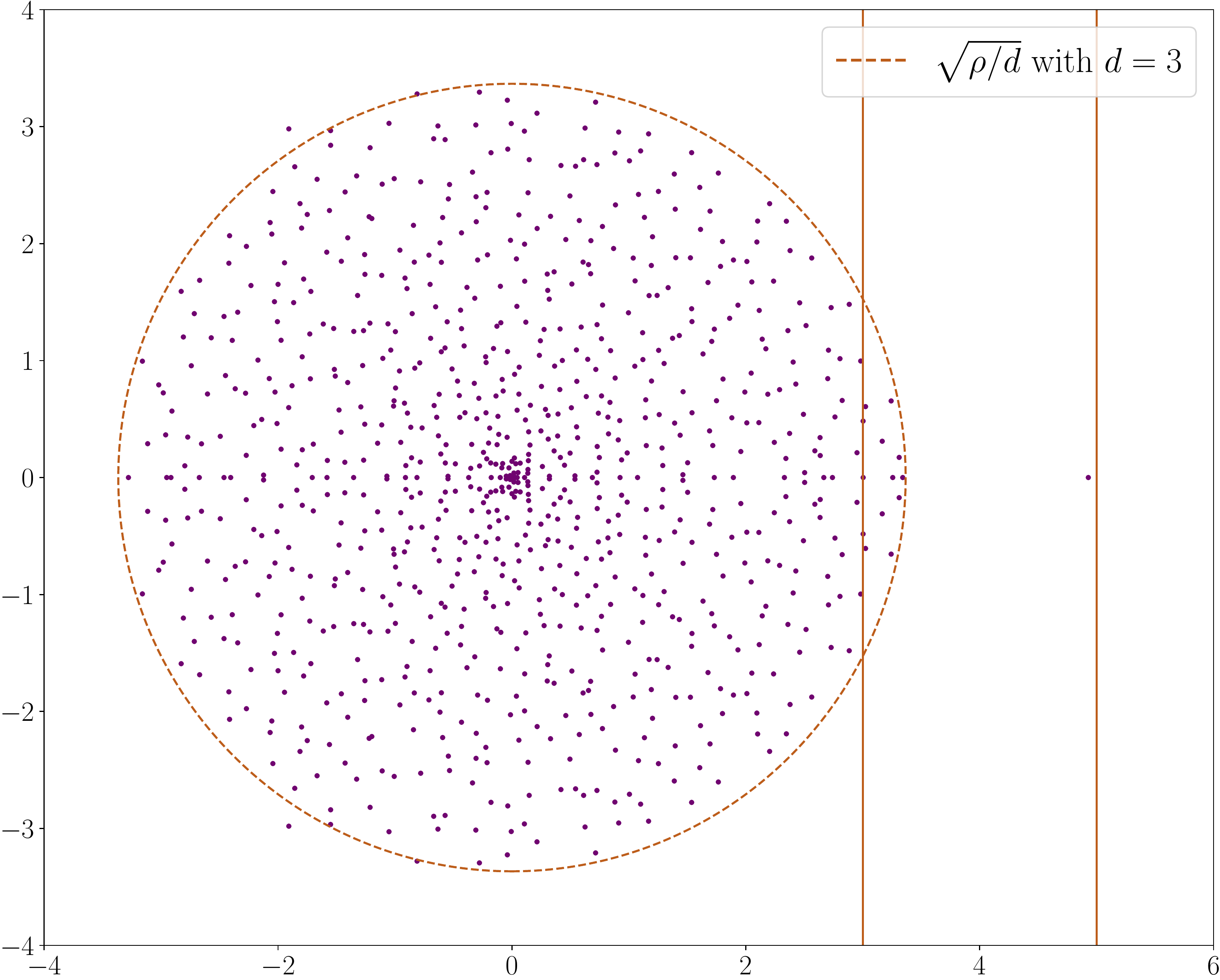}

\includegraphics[width=0.9\textwidth]{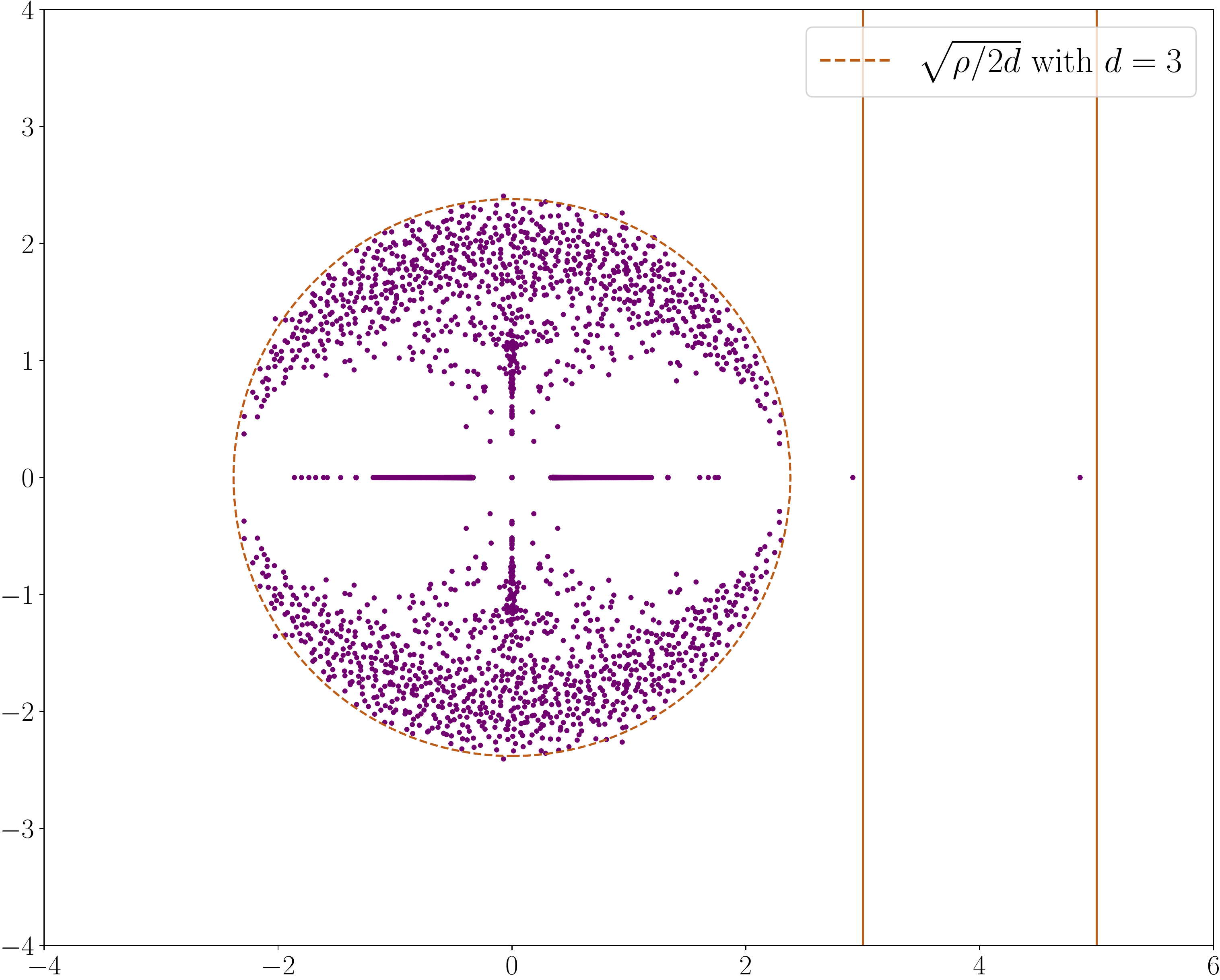}

\caption{The underlying $1000\times 1000$ symmetric matrix $P$ is as in Example \ref{example:nb} and has eigenvalues $0, 3$ and $5$. On the top panel, we see the eigenvalues of $A$ when the sparsity parameter is $d=3$, with one outlier above the threshold $\thresh \approx 3.317$. Below is the spectrum of the non-backtracking matrix $B$ defined in this paragraph; there are two outliers close to $3$ and $5$, above the threshold $\bar{\thresh } \approx 2.23$. Note that the second panel depicts around $6000$ points (the average size of $B$ with $n=1000$ and $d=3$). }\label{fig:illustration_nb}
\end{figure}

\section{Detailed results: rectangular matrices}
\label{sec:rectangular}

\newcommand{\nn}{\tilde{n}}
\newcommand{\dd}{\tilde{d}}
\newcommand{\AAA}{\tilde{A}}
\newcommand{\MM}{\tilde{M}}
\newcommand{\PPP}{\tilde{P}}
\newcommand{\QQ}{\tilde{Q}}
\newcommand{\rhoo}{\tilde{\rho}}
\newcommand{\LL}{\tilde{L}}
\newcommand{\threshh}{\tilde{\upvartheta}}
\newcommand{\rr}{\tilde{r}}
\newcommand{\ttau}{\tilde{\tau}}
\newcommand{\Gammaa}{\Gamma}

We now go back to our original problem, where $P$ is a rectangular $m \times n$ matrix. Without loss of generality, we will assume $m \leqslant n$ and we introduce the parameter $\alpha \in (0,1]$ defined as 
\begin{equation}
\alpha = \frac{m}{n}.
\end{equation}
Let $M$ be a $m \times n$ matrix whose entries are i.i.d. Bernoulli with parameter $d/n$: 
\[\PP(M_{x,y} = 1) = 1- \PP(M_{x,y}=0)=\frac{d}{n}. \]
As before, the non-zero entries of $M$ correspond to the entries of $P$ that are observed. 

For convenience, most proofs in this section are deferred to Section \ref{sec:proofs:rect} at page \pageref{sec:proofs:rect}. 

\subsection{Setting and strategy: reducing non-symmetric to symmetric}

We first describe a useful strategy to use our results for symmetric problems $P$ and transfer them to the non-symmetric world. 

We start by introducing an auxiliary $(n+m) \times (n+m)$ random matrix $Z$ whose entries are Bernoulli random variables with the following distribution: for $x, y \in [n+m]$, 
\begin{align}\label{def:Z}
&\PP(Z_{x,y}=1,Z_{y,x}=1) = q & (x\neq y)\\
& \PP(Z_{x,y}= 0, Z_{y,x}=1 ) =\PP(Z_{x,y}= 1, Z_{y,x}=0 ) =\frac{1-q}{2} & (x\neq y)\\
&\PP(Z_{x,x}=1)=(1+q)/2
\end{align}
where $q \in (0,1)$ is a parameter which must satisfy the following identity:
\begin{equation}\label{pq_choice}
q=\frac{d}{n}\left( \frac{1+q}{2}\right)^2.
\end{equation}
It is easy to check that such a $q$ exists. It is given by $2d(1-\sqrt{1-d/n})/n -1 \approx d/4n$. We finally define a $(n+m)\times (n+m)$ matrix with zero-one entries by
\begin{equation}
\MM \defeq Z \odot \begin{pmatrix}
0 & M \\ M^*  & 0
\end{pmatrix}.
\end{equation}

\begin{lem}\label{lem:Ztilde}
If $q$ satisfies \eqref{pq_choice} then the entries of $\MM$ are independent Bernoulli random variables with parameter
\[\frac{2q}{1+q} = (1+o(1))\frac{d}{2n}. \]
\end{lem}
 The proof is at Subsection \ref{proof:Ztilde}. The mask $\MM$ is thus an i.i.d. Bernoulli matrix with parameter $2q/(1+q)=d/2n+o(1/n) =: \dd /\nn$ with $\nn \defeq n+m$. The parameter $\dd$ is close to $(1+\alpha)d/2$, hence we have $\nn/\dd \approx 2n/d$. We can now apply our results from the first section, especially Theorem \ref{thm:1}, to our new estimator $\AAA$ which we define now. First, we note $\PPP$ the Hermitization of $P$:
\begin{equation}\label{eq:HermP}
\PPP=\begin{pmatrix}
0 & P \\P^* & 0
\end{pmatrix}.
\end{equation}
The link between $P$ and $\PPP$ (especially between the spectral decomposition of $\PPP$ and the SVD of $P$) is well-known in the literature ; we recall it at Subsection \ref{subsec:Girko}.  Our estimator is simply going to be 
\begin{equation}\label{def:AAA}
\AAA \defeq \left( \frac{2n}{d}\right) \MM \odot \PPP. 
\end{equation}
It is a block matrix with the following form:
\begin{equation}\label{def:blocks}  \AAA=\begin{pmatrix}
0 & A_1 \\ A_2^* & 0
\end{pmatrix}\end{equation}
where $A_1,A_2$ are $m \times n$ real matrices. Note that all the information contained in the original problem is kept intact: each revealed entry of $P$ is present at least once (maybe twice) in this new estimator $\AAA$.

\subsection{Results}
Just as before, we first gather the main definitions involved in our result. We emphasize the differences with the quantities in the preceding sections with a tilde. 

\begin{defins}[complexity parameters of $P$]
Let $P=\sum \sigma_i \zeta_i \xi_i^*$ be an $m \times n$ matrix and $\PPP$ is Hermitization as in \eqref{eq:HermP}. The amplitude, stable rank and incoherence describe the complexity of the matrix $P$.
\begin{enumerate}
\item \textbf{Size}: $\nn= m+n$.
\item \textbf{Amplitude parameter} $\LL$: \begin{equation}
\LL = \nn \max_{x,y}|P_{x,y} | = (1+\alpha)L.
\end{equation}
Equivalently, it is the scaled $L^1$ to $L^\infty$ norm of $\PP$.
\item \textbf{Stable numerical rank} $r$: $$
r = \frac{\| P \|_F^2 }{\| P \|^2} = \frac{ \sum_{k=1}^n \mu_k^2 }{\mu_1^2}.
$$
\item \textbf{Incoherence parameter} $b$: any scalar $b \geq 1$ such that for every $k$ in $ [n]$ we have
\begin{equation}
\max_{x \in [n]} |\xi_k(x)|, |\zeta_k(x)| \leq  \frac{b}{\sqrt{n}}.
\end{equation}
\end{enumerate}
\end{defins}

\begin{defins}[detection parameters of $P$ and $d$]Let $\PPP$ be as in \eqref{eq:HermP} and $d>1$ be a real number. The detection threshold, rank and gap describe what parts of $\PPP$ (and $P$) can be detected and how easily. 
\begin{enumerate}
\item \textbf{Variance matrix} $\QQ$: \begin{equation}\label{def:QQ}
\QQ_{x,y} = \nn |\PPP_{x,y}|^2 \qquad \qquad  \rhoo = \Vert \QQ \Vert.
\end{equation}
\item \textbf{Detection threshold} $\threshh$: any number $\threshh$ such that 
 \begin{equation}\label{def:thresh''}
\threshh \geq \max \{\threshh_1,  \threshh_2 \},
\end{equation}
where the `theta parameters' are defined by
\begin{equation}\label{def:thresh'''}
\threshh_2 = \sqrt{\frac{\rhoo}{\dd}} \quad \text{ and } \quad \threshh_1=\frac{L}{\dd}.
\end{equation}
\item \textbf{Detection rank} $\rr_0$:  number of singular values of $P$ which are strictly larger than  $\threshh$, i.e.
\begin{equation}
\sigma_1  \geqslant \cdots \geqslant \sigma_{\rr_0}  > \threshh \geqslant   \sigma_{\rr_0+1}  \geqslant \cdots \geqslant \sigma_n.
\end{equation}
\item \textbf{Detection hardness or gap} $\ttau_0$: \begin{equation}
\frac{\threshh}{\sigma_{r_0}}  = \ttau_0 \in (0,1).
\end{equation}
\end{enumerate}
\end{defins}

Let us make a few remarks on these definitions.
\begin{itemize}
\item It is clear that if $Q$ is the non-square matrix defined by $Q_{x,y}=n|P_{x,y}|^2$, and $\rho = n\Vert Q \Vert$, thus corresponding to the base problem, then we have $\rhoo=(1+\alpha)\rho$. 
The thresholds in \eqref{def:thresh'''} thus satisfy
\begin{equation*}
\threshh_2 = \sqrt{\frac{\rhoo}{\dd}} \approx \sqrt{\frac{2\rho}{d}} \quad \text{ and } \quad \threshh_1=\frac{\LL}{\dd} \approx \frac{2L}{d}.
\end{equation*}
\item We know (see the link between $P$ and $\PPP$ at Subsection \ref{subsec:Girko}, and more precisely \eqref{def:girkovec}) that $|\varphi_k|_\infty \leqslant b/\sqrt{2}\leqslant b$.  
\item The number of eigenvalues of $\PPP$ with modulus greater than $\threshh$ is $2\rr_0$. Each singular value $\sigma_i$ gives rise to two eigenvalues of $\PPP$ at $+\sigma_i$ and $-\sigma_i$, as recalled in Subsection \ref{subsec:Girko}. 
\item It is easy to see that the stable   rank $\rr$ of $\PPP$ is equal to $2r$, where $r$ is the stable   rank of $P$. 
\end{itemize}

As in the square case, we need to define `theoretical correlations' between eigenvectors. They now depend on more parameters than before.

\begin{defin}[theoretical covariances]For $i,j \in [\rr_0]$ and for signs $\circ, \square$, the real numbers $\Gamma^{\circ, \square}_{i,j}$ are defined by 
\begin{equation}\label{def:rect:correlation}
\Gammaa^{\square, \circ}_{i,j} \defeq \sum_{s=0}^\ell \frac{\langle \mathbf{1},  \QQ^s \varphi^\square_i \odot \varphi^\circ_j \rangle}{(\square \sigma_i)(\circ \sigma_j) d)^s}
\end{equation}
where $\ell$ is the integer defined in \eqref{def:rect:l} thereafter. \end{defin}
We will also use $\gamma_i$ as a shorthand for $\Gamma^{+,+}_{i,i}=\Gamma^{-,-}_{i,i}$. It is easily checked to be bigger than $1$; however, there is no particular reason for $\Gamma^{+,-}_{i,j}$ to be bigger than $1$ or even positive in general, and indeed one can verify that $\lim_{d \to \infty} \Gamma^{+,-}_{i,j} = -\delta_{i,j}/2$. 

We are now ready to state ou main theorem for rectangular matrices: it directly follows from an application of Theorem \ref{thm:1} in our setting.

\begin{theorem}\label{thm:1-rectangular}Let $P$ and $\AAA$ as defined earlier. 
We define $D = \max(2\dd,1.01)$ and 
\begin{equation}\label{def:rect:l}
\ell =\left\lfloor (1/8) \log_{D} (\nn)\right\rfloor. 
\end{equation}
There exists a universal constant $c \geq 1$ such that if  the inequality 
\begin{equation}
C_0 \defeq c  \rr   \rr_0^{4} b^{44}  \ln(\nn)^{16}  \leq  \ttau_0^{-\ell},
\end{equation}
holds true then  with probability greater than $1-cn^{-1/4}$, the following event occurs:
\bigskip

1) Eigenvalues. There exists an ordering $\lambda_1, \ldots, \lambda_{r_0}$ of the $\rr_0$ eigenvalues of $\AAA$ with greater modulus and positive real part,  such that for all $i \in[\rr_0]$,
\begin{equation}
 |\lambda_i - \sigma_{i}|\leqslant C_0 \left| \frac{ \threshh }{\sigma_{i}}\right|^\ell |\sigma_{i}|,
\end{equation}
and these $\lambda_i$ are real. All the other eigenvalues with positive real part have modulus smaller than $C_0^{1/\ell}\threshh$. 

\bigskip

2) Eigenvectors. We denote by $\psi_i^\pm$ the unit right eigenvectors of $\pm \lambda_i$.  The relative spectral gap ratio at $\sigma_i$ is defined as
\begin{equation}
\label{hyp:spectral_sep'}
\ttau_{i,\ell} =  1 - \min_{j \in [n] \setminus \{i\}} | 1 - (\sigma_j/\sigma_i)^\ell |.
\end{equation}
Then, for every $i \in [\rr_0]$ and $j \in [\rr_0]$, and for every signs $\circ, \square \in \{+,-\}$, one has
\begin{equation}\label{eigenvector_errorbound+-}
\left| |\langle \psi_i^\circ , \varphi^\square_j \rangle| - \frac{\delta_{\circ, \square}\delta_{i,j}}{\sqrt{\gamma_i}}  \right| \leqslant \frac{C_0 \ttau_0^{\ell}}{1 - \ttau_{i,\ell} } \end{equation}
and
\begin{equation}\label{eq:corr_signs}
\left| |\langle \psi_i^\circ , \psi^\square_j \rangle| - \frac{|\Gamma^{\circ, \square}_{i,j}|}{\sqrt{\gamma_i \gamma_j}}  \right| \leqslant \frac{C_0 \ttau_0^{\ell}}{\sqrt{(1 - \ttau_{i,\ell})(1-\ttau_{j,\ell})} } \end{equation}

Finally, if $\psi^\pm_{i, {\rm left}}$ denotes the unit left eigenvector associated with $\pm \lambda_i$ with the convention \eqref{eq:choicephase}, then
\begin{align}\label{227}
&\left|\langle\psi^\square_i, \psi_{j,{\rm left}}^\circ \rangle  - \frac{\delta_{\square, \circ}\delta_{i,j}}{\sqrt{\gamma_i \gamma_j}} \right| \leqslant \frac{C_0 \ttau_0^{\ell}}{1 - \ttau_{i,\ell} }.
\end{align}
\end{theorem}

The main difference between this theorem and the original theorem for symmetric matrices lies in the threshold $\threshh$ in \eqref{def:thresh''}-\eqref{def:thresh'''}. With our method, the singular values of $P$ are detected only above the new threshold
\[\max \left\lbrace \frac{2L}{d}, \sqrt{\frac{2\rho}{d}} \right\rbrace\]
which is strictly bigger than $\max\{L/d, \sqrt{\rho/d}\}$, the original threshold.

\subsection{Smaller, square matrices}

The matrix $\AAA$ is a square matrix with size $m + n$, which is bigger than $m$ by a factor $1+\alpha^{-1}$. This can result in a higher computational cost, but we can do better and restrict ourselves to smaller matrices. Starting from the block decomposition \eqref{def:blocks}, we can use the alternative matrices 
\begin{equation}\label{def:XY}
X\defeq   A_1 A_2^* \qquad Y \defeq A_2^* A_1
\end{equation}
which are square matrices of respective sizes $m$ and $n$. The following example shows in details how to obtain them directly from the observations.

\begin{exemple}[From raw data to $X$ and $Y$]Suppose that the matrix where we store the observed entries is 
\[ \begin{pmatrix}
0 & 2 & 0 & 4 \\
1 & 0 & 0 & 4 
\end{pmatrix}\]
the zeros meaning that the corresponding entry is not observed. The steps to form the matrices $X$ and $Y$ are as follows: first, for each revealed entry, flip a coin as in Lemma \ref{lem:Ztilde} and put the entry right or left:
\[ \begin{pmatrix}
0 & 2 & 0 & 4 \\
1 & 0 & 0 & 4 
\end{pmatrix}=\begin{pmatrix}
0 & 2 & 0 & 0 \\
0 & 0 & 0 & 4 
\end{pmatrix}+\begin{pmatrix}
0 & 0 & 0 & 4 \\
1 & 0 & 0 & 0 
\end{pmatrix}\]
Second, normalize by $2n/d$ to get $A_1$ and $A_2^*$:
\begin{align*}
&A_1 = \frac{2n}{d}\begin{pmatrix}
0 & 2 & 0 & 0 \\
0 & 0 & 0 & 4 
\end{pmatrix} &&A_2 = \frac{2n}{d}\begin{pmatrix}
0 & 0 & 0 & 4 \\
1 & 0 & 0 & 0 
\end{pmatrix}
\end{align*}
Finally, multiply them to get the square matrices $X$ and $Y$: for $X$, which has size $m=2$, 
\[X = A_1A_2^* = \left( \frac{2n}{d}\right)^2 \times \begin{pmatrix}
0 & 2 & 0 & 0 \\
0 & 0 & 0 & 4 
\end{pmatrix} \times \begin{pmatrix}
0 & 1 \\ 0&0\\0&0\\4&0
\end{pmatrix}=  \left( \frac{2n}{d}\right)^2 \times \begin{pmatrix}
0 & 0 \\16 & 0
\end{pmatrix}  \]
and for $Y$ which has size $n=4$:
\[Y = A_2^*A_1 = \left( \frac{2n}{d}\right)^2\times \begin{pmatrix}
0 & 1 \\ 0&0\\0&0\\4&0
\end{pmatrix} \times \begin{pmatrix}
0 & 2 & 0 & 0 \\
0 & 0 & 0 & 4 
\end{pmatrix}  =  \left( \frac{2n}{d}\right)^2 \times \begin{pmatrix}
0 & 0 & 0 & 4 \\
0 & 0 & 0 & 0 \\
0 & 0 & 0 & 0 \\
0 & 8 & 0 & 0
\end{pmatrix}.  \]

\end{exemple}

The elementary properties of those matrices are gathered in the following lemma, whose proof is a mere verification.

 \begin{lem}[structure of $\AAA$]\label{lem:AAAsymmetry}
Let $\lambda$ be a nonzero eigenvalue of $\AAA$ associated with a unit right-eigenvector 
\[\begin{pmatrix}
u \\ v 
\end{pmatrix}. \]
Then, $-\lambda$ is an eigenvalue of $\AAA$, and a unit right-eigenvector is given by
\begin{equation}\label{lem:def:uv}\begin{pmatrix}
-u \\ v 
\end{pmatrix}. \end{equation}
If $X\defeq   A_1 A_2^* $ and $ Y \defeq A_2^* A_1$, then we have $\AAA^2 = \mathrm{diag}(X,Y)$. Moreover, 
\begin{equation}
\Spec(X)=\Spec(Y) = \{\lambda^2 : \lambda \in \Spec(\AAA) \}.
\end{equation}
If a nonzero eigenvalue $\lambda$ has multiplicity $m_\lambda$ in $\AAA$, then $\lambda^2$ has also multiplicity $m_\lambda$ in $X$ and in $Y$. If \eqref{lem:def:uv} was a unit right-eigenvector of $\lambda$, then $u \neq 0$ and $u/|u|$ is a unit right-eigenvector of $X$ associated with $\lambda^2$. Similarly, $v \neq 0$ and $v/|v|$ is a unit right-eigenvector of $Y$ associated with $\lambda^2$. 
 \end{lem}

We can thus detect the \emph{squares} of the singular values of $P$, and the singular values themselves, by only looking at $X$ or $Y$. Equivalently, we can estimate the singular vectors $\zeta_i, \xi_i$ by only looking at the eigenvectors of $X$ or $Y$. To do this, fix $i \in [\rr_0]$. We will note $\chi_i$ a unit right-eigenvector of $X$ associated with the eigenvalue $\lambda_i^2$, and $\chi'_i$ a left eigenvector (recall convention \eqref{eq:choicephase}). Similarly, we will note $\pi_i, \pi'_i$ for the eigenvectors of $Y$. 
\newcommand{\triangleu}{\bigtriangledown}
We will need a variation of the quantities $\gamma_i$. We define
\begin{align}\label{def:triangle}
\zeta_i \triangle \zeta_j \defeq \begin{pmatrix}
\zeta_i \odot \zeta_j \\ 0
\end{pmatrix} &&\xi_i \triangleu \xi_j \defeq \begin{pmatrix}
0 \\ \xi_i \odot \xi_j
\end{pmatrix}.
\end{align}
Equivalently, $\zeta_i \triangle \zeta_j = \varphi^+_i \odot \varphi^+_j - \varphi^+_i \odot \varphi^-_j$  and $\xi_i \triangleu \xi_j = \varphi^+_i \odot \varphi^+_j + \varphi^+_i \odot \varphi^-_j$. Then, we define:
\begin{align}\label{def:gammatri}
&\Gamma_{i,j}^\triangle \defeq \sum_{s=0}^\ell \frac{\langle \mathbf{1}, \QQ\zeta_i \triangle \zeta_j \rangle}{(\sigma_i^2\dd)^s} && \Gamma_{i,j}^\triangleu \defeq \sum_{s=0}^\ell \frac{\langle \mathbf{1}, \QQ\xi_i \triangleu \xi_j \rangle}{(\sigma_i^2\dd)^s}
\end{align}
and we set $\gamma^\triangle_i = \Gamma^\triangle_{i,i}$ and $ \gamma^\triangleu_i = \Gamma^\triangleu_{i,i}$.
It is straightforward to check that $\gamma_{i}^\triangle+\gamma_{i}^\triangleu=2\gamma_i$ and that 
\begin{align}&\Gamma^\triangle_{i,j}= \Gamma_{i,j}^{+,+} - \Gamma_{i,j}^{+,-}&&\Gamma^\triangleu_{i,j}= \Gamma_{i,j}^{+,+} + \Gamma_{i,j}^{+,-}. \end{align}

\begin{theorem}\label{thm:smallsquare}
Let, $X,Y$ be the matrices defined in \eqref{def:XY}. We place ourselves under the event of Theorem \ref{thm:1-rectangular}. Then,  there exists an ordering $\nu_1, \dotsc, \nu_{\rr_0}$ of the $\rr_0$ eigenvalues with greater modulus of $X$, such that 
\begin{equation}
|\sqrt{\nu_i} - \sigma_i| \leqslant  \frac{C_0 \ttau_0^{\ell}}{1 - \ttau_{i,\ell} }
\end{equation}
Let $\chi_i$ and $\chi'_i$ be two unit right and left eigenvectors associated with $\nu_i$ with positive scalar product, then 
\begin{align}\label{2231}
&\left| |\langle \chi_i, \zeta_j \rangle|-\frac{\delta_{i,j}}{\sqrt{\gamma_{i}^\triangle}}  \right|\leqslant  \frac{C_0 \ttau_0^{\ell}}{1 - \ttau_{i,\ell} } &&\left| |\langle \chi'_i, \zeta_j \rangle|-\frac{\delta_{i,j}}{\sqrt{\gamma_{i}^\triangleu}}  \right|\leqslant  \frac{C_0 \ttau_0^{\ell}}{1 - \ttau_{i,\ell} }
\end{align}
and
\begin{align}
&\left| |\langle \chi_i, \chi_j \rangle|- \frac{\Gamma^{\triangle}_{i,j}}{\sqrt{\gamma_{i}^\triangle\gamma_{j}^\triangle}}  \right|\leqslant   \frac{C_0 \ttau_0^{\ell}}{\sqrt{(1 - \ttau_{i,\ell})(1-\ttau_{j,\ell})} } \\ &\left| \langle \chi_i, \chi'_j \rangle -\frac{\delta_{i,j}}{\gamma_{i}^\triangle}  \right|\leqslant \frac{C_0 \ttau_0^{\ell}}{\sqrt{(1 - \ttau_{i,\ell})(1-\ttau_{j,\ell})} }  .
\end{align}
A similar statement holds for $Y$, its eigenvalues $\nu_i$ and its eigenvectors $\pi_i$, in particular:
\begin{align*}
&\left| |\langle \pi_i, \pi_j \rangle|- \frac{\Gamma^{\triangleu}_{i,j}}{\sqrt{\gamma_{i}^\triangleu\gamma_{j}^\triangleu}}  \right|\leqslant \frac{C_0 \ttau_0^{\ell}}{\sqrt{(1 - \ttau_{i,\ell})(1-\ttau_{j,\ell})} }  \\
&\left| \langle \pi_i, \pi'_j \rangle-\frac{\delta_{i,j}}{\gamma_{i}^\triangleu}  \right|\leqslant \frac{C_0 \ttau_0^{\ell}}{\sqrt{(1 - \ttau_{i,\ell})(1-\ttau_{j,\ell})} } .
\end{align*}
Finally, if the orientation of eigenvectors is chosen so that $\langle \chi_i, \chi'_i\rangle \geqslant 0$, then $\langle \chi_i, \zeta_i\rangle$ and $\langle \chi'_i, \zeta_i\rangle$ have the same sign. 
\end{theorem}

The proof is in Section \ref{sec:proofs:rect} at page \pageref{sec:proofs:rect} and the result is illustrated at Figure \ref{fig:illustration_XY}.

The preceding theorem completely describes the behaviour of the most informative parts in the spectral decomposition of the `smaller matrices' $A_2^* A_1$ and $A_1^* A_2$. They will also be used later in the design of `optimal' estimators of the hidden matrix $P$. The theoretical covariances $\Gamma^\triangle_{i,j}$ can be difficult to compute in general, however we will see in the following sections that 
\begin{enumerate}[(i)]
\item when the rank of $P$ is one, which is in itself an important example in applications, all the computations can explicitly be done (see Section \ref{sec:rankone} after); 
\item When they cannot explicitly be computed, then can consistently be estimated by the mere results of Theorem \ref{thm:smallsquare}. This will be done in Section \ref{sec:MC}.
\end{enumerate}

\begin{exemple}\label{ex:XY}
We took $P$ to be a $1000 \times 5000$ matrix with SVD $P = 5\zeta_1\xi_1+3\zeta_2\xi_2$, the singular vectors being taken uniformly at random over the unit sphere. With this matrix, we have $\rho \approx 170$ and the threshold is  given by
\begin{equation}
\threshh_2 = \sqrt{\frac{2\rho}{d}} \approx \frac{18.47}{\sqrt{d}}.
\end{equation}
The singular value $3$ will only be detected if $3> \threshh_2$ or equivalently if $d> (18.47/3)^2 \approx 37.91$. With $d=50$, the two outliers $\nu_1, \nu_2$ of $X$ clearly appear close to the locations $\sigma_1^2=5^2=25$ and $\sigma_2^2=3^2=9$, same thing for $Y$. 
\end{exemple}

\begin{figure}\centering
\includegraphics[width=0.99\textwidth]{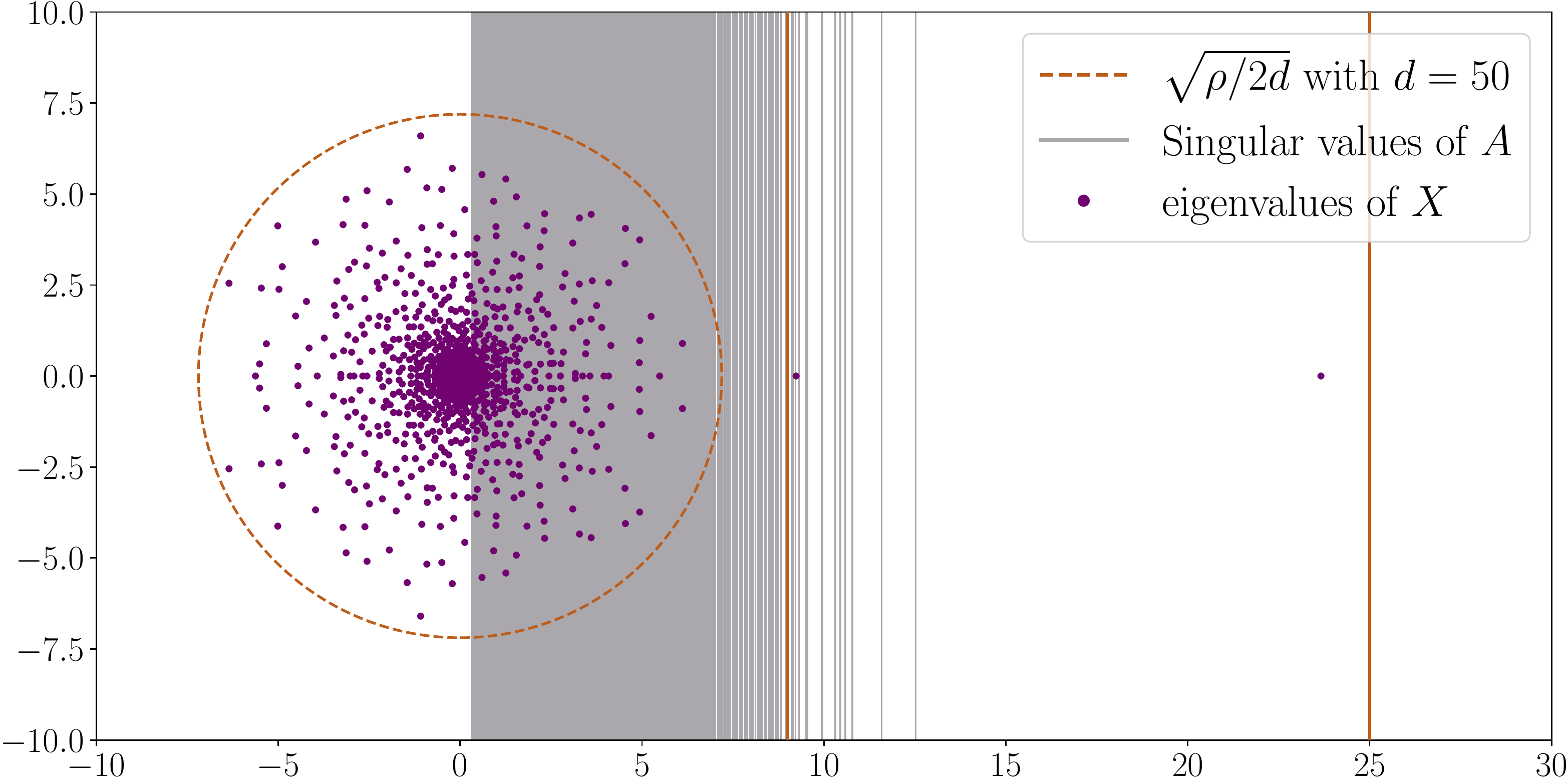}

\caption{Here we chose $P$ as in Example \ref{ex:XY} with singular values $5$ and $3$. The singular values of the masked matrix, $A$, are depicted with the gray lines. They are uninformative, contrary to our matrices $X$ and $Y$ defined in \eqref{def:XY}. The spectra of $X$ and $Y$ are identical except for the multiplicity of the value $0$, and is depicted in the picture. With $d=50$ the two outliers can be seen close to $\sigma_1^2=25$ and $\sigma_2^2=9$. }\label{fig:illustration_XY}
\end{figure}

\subsection{Statistical estimation aspects}
\label{subsec:singvecest}

We will now use the results of the preceding section for the statistical estimation of the singular vectors of $P$. Fix some $i \in [\rr_0]$. In the spectrum of $X$, there is one eigenvalue $\nu_i$ close to $\sigma_i^2$. It gives rise to two unit eigenvectors, $\chi_i$ on the right and $\chi'_i$ on the left, which contain information on $\zeta_i$. Similarly, the unit eigenvectors $\pi_i, \pi'_i$ of $Y$ associated with $\eta_i$ contain information on $\xi_i$.

We present two estimators for $\zeta_i$ and two for $\xi_i$: using one eigenvector without any modification (`simple'), or averaging the two left and right eigenvectors:

\newcommand{\zsim}{\hat{\zeta}^{\rm sim}}
\newcommand{\ztot}{\hat{\zeta}^{\rm tot}}
\newcommand{\zavg}{\hat{\zeta}^{\rm avg}}
\newcommand{\xsim}{\hat{\xi}^{\rm sim}}
\newcommand{\xtot}{\hat{\xi}^{\rm tot}}
\newcommand{\xavg}{\hat{\xi}^{\rm avg}}
\newcommand{\zgen}{\hat{\zeta}^{\#}}
\newcommand{\xgen}{\hat{\xi}^{\#}}
\newcommand{\cgenX}{{\mathfrak{C}}_1^{\#}}
\newcommand{\cgenY}{{\mathfrak{C}}_2^{\#}}
\newcommand{\csimX}{{\mathfrak{C}}_1^{{\rm sim}}}
\newcommand{\cavgX}{{\mathfrak{C}}^{1,{\rm avg}}}
\newcommand{\ctotX}{{\mathfrak{C}}^{1,{\rm tot}}}
\newcommand{\csimY}{{\mathfrak{C}}^{2,{\rm sim}}}
\newcommand{\cavgY}{{\mathfrak{C}}^{2,{\rm avg}}}
\newcommand{\ctotY}{{\mathfrak{C}}^{2,{\rm tot}}}

\begin{align*}
&\zsim_i = \chi_i &&\xsim_i = \pi_i\\
&\zavg_i = \frac{\chi_i+\chi'_i}{|\chi_i + \chi'_i|}&&\xavg_i = \frac{\pi_i+\pi'_i}{|\pi_i + \pi'_i|}.
\end{align*}

The following theorem gives the full correlation between these estimators themselves, as well as their performance at estimating $\zeta_i, \xi_j$. First, we define 
\begin{align*}
&c_{1,i}^{\rm sim} =  \frac{1}{\sqrt{\gamma_{i}^\triangle}}&& c_{2,i}^{\rm sim} = \frac{1}{\sqrt{\gamma_{i}^\triangleu}} \\
&c_{1,i}^{\rm avg} =   \sqrt{\frac{2}{\gamma_i^\triangle +1}}&&c_{2,i}^{\rm avg} =\sqrt{\frac{2}{\gamma_i^\triangleu +1}}.
\end{align*}

\begin{theorem}[statistical estimators]\label{thm:stats}We place ourselves on the high-probability event of Theorem \ref{thm:1-rectangular}. Let $i,j \in [\rr_0]$. Then, we have 
\begin{align*}
& |\langle \zsim_i, \zeta_j \rangle |  =(1+o(1)) c_{1,i}^{\rm sim}, && | \langle \zavg_i, \zeta_j \rangle|  =(1+o(1)) c_{1,i}^{\rm avg} \\
&|\langle \xsim_i, \xi_j \rangle | =(1+o(1)) c_{2,i}^{\rm sim}, && |\langle \xavg_i, \zeta_j \rangle | =(1+o(1)) c_{2,i}^{\rm avg} \\
\end{align*}

Moreover, these estimators satisfy:
\begin{align}
&\langle \zsim_i, \zsim_j\rangle =(1+o(1))\frac{\Gamma^\triangle_{i,j}}{\sqrt{\gamma_i^\triangle \gamma_j^\triangle}}\label{236stats}\\
&\langle \zavg_i, \zavg_j\rangle =(1+o(1))\frac{\Gamma^\triangle_{i,j}+ \delta_{i,j}}{\sqrt{(\gamma_i^\triangle+1)(\gamma_j^\triangle+1)}}. \label{237stats}
\end{align}
\end{theorem}

The proof is in Section \ref{sec:proof:rect:stats} at page \pageref{sec:proof:rect:stats}.
One can observe that $c^\#_{k,i}$ goes to $1$ when $d \to \infty$, indicating that in the high-degree regime where $d$ is high, total reconstruction is nearly achieved.  Moreover, it is easy to see, using Definition \eqref{def:rect:correlation}, that 
\begin{align} \label{eq:estimation_inequalities}
&c_{1,i}^{\rm sim} \leqslant c_{1,i}^{\rm avg}  \leqslant 1 & {\rm and}& &&c_{2,i}^{\rm sim} \leqslant c_{2,i}^{\rm avg}  \leqslant 1 .
\end{align}
The elementary identity $|\zgen - \zeta|^2 = 2(1- \langle \zgen, \zeta \rangle) \approx 2(1- c_{1,i}^\#)$ shows that the closer $c^\#_{1,i}$ is to $1$, the better the estimator. The meaning of the inequalities in \eqref{eq:estimation_inequalities} is that ${\rm avg}$ estimator is better than the other, as it incorporates more spectral information.

\subsection{Non-backtracking matrices}
For better performances at a higher computational cost, it is naturally also possible to use the non-backtracking matrix in conjunction with the Hermitization $\PPP$ of the matrix $P$ defined in \eqref{eq:HermP}. The mask matrix is then 
\begin{equation*}
\bar M = \begin{pmatrix}
0 & M \\M^* & 0
\end{pmatrix}.
\end{equation*}
We can thus define the weighted non-backtracking matrix $B$ associated to this mask matrix with weights $\PPP$ as in Subsection \ref{subsec:NB}. Applying directly Theorem \ref{thm:1nb}, we then obtain a non-backtracking version of Theorem \ref{thm:1-rectangular}. This can be used to do statistical estimation of the singular vectors as in Subsection \ref{subsec:singvecest}. To avoid too much repetitions, we leave the details of the statements to the reader since they follow from exactly the same considerations than above. 

\section{Application to matrix completion}
\label{sec:MC}

\newcommand{\estwopt}{\hat{w}^{\rm opt}}

\newcommand{\wopt}{w^{\rm opt}}
\newcommand{\MSE}{\mathrm{MSE}}
\newcommand{\Pest}{\hat{P}}

Let us place ourselves in the general, rectangular case, where the rectangular matrix
\[P = \sum_{i=1}^r \sigma_i \zeta_i \xi_i^* \]
satisfies the suitable incoherence and rank assumptions from the preceding sections. We want to find back $P$ from the observation of its masked version $P \odot M$, where the probability of uncovering each entry is $d/n$ for a fixed $d$. Clearly, our theoretical results only allow to recover the singular values and vectors with $i \in [\rr_0]$. We will note
\[P_0 \defeq \sum_{i=1}^{\rr_0} \sigma_i \zeta_i\xi_i^* \]
the part of the matrix $P$ which can be recovered. It is clear from the previous results that $P=P_0$ if $\sigma_{\min}<\threshh$.

\subsection{Mean squared error optimal matrix recovery}

Suppose that we dispose of estimators $\hat{\zeta_i}$ and $\hat{\xi_i}$ of the singular vectors $\zeta_i, \xi_i$ --- we do not specify what they are for the moment. Then, we can try to estimate $P_0$ by a matrix $\hat{P}$ which can be written $\sum_{i=1}^{\rr_0} w_i \hat{\zeta_i}\hat{\xi_i}^*$. This amounts to solving the optimisation problem (see \cite{raj}): 
\begin{equation}
\wopt = (\wopt_1, \dotsc, \wopt_{\rr_0}) = \arg \min_{w_1, \dotsc, w_{\rr_0}>0} \left\Vert P_0 - \sum_{i=1}^{\rr_0} w_i \hat{\zeta_i}\hat{\xi_i}^*\right\Vert_F.
\end{equation}
This problem can be solved using elementary analysis, the solution being
\begin{equation}\label{eq:wopt}
\wopt_i = \left(\mathrm{Re} \left[ \sum_{j=1}^{\rr_0}  \sigma_j \langle \zeta_j, \hat{\zeta_i}\rangle \langle \xi_j, \hat{\xi_i}\rangle \right] \right)_+, 
\end{equation}
see \cite{raj}, Theorem 2.1, statement a) and the proof therein. In order to achieve small mean-square error in this sense, one must dispose of efficient estimators in the sense that they have to be strongly correlated with the original eigenvectors; we need $\langle \zeta_i, \hat{\zeta_i}\rangle$ as close to $1$ as possible. If we use the estimators from the preceding section, namely $\zsim,\ztot$, we obtain different behaviours for the corresponding optimal $\wopt_i$, which will be named $w^{\rm sim}_i, w^{\rm avg}_i$. From \eqref{eq:wopt}  and Theorem \ref{thm:stats} we thus get
\begin{equation}\label{eq:wopt-specific}
w^{\#}_i = (1+o(1))  \sigma_i c_{i,1}^\#c_{i,2}^\# \qquad (\# =\rm sim, avg  )
\end{equation}
with high probability. The asymptotic expressions are 
\begin{align}
&w^{\rm sim}_i=\frac{(1+o(1))\sigma_i }{\sqrt{\gamma_{i}^\triangle \gamma_{i}^\triangle}} \\
&w^{\rm avg}_i= \frac{(1+o(1))2\sigma_i}{\sqrt{(\gamma_{i}^\triangle+1)(\gamma_{i}^\triangleu+1)}}.
\end{align}

Let us note $\Pest^\#$ the matrix obtained with this method: $\Vert P_0 - \Pest^\#\Vert_F$ is the \emph{optimal} mean-square error $\MSE_\star^\#$ we can get with our estimators using $\# \in \{\rm sim, avg\}$, and it is given by
\begin{equation}
\MSE_\star^\# = \Vert P_0 - \Pest^\# \Vert^2_F.
\end{equation}

In general, it is not possible to use \eqref{eq:wopt-specific} and the subsequent explicit expressions, because the formulas for $w^\#_i$ are not statistics, they depend through $c_{i,1}^\#$ on hidden quantities contained in $P$, namely the $\gamma_i^\triangle$ and $\sigma_i$. However, we can efficiently estimate these quantities. First, our analysis provides a number of ways to estimate $\sigma_i$, the simplest being to simply set 
\begin{equation}\label{def:estimator_sigma}
\hat{\sigma}_i = \sqrt{\nu_i}.
\end{equation}

We now want to estimate $c^\#_{i,k}$ directly from the data, and a delightful consequence of Theorem \ref{thm:smallsquare} is that we can estimate these quantities directly from data. The key result here is that the inner product between the left-eigenvector $\chi_i$ and the right-eigenvector $\chi'_i$ is indeed asymptotically equal to the square of the inner product between $\chi_i$ and $\zeta_i$, namely $|\langle \chi_i, \chi_i'\rangle|\approx 1/\gamma_i^\triangle = (c_{i,1}^{\rm sim})^2$, thanks to equation \eqref{236stats}. By setting 
\begin{align}\label{def:estimator:csim}
&\widehat{c^{\rm sim}_{i,1} }\defeq \sqrt{|\langle \chi_i, \chi'_i\rangle|}&&\hat{c^{\rm sim}_{i,2} }\defeq \sqrt{|\langle \pi_i, \pi'_i\rangle|}
\end{align}
we have obtained consistent estimators of $c^{\rm sim}_{i,k}$. Similarly, 
\begin{align}\label{def:estimator:cavg}
&\widehat{c^{\rm avg}_{i,1}} \defeq \sqrt{\frac{2|\langle \chi_i, \chi_i'\rangle|}{1+|\langle \chi_i, \chi_i'\rangle|}}&&\widehat{c^{\rm avg}_{i,2} }\defeq \sqrt{\frac{2|\langle \chi_i, \chi_i'\rangle|}{1+|\langle \pi_i, \pi_i'\rangle|}}
\end{align}
are consistent estimators of $c^{\rm avg}_{i,k}$, as can be directly checked from Theorem \ref{thm:smallsquare}.

We can now replace the theoretical optimal quantity $w^\#_i$ by a quantity $\hat{w}^\#_i$ which is directly computable on the data. In practice this leads to the following two estimators:
\begin{align}
\hat{w}^{\rm sim}_i &\defeq \hat{\sigma}_i \widehat{c^{\rm sim}_{i,1}}\widehat{c^{\rm sim}_{i,2}} \label{what:sim} \\
\hat{w}^{\rm avg}_i &\defeq   \hat{\sigma}_i \widehat{c^{\rm avg}_{i,1}}\widehat{c^{\rm avg}_{i,2}}   \label{what:avg}
\end{align}
Note that with the definitions they can be written in greater detail as
\begin{align*}
&\hat{w}^{\rm sim}_i = \sqrt{\nu_i | \langle \chi_i,  \chi_i'\rangle \langle \pi_i,  \pi_i'\rangle|} &&\hat{w}^{\rm avg}_i =\sqrt{\frac{4\nu_i |\langle \chi_i,  \chi_{i}'\rangle \langle \pi_i,  \pi_{i}'\rangle|}{(1+|\langle \chi_i, \chi_{i}'\rangle|)(1+|\langle \pi_i, \pi_{i}'\rangle|)}}.
\end{align*}
\subsection{Procedure}\label{sec:procedure}

The methods described above require a few pre-processing of the problem. The starting data are the observed entries of $P$. Then, one has to generate the auxiliary matrix $Z$ defined earlier in \eqref{def:Z} , and form the new estimator $\AAA$ defined in \eqref{def:AAA}, or even better, to directly form the two matrices $X$ and $Y$ from the preceding corollary. The matrix $Z$ is not especially difficult to generate, but its role here is more theoretic because it allowed us to directly transfer our results from the symmetric setting to this new setting, the key here being that $Z \odot M$ is a Bernoulli matrix. 

However, in practice, we see that the probability that $Z_{x,y}=Z_{y,x}=1$ is proportional to $1/n$, hence extremely small, while the probability of having $Z_{x,y}=1$ and $Z_{y,x}=0$ or the other way round is indeed very close to $1/2$. 

For practical purposes, it is better to replace $Z$ with a matrix $Z'$ whose entries above the diagonal are i.i.d. Bernoulli with parameter $1/2$, and the entries below the diagonal are simply $Z'_{x,y}=1-Z'_{x,y}$. The procedure would then be as follows: 
\begin{enumerate}
\item Let  $T=P \odot M$ be the $m \times n$ observed matrix. 
\item Let $Z'$ be an $m \times n$ matrix with i.i.d. Bernoulli $1/2$ entries. We set $C_1= T \odot Z$ and $C_2 = T-X$. 
\item Our estimators are  
\begin{equation}\label{def:XYalt}X=\left( \frac{2n}{d}\right)^2 C_1C_2^* \qquad \qquad Y =\left( \frac{2n}{d}\right)^2  C_2^* C_1 \end{equation}
and the spectral statistics of $X,Y$ have the properties described in Theorem \ref{thm:smallsquare}.
\end{enumerate}

The algorithmic description of this method is described in Algorithm \ref{alg:new} at page \pageref{alg:new}.

\begin{algorithm}[t]
\centering
\caption{Statistically optimal matrix completion using asymmetric eigen-decomposition}

\label{alg:new}
\begin{algorithmic}[1]\label{alg}
\STATE Choose the estimation method $\# = {\rm sim, avg, tot }$.
\STATE Input: $T=  P \odot M$ the masked matrix. 
\STATE Input: $\hat{r}$ = estimate of the detection rank $\rr_0$. 
\STATE Generate $X,Y$.
\STATE Compute the $2\hat{r}$ highest eigenvalues $\nu_i$ of $X$, the associated right-eigenvectors $\chi_i$, the left eigenvectors $\chi_i'$. Select an orientation such that $\langle \chi_i, \chi'_i> \geqslant 0$.
\STATE Compute the $2\hat{r}$ highest eigenvalues $\eta_i$ of $Y$, the associated right-eigenvectors $\pi_i$, the left eigenvectors $\pi_i'$. Select an orientation such that $\langle \pi_i, \pi'_i> \geqslant 0$.
\FOR {$i = 1, \dotsc, \hat{r}$}
\STATE  Compute the empirical correlations between the eigenvectors.
\STATE Compute $\hat{w}_i^\#$ as in \eqref{what:sim}-\eqref{what:avg}.
\STATE Set 
\[\hat{\zeta}_i = \hat{\zeta}_i^{\#} \ANDalt \hat{\xi}_i = \hat{\xi}^\#_i .\]
\ENDFOR
\RETURN $P^\# =\sum_{i=1}^{\hat{r}} w_i^\# \hat{\zeta}_i^{\#} (\hat{\xi}_i^{\#})^*$.
\end{algorithmic}
\end{algorithm}

\subsection{Mean square errors}

The following proposition gives a theoretical expression for the mean square error. We introduce a notation: 
\begin{align*}
&(\mathfrak{C}_1^{\rm sim})_{i,j}= \frac{\Gamma_{i,j}^\triangle}{\sqrt{\gamma^\triangle_i\gamma^\triangle_i}} &&(\mathfrak{C}_2^{\rm sim})_{i,j}= \frac{\Gamma_{i,j}^\triangleu}{\sqrt{\gamma^\triangleu_i\gamma^\triangleu_i}} \\
&(\mathfrak{C}_1^{\rm avg})_{i,j}= \frac{\Gamma_{i,j}^\triangle+\delta_{i,j}}{\sqrt{(\gamma^\triangle_i+1)(\gamma^\triangle_i+1)}} &&(\mathfrak{C}_2^{\rm avg})_{i,j}= \frac{\Gamma_{i,j}^\triangleu+\delta_{i,j}}{\sqrt{(\gamma^\triangleu_i+1)(\gamma^\triangleu_i+1)}}.
\end{align*}
These matrices are indeed the Gram matrices of the estimators $\zsim_i, \zavg_i$ and $\xsim_i, \xavg_i$, thanks to the results in Theorem \ref{thm:stats} (up to conjugation by a unitary matrix of signs).

\begin{prop}[minimum square error]\label{prop51}On an event with probability tending to $1$, the mean square error obtained with the aforementioned estimators is asymptotically given by
\begin{equation}
\MSE^\#_\star \approx \sum_{i=1}^{\rr_0} \sigma_i^2\left(1- 2(c^\#_{1,i} c^\#_{2,i})^2 \right) + \sum_{i,j \in [\rr_0]} \sigma_i \sigma_j [(c^\#_{1,i}c^\#_{2,j})^2(\mathfrak{C}_1^\#)_{i,j}(\mathfrak{C}_2^\#)_{i,j}].
\end{equation}
\end{prop}

The proof is at Subsection \ref{proof:prop51}. 

The expression for the MSE in the preceding theorem can explicitly be computed provided we can compute the $\Gamma_{i,j}$ and the $\gamma_i^\triangle$ and $\gamma_i^\triangleu$, which might be difficult. However, when the rank of $P$ is $1$, things are really simple since in this case it is easy to check that $(\mathfrak{C}_k)_{1,1} = 1$, and consequently we will simply have 
\[\MSE_\star^\# = \sigma_1^2\left(1- (c^\#_{1,i} c^\#_{2,i})^2 \right)+o(1).\]

It is easily understood that when $d \to \infty$, the matrices $\mathfrak{C}_1,\mathfrak{C}_2$ converge towards the identity $I$, and the quantities $c^\#_{k,i}$ converge to $1$, hence 
\[\lim_{d \to \infty} \limsup_{n \to \infty} \MSE^\#_\star = 0 \]
thus ensuring that in the $d \to \infty$ regime, recovery is almost exact.

\section{The rank-one case}\label{sec:rankone}

In this part, we develop in greater detail our theory when the underlying problem $P$ has rank $1$. In this case, many computations can explicitly be done without too much difficulty and provide a better understanding of the different parameters at stake.

 We begin with the simple case when $P$ is Hermitian, and then illustrate our statistical results when $P$ is rectangular with rank $1$.

\subsection{Warm-up: symmetric problems}\label{sec:numerics}
Here, the first basic model is the rank-one symmetric completion problem already explored in Subsection \ref{subsec:square_rank1}; the underlying matrix is $P =  \varphi \varphi^*$ and the main parameters $\thresh_2$ and $\gamma$ we computed in \eqref{eq:rankgamma1}. 
We gather these results in the following proposition.

\begin{prop}\label{prop:square_rank1}
Suppose that $P = \varphi \varphi^*$. Then, $\rho = n|\varphi|_4^4$. The detection threshold is given by
\begin{equation}
\thresh_2 = \sqrt{\frac{n|\varphi|_4^4}{d}}.
\end{equation}
The parameter $\gamma_1$ is given by
\begin{align}
\gamma_1 &=  \frac{1 - (\thresh_2^2)^{\ell+1}}{1-\thresh_2^2} \\&= \frac{1+o(1)}{1- \frac{n|\varphi|_4^4}{d}}.
\end{align}
\end{prop} But now, the eigenvector $\varphi$ is going to be taken at random among various distributions, a common model in the literature. More precisely, we take a family $(B(x))_{x \in [n]}$ of i.i.d. random variables, we set $S = \sum |B(x)|^2$ and we define
\[\varphi(x) = \frac{B(x)}{S^{1/2}}. \]
From Theorem \ref{thm:1} and equation \eqref{eq:rankgamma1}, the phase transition in $d$ for weak recovery is given by 
\begin{align}\label{dcrit_rank1}n|\varphi|^4_4 &= n\frac{\sum_{x \in [n]}|B(x)|^4}{\left(\sum_{x \in [n]}|B(x)|^2 \right)^2} \\ &= \left( \frac{\sum_{x \in [n]}|B(x)|^2}{n}\right)^{-2}  \left( \frac{\sum_{x \in [n]}|B(x)|^4}{n}\right) 
\end{align}
which is easily computed using the Law of Large Numbers: if $B$ is a generic random variable with the same distribution as each $B(x)$,  and having a finite fourth moment, then almost surely one has
\[n|\varphi|^4_4 = \frac{\EE[|B|^4]}{\EE[|B|^2]^2}(1+o(1)) \sim \KURT_B\] 
where $\KURT_B$ is the non-centered kurtosis, i.e. the ratio of the fourth moment to the squared second moment; when the distribution is centered this is the classical kurtosis, defined as the ratio of the fourth centered moment to the fourth power of the standard deviation. Of course, the Cauchy-Schwarz inequality tells us that $\KURT$ is always greater than $1$ and this bound is attained for random variables with constant modulus --- in particular, for the centered $\BER(1/2)$ distribution. Table \ref{table:kurtosis} collects some values of the kurtosis.

\begin{table}
\begin{center}
\begin{tabular}{c|c}\hline
distribution $B$  & $\KURT_B$, asymptotic value of $n|\varphi|^4_4$ \\ 
\hline
Bi-sided exponential (Laplace), $f(x)\propto e^{-|x|}$ & 6 \\
Hyperbolic secant & $ 5$ \\
Standard normal & $ 3$ \\
Uniform on $[0,1]$ & $ 9/5 \approx 1.8 $\\
$\BER(c)$ & $ 1/c$ \\
Centered $\BER(1/2)$ & 1 \\
Generalized normal: $f(x)\propto e^{-|x|^\beta}$ & $\frac{\Gamma(5/\beta)\Gamma(1/\beta)}{\Gamma(3/\beta)^2}$ \\
\hline
\end{tabular}
\end{center}
\caption{Some values of the kurtosis, appearing in our threshold as $n|\varphi|_4^4 \sim \KURT_B$.}\label{table:kurtosis}
\end{table}

\begin{remark}
As mentioned in the introduction, the real threshold is $\max \{\thresh_1, \thresh_2\}$, and in this setting it is equal to $\thresh_2$ if and only if $\sqrt{\rho/d}> L/d$, which reduces to $n|\varphi|_4^4 > L^2/d$. But in our rank-one models, if the sampling distribution of entries of $\varphi$ is unbounded, then $L$ will grow to $\infty$, even if very slowly: for standard normal random entries, we will have $|\varphi|_\infty \asymp \sqrt{\log n/n}$ so $L \asymp \log n$ and our theoretical threshold should (asymptotically) be $L/d$. However, we can actually bypass this limitation, using tools introduced in \cite{BQZ} and further explained in the last section of this paper, Section \ref{sec:techdiscuss}. The key here is that $L$ might be replaced by an essential supremum $L'$, which accounts for the maximum of the entries of $P$ after deleting a small subset of these entries. The simulations suggest that these refinements do indeed confirm that $\thresh_2$ is the right threshold of interest. 
\end{remark}

\subsection{Numerical validation of theoretical results}\label{sec:numerical sim rank1}

\begin{figure}[H]
    \centering
    \subfloat[Bernoulli]{
    \includegraphics[width=0.95\textwidth]{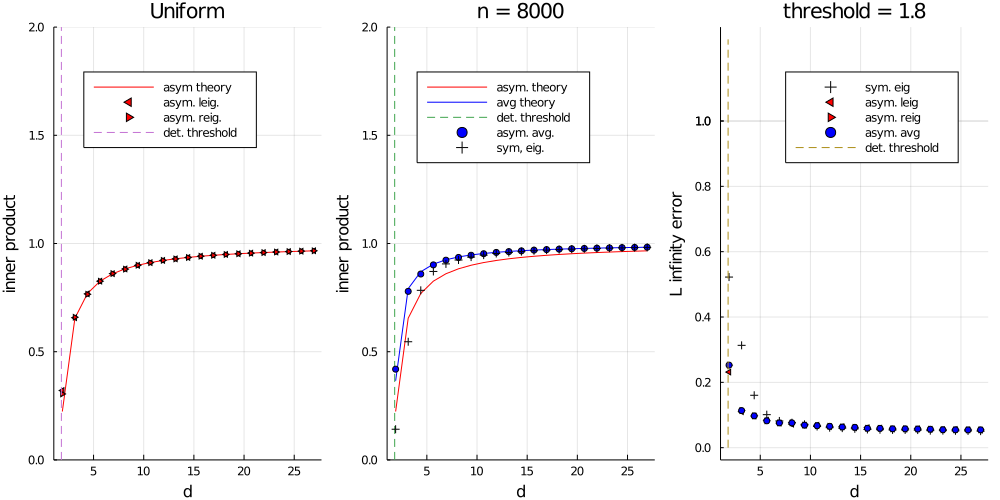} 
    \label{fig:sim_uniform_symmetric}
   }\\
    
    \subfloat[Normal]{
    \includegraphics[width=0.95\textwidth]{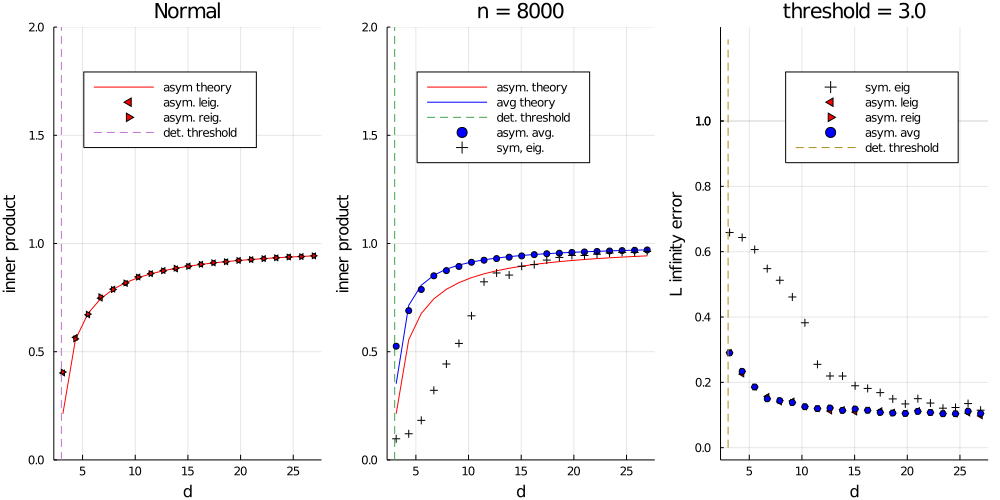}
    \label{fig:sim_normal_symmetric}
    }\\
    \subfloat[Hyperbolic Secant.]{
    \includegraphics[width=0.95\textwidth]{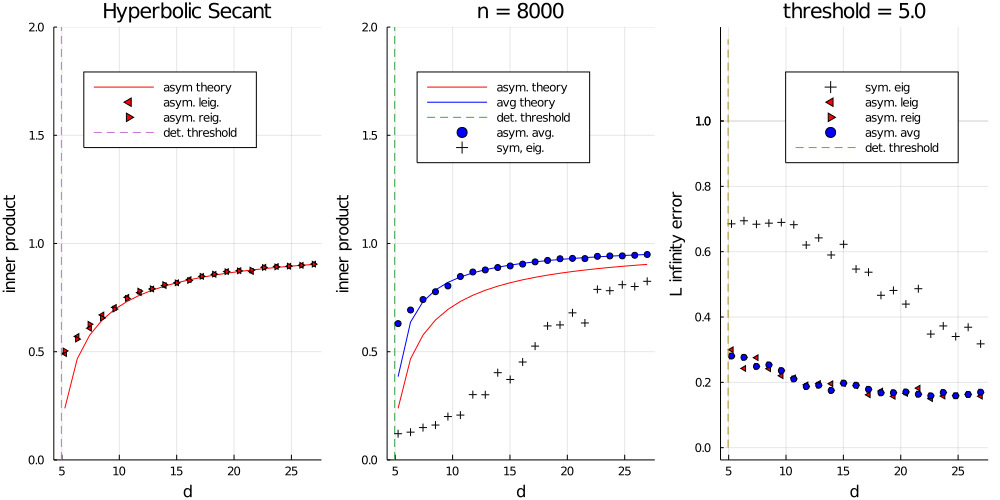}
    \label{fig:sim_hypersecant_symmetric}
    }
    
\caption{Numerical validation of the predictions in Proposition \ref{prop:square_rank1} for the setup described in Section \ref{sec:numerical sim rank1} where $n = 8000$ and various values of $d$.}
\label{fig:numerical sim 1}
\end{figure}

Figure \ref{fig:numerical sim 1} shows the agreement between the theoretical predictions in Proposition \ref{prop:square_rank1} for the inner product between the left and right eigenvector and the ground truth vector and experiment for the setting where the rank one $P$ with $n = 8000$  has (unit-norm) eigenvectors vectors drawn from the uniform, normal and hyperbolic secant distributions.

The detection threshold, as in Proposition \ref{prop:square_rank1} is a function of the kurtosis and Figure \ref{fig:numerical sim 1} confirms this prediction as well as that of the predicted inner products of the left/right eigenvectors with respect to the ground truth eigenvector and the prediction for the improved performance of the  eigenvector estimate formed by averaging the left and right vectors.  

The third column of Figure \ref{fig:numerical sim 1} illustrates the improved accuracy of the estimated vectors in the $\ell_\infty$ error sense relative to the eigenvector obtained using the symmetric eigen-decomposition.  This goes beyond the statement of our results and an analysis of this improvement  represents a natural follow-up of our line of work. 

Note, too that the estimation performance gap between the asymmetric method and the symmetric method increases with the kurtosis of the eigenvector element. Intuitively this has to do with the localization of the eigenvectors of the symmetric matrix in the very sparse regime which is bypassed when we induce asymmetry. 

\subsection{Using non-backtracking matrices}

We illustrate in this section the behaviour of the non-backtracking statistics from Section \ref{subsec:NB}, and especially Corollary \ref{cor:1nb}. To do this, we recall that we first symmetrize the observation and then build the non-backtracking matrix $B$. 

\bigskip

\noindent \emph{Working at the non-backtracking level}. The left/right eigenvectors associated with the unique outlier of $B$ will be called $\psi, \psi'$, and lives in $\mathbb{C}^m$. The result in \eqref{eigenvector_errorboundnb} says that if $\varphi^+$ is the lifting of $\varphi$, then 
\[\langle \psi, \varphi^+ \rangle = \frac{1+o(1)}{\sqrt{\gamma_i}}, \]
just as in the preceding paragraph. The main difference now is that the left/right inner product is given thanks to \eqref{eigenvector_errorboundLRnb} by
\begin{equation}
\langle \psi, \psi'\rangle = \frac{1+o(1)}{\sqrt{\gamma_i\hat{\gamma}_i}}.
\end{equation}
However, when the rank is $1$, the quantity $\hat{\gamma}_i$ is found in \eqref{gamchapNB} to be equal to $(\rho/\mu_1^2)\gamma_i = \rho \gamma_i$, which in our case is exactly 
\[\hat{\gamma}_i = n|\varphi|_4^4 \gamma_i. \]
Consequently, the inner left/right product is given by 
\begin{equation}\label{wnb_lr_dot}
\langle \psi, \psi'\rangle = \frac{1+o(1)}{\gamma_i \sqrt{n|\varphi|_4^4}} = \frac{1-\frac{n|\varphi|_4^4}{d}}{\sqrt{n|\varphi|_4^4}}+o(1).
\end{equation}
Asymptotically, the left and right non-backtracking eigenvectors are thus far from being aligned even when $d \to \infty$, since in this regime their angle converges towards $1/\sqrt{n|\varphi|_4^4}$ which is generally strictly smaller than $1$ as soon as $\varphi$ is not the constant unit vector.

\bigskip

\noindent\emph{Working with lowered eigenvectors. }As in Corollary \ref{cor:1nb}, we can also `lower' the eigenvectors to the dimension $n$. To do this we simply follow one of the procedures described above Corollary \ref{cor:1nb}; the inner product between the estimators $\hat{\varphi}$ and $\check{\varphi}$ and the real ground-truth eigenvector $\varphi$ is $1/\sqrt{\gamma_i}$ as above.

Figure \ref{fig:numerical sim back} shows agreement between theory and experiment using the weighted non-backtracking matrix. The leftmost subplot  in Figure \ref{fig:uniform wnb} confirms the accuracy of the inner product prediction in Theorem \ref{thm:1nb}, and the equivalent performance of the left and right lowered vectors of the weighted non-backtracking matrix as predicted in Corollary \ref{cor:1nb}. The rightmost subplot in Figure \ref{fig:uniform wnb} confirms the accuracy of the theoretical prediction for the inner product between the left and right (raised) eigenvectors of the weighted non-backtracking matrix given by Theorem \ref{thm:1nb} and \eqref{wnb_lr_dot} and that between the lowered left and right eigenvectors. The middle plot in Figure \ref{fig:uniform wnb} shows the improvement in eigenvector estimation due to the weighted non backtracking matrix relative to that obtained from the (symmetric) eigendecomposition and the average of the left and right vectors from the randomized asymmetric eigendecomposition. Figure \ref{fig:normal wnb} plots the same quantities as Figure \ref{fig:uniform wnb}, except over a single trial and with normally distributed eigenvectors -- the plots confirm the accuracy of the asymptotic predictions and the concentration of measure implied in Theorem \ref{thm:1nb}.

\bigskip

\begin{figure}
    \centering
    \subfloat[Uniformly distributed eigenvector: results averaged over 50 trials. ]{\includegraphics[width=0.95\textwidth]{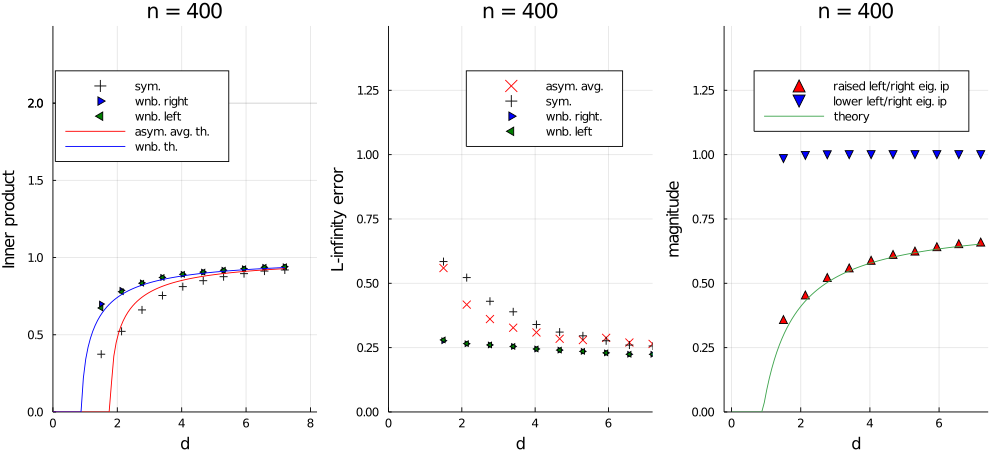}
    \label{fig:uniform wnb}
    }\\
    \subfloat[Normally distributed eigenvector: result obtained from 1 trial.]{\includegraphics[width=0.95\textwidth]{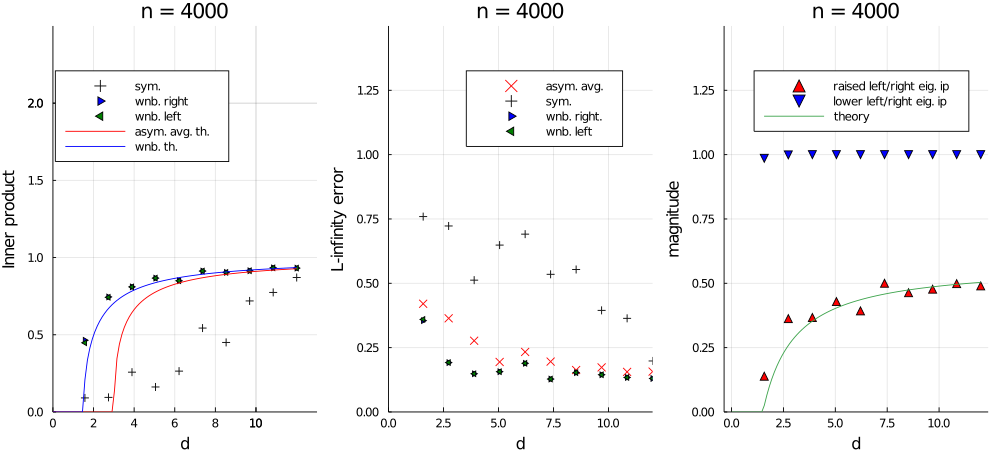}
    \label{fig:normal wnb}
    }\\
\caption{Numerical validation for the improved accuracy of the eigenvectors derived from the weighted non-backtracking matrix relative to those obtained from the eigenvectors (symmetric) eigen-decomposition  of $P$ and the averaged left and right eigenvectors of the randomized asymmetric matrix. The rank one matrix was generated as described in Section \ref{sec:numerics}  with uniformly distributed (top) and normally distributed (bottom) eigenvectors. The setup and the theoretical predictions are described in Proposition \ref{prop:square_rank1}.}
    \label{fig:numerical sim back}
\end{figure}

\subsection{Rectangular rank-one: explicit computations}

Suppose now that that $P = \zeta \xi^*$ where $\zeta, \xi$ are unit vectors. The size of the matrix is $m \times n$ with $m/n = \alpha$. In this case, the whole problem relies on the computation of the quantities $\gamma^\triangle_i$. They might be difficult to compute in the general case, but here these quantities can entirely be computed in terms of the 4-norm of $\zeta, \xi$, as in the symmetric case. The expressions are a little bit more intricate, but once $|\xi|_4, |\zeta|_4$ are known, the dependence in $d$ is simple. Note that $\gamma_i^\triangle, \gamma_i^\triangleu$ both depend on $\ell$, see the definition in \eqref{def:rect:correlation}; however, in the computations (which are deferred to Section \ref{sec:proofs:rect}), this dependence can be neglected because it only gives rise to terms which are seemingly complicated, but who in the end behave like $\threshh^\ell$ which goes to zero. This is why we encapsulated them in the $o(1)$ notation.

\begin{prop}\label{prop:rect:rank1}
Suppose that $P = \zeta \xi^*$. The detection threshold is given by
\begin{equation}
\threshh_2 = \sqrt{\frac{2n|\zeta|_4^2|\xi|_4^2}{d}}.
\end{equation}
The parameters $\gamma^\triangle_i, \gamma^\triangleu_i$ are given by
\begin{equation}\label{rect:gamma12+1}
\gamma^\triangle_i =  \frac{1+o(1)}{2|\xi|_4^2} \left(  \frac{|\zeta|_4^2+|\xi|_4^2}{1-\frac{2n|\zeta|_4^2|\xi|_4^2}{d}}+\frac{|\xi|_4^2-|\zeta|_4^2}{1+\frac{2n|\zeta|_4^2|\xi|_4^2}{d}} \right)
\end{equation}
and
\begin{equation}\label{rect:gamma12+2}
\gamma^\triangleu_i=\frac{1+o(1)}{2|\zeta|_4^2} \left(  \frac{|\zeta|_4^2+|\xi|_4^2}{1-\frac{2n|\zeta|_4^2|\xi|_4^2}{d}}+\frac{|\zeta|_4^2-|\xi|_4^2}{1+\frac{2n|\zeta|_4^2|\xi|_4^2}{d}} \right).
\end{equation}
\end{prop}
The proof of Proposition \ref{prop:rect:rank1} is only a computation, although more tedious than when the underlying problem $P$ is Hermitian. We postponed it in Section \ref{sec:proofs:rect}, at page \pageref{proof:rect:rank1}. 

\subsection{Rectangular rank-one: numerical validation}

Here, $P = \zeta \xi^*$, but we generate $\zeta, \xi$ using the same model as for the symmetric case: we put 
\[\zeta(x) = \frac{A(x)}{S_A^{1/2}} \qquad \xi(x) = \frac{B(x)}{S_B^{1/2}}\]
where $S_A = \sum |S(x)|^2$ and $S_B = \sum_y |B(y)|^2$, and $A(x)$ are i.i.d. samples from a common distribution $A$ and $B(x)$ are i.i.d. samples from another distribution. We suppose that both of them have finite fourth moments. With this model, the Law of Large Numbers entails
\begin{align*}
&S_A \sim n \EE[A^2]&& S_B \sim n \alpha \EE[B^2] \\
&|\zeta|_4^4 \sim \frac{1}{n} \frac{\EE[A^4]}{\EE[A^2]^2}= \frac{\KURT_A}{n} && |\xi|_4^4 \sim \frac{1}{\alpha n} \frac{\EE[B^4]}{\EE[B^2]^2}=\frac{\KURT_B}{\alpha n}.
\end{align*}
almost surely as $n \to \infty$. 

In this case, Proposition \ref{prop:rect:rank1} say that the detection threshold in $d$ is equal to 
\[2n|\zeta|_4^2|\xi|_4^2 =(1+o(1)) 2\sqrt{\frac{\KURT_A \KURT_B}{\alpha}}.\]
For example, if $A,B$ have the same kurtosis $k$, then the threshold for the birth of outliers close to $1$ in the spectra of $X$ and $Y$ (defined in \eqref{def:XY}) is $2k / \sqrt{\alpha}$.

More precisely, with high probability, the following happens. 
\begin{enumerate}[(i)]
\item  If $d \leq 2n|\zeta|_4^2|\xi|_4^2$, then all the eigenvalues of $X$ and $Y$ have modulus smaller than $\sqrt{2n|\zeta|_4^2|\xi|_4^2/d}+o(1)$.
\item  If  $d > 2n|\zeta|_4^2|\xi|_4^2$, then all the eigenvalues of $X$ and $Y$ have modulus smaller than $\sqrt{2n|\zeta|_4^2|\xi|_4^2/d}+o(1)$, except one eigenvalue $\nu$ of $X$ with $\nu = 1 + o(1)$ and one eigenvalue $\eta$ of $Y$ with $\eta = 1 + o(1)$.  
\end{enumerate}

We denote by $\chi, \chi'$ the unit right and left eigenvectors of $X$ associated with the outlier, when it exists. Similarly, we denote $\pi,\pi'$ the right and left eigenvectors of $Y$. Recall the convention \eqref{eq:choicephase} on the positivity of the scalar product of left and right eigenvectors. 

\bigskip

\noindent \emph{Estimators: definition and accuracy. }We will use the estimators defined in Subsection \ref{subsec:singvecest}: 

\begin{align*}
&\zsim = \chi &&\zavg = \frac{\chi+\chi'}{|\chi + \chi'|}
\end{align*}
and similar estimators for $\xi$. We computed $\gamma_i^\triangle$ in Proposition \ref{prop:rect:rank1}: when $\KURT_A=\KURT_B=k$, a few manipulations show that
\begin{align}
&\gamma^\triangle_1 \approx \frac{1+\frac{2k}{d}}{1-\frac{4k^2}{d^2\alpha}} &&\gamma^\triangleu_1 \approx  \frac{1+\frac{2k}{d\alpha}}{1-\frac{4k^2}{d^2\alpha}}.
\end{align}
Consequently, from Theorem \ref{thm:smallsquare} that the left/right inner product is 
\begin{align*}&\langle \chi, \chi'_i \rangle = \frac{1+o(1)}{\gamma_i^\triangle}&& \langle \pi, \pi'_i \rangle = \frac{1+o(1)}{\gamma_i^\triangleu}. \end{align*}
The formulas for the correlations in Theorem \ref{thm:stats} give $\langle \zgen, \zeta \rangle = c_1^\#$, and we get 
\begin{align}
&|\langle \zsim, \zeta \rangle|\approx \sqrt{\frac{1-4k^2/d^2\alpha}{1+2k/d}} &&|\langle \xsim, \xi \rangle|\approx \sqrt{\frac{1-4k^2/d^2\alpha}{1+2k/d\alpha}}\\
&|\langle \zavg, \zeta \rangle|\approx \sqrt{\frac{1-4k^2/d^2\alpha}{1+k/d-2k^2/d^2\alpha}} &&|\langle \xavg, \xi \rangle|\approx \sqrt{\frac{1-4k^2/d^2\alpha}{1+k/d\alpha-2k^2/d^2\alpha}}.
\end{align}

Finally, the MSE (in the sense of Proposition \ref{prop51}) in this rank-one context considerably simplifies. Indeed, we have $\sigma_1=1$ and above the threshold $\rr_0=1$, so that  
\begin{align*}
\MSE_\star^\#&= 1 - (c_{1,i}^\# c_{2,i}^\#)^2 +o(1)
\end{align*}
and using the definitions of $c_{k,i}^\#$ from above Theorem \ref{thm:stats} we find
\begin{align*}
&\MSE_\star^{\rm sim}= 1 - \frac{1}{\gamma_i^\triangle \gamma_i^\triangleu}+o(1)\\&\MSE_\star^{\rm avg}=1-\frac{4}{(1+\gamma_i^\triangle)(1+\gamma_i^\triangleu)}+o(1).
\end{align*}
\bigskip

\noindent \emph{Illustrations. }
Figure \ref{sim:mat completion} shows the agreement between theory and experiment for the rectangular setting with respect to the predicted inner product between the averaged left and right eigenvector of the $X$ (resp. Y) matrix corresponding to the largest real eigenvalue and the left (resp. right) singular vector of the underlying matrix. The plot also confirms our prediction in Subsection \ref{sec:procedure} that the accuracy of the left (resp. right) singular vector estimated thus with respect to the ground truth vector can   can be determined from the inner product between the left and right eigenvectors of $X$ (resp. $Y$). This underpins the statistically optimal (in the MSE sense) data-driven matrix completion  algorithm 

We could lower the MSE of the recovered by using the weighted non-backtracking variant of the method  -- it is computationally too expensive for the the $m, n$ and $d$ values considered here.

\begin{figure}
    \centering
    \subfloat[Hyperbolic Secant.]{
    \includegraphics[width=0.85\textwidth]{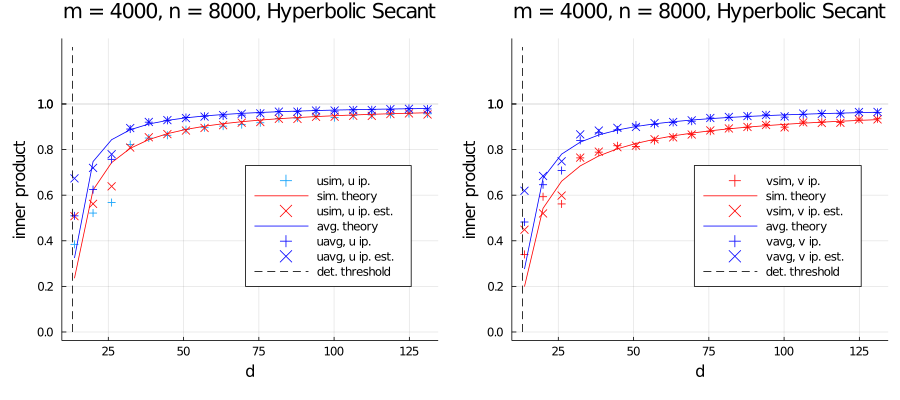}}
    \hspace{0.5cm}
    \subfloat[Hyperbolic secant.]{
    \includegraphics[width=0.85\textwidth]{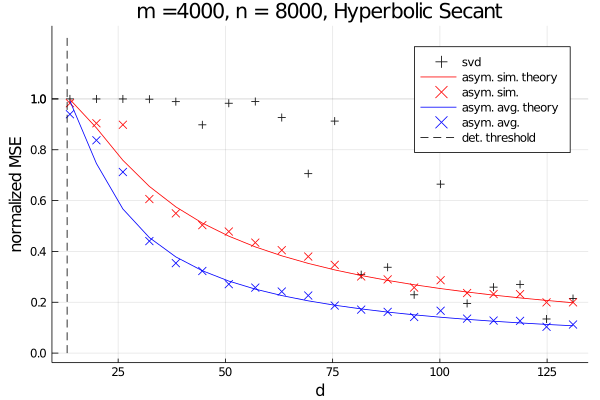}} \\

\caption{Matrix completion normalized optimal MSE for Hyperbolic Secant distributed singular vectors for one trial for am $m \times n$ rank one matrix with Hyperbolic Secant distributed (unit norm) left and right singular vectors. Note the accuracy of the theoretical predictions, the improvement in performance relative to the SVD and underlying  predicted concentration that makes the asymptotic theory closely match the result from a single trial. }
\label{sim:mat completion}
\end{figure}

\clearpage

\section{Related work}\label{sec:related}

A first version of this paper appeared in the PhD manuscript of the second author, in 2019 (\cite{these}). 

\subsection{Completion and sparsification}The problem of sparse completion consists in observing a very sparse sample of elements of a general object (a matrix, a subspace) carrying some structure (low-rank, delocalized), and trying to reconstruct it. The problem of \emph{matrix completion} has attracted a gigantic amount of attention from researchers in applied mathematics since the last 15 years; the general philosophy can be grasped by a handful of seminal papers from Candès and Tao \cite{candes_tao} and Candès and Recht (\cite{candes_recht}),  Keshavan Montanari and Oh \cite{montanari} and Chatterjee (\cite{chatterjee}). The survey \cite{davenport2016overview} gives a global view of the field. 

 The dual problem of completion is \emph{sparsification}, where given a matrix $P$, one seeks a procedure to keep only a handful of entries of $P$ without altering too much its properties (\cite{achlioptas2007fast, drineas_zouzias, kundu_drineas, orourkevuwang}). 

Those papers, although different in their methods, show that completing a matrix from the observation of $nd$ of its entries can only be done if the underlying matrix $P$ is not too complicated (i.e. low-rank and sufficiently incoherent), and in that case $P$ can efficiently be recovered only if $d$ is of order $\ln(n)$ --- the so-called \emph{information-theoretic threshold} for completion. In \cite{montanari}, there are results for $d$ fixed, but they are not sharp at all and do not allow any precise asymptotics on specific eigenvalues as we do. To our knowledge, the few works on completion from $d=O(n)$ entries (see for instance Gamarnik, Li and Zhang \cite{gamarnik2017matrix} and references therein) is focused on $\epsilon$-approximating the whole hidden matrix $P$, and never on exact estimation of a specific part of the matrix.

\subsection{Random matrices and \erd graphs}\label{subsec:erd_bibli}rom the random matrix point of view, this is all about the spectrum of (sparse) random matrices, or on the eigenvalues of weighted (sparse) random graphs. Estimating the spectral properties of the simplest of random graphs, such as \erd, is already quite difficult (\cite{krivelevich2003largest}). The complete description of the behavior of the greatest eigenvalues of \erd graphs have been totally explained, in the $d=o(n)$ sparse setting, only recently by different works: Benaych-Georges, Bordenave, Knowles (\cite{benaych2017largest, benaych_radii}) and Alt, Ducatez and Knowles (\cite{alt2019extremal}). Recently, Tikhomirov and Youssef gave similar results for eigenvalues of \erd graphs with i.i.d. Gaussian weights on the edges (\cite{youssef_tikho_outliers}); here, the underlying matrix $P$ is thus drawn from GOE, and does not meet the usual assumptions of matrix completion. We finally mention a significant result on inhomogeneous \erd graphs by Chakrabarty, Chakraborty and Hazra \cite{chakrabarty2019eigenvalues} complementing \cite{benaych2017largest}. 

In those works, it turns out that the behaviour of the (suitably normalized) high eigenvalues of \erd graphs is governed by the high degrees of the graph when $d<\ln(n)$, and stick to the edge $\pm 2$ of the limiting semi-circle law in then $d \to \infty$. The exact threshold for the disappearance of outliers happens at $d_\star= \ln(4/e)^{-1}\ln(n)$ (\cite{alt2019extremal, youssef_tikho_outliers}). Those results hold for \emph{undirected} \erd graphs, and we are not aware of any similar results for \emph{directed} \erd graphs, and even less in the really sparse regime where $d$ is fixed. Indeed, only the convergence of the global spectrum towards the circle law is now proven (when $d >\ln(n)^2$) by Basak and Rudelson (\cite{basak2017circular}). Many questions and intuitions are given in the physicist survey \cite{physicist_sparse}. Among them are listed (but not proved) our results on eigenvalues of \erd graphs. Our results on eigenvectors completes the picture.

\subsection{Phase transitions}Our main result is a phase transition for the top eigenvalues of sparse non-Hermitian matrices: the whole bulk is confined in a circle of radius $O(1/\sqrt{d})$, and depending on the strength of the noise $d$, a few outliers appear and they are aligned with the corresponding eigenvalues of the original matrix $P$, and their eigenvectors have a nontrivial correlation with the original eigenvector.

This is of course similar to the celebrated BBP transition (\cite{transitionBBP}), and many similar transitions are already available in the literature of PCA or low-rank matrix estimation (\cite{transitionBenaych, transitionMiolane} and references therein). Apart from \cite{transitionMiolane_nonsym}, which has a very different setting than ours, there are no results for phase transitions in low-rank non-symmetric matrix estimations, or in sparse settings.

\subsection{`Asymmetry helps'}One of the key features of this paper is that it deals with top eigenvalues of non-symmetric matrices. While the global behaviour of the spectrum of random matrices is now well understood (see the survey \cite{circular_survey} on the circular law, or \cite{touboul, physicist_sparse} for physicist's point of views), finer properties are less known. 

Generally speaking, it is easier to deal with eigenvalues of Hermitian matrices, notably thanks to the variational characterizations of the eigenvalues. However, in many problems from applied mathematics, it turns out that the spectrum of Hermitian matrices can sometimes be less informative than the spectrum of other choices of non-Hermitian matrices. A striking instance of this fact was the so-called `spectral redemption conjecture' in community detection (\cite{krzakala2013spectral} and \cite{bordenave_lelarge_massoulie}), where the interesting properties were not captured by the spectrum of the adjacency matrix, but of a non-Hermitian matrix, the non-backtracking matrix. 

In the setting of matrix perturbation, this insight was remarkably exposed in a recent and inspiring paper by Chen, Cheng and Fan (\cite{chen2018asymmetry}). Their setting is more or less the same as ours: an underlying Hermitian matrix $P$, which is asymmetrically perturbed into an observed non-Hermitian matrix $A=P+H$, the entries of $H$ being all i.i.d. One might favor a singular value decomposition because of the conventional wisdom that SVD is more stable than eigendecomposition when it comes to non-Hermitian matrices; but this in fact not true, as shown in their Figure 1, and indeed the eigenvalues are more accurate than the singular values; \emph{verbatim},

\begin{center}
\emph{``When it comes to spectral estimation for low-rank matrices, arranging the observed matrix samples in an asymmetric manner and invoking eigen-decomposition properly (as opposed to SVD) could sometimes be quite beneficial.''  \cite[page 2]{chen2018asymmetry} }
\end{center}

This is the philosophy we would like to convey here; however, their result hold only on the not-so-sparse regime where $d>\ln(n)$. We extend all their results to the fixed $d$ regime, with an explicit threshold for the detection of $P$ and exact asymptotics for perturbation of linear forms.

\subsection{Eigenvalues of perturbed matrices}Many works on completion or sparsification rely on a perturbation analysis of the eigenvalues/singular values of perturbed matrices. 

For example, one of the key points in many papers is that the sparsification procedure (from $P$ to $A$) alters the spectral properties of $P$, but not too much; indeed the top singular values or eigenvalues do not differ too much, hence keeping only the `greater' items in the SVD or the eigendecomposition of $A$ is sufficient to weakly recover $P$; that was the idea of \cite{montanari, chatterjee, drineas_zouzias} (and many of their heirs). The proofs usually rely on estimates on eigenvalues/singular values of the random matrix $A$, by combining concentration inequalities and eigenvalues inequalities (such as Weyl's one). but no sharp asymptotics can be obtained with those methods, a limitation already visible in the seminal paper from Friedman, Kahn, Szemeredi (\cite{friedman1989second} and Feige and Ofek (\cite{feige2005spectral}). This problem becomes unassailable when $d$ is really smaller than $\ln(n)$ or fixed, due to the fact that the underlying graphs are highly non-regular.

Our proof techniques globally rely on methods introduced by Massoulié and refined by Bordenave, Lelarge and Massoulié (\cite{massoulie_rama, bordenave_lelarge_massoulie}). This powerful and versatile trace method has now been used in various problems for estimating high eigenvalues of sparse random matrices, such as random regular graphs (\cite{bordenave2015}), biregular bipartite graphs (\cite{dumitriu}), digraphs with fixed degree sequence (\cite{coste2017}), bistochastic sparse matrices (\cite{bordenave_qiu_zhang}), multigraph stochastic blockmodels (\cite{2019arXiv190405981P}). However, our construction, and especially the pseudo-eigenvectors we chose, greatly simplifies the former analysis in \cite{massoulie_rama, bordenave_lelarge_massoulie}. This considerable simplification has been very recently been applied to the non-backtracking spectrum of inhomogeneous graphs in Massoulié and Stephan in \cite{stephan2020nonbacktracking}, it follow from our methods which was introduced in a preliminary version of this work contained in \cite{these}.

\subsection{Eigenvectors of perturbed matrices}Eigenvector perturbation has also attracted a lot of attention, mainly around variants of the Davis-Kahan theorem (\cite{yu2014useful}). As mentioned in \cite{unperturbed}, many algebraic bounds (such as Weyl's inequality or the Davis-Kahan theorems) are tight in the worst case, but wasteful in typical cases. Our proof method does not rely on those general bounds,  and naturally integrates the perturbation of eigenvectors in combination with the now classical Neumann trick (see \cite{unperturbed, chen2018asymmetry}). 



\bibliography{bibli}

\newpage

\section{An algebraic perturbation lemma}
\label{sec:perturbation}
We present an eigenvalue-eigenvector perturbation theorem, which extends some the results from \cite[Section 4]{bordenave_lelarge_massoulie} by taking into account the lack of normality of the structures at stake. We formulate this tool in a separate section because it can be of independent interest. 

Let us first give a simple description of the result: if $u_i, v_i$ are vectors such that $\langle u_i, v_j\rangle \approx \delta_{i,j}$, then every matrix close to $S = \sum \theta_i u_i v_i^*$ has eigenvalues close to the $\theta_i$, provided the $u_i$'s are sufficiently well-conditioned. Moreover, if the $\theta_i$ are well-separated, the corresponding right-eigenvectors of $A$ are close to the $u_i$.  Theorem \ref{thm:linalg} quantifies this for eigenvalues of generic matrices and Theorem \ref{thm:linalg:powers} quantifies this for eigenvalues and eigenvectors of matrix powers. The novelty here is that the vectors $u_i$ need not form an orthonormal family for the result to hold, and the same for the $v_i$'s. 

\bigskip

We first recollect the Bauer-Fike theorem:

\begin{theorem}[Bauer-Fike, \cite{bauer}, \cite{MR1061154}, chapter IV]\label{th:BF}
Let $S$ be a diagonalizable matrix, $S=P\Sigma P^{-1}$ with $\Sigma=\mathrm{diag}(\theta_1, \dotsc, \theta_n)$ and let $A = S +E$ be a matrix. Then, all the eigenvalues of $ A$ lie inside the union of the balls $B(\theta_i, \varepsilon)$ where $\varepsilon = \Vert E \Vert \Vert P \Vert \Vert P^{-1} \Vert$. Moreover, if $J \subset [n]$ is such that 
\[\left(\cup_{j \in J}B(\theta_j, \varepsilon) \right) \cap \left(\cup_{j \notin J}B(\theta_j, \varepsilon) \right) = \varnothing, \]
then the number of eigenvalues  (with multiplicities) of $A = S+E$ inside $\cup_{j \in J}B(\theta_j, \varepsilon)$ is exactly $|J|$.
In particular, for each $i \in [n]$, if $m_i$ is the number of distinct eigenvalues of $\Sigma$ which are in the connected component of $\cup_{j \in J}B(\theta_j, \varepsilon)$ containing $\theta_i$, there exists an ordering of the eigenvalues $\lambda_1, \ldots, \lambda_n$ of $A$ such that 
$$
|\lambda_i - \theta_i| \leq (2 m_i -1)\veps.
$$
\end{theorem}

The last statement is usually stated with $m_i=n$ (as in Theorem \cite[Theorem 3.3]{MR1061154}, see Figure \ref{fig:fik} for an illustration (and first paragraph p170 in \cite{MR1061154} for further explanation).
\begin{figure}[H]\centering
\includegraphics[width=7cm]{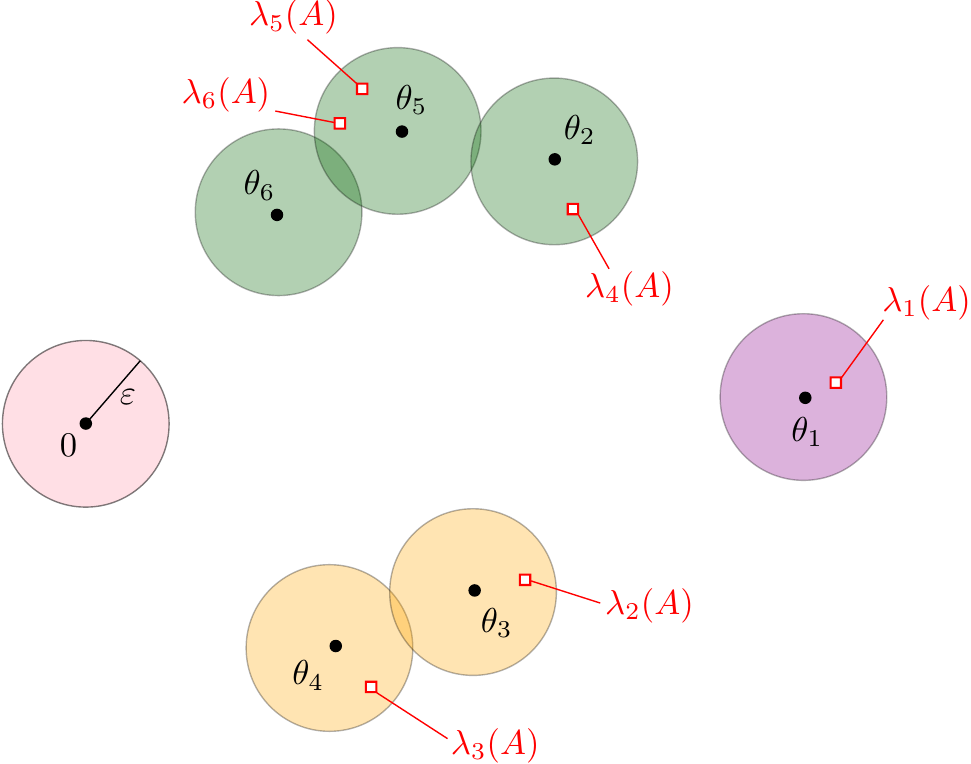}
\caption{The eigenvalue $\lambda_i$ need not be in $B(\theta_i, \varepsilon)$ because there might be some overlap with the closest balls, such as for the green or yellow ones in the drawing. However, if $m_i$ is the number of distinct balls in the connected component $B(\theta_i, \varepsilon)$, the eigenvalue $\lambda_i$ will always be within distance $(2m_i -1) \varepsilon$ of some $\theta_{\pi(i)}$ and $\pi$ permutation in $\mathfrak{S}_n$: for instance in the picture above one could take $\pi=(234)$. }\label{fig:fik}
\end{figure}

In the first section, we prove our general eigenvalue perturbation lemma. In the second section, we give a variant for powers of matrices, that incorporates a control over arguments of complex eigenvalues, and which also contains the eigenvector perturbation result. 

\subsection{Eigenvalue perturbation}

Let $u_1, \dotsc, u_r, v_1, \dotsc, v_r$ be two families of nonzero vectors in $\mathbb{R}^n$. Let us note $U=(u_1, \dotsc, u_r)$ and $V=(v_1, \dotsc, v_r)$; those are real matrices with $n$ lines and $r$ columns. Our `nearly diagonalizable' matrix will be $S = U\Sigma V^*$ with $\Sigma = \mathrm{diag}(\theta_1, \dotsc, \theta_r)$, the $\theta_i$ being complex numbers. The center of our investigations will be some real square matrix $A \in \mathscr{M}_{n}(\mathbb{R})$, not necessarily diagonalizable, but close to $S$ in operator norm. We make the following assumptions. 
\begin{enumerate}
\item There is some $\eta >0$ such that
\begin{equation}
\Vert A - S \Vert \leqslant \eta.
\end{equation}
\item\label{cond2} The matrices $U$ and $V$ are well-conditionned, in the following sense: 
\begin{itemize}
\item For some $N\geq 1$ we have $\Vert U \Vert \leqslant N$ and $\Vert V \Vert \leqslant N$.
\item For some  $h >0$ we have 
\begin{equation}\label{linalg:conditioning}
\lambda_{\min}( V^* V ) \geqslant h \qquad \text{ and } \qquad \lambda_{\min}( U^* U)  \geqslant h,
\end{equation}
where $\lambda_{\min}$ is the smallest eigenvalue. 
\item The matrices $U$ and $V$ are nearly pseudo-inverses: there is a  $\delta\geq 0$ such that
\begin{equation}\label{linalg:hyp_U*V}
\Vert U^* V - I_r \Vert \leqslant \delta.
\end{equation}
\end{itemize}
\end{enumerate}

\begin{theorem}\label{thm:linalg}
Set $\theta_i = 0$ for all $i \in [n] \backslash [r]$ and $$\veps = 12 N^ 3  \left(\eta+\frac{ 4 \sqrt r \delta N^4 \max_i |\theta_i| }{h} \right).$$
Under the preceding assumptions, we can apply the conclusion of Theorem \ref{th:BF} to $A$ and $\veps$.  
\end{theorem}

\begin{proof}
By homogeneity, we may assume that $\max_i |\theta_i|= \|\Sigma\| \leq 1$. Note also that $\| S \| \leq \| U \| \| \Sigma \| \| V\| \leq N^2$ and $\| A \| \leq \| A - S\| + \| S \| \leq \eta + N^2$. Hence the maximal distance between any pair of eigenvalues of $A$ and $S$ is at most $2(\eta + N^2)$. It follows that the statement is trivial if the following inequality does not hold: $$\delta \leq \frac{h }{ 4 N^2 \sqrt r}.$$
We will thus assume that the above inequality hold. We begin by defining a matrix $\bar{U}$ close to $U$, which is really a pseudo-inverse of $V$; this will be achieved thereafter in \eqref{UstarV}. To do this, we define the vector spaces
\[H_i = \mathrm{vect}(v_j : j \neq i).\] 
Since $\lambda_{\min} ( V^* V) \geq h >0$, $V$ has full rank, hence $H_i$ has dimension $r-1$. The orthogonal projection on $H_i$ is given by 
\[\proj_{H_i}(w)= V_i ({V_i}^* V_i)^{-1} V_i^* w\] where $V_i$ is $V$ whose $i$-th column $v_i$ has been deleted. Note that ${V_i}^* V_i$ is a principal submatrix of $V^* V$, hence it is nonsingular itself; moreover, its eigenvalues interlace those of $V^* V$ and in particular, its smallest eigenvalue is greater than $h$ through \eqref{linalg:conditioning}; when taking the inverse, we get $\Vert ({V_i}^* V_i)^{-1}\Vert \leqslant 1/h$. 

We now consider the vectors defined by $\tilde{u}_i \defeq u_i - \proj_{H_i}(u_i)$ and 
\begin{equation}
\bar{u}_i \defeq \frac{\tilde{u}_i}{\langle \tilde{u}_i, v_i \rangle} =\frac{u_i - \proj_{H_i}(u_i)}{\langle u_i - \proj_{H_i}(u_i), v_i\rangle}.
\end{equation}

We set $\bar{U}=(\bar{u}_1, \dotsc, \bar{u}_r)$; we want to prove that $\bar{U}$ is close to $U$ and that $\bar{U}^*V=I_r$. 
 Let $e_j$ denotes the $j$-th element of the canonical basis of $\mathbb{R}^n$. By \eqref{linalg:hyp_U*V}, we have $|V^* u_i-e_i|=|V^* Ue_i - I_r e_i|\leqslant \Vert V^* U - I_r \Vert = \Vert U^* V - I_r \Vert \leqslant \delta$, thus we also have $|{V_i}^*u_i|^2 = \sum_{j \neq i} |\langle v_j, u_i \rangle |^2 \leqslant |V^* u_i - e_i|^2 \leqslant \delta^2 $, and finally
\[|u_i - \tilde{u}_i|= |\proj_{H_i} (u_i)| = |V_i ({V_i}^* V_i)^{-1} V_i^* u_i |\leqslant \Vert V \Vert \Vert (V^*_i V_i)^{-1}\Vert \delta \leqslant \frac{ \delta N }{h}. \] 
Moreover, $\langle u_i, v_i\rangle - 1 $ is the $i$-th diagonal entry of $U^*V-I_r$ and thus its modulus is smaller than $\Vert U^* V - I_r\Vert$, so we have $|\langle u_i, v_i\rangle - 1 |\leqslant \delta$, and
\begin{align*}
|\langle \tilde{u}_i, v_i \rangle - 1 |&\leqslant |\langle u_i, v_i\rangle - 1| + |\langle u_i, v_i\rangle - \langle \tilde{u}_i, v_i \rangle | \\
&\leqslant \delta + |u_i - \tilde{u}_i||v_i| \\
&\leqslant \delta (1+ N^2 h^{-1} ) \\
& \leqslant 2 \delta N^2  h^{-1},
\end{align*}
where we have used that $h \leq \Vert V \Vert^2 $. When $ 0 \leq t \leq 1/2$, we have $|(1+t)^{-1} - 1 |\leqslant 3t/2$, thus, since $\delta \leqslant h / ( 4 N^2)$, we have 
\[\left|\frac{1}{\langle \tilde{u}_i, v_i \rangle} - 1 \right|\leqslant  3 \delta N^2  h^{-1}. \]

We now write
\begin{align*}
| \bar{u}_i-u_i | &= \left| \frac{\tilde{u}_i}{\langle \tilde{u}_i, v_i\rangle}-u_i \right| \\
&\leqslant  |\tilde{u}_i - u_i| +  \left| \frac{1}{\langle \tilde{u}_i, v_i\rangle}-1  \right| |\tilde{u}_i| \\
&\leqslant \delta N h^{-1}+ 3 \delta h^{-1}  N^2  \Vert U \Vert  \\
&\leqslant \delta  N h^{-1}  (1+3N^2).
\end{align*}
The last term is bounded by $4\delta N^3/h$, so using the elementary inequalities $\Vert M \Vert \leqslant \Vert M \Vert_F $, we get our estimation expressing how $\bar{U}$ and $U$ are close: \begin{equation}\label{uubar}
\Vert \bar{U}-U \Vert \leqslant \frac{4 \sqrt r\delta N^3}{h}
\end{equation}
Finally, from the definitions of $H_i$ and $\bar{u}_i$, we have $(\bar{U}^* V)_{i,j} = \langle \bar{u}_i, v_j \rangle = 0$ if $i \neq j$, and $\langle \bar{u}_i, v_i \rangle =1$, a crucial fact which can also be written as
\begin{equation}\label{UstarV}
\bar{U}^* V = V^* \bar{U} = I_r.
\end{equation}
Together, \eqref{uubar}-\eqref{UstarV} achieve our preliminary and show that $\bar{U}$ is the suitable pseudo-inverse of $V$ which is close to $U$. We now study the matrix $S$. 

 We set $\Sigma = \mathrm{diag}(\theta_i )$ and $\bar{S}=\bar{U}\Sigma V^*$; we have
\begin{align*}\Vert S - \bar{S}\Vert &\leqslant \Vert V \Vert \Vert \Sigma \Vert  \Vert \bar{U} - U \Vert \\
&\leqslant   \frac{4\sqrt r\delta  N^4}{h} \defeq \eta', \end{align*}
where we have used that $\| \Sigma \| \leq 1$.  We now claim that $\bar{S}$ is diagonalizable with eigenvalues $\theta_1, \dotsc, \theta_r$. Indeed, $\bar{u}_i$ is an eigenvector with eigenvalue $\theta_i \neq 0$, and every basis of $\mathrm{im}(V)^\perp$ is a family of eigenvectors associated with the eigenvalue zero. We note $\bar{S}=P^{-1}\Sigma'P$ with $\Sigma'=\mathrm{diag}(\theta_1, \dotsc, \theta_r, 0, \dotsc, 0)$ and $P$ its diagonalization matrix. 

The matrices $A$ and $\bar{S}$ are close:
\[\Vert A- \bar{S}\Vert \leqslant \Vert S - \bar{S}\Vert + \Vert A - S \Vert \leqslant \eta' +  \eta.\]

We may thus apply Bauer-Fike Theorem \ref{th:BF} to $A$ and $\bar S$ with $\varepsilon = (\eta+\eta') \Vert P \Vert \Vert P^{-1} \Vert$. It thus remains to prove $\Vert P \Vert \Vert P^{-1} \Vert \leq 12N^3$. In the remaining of the proof. We compute $P$ and $\Vert P \Vert \Vert P^{-1} \Vert$.

 Let $K = \SPAN(v_1, \dotsc, v_r)^{\perp}=\mathrm{im}(V)^\perp=\ker(V^*)$; the dimension of $K$ is $n-r$. Let us choose any orthonormal basis $(w_{r+1}, \dotsc, w_n)$ of $K$ and set up $P=(\bar{U}, W)$ where $W$ is the $n \times (n-r)$ matrix whose columns are the $w_k$'s. Then, the family $(\bar{u}_1, \dotsc, \bar{u}_r, w_{r+1}, \dotsc, w_n)$ is a diagonalization basis for the matrix $S$: more precisely, we have $S\bar{u}_i = \bar{U}\Sigma V^* \bar{u}_i = \theta_i \bar{u}_i$, and $Sw_j = 0$. We now claim that the inverse of $P$ is given by
 \begin{equation}
 P^{-1} = \begin{pmatrix}
 V^* \\ -W^*\bar{U}V^* + W^* 
 \end{pmatrix}.
 \end{equation}
We can directly check this using the relations \eqref{UstarV}, the orthonormality relation $W^* W = I_{n-r}$ and $V^* W = 0$, which stems from the choice of $W$ as a basis for $\ker(V^*)$. Indeed, 
 \begin{align*}
 \begin{pmatrix}
 V^* \\ -W^*\bar{U}V^* + W^* 
 \end{pmatrix}P&= \begin{pmatrix}
 V^* \\ -W^*\bar{U}V^* + W^* 
 \end{pmatrix} \begin{pmatrix}
 \bar{U} & W 
 \end{pmatrix} \\
 &= \begin{pmatrix}
 V^* \bar{U} & V^* W \\ -W^* \bar{U} V^* \bar{U} + W^* \bar{U} & -W^* \bar{U} V^* W + W^* W
 \end{pmatrix} \\
 &=\begin{pmatrix}
I_r & 0 \\ -W^* \bar{U}  + W^* \bar{U} & W^* W
 \end{pmatrix} \\
  &=\begin{pmatrix}
I_r & 0 \\ 0 & I_{n-r}
 \end{pmatrix}  = I_n.
 \end{align*}
To compute the condition number of $P$ we use the elementary Lemma \ref{lem:norme}, stated hereafter. Clearly, $\Vert W \Vert = 1$, hence by the lemma
\[\Vert P \Vert \leqslant \sqrt{2}\max ( 1 , \Vert \bar{U}\Vert ) . \]
For $P^{-1}$ we note that $-W^* \bar{U}V^* + W^* = W^* (I_n -\bar{U}V^* )$, hence $\Vert -W^* \bar{U}V^* + W^* \Vert \leqslant \Vert W \Vert \Vert I_n - \bar{U}V^* \Vert$, and
\[\Vert P^{-1} \Vert \leqslant \sqrt{2} ( 1+ \Vert \bar{U}\Vert \Vert V \Vert ). \]
We thus get
\[\Vert P \Vert \Vert P^{-1} \Vert \leqslant 2 \max(1,\Vert \bar{U} \Vert) (1+\Vert \bar{U} \Vert \Vert V \Vert ). \]

From \eqref{uubar}, $\Vert \bar{U} \Vert \leqslant  \| U \| + \| \bar U - U \| \leqslant N + 4 \sqrt r\delta N^3/h \leqslant 2 N$ from our assumption on $\delta$. Finally, we get
\[\Vert P \Vert \Vert P^{-1} \Vert \leqslant 4 N ( 1 + 2 N^2) \leq 12 N^3. \]
It concludes the proof.
\end{proof}

As promised, here is a simple lemma used in the preceding proof.

\begin{lem}\label{lem:norme}
Let $M_1 \in \mathscr{M}_{n,r}(\mathbb{R})$ and $M_2 \in \mathscr{M}_{n,n-r}(\mathbb{R})$ be two matrices; we set $M = (M_1, M_2) \in \mathscr{M}_{n,n}(\mathbb{R})$. Then 
\[\Vert M \Vert \leqslant \sqrt{2}\max \{\Vert M_1 \Vert, \Vert M_2 \Vert \}. \] 
\end{lem}

\begin{proof}
For any $x \in \mathbb{R}^r,y \in \mathbb{R}^{n-r}$ and $z=(x^*,y^*)^*$ we have 
\begin{align*}|Mz|=|M_1x+M_2y|&\leqslant \Vert M_1 \Vert |x|+ \Vert M_2 \Vert |y| \\
&\leqslant \max \{\Vert M_1 \Vert, \Vert M_2 \Vert \} \sqrt{2}\sqrt{|x|^2+|y|^2}\\
&=\max \{\Vert M_1 \Vert, \Vert M_2 \Vert \} \sqrt{2}|z|  \end{align*}
which is valid for any $z \in \mathbb{R}^n$. 
\end{proof}

\subsection{Eigenvector perturbation}

Perturbation theory for invariant subspaces is a delicate matter. For simplicity, we will restrict ourselves here to the case of invariant subspaces of dimension $1$, that is eigenvectors of simple eigenvalues.

We consider the matrix $A = S + E$ with $S = U \Sigma V^*$ as in Theorem \ref{thm:linalg}. If $\Sigma = \mathrm{diag}(\theta_1, \dotsc, \theta_r)$, we define the spectral ratio gap at $\theta_i$ as 
\begin{equation}
\label{linalg:separation_spectrum}
g_i = \min_{1 \leq j \neq i \leq n}|1 - \theta_j / \theta_i|.
\end{equation}
where we have set $\theta_i = 0$ for all $i \in [n] \backslash [r]$. 

\begin{theorem}\label{thm:linalgvec}
Under the assumptions of Theorem \ref{thm:linalg}, we consider the ordering of the eigenvalues of $A$, $\lambda_1,\ldots,\lambda_n$ defined in Theorem \ref{thm:linalg}. Set $$\veps = 12 N^ 3  \left(\eta+\frac{ 4 \sqrt r \delta N^4 \max_i |\theta_i| }{h} \right).$$
Then for any $i \in [r]$, if $g_i \theta_i > 2 \veps$, $\lambda_i$ is a simple eigenvalue and any unit right-eigenvector of $A$ associated with $\lambda_i$, denoted $\psi_i$, satisfies
\begin{equation}
\left|\langle \psi_i , \frac{u_{i}}{|u_{i}|} \rangle \right|^2 \geq 1 -  \frac{12}{g_i^2 h^2}\left( \frac{\veps}{\min_{j\in [r]} |\theta_j|} + \frac{ N^2 \eta }{|\theta_i| - 2 \veps} \right)^2,
\end{equation}
and the same result holds for a unit left-eigenvector and $v_i/|v_i|$.
\end{theorem}

The proof is inspired by the so-called Neumann trick. We start with an elementary statement. 

\begin{lem}\label{lem:oih} Let $D=\mathrm{diag}(d_1,\ldots,d_r)$ with $d_i$ complex numbers such that $\min_{j \neq 1} |d_1 - d_j|=g>0$. If $w = (w_1,\ldots, w_r) \in \mathbb{C}^r$ is a unit vector then $| w_1 |^2\geqslant  1 - |d_1 w -D w|^2/ g ^2$. 
\end{lem}

\begin{proof}
We have  $1-|w_1|^2=|w_2|^2+\dotsb + |w_r|^2 \leqslant |(d_1 I_r-D)w |^2/g^2 $.
\end{proof}

\begin{proof}[Proof of Theorem \ref{thm:linalgvec}]
We fix some $i \in [r]$, and we note $\lambda = \lambda_i$, $\theta = \theta_i$ and $g = g_i$. We set $E = A - S$, so that, by assumptions, $\Vert E \Vert < \eta$. By scale invariance, we may assume without loss of generality that $(\min_i |\theta_i|)^{-1} = \| \Sigma^{-1} \| =1$. Since $g \theta > 2 \veps$,  from Theorem \ref{thm:linalg}, it implies that $|\lambda  - \theta | < \veps$ (the ball of radius $B(\theta_,\veps)$ does not intersect any of the balls of radius $B(\theta_j,\veps)$ with $j \ne i$). In particular $|\lambda|  > |\theta| - \veps \geq \gamma - \veps > \veps > \eta$  and hence $\lambda$ is not an eigenvalue of $E$ and $\det(E-\lambda I_n)\neq 0$. Then, using Sylvester's identity, 
\begin{align*}
0=\det(A - \lambda I_n)&=\det(U \Sigma V^*+E-\lambda I_n) \\
&=\det(E-\lambda I_n)\det(I_n+U{\Sigma}V^*(E-\lambda I_n)^{-1}) \\
&= \det(E-\lambda I_n) \det(I_r+{\Sigma}V^* (E-\lambda I_n)^{-1} U )
\end{align*}
which implies that $ \det(I_r+{\Sigma}V^* (E-\lambda I_n)^{-1} U )=0$. Consequently, there is a vector $w \in \mathbb{C}^r$ with $|w|=1$ such that 
\begin{equation}\label{alg:w}
0=w + \Sigma V^* (E-\lambda I_n)^{-1} U w.
\end{equation}

We set $\phi=\lambda (E-\lambda)^{-1}Uw$. This vector cannot be zero (because $w \neq 0$ and ${\Sigma}V^*$ is nonsingular), and it is actually an eigenvector of $A$ associated with the eigenvalue $\lambda$: 
\begin{align*}
(A - \lambda I_n)\phi &= \lambda (U{\Sigma}V^*  +E - \lambda I_n)(E-\lambda I_n)^{-1}Uw \\
&= \lambda  (U{\Sigma}V^*(E-\lambda I_n)^{-1}U +U)w \\
&= \lambda U ({\Sigma}V^*(E-\lambda I_n)^{-1}U +I_r)w =0.
\end{align*}

Our goal is now to prove that $w$ is close to the vector $e_i$ of the canonical basis. We set $C= \lambda V^* (E - \lambda I_n )^{-1} U + I_r$. In \eqref{alg:w}, we see that $0=\lambda w - {\Sigma}w +{\Sigma}Cw$, hence $|\lambda \Sigma^{-1} w- w |\leqslant  \Vert C \Vert$ and we want to bound $C$. This is done as follows:
\begin{align*}
\Vert C \Vert=\Vert \lambda V^* (E - \lambda I_n )^{-1} U + I_r \Vert &= \Vert -V^* U + I_r +V^*[\lambda (E - \lambda I_n)^{-1}+I_n]U \Vert \\
&\leqslant \Vert I_r - V^* U \Vert + \Vert U \Vert \Vert V \Vert   \Vert \lambda (E - \lambda I_n)^{-1} + I_n \Vert .
\end{align*}
From the Neumann series expansion, 
\[\Vert (E-\lambda)^{-1} -(-\lambda^{-1})I_n \Vert \leqslant \frac{1}{|\lambda|}\frac{\Vert E \Vert/|\lambda| }{1-\Vert E \Vert/|\lambda|}.\]
We thus obtain the bound:
\[ \Vert C \Vert \leqslant  \delta +  \frac{N^2 \eta}{|\lambda|-\eta}=: \delta' \]
and we obtain $|\lambda \Sigma^{-1} w - w |<\delta'$.  We thus have $|\theta  \Sigma^{-1} w -  w|\leqslant |\theta - \lambda| \| \Sigma^{-1} \| +|\lambda  \Sigma^{-1} w -   w|<\varepsilon+\delta'$.

We apply Lemma \ref{lem:oih} to $D = \theta \Sigma^{-1}$ and  we get that $|w_i|^2 \geqslant  1 - (\varepsilon +\delta')^2 /g^2 $. We may assume without loss of generality that $w_i$ is a non-negative real number. We deduce that $| w - e_i|^2 = 2 - 2 \Re (w_i) \leq 2(\varepsilon +\delta')^2 /g^2 $. Then we find that 
\begin{align*}|\phi - u_i|& \leqslant \Vert \lambda(E - \lambda)^{-1} U  - I_r \Vert + | w - e_i|
\\ & \leq \frac{\| U \|  \Vert E \Vert}{|\lambda|-\Vert E \Vert} + \frac{\sqrt 2(\varepsilon +\delta')}{g} =: \veps' .
\end{align*}

But keep in mind that $\phi$ might not be normalized. A unit-norm right eigenvector of $A$ associated with $\lambda$ is $\phi/|\phi|$. However $||\phi|-|u_i|| \leq | \phi - u_i|$, so that $|\phi/|\phi|-u_i/|u_i||\leqslant 2|u_i-\phi|/|u_i|$, and finally, using $|u_i|>h$, we get 
\begin{equation*} \label{eq:ndend}|\phi/|\phi| - u_i/|u_i||\leqslant 2 \veps' / h. \end{equation*}
 Simplifying the error term leads to the bound \eqref{linalg:vecs} (observe that $g \leq 1$ and any unit eigenvector $\psi$ associated with $\lambda$ satisfies $|\langle \psi,u_i/|u_i| \rangle|^2 = |\langle \phi/|\phi|,u_i/|u_i| \rangle|^2 = 1 - |\phi/|\phi| - u_i/|u_i||^2/2$). \end{proof}

\subsection{Powers of matrices}

We now use the preceding perturbation theorems, but for powers of matrices. This will be necessary since the perturbations that we will encounter are not small in norm unless we raise them at a high power. We need some variations on the hypothesis which are tuned for our needs. We emphasize the fact that they are certainly not optimal, but rather suited to our subsequent needs. Let $U,U',V,V'$ be four $n \times r$ matrices and $\ell,\ell'$ be two integers. We set $S = U\Sigma^\ell V^*$ and $S'=U'\Sigma^{\ell'}(V')^*$, where $\Sigma = \mathrm{diag}(\theta_1, \dotsc, \theta_r)$ and the $\theta_i$'s are real numbers (unlike the previous paragraphs where they were complex).

Let $A \in \mathscr{M}_{n}(\mathbb{C})$.  We make the following assumptions. 
\begin{enumerate}
\item The integers $\ell$ and $\ell'$ are mutually prime. 
\item There are numbers $\eta,\eta'>0$ such that
\begin{equation}
\Vert A^\ell - S \Vert \leqslant \eta \quad \text{and} \quad \Vert A^{\ell'} - S' \Vert \leqslant \eta'
\end{equation}
\item The matrices $U, U', V, V'$ are well-conditioned:
\begin{itemize}
\item The operator norms of $U$ and $V$ are smaller than $N$.
\item The smallest eigenvalue of $U^*U, V^*V$ are greater than $h >0$.
\item The operator norm of $ U^* V - I_r $ is smaller that $\delta$.
\item The above properties hold for $U'$ and $V'$ for some constants $N',h',\delta'$. 
\end{itemize}
\end{enumerate}

The following perturbation theorem holds.

\begin{theorem}\label{thm:linalg:powers}
Under the above assumptions, we set $$\veps = 12 N^ 3  \left(\eta+\frac{ 4 \sqrt r \delta N^4 \max_i |\theta_i|^\ell }{h} \right)$$ and $\veps'$ is defined similarly with the parameters $(\ell',\eta',N',h',\delta')$. Assume that for all $i \in [r]$,  
\begin{equation}\label{linalg:separation:powers}
|\theta_i|^\ell > 2 \ell' r \veps \hbox{ and } \quad |\theta_i|^{\ell'} > 2 \ell r  \veps'.
\end{equation}
Then, there is an ordering of the $r$ largest eigenvalues of $A$ in modulus, say $\lambda_1, \ldots, \lambda_r$ such that  
\begin{equation}\label{linalg:result_powers}
|\lambda_{i} - \theta_{i} |<  \frac{4 r \veps}{ \ell |\theta_i |^{\ell-1}} .
\end{equation}
All the other $n-r$ eigenvalues of $A$ have modulus smaller than $\varepsilon^{1/\ell}$. 

Moreover, for any $i \in [r]$, if $g_i \defeq \min_{j \neq i}|1 - \theta_j^\ell / \theta_i^\ell| > 2 \veps \theta_i^{-\ell}$, then $\lambda_i$ is a simple eigenvalue and any unit right-eigenvector of $A$ associated with $\lambda_i$, denoted $\psi_i$, satisfies
\begin{equation}\label{linalg:vecs}
\left|\langle \psi_i , \frac{u_{i}}{|u_{i}|} \rangle \right|^2 \geq 1 -  \frac{12}{g_i^2 h^2}\left( \frac{\veps}{\min_{j \in [r]} |\theta_j|^\ell } + \frac{N^2 \eta }{|\theta_i|^\ell - 2 \veps} \right)^2
\end{equation}
and the same result holds for a unit left-eigenvector and $v_i/|v_i|$.
\end{theorem}

\begin{proof}
We start with the statement on eigenvalues. We apply Theorem \ref{thm:linalg} to $A^\ell$ and $S$ and to $A^{\ell'}$ and $S'$. There are two permutations $\pi, \pi' \in \mathfrak{S}_r$ such that $|\lambda_i^\ell - \theta_{\pi(i)}^\ell|\leqslant \veps_0 = (2r-1)\varepsilon$ and $|\lambda_i^{\ell'} - \theta_{\pi'(i)}^\ell|\leqslant \veps'_0 = (2r-1)\varepsilon'$. Indeed, by assumption \eqref{linalg:separation:powers}, the ball of radius $\veps$ centered at $0$ does not intersect any of the balls centered at $\theta_i$, $i \in [r]$. We may assume that $\pi$ is the identity.

We fix some $i \in [r]$. Our goal is to show that indeed, $\lambda_i$ is close to $\theta_{i}$. We first  control the argument of $\lambda_i$ by using that two mutually prime powers of $\lambda_i$ are close to the real axis (since the $\theta_j$'s are real). More precisely, we set $x= \theta_{i}$ and $y=\theta_{\pi'(i)}$, so that 
\[|\lambda_i^\ell/x^\ell - 1 | \leqslant u \defeq \frac{\varepsilon_0}{|x|^\ell}\]
and
\[|\lambda_i^{\ell'}/y^{\ell'} - 1 | \leqslant \,u' \defeq \frac{\varepsilon'_0}{|y|^{\ell'}}.\]
The polar decomposition of $\lambda_i$ is written $\lambda_i = |\lambda_i| e^{\ic \omega}$ with $\omega$ real. The argument of $(\lambda_i/x)^\ell$ is between $-\tau$ and $\tau$ where $\tau = \arctan(u/2)/2$, and we have $|\arctan(t)|\leqslant |t|$, so the argument of $(\lambda_i/x)^\ell$ has smaller absolute value than $u/4$, a similar fact holding for $(\lambda_i/y)^{\ell'}$. As a consequence, there are two integers $p,p'$ and two numbers $s \in (-u/4, u/4)$ and $s' \in (-u'/4,u'/4)$ such that 
\[
\ell\omega  = p\pi + s \qquad \ell'\omega  = p'\pi + s'
\]
This implies 
\[p\ell' - p'\ell = \frac{s'\ell - s\ell'}{\pi}.\]
The LHS is an integer and in the RHS, and the terms $s'\ell$ and $s\ell'$ have magnitude smaller than $u'\ell/4$ and $u\ell'/4$, which were supposed to be smaller than $1$ in our hypothesis \eqref{linalg:separation:powers}, so the whole RHS can only be zero and $p\ell'=p'\ell$. However, as $\ell \wedge \ell'=1$ we see that $\ell$ divides $p$ and $\ell'$ divides $p'$, so $\omega = k\pi + s/\ell$ for some $k \in \mathbb{Z}$: the complex number $\lambda$ has argument close to $0$ or $\pi$. We can also see that if $\theta_{i}$ is positive, then $k$ is even and we can indeed take $\omega=s/\ell$. Otherwise, $k$ is odd and we can take $\omega = \pi+s/\ell$.

We may now come back to the relation $|\lambda_i^\ell - \theta_{i}^\ell|\leqslant \varepsilon_0$, which can also be written as 
\begin{equation*}
\lambda_i^\ell = \theta_{i} ^\ell (1+z)
\end{equation*}
with $|z|\leqslant \varepsilon_0  /|\theta_{i}|^\ell$. When taking the modulus, we get $|\lambda_i|=|\theta_{i}||1+z|^\frac{1}{\ell}$, and from the inequality $||1+z|^\frac{1}{\ell}-1|\leqslant |z|/\ell$, we finally find out that
\begin{equation*}
||\lambda_i|-|\theta_{i}||\leqslant \frac{ \varepsilon_0}{\ell |\theta_{i}|^{\ell-1}}. 
\end{equation*}

We may finally combining the bounds on arguments and modulus. If $\theta_{i}$ is positive, we saw that the argument of $\lambda_i$ is $s/\ell$. Writing $\theta_i = \theta_i e^{\ic s/\ell} + \theta_i - \theta_i e^{\ic s/\ell}$, we find 
\begin{align*}
|\lambda_i - \theta_{i}|&\leqslant 
 ||\lambda_i|-\theta_{i}|+| \theta_i||e^{\ic s/\ell} - 1 |  \\
&\leqslant (\varepsilon_0/|\theta_{i}|^\ell+  |s||\theta_i|) /\ell \\
&\leqslant 2\varepsilon_0  / (\ell |\theta_{i}|^{\ell-1}).
\end{align*}
This gives the claimed statement. On the other hand, if $\theta_{i}$ is negative, then the argument of $\lambda_i$ is $\pi+s/\ell$ and in this case, $\lambda_i = -|\lambda_i|e^{\ic s/\ell}$ and the same argument holds.

The proof of the statement for the eigenvector is then a consequence of Theorem \ref{thm:linalgvec} applied to $A^\ell$.
\end{proof}

%
%
%

\section{Proof of Theorem \ref{thm:1}}\label{sec:proofs}

\subsection{Notation}
We fix a matrix $P$ as described above Theorem \ref{thm:1}.

If $r_0 \geq 1$, we define $\Phi = (\varphi_1, \dotsc, \varphi_{r_0})$ and $\Sigma = \mathrm{diag}(\mu_1, \dotsc, \mu_{r_0})$. The columns of $\Phi$ form an orthonormal family, hence $\Phi^* \Phi = I_{r_0}$.

Recall the parameter $\ell$ defined in \eqref{def:l}. The `candidate eigenvectors' are $u_i= A^\ell \varphi_i/\mu_i^\ell$ and $v_i = (A^*)^\ell \varphi_i /\mu_i^\ell$, or to put it in matrix form they are the columns of
\begin{equation*}
U = A^\ell \Phi \Sigma^{-\ell} \qquad \text{ and } \qquad V = (A^*)^\ell \Phi \Sigma^{-\ell}.
\end{equation*}
We set
\begin{equation*}
S = U\Sigma^\ell V^*.
\end{equation*}

We finally introduce the vector spaces 
\begin{equation}\label{eq:defH}
H = \mathrm{vect}(v_1, \dotsc, v_{r_0})=\mathrm{im}(V) \quad \hbox{ and } \quad H' = \mathrm{vect}(u_1, \dotsc, u_{r_0})=\mathrm{im}(U).
\end{equation}

Finally, if $r_0 =0$, then $S$ is simply set to be the zero matrix, and $H,H'$ are the trivial vector spaces.

\subsection{Comments on the assumptions} 

We note that the problem is homogeneous. From now on, we will thus assume without loss of generality that 
$$
\mu_1 = 1.
$$

We observe also that the statement is trivial to check if $\ell = 0$. In particular, we will assume in the sequel without loss of generality that 
\begin{equation}\label{eq:boundd}
\log (n) \geq 8 \log (2d).
\end{equation}

We will also check easily in \eqref{lower_bound_on_rho} that $\rho \geq \mu^2_1 = 1$. In particular 
$$
\thresh \geq \frac{1}{\sqrt d}\quad \hbox{ and } \quad \tau^{-1}_0 \leq \sqrt d
$$
It follows that in the statement of  Theorem \ref{thm:1}, we have
\begin{equation}\label{eq:taudn}
\tau^{-2\ell}_0 \leq d^{\ell} \leq n^{1/8}.
\end{equation}
The assumption $C_0 \leq \tau_{0}^{-2\ell}$ implies that 
\begin{equation} \label{eq:defb}
C_0 = c  r    r_0^{4} b^{44}  \ln(n)^{16}  \leq n^{1/8}.
\end{equation}

\subsection{Algebraic structure of $A,U,V$ with respect to $H$}
The behaviour of the matrices $U,V$ is dictated by a \emph{theoretical covariance matrix} $\Gamma^{(\ell)}$, which is a good approximation of the Gram matrices of the columns of $U$ and $V$. It is defined as follows: let $i,j$ be in $[r_0]$ and $t$ be an integer. We will note $\varphi^{i,j}$ for the Hadamard product between $\varphi_i$ and $\varphi_j$:
\[\varphi^{i,j}(x) = \varphi_i(x) \varphi_j(x). \]
Then, we define the matrix $\Gamma^{(t)} \in \mathscr{M}_{r}( \mathbb{R} )$:
\begin{equation}\label{def:gammat}
\Gamma^{(t)}_{i,j} = \sum_{s=0}^t \frac{\langle \mathbf{1}, Q^s \varphi^{i,j}\rangle}{(\mu_i \mu_j d)^s}.
\end{equation}
The diagonal entries $\Gamma_{i,i}^{(\ell)}$ are exactly the $\gamma_i$ appearing in Theorem \ref{thm:1}. The next lemma gathers useful properties of $\Gamma^{(t)}$.

\begin{lem}[Properties of $\Gamma^{(t)}$]\label{lem:Gammat}For any $t$, the matrix $\Gamma^{(t)}$ is a semi-definite positive matrix with eigenvalues greater than $1$, and with
\[1\leqslant \Vert \Gamma^{(t)} \Vert \leqslant  r_0 b^8 \frac{1 - \tau_0^{2(t+1)} }{1 - \tau_0^2}.\]
\end{lem}
 
 This lemma will be proved in Section \ref{sec:algebra_incoherence}. The main tool for the subsequent analysis of $U$ and $V$ is the following theorem, which could also be of independent interest. 

\begin{theorem}[Algebraic structure of $U$ and $V$]\label{thm:algebra}
There are a universal constant $c >0$ and  an event with probability greater than $1-c n^{-1/4}$ such that the following holds:
\begin{equation}\label{eq:PhiAlPhi}
\Vert V^* A^\ell U - \Sigma^\ell \Vert \leqslant C_1  n^{-1/4} \tau_0^{2\ell}\thresh^{\ell}
\end{equation}
\begin{equation}\label{eq:U*V}
\Vert U^* V - I_{r_0} \Vert \leqslant C_1 n^{-1/4} \tau_0^{2 \ell}
\end{equation}
\begin{equation}\label{eq:PhiAlPhi0}
\Vert \Phi^* U - I_{r_0} \Vert \leqslant C_1  n^{-1/4} \tau_0^{2\ell}
\end{equation}
\begin{equation}\label{eq:U*U}
\Vert U^* U -   \Gamma^{(\ell)} \Vert \leqslant C_1 n^{-1/4} \tau_0^{2 \ell}
\end{equation}
\begin{equation}\label{eq:V*V}
\Vert V^* V -   \Gamma^{(\ell)} \Vert \leqslant C_1 n^{-1/4} \tau_0^{2 \ell}
\end{equation}
\begin{equation}\label{norm_on_orthogonal}
\Vert A^\ell \proj_{H^\perp} \Vert \leqslant C_2 \thresh ^\ell 
\end{equation}
\begin{equation}\label{norm_on_orthogonal2}
\Vert  \proj_{{H'}^\perp}  A^\ell \Vert \leqslant C_2 \thresh^{\ell},
\end{equation}
where $C_1 =  c r_0b^2\ln(n)^{5/2}$ and $C_2 =c b^{20} r \ln(n)^{12} $.
\end{theorem}

\begin{remark}\label{remark:prime}
The same statements holds, with the same constants, if we replace $\ell$ by $\ell' = \ell+1$. Note that $\ell'$ and $\ell$ are then mutually prime.
\end{remark}

The proof of this theorem occupies the next sections of this paper. We now use this theorem to prove all the results mentioned before. 

\subsection{Proof of Theorem \ref{thm:1}} 
\label{subsec:proofth1}
We consider the intersection of the two events of Theorem \ref{thm:algebra} for $\ell$ and $\ell' = \ell+1$ which has probability at least $1 - cn^{-1/4}$. Let us call $\mathcal E_n$ this event.  Our goal is to apply Theorem \ref{thm:linalg:powers} to $A^\ell$ and $S$ and to $A^{\ell'}$ and $S'$ on $\mathcal E_n$. Our main task will be to check the three conditions of Theorem \ref{thm:linalg:powers}; we will focus on checking them for $U$ and $V$, the statements for $U',V'$ working obviously in the same way. 

We will use that \eqref{eq:defb} implies that 
$ C_1 n^{-1/4} $ goes to $0$ as $n$ goes to infinity, uniformly in $r,b$ satisfying \eqref{eq:defb}. In particular, for all $n \geq n_0$ large enough
\begin{equation}
\label{eq:errC1}
C_1 n^{-1/4} \leq 1/2.
\end{equation}
In the sequel, we always assume that $\mathcal E_n$ holds and that $n \geq n_0$. 

\medskip

\noindent \emph{Condition 1 of Theorem \ref{thm:linalg:powers}}. This is settled by Remark \ref{remark:prime}.

\medskip

\noindent \emph{Condition 3 of Theorem \ref{thm:linalg:powers}}. From \eqref{eq:U*U} and Lemma \ref{lem:Gammat}, we have,  
\begin{equation} \label{norm_of_U}
\Vert U \Vert^2  =  \Vert U^* U \Vert  \leqslant (\ell+1)r_0 b^{8} +  C_1 n^{-1/4}  \leqslant r_0 b^{8} \ln(n)
\end{equation}
the last line coming from \eqref{eq:errC1}. The same inequality holds for $\Vert V \Vert^2$. We find that 
\[\|U\| \vee \| V \| \leq  N \defeq \sqrt {r_0 b^{8} \ln(n)}.\]

For the conditioning properties of $U,V$, we deduce from \eqref{eq:errC1} and Lemma \ref{lem:Gammat} that
\begin{equation}\label{lambdamin}
      \Vert U^*U \Vert \wedge \Vert V^*V \Vert  \geq   \lambda_{\min}(\Gamma^{(\ell)}) - C_1 n^{0.1} \geqslant h \defeq 1/2. 
\end{equation}

Similarly, from \eqref{eq:U*V}, we have
\begin{equation}\label{U*V}
    \Vert U^* V - I_{r_0} \Vert \leqslant C_1 n^{-1/4} \tau_0^{2\ell} \defeq \delta.
\end{equation}

This gives condition $3$ for $U,V$. We note that \eqref{lambdamin} implies that:
 \begin{equation}
 \label{V*V-1}
    \Vert (U^* U)^{-1} \Vert \vee     \Vert (V^* V)^{-1} \Vert \leqslant 2.
 \end{equation}
 
\medskip

\noindent \emph{Condition 2 of Theorem \ref{thm:linalg:powers}}. This condition requires more work. We have 
$$\proj_H=V(V^* V)^{-1}V^* \quad \hbox{ and } \quad \proj_{H'}=U(U^* U)^{-1}U^*.$$ 
We also note that $\proj_{H^\perp}=I_n - \proj_H$ the projection matrix on $H^\perp$. Since $S \proj_{H^\perp}= \proj_{{H'}^\perp} S = 0$, we have  $S = \proj_{{H'}^\perp} S \proj_H $ and 
\begin{align}
    \Vert A^\ell - S \Vert     &\leqslant  \Vert  \proj_{{H'}} A^\ell \proj_H - S \Vert  + \Vert A^\ell \proj_{H^\perp} \Vert +  \Vert  \proj_{{H'}^\perp}  A^\ell \Vert \nonumber \\
    &\leqslant   \Vert U \Vert  \Vert (U^* U)^{-1}U^*  A^\ell V(V^* V)^{-1}  -  D^\ell   \Vert   \Vert V \Vert + \Vert A^\ell \proj_{H^\perp} \Vert + \Vert  \proj_{{H'}^\perp}  A^\ell \Vert, \label{AL-S}
\end{align}
where at the second line, we have used that $U D^\ell V^* = S$.  To bound the above expression, we will use that  the fact that $V^* U - I_{r_0}$ is small implies that $U$ is close to $\tilde U \defeq V (V^* V)^{-1}$ and that $V$ is close to $\tilde V \defeq U (U^* U)^{-1}$. More precisely, we write
\begin{align*}
    \proj_{H}U &= \tilde U V^* U \\
    &= \tilde U + E_1 
\end{align*}
where $E_1  \defeq  \tilde U (V^* U-I_{r_0})$. From \eqref{norm_of_U}-\eqref{U*V}-\eqref{V*V-1}, we find that $$\Vert E_1 \Vert \leqslant 2 \delta N.$$
We may thus decompose $U$ as follows: 
\[U = \tilde U + E_1 +  \proj_{H^\perp}U .\]
Similarly, we find 
$$
V =  \tilde V + E_2 + \proj_{{H'}^\perp}V
$$
with $E_2 = \tilde V(U^*V-I_{r_0})$ and $\Vert E_2 \Vert \leqslant 2 \delta N$. We get the following: 
\begin{align*}
\Vert \tilde V^*  A^\ell \tilde U -  \Sigma^\ell   \Vert & \leq \Vert \tilde V^*  A^\ell U  -  \Sigma^\ell   \Vert + \Vert \tilde V \Vert \| A^\ell \| \| E_1 \| + \Vert A^\ell \proj_{H^\perp} \Vert \| U\| \\
& \leq \Vert   V^*  A^\ell U  -  \Sigma^\ell   \Vert +  \Vert \tilde V \Vert \| A^\ell \| \| E_1 \| + \Vert A^\ell \proj_{H^\perp} \Vert \| U\| +  \Vert U \Vert \| A^\ell \| \| E_1 \| + \Vert \proj_{{H'}^\perp} A^\ell \Vert \| V\|.
\end{align*} 
We use this last inequality in \eqref{AL-S} and use \eqref{eq:PhiAlPhi}. We obtain the bound:
\begin{align*}
    \Vert A^\ell - S \Vert     &\leqslant  \delta N^2 \thresh^\ell  + 4 N^3 C_2 \thresh^\ell + 4 \delta N^4 \| A^\ell \|.
\end{align*}
From our choice of parameters \eqref{eq:defb}, we have
\begin{equation}\label{eq:deN2}
4 \delta N^4 \leq 1/2.
\end{equation}
It follows that
\begin{align}
    \Vert A^\ell - S \Vert     &\leqslant   5 N^3 C_2 \thresh^\ell + 4 \delta N^4 \| A^\ell \|.\label{AL-S2}
\end{align}

We may use \eqref{AL-S2} and $\| S \| \leq \| U \| \| \Sigma^\ell \| \|V \| \leq N^2$ to upper bound $\| A^\ell \|$.  
Indeedn we write $\| A^\ell \| \leq \| S \| +  \Vert A^\ell - S \Vert \leq N^2 + \Vert A^\ell - S \Vert$, we find from \eqref{AL-S2} and \eqref{eq:deN2}: 
\begin{equation}\label{eq:normAl}
\| A ^\ell \| \leq 2N^2   + 10 N^3 C_2 \thresh^\ell.
\end{equation}
Putting this last expression back in \eqref{AL-S2} and using $4 \delta N^4 \leq 1/2$, we obtain the bound:
\begin{align}
    \Vert A^\ell - S \Vert     &\leqslant  \eta \defeq 10 N^3 C_2 \thresh^\ell + 8 \delta N^6.\label{AL-S3}
\end{align}

The conclusion of Theorem \ref{thm:1} is then a consequence of Theorem \ref{thm:linalg:powers}. Let us now estimate roughly the quantity $\veps$ in Theorem \ref{thm:linalg:powers}. We first note that $2 \delta  N^6$ is larger than $\delta N^4 \sqrt r_0 / h$, since $h = 1/2$ and $N \sqrt{ r_{0}} \leq N^2$. It follows that $\veps$ in Theorem \ref{thm:linalg:powers} is bounded by $24 \eta N^3$. We claim also that
\begin{equation}\label{eq:deltath}
\delta N^6 \leq  N^3 C_2 \thresh ^\ell.
\end{equation}
Indeed,  $n^{-1/4} = n^{-\kappa} \leq (2d)^{-2 \ell}$ and, since $\rho \geq \mu_1^2 = 1$ (see forthcoming bound \eqref{lower_bound_on_rho}),  we have $\thresh \geq 1/\sqrt d$. It follows that $n^{-1/4} \leq \thresh^{\ell}$. To prove \eqref{eq:deltath}, we then need to check that $\tau_0^{2\ell} C_1 N^3 \leq C_2$, the latter is immediate using \eqref{eq:defC0}.

 It follows that, for some universal constant $c ,c'>0$, 
$$\veps \leq c N^6 C_2 \thresh ^\ell\leq c' r r_0^3  b^{44} \ln(n)^{15}  \thresh ^\ell .$$


We set $C_0 = 4 \veps r_0 \ln(n) / \thresh^ \ell$. After crude rearrangement and simplifications, the statement of Theorem \ref{thm:1} for eigenvalues follows easily from  Theorem \ref{thm:linalg:powers} applied to $r =r_0$.

For eigenvectors, Theorem \ref{thm:linalg:powers} allows us to describe the behaviour of the eigenvectors of $A$. We may assume without loss of generality that 
$$
\frac{C^2_0 \tau_0^{2\ell}}{(1 - \tau_{i,\ell})^2} \leq 1,
$$
since otherwise the statement is trivial (in particular $\tau_{i,\ell} <1 $). In particular, for any $i \in [r_0]$, we have 
$$
(1 - \tau_{i,\ell})\theta_i ^\ell \geq (1 - \tau_{i,\ell}) \tau_0^{-\ell}  \thresh^\ell \geq C_0 \thresh^\ell \geq 4 \veps.
$$
We are thus in position to apply  the eigenvector part in Theorem \ref{thm:linalg:powers}. 

Note that $|\theta_i|^\ell - 2 \veps \geq |\theta_i|^\ell /2$ and thus $N^2 \eta / |\theta_i|^\ell - 2 \veps \leq 2 N^2 \eta \leq \veps / (6 N)$. We choose the orientation of $\psi_i$ to ensure that $\langle \psi_i, u_i\rangle$ is non-negative. We deduce from \eqref{linalg:vecs} that,
$$
\left|  \psi_i  -  \frac{u_i}{|u_i|} \right|^2 = 2 - 2 \langle \psi_i ,  \frac{u_i}{|u_i|} \rangle \leq \frac{40 \veps^2 }{( 1 - \tau_{i,\ell})^2 |\theta_{r_0}|^2 } \leq \frac{40 C_0^2 \tau_0^{2\ell} }{( 1 - \tau_{i,\ell})^2}  .  
$$

We then write
\[\left| \langle \psi_i, \varphi_j\rangle - \langle \frac{u_i}{|u_i|}, \varphi_j\rangle \right| \leq \left|  \psi_i  -  \frac{u_i}{|u_i|} \right|. \]

Thanks to \eqref{eq:PhiAlPhi0}, we have $\langle u_i / |u_i|, \varphi_j\rangle$ is close to $\delta_{i,j}/|u_i|$ with error bounded by $\delta / |u_i|$. Finally, \eqref{eq:U*U} implies that 
\[\left | |u_i|^2 -\gamma_i \right| =  \left||u_i|^2 - \Gamma^{(\ell)}_{i,i} \right|\leqslant  \delta . \]
Recall also that by Lemma \ref{lem:Gammat}, $\Gamma^{(\ell)}_{i,i} \geq 1$. We thus find
$$
\left| |u_i| -\sqrt{\gamma_i }\right| = \frac{\left| |u_i|^2 -\gamma_i \right| }{ \left| |u_i| + \sqrt{\gamma_i} \right|} \leq \delta.
$$ 
When gathering all those bounds, we get that $|\langle \psi_i, \varphi_j\rangle - \delta_{i,j}/\gamma_i|$ is bounded by $c C_0 \tau_0^{\ell} / (1 - \tau_{i,\ell}$ where $c$ is a universal constant. Defining a new constant $C_0$ equal to $cC_0$, we thus obtain the bound displayed in \eqref{eigenvector_errorbound}. The same proof with \eqref{eq:U*V} in place of \eqref{eq:PhiAlPhi0} gives the bound \eqref{eigenvector_errorboundLR}. It concludes the proof of Theorem \ref{thm:1}.

\subsection{Proof of Corollary \ref{cor:thm1}}

On the event of Theorem \ref{thm:1}, we have 
$$
\left|\langle \frac{\psi_i + \psi_i'}{|\psi_i + \psi_i'|} , \varphi_i \rangle  - \frac{2}{\sqrt{\gamma_i} |\psi_i + \psi_i'| }\delta_{i,j}\right|  \leq  \frac{ 2 C_0 \tau_0^{\ell}}{ 1 - \tau_{i,\ell}}  =: 2 \veps. 
$$
Moreover, since $
|\psi_i + \psi_i'|^2 =  2 + 2 \langle \psi_i, \psi'_i \rangle$, we have 
$$
| |\psi_i + \psi_i'|^2 - 2 - 2 / \gamma_i | \leq 2 \veps.
$$
In particular, since $\gamma_i \geq 1$,
$$
| |\psi_i + \psi_i'| - \sqrt{2 + 2 / \gamma_i} | \leq 2 \veps / (|\psi_i + \psi_i'| + \sqrt{2 + 2 / \gamma_i}  ) \leq \sqrt{2} \veps.
$$
If $\veps \leq 1/2$, then $|\psi_i + \psi_i'| \geq 1$ and we find 
$$
\left| \frac{1}{|\psi_i + \psi_i'|} - \frac{1}{\sqrt{2 + 2 / \gamma_i}} \right| \leq \frac{\sqrt{2} \veps}{\sqrt{2 + 2 / \gamma_i}} \leq \veps.
$$

We thus have checked that if $\veps \leq 1/2$ then 
$$
\left|\langle \frac{\psi_i + \psi_i'}{|\psi_i + \psi_i'|} , \varphi_i \rangle  - \frac{\sqrt 2}{\sqrt{ \gamma_i + 1} }\delta_{i,j}\right|  \leq  4 \veps. 
$$
Otherwise, $\veps > 1/2$ and this last inequality also holds since it is trivial in this case.

\section{Consequences of algebraic incoherence}\label{sec:algebra_incoherence}

Our goal in this section is to gather several useful estimates on $Q$ and $\Gamma^{(t)}$ linked with the incoherence properties \eqref{def:L} and \eqref{incoherence1}.  
%
%
%
%
%

\subsection{Incoherence of $P$ and $Q$} 

The parameter $L$ in \eqref{def:L}  is not independent of the parameter $b$\eqref{incoherence1}.

We introduce the scale invariant analog of the parameter $L$ for the matrix $Q$: we set
\begin{equation}
\label{eq:boundQ1}
K =  n \max_{x,y} Q_{xy} /\rho.
\end{equation}
We have 
$$
L = n \max_{x,y} |P_{xy}| = \sqrt{K \rho}.
$$
Notice also that the scalar $K$ is scale invariant.  We note also that the following bound holds:
\begin{equation*}
1 \leq K \leq \frac{L ^2}{\mu_1^2} .
\end{equation*}
Indeed, for the lower bound, we use that for any matrix $T$, $\| T\| \leq n \max_{x,y} |T_{xy}|$. For the upper bound, we use that $K = L ^2 / \rho^2$ and 
\begin{equation}\label{lower_bound_on_rho}
\rho \geqslant \frac{\langle \mathbf{1}, Q \mathbf{1}\rangle}{n} = \sum_{x,y} P_{x,y}^2 = \| P\|_F^2 \geq \mu_1^2.
\end{equation}

We note also that the parameter $L$ may be bounded as follows:
\begin{align*}
|P_{x,y}|&=\left| \sum_{k=1}^n \mu_k \varphi_k(x) \varphi_k(y) \right| \nonumber \\
&\leqslant |\mu_1| \sqrt{\sum_{k=1}^n  |\varphi_k(x)|^2}\sqrt{\sum_{k=1}^n  |\varphi_k(y)|^2} \nonumber \\
&\leqslant |\mu_1|  \frac{b^2}{n}.
\end{align*}
We deduce that $L$ and $K$ are bounded by:
\begin{equation}\label{eq:L_is_bounded}
L \leqslant |\mu_1| b^2 \quad \hbox{ and } \quad  K \leq b^4. 
\end{equation}

\subsection{Bounds on the entries of $Q$}
We start by bounding the entries of powers of $Q$.
For any $x$, we find
\begin{equation*}
\label{eq:boundQ11}\sum_{y} Q_{xy}  \leq K \rho.
\end{equation*}
It follows that for any $x,y$, 
\begin{equation*}
(Q^ 2)_{xy} = \sum_{z} Q_{xz } Q_{zy} \leq\frac{K^2 \rho^2}{ n}.
 \end{equation*}
Let $(\psi_k)$ be an ON basis of eigenvectors of $Q$ with eigenvalues $(\nu_k)$. Let $t \geq 2$, we write for any $x,y$, 
$$
( Q^t) _{x  y} =   \sum_k \nu_k ^t  \psi_k(x) \psi_k (y) \leq \rho^{t-2}  \sum_k \nu_k ^2 | \psi_k(x) | | \psi_k (y)| \leq \rho^{t-2} \sqrt{(Q^2)_{xx}} \sqrt{(Q^2)_{yy}},
$$ 
where the last step follows from Cauchy-Schwarz inequality. In particular, for any $t \geq 2$ and $x,y$
\begin{equation}
\label{eq:boundQt}
 (Q^t)_{xy} \leq  \frac{K^2 \rho^t}{n}
\end{equation}
It follows from \eqref{eq:boundQ1} that Equation \eqref{eq:boundQt} also holds for $t = 1$. The following immediate consequence will be crucial; the idea it conveys is that, for any vector $v,w$, $\langle v, Q^t w \rangle$ is essentially bounded by $|v|_1 |w|_1 \rho^t / n$, a result in the flavour of Perron-Frobenius theory.

\begin{lem}\label{propQ} For any integer $t \geqslant 1$ and any vectors $v,w \in \mathbb{R}^n$, 
\begin{equation*}
\langle v, Q^t w^2 \rangle \leqslant \frac{|v|_1|w|_1 K^2 \rho^{t}}{ n}.
\end{equation*}
\end{lem}

\begin{proof}
We simply write 
$$\langle v, Q^t w \rangle = \sum_{x,y} (Q^t)_{x,y} v(y) w (y) \leq \sum_{x,y} \frac{K^2 \rho^t}{n} |v(x)| |w (y)|
$$
where we have used \eqref{eq:boundQt}. The conclusion follows.
\end{proof}

\subsection{Proof of Lemma \ref{lem:Gammat}: the incoherence hypothesis for the covariance matrix}
%


Let us end this section by the proof of Lemma \ref{lem:Gammat}. By Lemma \ref{propQ},  We start by recalling the definition of the theoretical covariance $\Gamma^{(t)} \in \mathscr{M}_{r_0}(\mathbb{R})$:
\begin{equation*}
\Gamma^{(t)}_{i,j} = \sum_{s=0}^t \frac{\langle \mathbf{1}, Q^s \varphi^{i,j}\rangle}{(\mu_i \mu_j d)^s}.
\end{equation*} 
By Lemma \ref{propQ}, we find
\begin{align}\label{eq:lde}
    \langle \mathbf{1}, Q^t \varphi^{i,j} \rangle \leqslant  |\varphi^{i,j} |_1 K^2 \rho^{t} \leq   K^2 \rho^{t} 
\end{align}
where we have use Cauchy-Schwarz inequality:
$$
 |\varphi^{i,j} |_1 = \sum_x |\varphi_i(x)| |\varphi_j(x)| \leq |\varphi_i|_2 |\varphi_j|_2 = 1. 
 $$
Going back to the sum defining $\Gamma_{i,j}^{(t)}$, we get 
\begin{equation*}\label{eq:Gaij}
    |\Gamma_{i,j}^{(t)}|\leqslant K^2 \frac{1-(\rho/(\mu_i \mu_j d))^{t+1} }{1-\rho/(\mu_i \mu_j d)} \leq \frac{K^2 ( 1 -  \tau_0^{2(t+1)}) }{1-\tau_0^2}
\end{equation*}
and as a consequence, 
\[
    \Vert \Gamma^{(t)}\Vert \leqslant r_0 \max_{i,j}  |\Gamma_{i,j}^{(t)}| =  \frac{r_0 K^2( 1 -  \tau_0^{2(t+1)})  }{1-\tau_0^2}
\]
On the other hand, if we note 
\[
    C_{i,j}^{(s)} = \frac{\langle \mathbf{1}, Q^s \varphi^{i,j}\rangle}{(\mu_i \mu_j d)^s}
\]
then it is not difficult to see that $C^{(s)}$ is indeed a semi-definite positive (SDP) matrix; more precisely, if we introduce $\pi_s(x) = \sqrt{Q^s \mathbf{1}(x)}\geqslant 0$ and $\Pi_s = \mathrm{diag}(\pi_s)$, then 
\begin{equation}\label{eq:Cs}
    C^{(s)} = d^{-s} \cdot D^{-s}\Phi^* \Pi_s^2 \Phi D^{-s} 
\end{equation}
which is clearly SDP. The matrix $\Gamma^{(t)}$ is thus a sum of $C^{(0)}=I_{r_0}$ and $t-1$ SDP matrices, hence it is itself an SDP matrix and its eigenvalues are greater than the eigenvalues of $I_{r_0}$, hence the first statement.  


For further needs, we notice that if $i \in [r]\backslash [r_0]$, then $|\mu_i| \leq \thresh$ and we get from \eqref{eq:lde} that
\begin{equation}\label{eq:gammrr0}
\Gamma^{(t)}_{i,i} \leq \sum_{s=0}^t \frac{K^2 \thresh^{2s}}{|\mu_i|^{2s}} \leq K^2(t+1) \frac{\thresh^{2t}}{|\mu_i|^{2t}}.
\end{equation}

\section{Coupling graphs and trees}

The basic ingredients for the proofs of Theorem \ref{thm:1} and related statements are directed Galton-Watson trees and martingales defined on them. We start to introduce the notations and vocabulary for this. 

\subsection{Marked Galton-Watson trees and \erd digraphs}\label{sec:tree}

\subsubsection{Graph-theoretic definitions}

A marked graph with mark space $\dN$ is a digraph $(V,E)$, with possible loops, endowed with a mark function $\imath : V \to \dN$. We are going to note $\mathscr{G}_* $ the set of all rooted directed graphs on a common countable set $V$ and with mark space $\dN$. Formally, the elements of $\mathscr{G}_*$ are triples $(G,o,\imath)$, with $o$ the root, but in general we will drop the mark function $\imath$ and simply write $(G,o)$.

Let $(G,o,\imath) \in \mathscr{G}_*$ and $g=(V,E)$. If $W \subset V$ is a subset of $V$ containing the root, then the \textbf{induced subgraph} $(G,o,\imath)_W$ is defined as follows: the underlying graph is $G_W \defeq (V, E_W)$ where $(i,j) \in E_W$ if and only if $(i,j) \in E$ and both $i$ and $j$ are in $W$, and the mark function $\imath_W$ is given by $\imath_W(v) = \imath(v)$ for all $v \in W$.

The elements in $\mathscr{G}_*$ are digraphs, and therefore we need to make a distinction between directed paths and undirected paths. Let $g=(V,E)$ be a  digraph, 
\begin{itemize}
\item If $(x,y) \in E$ we note $x \to y$, 
\item if $x \to y$ or $y \to x$ or both, we note $x \sim y$. 
\end{itemize}
Every directed graph $G$ can be transformed into an undirected graph $\hat{G}=(V,\hat{E})$ by simply forgetting the direction of the edges: $(x,y) \in \hat{E}$ iff $x \sim y$ in $g$. 

If $u,v \in V$, a \textbf{directed path} or \textbf{dipath} from $x$ to $y$ is a sequence of vertices $x_0=x, x_1, \dotsc, x_k=y$ such that for every $s$ we have $u_s \to u_{s+1}$. A \textbf{path} is the same except that we only ask $x_s \sim x_{s+1}$. 
\begin{itemize}\item The length of the shortest directed path between $x$ and $y$ is denoted by $d^+(x,y)$. 
\item The length of the shortest directed path between $y$ and $x$ is also denoted by $d^-(x,y)=d^+(y,x)$. 
\item The length of the shortest path is denoted by $d(x,y)$. 
\item  When $G$ is a digraph graph, $\mathcal{P}_G(x,t)$ is the set of paths in $G$ starting from $x$ and having $t$ steps. 

\end{itemize}
The set of all $y$ such that $d^+(x,y) \leqslant t$ is the \textbf{forward ball} $B_G^+(x,t)$ and the set of all $y$ such that $d(x,y)\leqslant t$ is the \textbf{ball} $B_G(x,t)$. When no confusion can arise, we write $B^+$ or $B$ instead of $B_G, B_G^+$. 
If $t$ is an integer and $(G,x) \in \mathscr{G}_*$, then $(G,x)_t$ is the subgraph of $(G,x)$ induced by $B_G(x,t)$, as defined above, and similarly $(G,x)^+_t$ is the subgraph of $(G,x)$ induced by $B^+(x, t)$.

A \textbf{cycle} in the graph $G$ is a sequence of distinct vertices $(x_1, \dotsc, x_k)$ such that $x_s \sim x_{s+1}$ for every $s<k$ and $x_k \sim x_1$. The number $k$ is the length of the cycle. 

A \textbf{tangle-free} subgraph of $G$ is a subgraph of $G$ that contains at most one cycle. The graph $G$ is $t$-tangle free if for every vertex $x$, the ball $B_G(x,t)$ is tangle-free.

We will the following useful property of tangle freeness. If $G$ is $h$-tangle free then there is either zero or one cycle in $(G,x)_t$ for $t \leq h$. Hence for any $y$ at distance $t$ from $x$, there is at most two paths of length $t$ from $x$ to $y$. To put it another way, for any $x \in [n]$, we have  for $t \leq h$,
\begin{equation}\label{P(o,t)}
|\mathcal{P}_G(x, t)|\leqslant 2 |(G,x)_t|.
\end{equation}

\subsubsection{Definition of the graph $G$ and the marked Galton Watson tree}

We define that $G$ as the directed \erd graph (with loops) whose adjacency matrix is given by $M$: the vertex set is $[n]$ and each directed edge is present independently with probability $d/n$. Let $x$ be an arbitrary element of $[n]$. We root the graph $G$ at $x$, and we mark every vertex with itself: the mark of vertex $x \in [n]$ is simply the integer $\imath(x)=x$. The resulted marked graph $(G,x)$ is an element of $\mathscr{G}_*$. 

We now define the directed Galton-Watson tree $T$ in the following way. Starting from its root $o$,  every vertex has a $\POI(2d)$ number of children. Every edge $(u,v)$ is independently given a unique direction $u \to v$ or $v \to u$ with probability $1/2$. This yields a random directed tree. Equivalently, each vertex has a $\POI(d)$ number of `out-children' and a $\POI(d)$ number of `in-children'. 

Finally, every non-root vertex $o'$ is independently given a random mark $\imath(o')$ which is uniform on $[n]$.  The root is given a special mark $\imath(o)=\seed $.  The resulting element of $\mathscr{G}_*$ should be noted $(T_n, \seed)$ because it depends on $n$ through the marks, but we will simply note $(T,\seed )$. We shall say that the tree $(T,\seed )$ is \emph{grown from the seed $\seed $}.

\subsection{Growth properties: trees}

Let us first state several properties on the growth of the tree $T$ first, then on the graph $G$. They are directly drawn from \cite{bordenave_lelarge_massoulie}, see Section 8 for the tree, and Sections 9.1-9.2 for the graph.  Clearly, the underlying undirected tree obtained from $T$ by deleting the marks and orientations is simply a $\POI(2d)$ Galton-Watson tree, which allows u to use known results on the growth properties of GW trees. We recall $D   = 2d \vee 1.01$ was defined in Theorem \ref{thm:1}.

\begin{lem}[{\cite[Lemma 23]{bordenave_lelarge_massoulie}}]\label{le:growtr}
Let us note $S_t$ the number of vertices at distance $t$ from the root $o$ of $T$. There are two universal constants $c_0, c_1>0$ such that for all $\lambda>0$, 
\begin{equation}
\PP (S_t \leqslant \lambda {D}^t \text{ for all } t) \geqslant 1 - c_0 e^{-c_1\lambda}. 
\end{equation}
Moreover, there is a universal constant $c$ such that for every $p \geqslant 1$,  
\begin{equation}\label{LP-growth-bound_tree}
\mathbf{E}\left[ \max_{ t \geqslant 1 } \left( \frac{S_t}{D^t}\right)^p \right] \leqslant (c p) ^p. 
\end{equation}
\end{lem}
\begin{proof} The second claim is an immediate consequence of the first: we have 
$$
\mathbf{E}\left[ \max_{ t \geqslant 1 } \left( \frac{S_t}{(2d)^t}\right)^p \right] = p \int_0^\infty \lambda^{p-1} \PP \left(  \max_{ t \geqslant 1 } \left( \frac{S_t}{D^t}\right) \geq \lambda\right) d\lambda \leq p \int_0^\infty \lambda^{p-1} c_0 e^{-c_1\lambda} d\lambda.
$$
The first statement of the lemma is \cite[Lemma 23]{bordenave_lelarge_massoulie}, it is however not explicitly written in the proof of \cite[Lemma 23]{bordenave_lelarge_massoulie} that the constants $c_0,c_1$ are universal. An inspection of the proof shows this is indeed the case (we use here that $D$ is bounded away from $1$). 
\end{proof}

As a consequence we may easily upper bound the size of $(T,\seed)_t$. Indeed,  we use the inequality 
\begin{align*}
|(T,\seed)_t| &= 1+S_1 + \dotsb + S_t \\
&\leqslant  \max_{k \geqslant 0}(D^{-k} S_k ) \sum_{k=0}^t D^k \\
&\leqslant  2 \max_{k \geqslant 0}(D^{-k} S_k )  D^{t},
\end{align*}
where we have use that  $(D^{t+1} - 1 ) / ( D-1) \leq 2 D^t$ for all $D \geq 2$. Taking expectation , we deduce from Lemma \ref{le:growtr} that for any integers $p \geq 1$ and $t \geq 1$,
\begin{equation}\label{growth:expectation_balls_trees}
\left(\EE[|(T,\seed)|_t^p]\right)^{1/p} \leqslant 2cp  D^{t}.
\end{equation}

\subsection{Growth properties: graphs}\label{subsec:growth_graphs}

We now establish the same properties as before, but for the directed graph $G$ whose adjacency matrix is $M$. We start by proving that up to a depth of order $\log_{2d} n $, the graph $G$ has few cycles. We denote by $\hat{G}$ the graph $G$ in which the directions have been erased. Pick any vertices $x,y$. Then, from the union bound
\begin{align*}
\PP(\{x,y \} \in E( \hat{G}) )&= \PP( (x,y) \in E(G) \text{ or } (y,x) \in E(G)) \leqslant \frac{D}{n}.
\end{align*}
where as above we have set $D = 2d$. 
This shows that the edge distribution of $\hat{G}$ is stochastically dominated by the edge distribution of an undirected $(n, D/n)$ \erd graph (note that, unlike the usual definition, here we allow loops in the random graph). Events which are monotone for the deletion of edges (such has having a few number of cycles) are thus of smaller probability in $G$ than in $\ER(n, D/n)$. As a consequence some results in \cite{bordenave_lelarge_massoulie} directly transfer to our setting.

\begin{lem}[Growth rate, {\cite[Lemma 29]{bordenave_lelarge_massoulie}}]\label{lem:subexpgrowth}Let us denote by $S_t(x)$ the number of vertices in $G$ that are exactly at (unoriented) distance $t$ from vertex $x$. There are two universal constants $c_0, c_1>0$ such that for every positive $\lambda$ and every vertex $x \in [n]$, we have 
\begin{equation}
\mathbf{P}(S_t(x) \leqslant \lambda D^t \text{ for all } t) \geqslant 1 - c_1 e^{-c_0 \lambda}.
\end{equation}
Moreover, there is a universal constant $c$ such that for every $p \geqslant 1$,  
\begin{equation}\label{LP-growth-bound-2}
\mathbf{E}\left[ \max_{ t \geqslant 1 } \left( \frac{S_t(x)}{D^t}\right)^p \right] \leqslant  (c p)^p \quad  \hbox{ and } \quad  \mathbf{E}\left[ \max_{\substack{x \in [n] \\ t \geqslant 1 }} \left( \frac{S_t(x)}{D^t}\right)^p \right] \leqslant   (c \ln n )^p + (cp)^p.
\end{equation}
\end{lem}

\begin{proof}
Only the LHS inequality \eqref{LP-growth-bound-2} is not explicitly stated in \cite{bordenave_lelarge_massoulie}. It is proved as in the proof of Lemma \ref{le:growtr}.
\end{proof}

Let us apply this result to $|(G,x)_t| = 1+S_1(x)+\dotsb + S_t(x) $
 with $\lambda \defeq c_0^{-1}\ln(c_1 n^2)$. With probability greater than $1-1/n$, for any $t$ and for any $x \in [n]$, 
\begin{align*}
S_t(x) \leqslant c \ln(n) D^t
\end{align*}
where $c$ is a universal constant. On this event, one also has
\begin{align}
|(G,x)_t|&=1+S_1(x)+\dotsb + S_t(x) \nonumber \\
&\leqslant c\ln(n) (1+D+ \dotsb +D^t ) \nonumber  \\
&\leqslant 2 c \ln(n)D^t.\label{eq:GxtD}
\end{align}
Similarly, we use that $$|(G,x)_t| \leq 2 D^{t}   \max_{k \geqslant 0} \frac{S_k}{D^{k}}   $$ (as explained above \eqref{growth:expectation_balls_trees}) and we deduce from \eqref{LP-growth-bound-2} that for any $p \geq 1$, 
\begin{equation}\label{growth:expectation_balls_gr} 
\EE\left[ |(G,x)_t |^p \right]^\frac{1}{p} \leqslant  2 c p  D^t.
\end{equation}
\begin{equation}\label{E4}
\EE\left[\max_{x \in [n]}|(G,x)_t |^p \right]^\frac{1}{p} \leqslant  2 c(\ln(n) +  p)   D^t .\end{equation}

\subsection{Distance between neighborhoods in the graph and the GW tree}\label{sec:coupling}
%

In this subsection, we fix some $\kappa \geq 0$ and  we consider an integer $h$ such that 
\begin{equation}\label{eq:defhrad}
0 \leq h \leq \kappa \log_{D} (n)
\end{equation}
where we have set $D = 2d\vee 1.01$ as above.

Our goal is to study fine geometric properties of $(G,x)_h$ for $\kappa$ small enough. The above comparison trick between $G$ and undirected \erd with parameters $(n,D/n)$ implies that the following holds: 

\begin{lem}[Tangle-free,  {\cite[Lemma 30]{bordenave_lelarge_massoulie}}] \label{prop:tangle-free}Let $0 \leq \kappa \leq 0.49$ and $h$ an integer as in \eqref{eq:defhrad}. For some universal constant $c$, the graph $G$ is $h$-tangle free. with probability at least $1 - c n ^{2\kappa -1}$.  Moreover, for any vertex $x \in [n]$, the graph $(G,x)_{h}$ has no cycle with probability greater than $1-c n^{\kappa-1}$.
\end{lem}

\begin{proof}
In \cite[Lemma 30]{bordenave_lelarge_massoulie} there are no loops. The probability of having a loop at vertex $x$ is $d/n$. It is however immediate to check that the same proof extends  also in our case.
\end{proof}

We quantify the distance between neighborhoods of $G$ and $T$ up to the depth $h$. Let us recall some definitions. If $\dP_1, \dP_2$ are two probability measures on the space $(\Omega, \cF)$, their total variation distance is defined as
\[\DTV(\dP_1, \dP_2) = \min_{(X_1,X_2) \in \pi(\dP_1, \dP_2)} \mathbf{P}(X_1 \neq X_2) \]
where $\pi(\dP_1, \dP_2)$ denotes the set of \emph{couplings between $\mathbb{P}_1$ and $\mathbb{P}_2$}: pairs of random variables $(X_1, X_2)$ such that $X_1$ is distributed as $\dP_1$ and $X_2$ is distributed as $\dP_2$. It is a well-known fact (see \cite{peres}) that the total variation distance is also given by 
\[\DTV(\dP_1, \dP_2) = \max_{A \in \cF} \dP_1(A) - \dP_2(A) . \]
We note $\mathscr{L}(X)$ the probability distribution of a random variable $X$.

\begin{prop}[GW-tree approximation]\label{prop:DTV}
Let $0 \leq \kappa \leq 0.49$ and $h$ an integer as in \eqref{eq:defhrad}. There is a universal constant $c>0$ such that for every vertex $x$, 
\begin{equation}
\DTV \big(\mathscr{L}( (G, x)_{h}),\mathscr{L}( (T,x)_{h} )) \big) \leqslant c ( \ln n  )^2  n^{2\kappa-1}.
\end{equation}
\end{prop}

The proof of this fact is classical; one can adapt the arguments in \cite{bordenave_lelarge_massoulie} to our setting. The difference is that our graphs are directed and now have $[n]$ possible labels, but this only brings shallow difficulties. We sketch the main ideas. 

Let us recall the following very classical total variation distance:
\begin{equation}\label{eq:BinPoi}
\DTV\left(\BIN(n, \lambda/n) , \POI(\lambda) \right) \leqslant \frac{\lambda}{n}.
\end{equation}

\subsubsection*{Coupling between labelled graphs} As a consequence of Lemma \ref{prop:tangle-free} and \eqref{eq:GxtD}, with a probability greater than $1-c n^ {\kappa-1}$, the graph $(G,x)_{h}$ is a directed tree and contains no more than $ k\defeq c\ln(n)n^{\kappa}$ vertices. Let us note $E_{h}$ this event and perform a breadth-first exploration starting from $x$. This explorations finishes at a time $\tau \leqslant k$. At each step, we reveal a set of $\POI(d)$ out-vertices and $\POI(d)$ in-vertices. From \eqref{eq:BinPoi}, we make a total-variation error smaller than $d/(n-k)+d/(n - k) \leq 3 d / n$ for $n$ large enough. By repeatedly conditioning, the total variation error made on $E_{h}$ is not greater than $   3d \tau /n \leqslant 3d k / n$.  This gives a coupling between the unlabelled versions of $(G,x)_{h}$ and $(T,x)_{h}$ which fails with probability at most $(c+3cd) \ln (n)n^{\kappa-1} \leq 4 c d \ln(n)n^{\kappa-1}$.

We now bring the labels in.  With probability greater than $ 4 c d \ln(n)n^{\kappa-1}$, the coupling between the unlabelled versions of $(G,x)_h$ and $(T,x)_h$ succeeds and have size smaller than $k = c\ln(n)n^{\kappa}$. We then put the labels in the \erd graph by drawing a uniform ordered $k$-set from $[n]$, while we put the labels on the Galton-Watson tree by simply drawing $k$ i.i.d. uniform samples from $[n]$. We claim that the total variation distance between these two  random multi-sets is smaller than $k^2/n$, see below for the proof. Hence the labels agree with an extra total variation cost of $(c \ln(n) n^\kappa)^2/n$. 

In the end, the coupling created this way fails with probability at most $4c d \ln(n) n^{\kappa-1} +c^2(\ln n )^2n^{2\kappa-1}$  which is exactly what is needed, up to adjusting the constants (note that we may assume that $D \leq n^\kappa$ otherwise $h=0$ and the statement is trivial). This concludes the sketch of proof of Proposition \ref{prop:DTV} up to the claimed bound $k^2/n$ for the distance between the random multisets which we now explain.

\subsubsection*{Sampling with and without replacement} Let $m$ be an integer. We define two random multisets in the following way. Put $m$ identical balls with labels from $1$ to $m$ in a big urn. Draw the first ball and set $p_1$ and $q_1$ to be its label. Put the ball back in the urn. Then, suppose that one has constructed $(p_1, \dotsc, p_t)$ and $(q_1, \dotsc, q_t)$. Do the following : 
\begin{itemize}
\item Draw a ball from the urn and set $p_{t+1}$ to be the label of this ball. 
\item If this label is not already one of the $q_s$, set it onto $q_{t+1}$. Else, put the ball back in the urn and draw as many balls as needed to get a label which is not already one of the $q_s$.   Define $q_{t+1}$ to be this label.
\end{itemize}
It is clear that for every $k\leqslant m$,  $Q_k \defeq (q_1, \dotsc, q_k)$ is a uniform ordered $k$-set from $[m]$, while $P_k \defeq (p_1, \dotsc, p_k)$ is distributed as $k$ i.i.d. uniform elements in $[m]$.  The random variable $(P_k, Q_k)$ is thus a coupling between those two distributions. This coupling is successful if and only if $P_k$ has exactly $k$ distinct elements, which happens with probability 
\[\frac{(m-1)\dotsc (m-k-1)}{m^k}\geqslant  \left( 1-\frac{k}{m}\right)^k \geqslant 1-\frac{k^2}{m}.\]
The coupling thus fails with probability smaller than $k^2/m$, an upper bound for the total-variation distance between $P_k$ and $Q_k$.

\subsubsection*{Maximal coupling of trees and graphs}
Proposition \ref{prop:DTV} tells us that for every fixed $x$, there exists a random rooted marked tree $(T_x, x)$ defined on the same probabilistic space as $(G,x)$ and such that 
\[\PP((T_x, x)_{h} \neq (G, x)_{h}) \leqslant  c  (\ln n) ^2 n^{2\kappa -1}. \]
provided that $h \leq \kappa \log_{D} n$ with $\kappa \leq 0.49$ and $D = 2d$.
In the sequel, we will use this family of coupled trees $(T_x, x)$ for $x \in [n]$. 

%
%

\section{Graph functionals}\label{sec:functionals}

\subsection{Functionals on trees: computations}\label{sec:functionals_computations}

We now introduce a family of functionals on $\mathscr{G}_*$ that will be used several times in the sequel. Remember that when $(g,o)$ is a rooted marked graph, we note $\mathcal{P}_g(o,t)$ the number of paths in $g$ starting from the root $o$ and having $t$ steps, that is, $(t+1)$-uples $x=(x_0,x_1, \dotsc, x_t)$ with $x_0=o$ and $x_s \to x_{s+1}$.

\bigskip

 In this section, $\psi, \phi$ represent two vectors in $\mathbb{R}^n$ and $t$ is an integer. We define
\begin{equation}\label{def:functionals}
f_{\phi, \psi, t}(g,o) = \left( \frac{n}{d}\right)^t \phi(\imath(o))\sum_{\mathcal{P}_g(o,t)} P_{\imath(o), \imath(x_1)}  \dotsb  P_{\imath(x_{t-1}), \imath(x_t)} \times \psi(\imath(x_t)).
\end{equation}
We clearly have
\begin{equation}\label{def:functionalsAs}
f_{\phi, \psi, t}(G,x) = \phi(x) (A^t \psi)(x).
\end{equation}
We will also need another functional: 
\begin{equation}\label{def:functionalsF}
F_{\mu, \psi, t}(g,o)= f_{\mathbf{1}, \psi, t}(g,o) - \mu^{-1} f_{\mathbf{1}, \psi, t+1}(g,o)  .
\end{equation}
We have 
\[F_{\mu, \psi, t}(G,x) =A^{t} \psi(x)  - \frac{1}{\mu}A^{t+1} \psi(x). \]

Before moving to several computations on those observables, we state general regularity facts. We say that a function is {\bf $t$-local} if $f(g,o)$ only depends on $(g,o)_t$. Recall that $\thresh_1 = L/d$.

\begin{lem}\label{lemme:locality}
The function $f_{\phi, \psi, t}$ is $t$-local and satisfies
\begin{equation}\label{lemme:locality+bdd}
|f_{\phi, \psi,t}(g,o)|\leqslant |\phi|_\infty |\psi|_\infty |\mathcal{P}_g(o,t)|\thresh_1^t
\end{equation}
The function $F_{\psi, t}$ is $(t+1)$-local and satisfies
\begin{equation}\label{lemme:locality+bdd:F}
|F_{\mu, \psi,t}(g,o)|\leqslant |\psi|_\infty |\mathcal{P}_g(o,t+1)|(\thresh_1 ^t  + \thresh_1^{t+1}/|\mu|) .
\end{equation}
\end{lem}

\begin{proof}
The locality property is obvious from the definition, while for the bound it suffices to write
\begin{align*}
|f_{\phi, \psi, t}(g,o)|&\leqslant \left( \frac{n}{d}\right)^t |\phi|_\infty \sum_{\mathcal{P}_g(o,t)} \left( \frac{L}{n}\right)^t |\psi|_\infty \\
&\leqslant \thresh_1^t |\phi|_\infty |\psi|_\infty |\mathcal{P}_g(o,t)|.
\end{align*}
It is the same thing for $F_{\psi, t}$.
\end{proof}

The following crucial theorem gathers all the computations linked with expectations or variances of those functionals when specialized on a tree $(T_x, x)$ with the distribution described before.

\begin{theorem}\label{thm:tree_computations}
Let $\psi$ be any vector in $\mathbb{R}^n$ and $t$ be an integer. For any $i,j \in [r]$, the following identities are true.
\begin{align}
&\EE[f_{\psi, \varphi_j, t}(T, x)] = \psi(x)\varphi_j(x) \mu_j^t \label{TC1}\\
&\EE[f_{\psi, \varphi_i, t}(T, x) f_{\psi, \varphi_j, t}(T, x) ] = \mu_i^t \mu_j^t \psi(x)^2 \sum_{s=0}^t \frac{ Q^s \varphi^{i,j}(x) }{(\mu_i \mu_j d)^s}.\label{TC3}\\
&\EE[F_{\mu_i, \varphi_i, t}(T, x)^2] = \frac{Q^t\varphi^{i,i} (x)}{d^t}. \label{TC5} 
\end{align}
\end{theorem}

The proof consists in using the eigenvector equation to identify specific martingales and take advantage of their properties to compute those expectations and variances. It is postponed to Section \ref{sec:eigenwaves}.

\subsection{Functionals on graphs: concentration}

This section describes concentration of sum functionals on the graph $G$, having the form $\sum_{o \in [n]}f(G,o)$ where $f:\mathscr{G}_* \to \mathbb{R}$ is any measurable function. The tools and spirit of this section are identical to \cite[Section 9]{bordenave_lelarge_massoulie}, but slightly adapted to our needs. 

The first proposition deeply exploits the fact that $G$ is in fact a function of independent random variables $((M_{y,x},M_{y,x}))_{y \geq x}$. A generalized  Efron-Stein inequality will be very useful here. 

\begin{prop}[Moment inequality for graph functionals] \label{efron-stein-lemma} Let $f, \bar f : \mathscr{G}_* \to \mathbb{R}$ be two $t$-local functions such that $|f(g,o)|\leqslant \bar f(g,o)$ and $\bar f$ is non-decreasing by the addition of edges. Then, for some universal constant $c >0$, for all $p \geq 2$,
\begin{equation*}
\left(\EE \left[ \left| \sum_{o \in [n]} f(G,o) - \EE \sum_{o \in [n]} f(G,o) \right|^p \right]\right)^{1 / p } \leqslant c \sqrt {n} p^{3/2} D^{t}  \left(\EE\left[\max_{ x \in [n]} {\bar f}(G,x)^{2p}\right]\right)^ {1/(2p)},
\end{equation*}
and 
\begin{equation*}
\left(\EE \left[ \left| \sum_{o \in [n]} f(G,o) - \EE \sum_{o \in [n]} f(G,o) \right|^p \right]\right)^{1 / p } \leqslant c \sqrt {np} \left(p + \ln(n)\right) D^{t}  \left(\EE\left[\frac 1 n \sum_{ x \in [n]} {\bar f}(G,x)^{2p}\right]\right)^ {1/(2p)}.
\end{equation*}
\end{prop}

\begin{proof}
We define $E_x$ as the set of edges of the form $(x,y)$ or $(y,x)$ with $x \leq y$. From our assumptions on $M$, the variables $(E_x)_{x \in [n]}$ are independent. Moreover, there is a measurable function $F$ such that 
\[\sum_{o \in [n]}f(G,o)=F(E_1, \dotsc, E_n). \]
Let us denote $Y=F(E_1, \dotsc, E_n)$ this sum, and for any $x$ let us note $Y_x$ the same sum where $E_x$ has been emptied:
\[Y_x = F(E_1, \dotsc, E_{x-1}, \varnothing, E_{x+1}, \dotsc, E_n). \]
Equivalently, if $G_x$ indicates the graph $G$ where all the directed edges between  $x$ and a larger or equal vertex have been deleted, we have 
\[Y_x = \sum_{o \in [n]} f(G_x, o). \]
The moment inequality \cite[Theorem 15.5]{MR3185193} implies that there exists a universal constant $0 < c < 6$ such that for all $p \geq 2$, 
$$
 \EE \left[ \left| \sum_{o \in [n]} f(G,o) - \EE \sum_{o \in [n]} f(G,o) \right|^p \right] \leqslant (c \sqrt p)^p  \EE \left[ \left| \sum_{x \in [n]} (Y - Y_x)^2 \right|^{ p / 2}\right] .
$$
This is a generalization of Efron-Stein inequality (corresponding to $p=2$).

Now, fix $o$ and $x$ in $[n]$. The function $f$ is $t$-local, hence $f(G,o) - f(G_x, o)$ is always zero, except possibly if $x$ is in $(G,o)_t$, or equivalently if $o$ is in $(G,x)_t$. As a consequence, we have 
\begin{align*}
|Y-Y_x|
&\leqslant \sum_{o \in (G,x)_t} f(G,o)+f(G_x, o) \\
&\leqslant 2 |(G,x)_t| \max_{ o \in [n]} {\bar f}(G,o)
\end{align*}
where in the last line we used the fact that $\gamma$ is non-decreasing by the addition of edges. Recall that  Hölder inequality implies that for all $p \geq 2$,
$$
\left(\sum_{i=1}^n u^2_i\right)^{p/2} \leq n^{p/2 -1} \left(\sum_{i=1}^n  |u^{p}_i|\right) .
$$
Therefore,  
\begin{align*}
\EE \left[ \left| \sum_{x \in [n]} (Y - Y_x)^2 \right|^{ p / 2}\right] &\leqslant  n^{ p / 2 -1} 2^p \EE \left[    \sum_{x \in [n]} \max_{ o \in [n]} {\bar f}(G,o)^p |(G,x)_t|^p \right] \\
&\leq n^{p / 2 } \sqrt{\EE[|(G,x)_t|^{2p}]\EE\left[\max_{ o \in [n]} {\bar f}(G,o)^{2p} \right]}.
\end{align*}
where we have used Cauchy-Schwarz inequality at the second line.
Finally, we use \eqref{growth:expectation_balls_gr} and it concludes the proof of the first statement of the proposition.

For the second statement, we write instead:
\begin{align*}
\sum_{x \in [n]} (Y-Y_x)^2 & \leq 4 \sum_{ x \in [n]}  \left(\sum_{o \in (G,x)_t} \bar f(G,o)\right)^2\\
& \leq  4 \sum_{ x \in [n]} |(G,x)_t | \sum_{o \in (G,x)_t}\bar f(G,o)^2 \\
& = 4 \sum_{ o \in [n]} \bar f(G,o)^2  \sum_{x \in (G,o)_t} |(G,x)_t |  \\
& \leq 4 \max_{x \in [n]} |(G,x)_t |^2 \sum_{ o \in [n]} \bar f(G,o)^2.
\end{align*}
The rest of the proof follows exactly the same line.
\end{proof}

The next immediate lemma is a comparison principle between the expectation of a graph functional on the random graph $G$ and the same functional on the random tree. 
\begin{lem}\label{lem:EfGT}
Let $0 \leq \kappa \leq 0.49$ and $h$ an integer as in \eqref{eq:defhrad}. Let $f: \mathscr{G}_* \to \mathbb{R}$ be a $h$-local function. Then, some universal constant $c >0$, we have for all $x \in [n]$,
\begin{equation}
 \left| \EE f(G,x) -  \EE f(T,x)  \right| \leqslant c \ln (n) n^{\kappa -1/2} \sqrt{ \EE\left[f(G,x)^2 \right]  \vee   \EE[|f(T, x)|^2] }.
\end{equation}
\end{lem}
\begin{proof}
Let $\mathcal{E}(x)$ denote the event ``the coupling between $(G,x)_{h}$ and $(T_x, x)_{h}$ fails"; as our functionals are $h$-local, we have $f(G,x) = f(T_x, x)$ on $\mathcal{E}(x)$.  Proposition \ref{prop:DTV} implies that $\PP(\mathcal{E}(x)) \leqslant c (\ln n)^2 n^{2\kappa-1}$. Consequently, by the Cauchy-Schwarz inequality, 
\begin{align*}
\left| \EE   f(G,x) - \EE f(T, x)  \right| &\leqslant \EE[|f(G, x) - f(T_x,x) |\mathbf{1}_{\mathcal{E}(x)}] \\
&\leqslant  \sqrt{\PP(\mathcal{E}(x))}\left( \sqrt{\EE[{f}(G,x)^2]} +  \sqrt{\EE[{ f}(T,x)|^2]} \right)\\
&\leqslant  \sqrt c \ln (n) n^{\kappa + 1/2}  \left(\sqrt{\EE[{ f}(G,x)^2]}  +    \sqrt{\EE[{f}(T,x)|^2]} \right),
\end{align*}
which is bounded by the RHS in the claim, upon adjusting the constant.
\end{proof}

Proposition \ref{efron-stein-lemma} and Lemma \ref{lem:EfGT} can be combined to derive  general deviation inequalities for graph functionals. For simplicity in the next theorem, we consider the specific case of majorizing functions $\bar f(g,o)$ that we will encounter in the sequel.

\begin{theorem}\label{thm:concentration}
Let $0 \leq \kappa \leq 0.49$ and $h$ an integer as in \eqref{eq:defhrad}. Let $f : \mathscr{G}_* \to \mathbb{R}$ be a $h$-local function  such that $|f(g,o)|\leqslant \alpha  |(g,o)_h|^\beta$ for some $\alpha ,\beta >0$. Then, for some universal constant $c >0$, for any $s  \geq 1$, with probability greater than $1-n^{-s} $, we have
\begin{equation*}
\left| \sum_{x \in [n]} f(G,x) - \EE  \sum_{x \in [n]} f(T,x)  \right| \leqslant c e^\beta \alpha s^{3/2 + \beta} \ln(n)^{3/2 + \beta}  n^{\kappa(1+ \beta) +1/2}.
\end{equation*}
\end{theorem}

\begin{proof}
We set $\bar f(g,o) = \alpha  |(g,o)_h|^\beta$. By \eqref{growth:expectation_balls_gr}, we have for all $p \geq 1$ and $x \in [n]$, for some universal constant $c>0$,
$$
\EE [ \bar f(G,x)^{2p}] \leq  \alpha^{2p} \left(cp\right)^{2 p \beta}  D^{2p \beta h}.
$$ 
By the Chebyshev inequality and  the second claim of Proposition \ref{efron-stein-lemma}, adjusting the constant $c >0$, we have for all $t> 0$, 
\begin{align*}
\left| \sum_{x \in [n]} f(G,x) - \EE \sum_{x \in [n]} f(G,x)   \right| \geq   c \alpha s n^{\kappa(1 + \beta) +1/2} , \label{conc1}
\end{align*}
with probability at most
$$
\left(\frac{ p^{1/2 + \beta}  (p \vee \ln(n))}{t}\right)^p.   
$$ 
We take $t  = (e s \ln(n))^{3/2+\beta}$ with $s \geq 1$ and $p  = t^{1/(3/2+\beta)} / e \geq \ln(n)$. We obtain a bound with probability at least $1 - n^{-s}$. Then we use Lemma \ref{lem:EfGT} and up to adjusting the universal constant, we obtain the desired bound. 
 \end{proof}

\section{Near eigenvectors: proof of Theorem \ref{thm:algebra}}\label{sec:proofs:near_eigvecs}

In this Section we prove Theorem \ref{thm:algebra}, using the tools introduced earlier. For some $0 < \kappa < 1$ which will be fixed at the end, we set
$$
\ell = \lfloor \frac{\kappa}{2} \log_{D} (n)\rfloor.
$$
Here is the route taken: first, we prove different propositions related with precise bounds for the entries of the matrices $U,V$ or $\Phi^* A^\ell \Phi$. Often, the error terms look like
\[
c_0  (b  \ln n)^{c_1}  n^{2\kappa-1/2} \thresh^t,
\]
or small variants. For a good choice of $\kappa$,  this gives the requested bounds in  Theorem \ref{thm:algebra}. 

For functionals such as $\langle \varphi_i, A^t \varphi_j \rangle$, the plan is simple: we justify why those functionals can be well-approximated by the identities of Theorem \ref{thm:tree_computations} thanks to the deviation inequality Theorem \ref{thm:concentration}. 

Bounding $\Vert A^\ell \proj_{H^\perp}\Vert$ is however much more difficult and will be done through a tangle-free decomposition, in Subsection \ref{sec:trace}. Performing the expected high-trace method requires some care and we postponed this part to Section \ref{sec:high-trace}. 

\subsection{Entry-wise bounds for Theorem \ref{thm:algebra}}

\begin{prop}\label{prop:b1} Assume $0 \leq \kappa \leq 0.33$. There is a universal constant $c>0$ such that, with probability greater than $1- c n^{3 \kappa-1}$, for any $i,j \in [n]$ and $t \leq 3 \ell$, the following holds: 
\begin{equation}\label{ii}
\left| \langle \varphi_i, A^t \varphi_j\rangle- \mu_j^t \delta_{i,j} \right|\leqslant  c  b^2(\ln n)^{5/2} n^{3\kappa-1/2} \thresh_1^t.
\end{equation}
\end{prop}

\begin{proof}
Fix $i,j \in [n]$ and $t \leq 3\ell$. Using the notation already introduced in \eqref{def:functionals}, we define a function $f$ by
\[f(g,o) = \mathbf{1}_{(g,o)_t \text{ is tangle free }}f_{\varphi_i, \varphi_j, t}(g,o).\]
This function is clearly $t$-local and from \eqref{lemme:locality+bdd}, \eqref{incoherence1}, and \eqref{P(o,t)}, 
\begin{align*}
|f(g,o) &\leqslant |\varphi_i|_\infty |\varphi_j|_\infty \thresh_1^t |\mathcal{P}_g(o,t)| \mathbf{1}_{(g,o)_t \text{ is tangle free } }\\
&\leqslant \frac{2 b^2}{n}\thresh_1^t |(g,o)_t|.
\end{align*}
Moreover, on the event that $G$ is $t$-tangle free, from \eqref{def:functionalsAs} we have 
$$
\langle \varphi_i, A^t \varphi_j\rangle = \sum_{x \in [n]} f(G,x) 
$$
We now apply the concentration result in Theorem \ref{thm:concentration} to the function $f$ with $s = 4$, $\alpha = 2b^2 \thresh_1^t /n$, $\beta = 1$ and $\kappa' = 3\kappa/2$. The error bound in Theorem \ref{thm:concentration} is thus, for some new constant $c>0$,
\[  c  b^2(\ln n)^{5/2} n^{3\kappa-1/2} \thresh_1^t. \]
 Moreover, as computed in Theorem \ref{thm:tree_computations} - Equation \eqref{TC1}, we have 
\[\EE  \sum_{x \in [n]} f(T, x)   =\EE \sum_{x \in [n]} f_{\varphi_i, \varphi_j, t}(T_x, x) = \mu_i^t \delta_{i,j}. \]
Combined with Lemma \ref{prop:tangle-free} to control the probability that the graph $3\ell$-tangle free and $\ell \leq \ln(n)$ for the union bound, this concludes the proof.
\end{proof}

\begin{prop}\label{prop:b2} Assume $0 \leq \kappa \leq 0.49$.  There is a universal constant $c>0$ such that, with probability greater than $1-  c n^{2 \kappa-1}$,  for any $i,j \in [n]$ and $t \leqslant \ell$, the following holds:
\begin{equation}\label{iii}
\left| \langle A^t \varphi_i, A^t \varphi_j\rangle-  \mu_i^t \mu_j^t \Gamma^{(t)}_{i,j} \right|\leqslant  c  b^2 (\ln n)^{7/2} n^{3\kappa/2 -1/2}\thresh_1^{2t}
\end{equation}
\begin{equation}\label{iii-*}
\left| \langle (A^*)^t \varphi_i, (A^*)^t \varphi_j\rangle-  \mu_i^t \mu_j^t \Gamma^{(t)}_{i,j} \right|\leqslant  c  b^2 (\ln n)^{7/2} n^{3\kappa/2 -1/2}\thresh_1^{2t}.
\end{equation}
\end{prop}

\begin{proof} Since $A^*$ and $A$ are identical in distribution, one only has to prove the first inequality.  The proof is the same as for Proposition \ref{prop:b1}. Fix $i,j \in [n]$ and $t \leq \ell$. The right function here is $f$ defined by
\[f(g,o) = \mathbf{1}_{(g,o)_t \text{ is tangle free }}f_{\mathbf{1}, \varphi_i, t}(g,o)f_{\mathbf{1}, \varphi_j, t}(g,o).\]
This function is clearly $t$-local and  from \eqref{lemme:locality+bdd}, \eqref{incoherence1}, and \eqref{P(o,t)},
\begin{align*}
|f(g,o)| &\leqslant |\varphi_i|_\infty |\varphi_j|_\infty \thresh_1^{2t} |\mathcal{P}_g(o,t)|^2 \mathbf{1}_{(g,o)_t \text{ is tangle free } } \\
&\leqslant \frac{4b^2}{n} \thresh_1^{2t}|(g,o)_t|^2 .
\end{align*}
We observe that if the graph $G$ is $t$-tangle free, we have
\[  \langle A^t \varphi_i, A^t \varphi_j\rangle =  \sum_{x \in [n]} f(G,x). \]
We now apply Theorem \ref{thm:concentration}  with $s = 4$, $\alpha = 4b^2 \thresh_1^{2t} /n$, $\beta =2$ and $\kappa' = \kappa/2$. The error bound in Theorem \ref{thm:concentration} is thus, for some new $c>0$,
\[ c b^2 (\ln n)^{7/2} n^{3\kappa/2 -1/2}\thresh_1^{2t} .   \]
 Moreover, as computed in Theorem \ref{thm:tree_computations} - Equation \eqref{TC3}, we have 
\[\EE \left[\sum_{x \in [n]} f(T, x) \right] =\EE \left[\sum_{x \in [n]} f_{\mathbf{1}, \varphi_i, t}(T_x, x)f_{\mathbf{1}, \varphi_j, t}(T_x, x)  \right] = \mu_i^t \mu_j^t \Gamma^{(t)}_{i,j}. \]
Combined with Lemma \ref{prop:tangle-free}, this concludes the proof.\end{proof}

\subsection{Control over the growth of a process}

In this subsection, we establish the following proposition. In words, it asserts that if $\langle (A^*)^\ell \varphi_i , w \rangle = 0$ then  $\langle (A^*)^t \varphi_i , w \rangle$ is quite small for all $t \leq \ell$.
\begin{prop}\label{prop:iv} Assume $0 \leq \kappa \leq 0.33$. There exists a universal constant $c >0$ such that, with probability greater than $1- c n^{2\kappa-1}$, one has for any $t \leqslant \ell$, for any $w \in H^\perp$ with $|w| =1$ and for any $i \in [r_0]$ the following bound:
\begin{equation*}\label{iv}
    |\langle (A^*)^t \varphi_i, w \rangle| \leqslant c \ell  b^4 \thresh^t .
\end{equation*}
\end{prop}

\begin{proof}
We follow the usual strategy. First, we note that by the mere definition of $H$, when $w \in H^\perp$  we have $\langle (A^*)^\ell \varphi_i, w \rangle=0$. Consequently,
\[
    \mu_i^{-t} \langle (A^*)^t \varphi_i, w \rangle = \mu_i^{-t} \langle (A^*)^t \varphi_i, w \rangle - \mu_i^{-\ell}\langle (A^*)^\ell \varphi_i, w \rangle 
\]
and from a telescopic sum we get 
\begin{align}
    |\mu_i^{-t} \langle (A^*)^t \varphi_i, w \rangle| &=\left| \sum_{k=t}^{\ell-1} \mu_i^{-k}\langle (A^*)^{k} \varphi_i, w \rangle - \mu_i^{-k} \langle (A^*)^{k+1} \varphi_i, w \rangle \right| \nonumber \\
      &\leqslant \sum_{k=t}^{\ell-1} |\mu_i|^{-k}\left| (A^*)^{k} \varphi_i - \frac{1}{\mu_i}  (A^*)^{k+1} \varphi_i  \right|    \label{sum:growth}
\end{align}
where in the last line we used the Cauchy-Schwarz inequality and $|w| =1$. Let us fix $1 \leq k \leq \ell-1$. We have 
\[
    \left| (A^*)^{k} \varphi_i - \frac{1}{\mu_i}  (A^*)^{k+1} \varphi_i  \right|^2  = \sum_{x \in [n]} F_{\mu_i, \varphi_i, k}(G,x)^2
\]
where $F_{\mu_i, \varphi_i, k}$ was defined in \eqref{def:functionalsF}. From Theorem \ref{thm:tree_computations} - Equation \eqref{TC5},  we have for any $x \in [n]$
\[\EE[F_{\mu_i, \varphi_i, k}(T , x)^2] = \frac{Q^k\varphi^{i,i}(x)}{d^k}.\]
Therefore, summing over $x$ gives
\[
    \EE\left[\sum_{x \in [n]}F_{\mu_i, \varphi_i, k}(T, x)^2 \right] = \frac{\langle \mathbf{1}, Q^k \varphi^{i,i} \rangle}{d^k}. 
\]
We are going to use the concentration bound from Theorem \ref{thm:concentration}. We define the function  $f$ by
\[f(g,o) = \mathbf{1}_{(g,o)_{k+1} \text{ is tangle free }}F_{\mu_i, \varphi_i, k}(g, o)^2 .\]
This function is  $(k+1)$-local and from \eqref{lemme:locality+bdd:F}, \eqref{incoherence1}, and \eqref{P(o,t)}, we have
\begin{align}
 f(g,o) & \leqslant  |\varphi_i|^2_\infty |\mathcal{P}_g(o, k+1)|^2 (\thresh_1^{k}  + \thresh_1^{k+1} /|\mu_i|)^2 \mathbf{1}_{(g,o)_{k+1} \text{ is tangle free } } \\
&\leqslant \frac{8 b^2}{n}    \thresh_1^{2k}|(g,o)_{k+1}|^2,
\end{align}
where we have used $|\mu_i| \geq \thresh_1$. Then, applying Lemma \ref{prop:tangle-free} and Theorem \ref{thm:concentration} (with $s=4$, $\alpha =  8 b ^2 \thresh_1 ^{2k} /n$ and $\beta = 2$) yields, with probability at least $1 -  c n^{2 \kappa -1}$, for all $k \leq \ell-1$ and $i \in [r_0]$, 
\begin{equation*}\label{conc_F} 
    \left|\left| (A^*)^{k} \varphi_i - \frac{1}{\mu_i}  (A^*)^{k+1} \varphi_i  \right|^2  - \frac{\langle \mathbf{1}, Q^k \varphi^{i,i} \rangle}{d^k} \right| \leqslant c^2   b^2 (\ln n)^{7/2}    n^{3\kappa/2-1/2} \thresh_1 ^{2k}.
\end{equation*}

We now sum all $k$ between $t$ and $\ell-1$ in \eqref{sum:growth}. We find
 \[
    |\mu_i ^{-t}\langle (A^* )^t \varphi_i, w \rangle | \leqslant \sum_{k=t}^{\ell-1}\sqrt{\frac{\langle \mathbf{1}, Q^k \varphi^{i,i} \rangle}{(\mu_i^2d)^k}} +    c b (\ln n)^{7/4}    n^{3\kappa/4-1/4} \sum_{k=t}^{\ell -1}  \frac{\thresh_1 ^{k}}{|\mu_i|^{s}} .
 \]
 where we used $\sqrt{u+v} \leqslant \sqrt{u} + \sqrt{v}$. Then, we use \eqref{eq:lde} which tells that $\langle \mathbf{1}, Q^k \varphi^{i,i} \rangle\leqslant K^2  \rho^k$. We thus get from \eqref{sum:growth} the inequality
 \begin{align*} 
    |\langle (A^* )^t \varphi_i, w \rangle|& \leqslant |\mu_i|^t \left( K  + cb  (\ln n)^{7/4}    n^{3\kappa/4-1/4}  \right)\sum_{k=t}^{\ell-1} \frac{ \thresh^k }{|\mu_i|^k }  \\
        & \leqslant (\ell -t)  (K   + c b   (\ln n)^{7/4}    n^{3\kappa/4-1/4} ) \thresh ^t  .
 \end{align*}
where we used the fact that $|\mu_i|> \thresh$. We finally use \eqref{eq:L_is_bounded}. 
\end{proof}

\subsection{Norm of the  matrix restricted to $H^\perp$}\label{sec:trace}

Our goal in this section is to prove the inequality \eqref{norm_on_orthogonal} on $\Vert \uA^{(k)} \Vert$. We first describe the tangle-free decomposition introduced in \cite{massoulie_rama,bordenave_lelarge_massoulie}. The main technical estimate will be postponed to Section \ref{sec:high-trace}.

We notice that if the graph $G$ is $\ell$-tangle free then $A^{\ell} = A^{(\ell)}$ where 
$$
    A^{(\ell)}_{x,y} = \PAR{\frac{n}{d} }^\ell \sum_{F^{\ell}_{x,y}} \prod_{t = 1}^\ell P_{x_{t-1} x_t} M_{x_{t-1} x_t},
$$
and the sum runs over the set $F_{x,y}^\ell$ of all paths $(x_0, \dotsc , x_\ell)$ such that $x_0 = x$, $x_\ell = y$ and the graph of the path is tangle-free --- we recall that tangle-free means that there are no more than one cycle, see the definitions in Subsection \ref{sec:tree} on page \pageref{sec:tree}. More generally, $F^t$ denotes the set of all tangle-free paths of length $t$, whatever their endpoints. We also define the matrices $\uM$ 
and $\uA^{(\ell)}$ by $A^{(0)} =  \uA^{(0)} = I_n$, and 
\begin{eqnarray}
    \uM_{x,y} &= &M_{x,y} - \frac d n \nonumber \\
    \uA^{(\ell)}_{x,y} &= & \PAR{\frac{n}{d} }^\ell \sum_{ F^{\ell}_{x,y}} \prod_{t = 1}^\ell P_{x_{t-1} x_t} \uM_{x_{t-1} x_t} \label{eq:defuA}.
\end{eqnarray}

We use the convention that the product over an emptyset is $1$. Then we may write  for any $a, b \in \dR^{\ell}$, 
\begin{equation}\label{eq:telesc}
    \prod_{t=1}^\ell a_t = \prod_{t=1}^\ell b_t + \sum_{k=1}^{\ell} \PAR{\prod_{t =1}^{k-1} b_t } ( a_k - b_k) \PAR{ \prod_{t =k+1}^{\ell} a_t}.
\end{equation}
We thus get
$$
    A^{(\ell)}_{x,y} = \uA^{(\ell)}_{x,y} + \sum_{k=1}^{\ell} \PAR{\frac{n}{d} }^\ell \sum_{F^{\ell}_{x,y}} \prod_{t = 1}^{k-1} P_{x_{t-1} x_t} \uM_{x_{t-1} x_t} \PAR{ \frac{d}{n} P_{x_{k-1} x_k} }\prod_{t = k+1}^{\ell} P_{x_{t-1} x_t} M_{x_{t-1} x_t} .
$$

This can then be rewritten as the following identity in $\mathscr{M}_{n}(\mathbb{R})$: 
\[
    A^{(\ell)} = \uA^{(\ell)} +  \sum_{k=1}^{\ell} \uA^{(k-1)} P A^{(\ell - k)} -\sum_{k=1}^{\ell}  R_k^{(\ell)}  ,
\]
where 
\[
    (R_k ^ {(\ell)}) _{x,y}  =\PAR{\frac{n}{d} }^{\ell-1} \sum_{(x_0, \dotsc, x_\ell)\in T_{x,y}^{k,\ell}} \prod_{t = 1}^{k-1} P_{x_{t-1} x_t} \uM_{x_{t-1} x_t} P_{x_{k-1} x_k} \prod_{t = k+1}^{\ell} P_{x_{t-1} x_t} M_{x_{t-1} x_t} 
\]
where the sum is over all `paths'  $(x_0, \ldots, x_{\ell})$ such that $(x_0, \ldots, x_{k-1}) \in F^{k-1}$, $(x_{k}, \ldots, x_{\ell}) \in F^{\ell - k}$ but $(x_0, \ldots, x_{\ell})$ is not in $F^{\ell}$. 

We now use the spectral decomposition $P=\mu_1 \varphi_1 \varphi_1^*+\dotsb+\mu_n\varphi_n\varphi_n^* $. For any unit vector $w$, we have 
\[
    \uA^{(k-1)} P A^{(\ell - k)} w =\uA^{(k-1)} \sum_{j=1} ^{n}  \mu_j   \varphi_j \langle \varphi_j, A^{(\ell - k)} w \rangle.
\]
Hence, from the orthogonality of the $\varphi_j$'s, 
\begin{eqnarray}
    | \uA^{(k-1)} P A^{(\ell - k)} w |  &\leqslant &\|  \uA^{(k-1)} \| \left|\sum_{j=1}^{n}  \mu_j \varphi_j  \langle \varphi_j , A^{(\ell - k)} w \rangle \right|  \nonumber .\\
    & = & \|  \uA^{(k-1)} \| \sqrt{ \sum_{j=1}^{n} \mu^2 _j   \langle \varphi_j ,  A^{(\ell - k)}w \rangle ^2}   . \label{remainder_section:eq1}
\end{eqnarray}

From Proposition \ref{prop:iv}, with probability at least $1 - c n^{2\kappa-1}$, the following holds for any $t\leq \ell$ and $i \in [r_0]$ and $w \in H^\perp$: 
\[|\langle \varphi_i, A^t w \rangle | \leqslant c b^4 \ell  \thresh^t.\] 

From Proposition \ref{prop:b2}, with probability at least $1 - c n^{2\kappa-1} $, for all $i \in [n] \backslash [r_0]$ and $t \leq \ell$, we have 
$$|(A^*)^t \varphi_i|^2 \leqslant \mu_i^{2t}\Gamma_{i,i}^{(t)} + c b^2 (\ln n)^{7/2} n^{3\kappa/2  -1/2} \thresh^{2t}.$$ 
However, from Equations \eqref{eq:gammrr0}-\eqref{eq:L_is_bounded}, we have $\mu_i^{2t}\Gamma_{i,i}^{(t)} \leqslant   b^8 (t+1)  \thresh^{2t}$. As a consequence, for some universal constant $c >0$, for all $i \in [n]\backslash [r_0]$ and $t \leq \ell$,
\[   |\langle \varphi_i, A^t w \rangle | \leqslant |w||(A^*)^t \varphi_i| \leqslant c b^4 \sqrt \ell \thresh^t.
\]

On the union of the two events events and $G$ $\ell$-tangle free, the whole square root in \eqref{remainder_section:eq1} is bounded by, for all $w \in H^\perp$: 
\begin{align*}
    \sqrt{ \sum_{j=1}^{r} \mu^2 _j   \langle \varphi_j ,  A^{(\ell - k)}w \rangle ^2} &\leqslant  c b^4 \ell  \thresh^{\ell-k}    \sqrt{ \sum_{j=1}^{n}\mu_j^2 } = c  b^4  \ell r \thresh^{\ell-k} , 
\end{align*}
where we used that $\mu_1 = 1$. We get the following lemma.

\begin{lem}\label{lem:decomposition_tangle_free}
With probability at least $1-c n^{2\kappa-1}$, one has 
\begin{equation*}
\Vert A^\ell \proj_{H^\perp}\Vert \leqslant \| \uA^{(\ell)} \| + c  b^4 \ell r  \sum_{k=1}^\ell \| \uA^{(k-1)} \| \thresh^{\ell-k} + \sum_{k=1}^{\ell} \Vert  R_k^{(\ell)} \Vert .
\end{equation*}
\end{lem}

We now need bounds on $\| \uA^{(k-1)} \|$ and $\Vert  R_k^{(\ell)} \Vert$. 

\begin{prop}\label{prop:norms}  There exists a universal constant $c>0$ such that if $ n \geq cK^{42}$, with probability at least $1 - 1/\sqrt{n}$, the following holds for any $k  \in  [\ell]$:
\begin{equation} \label{eq:normsA}
\Vert \uA^{(k)} \Vert \leqslant  c \ln(n)^{10} K^4 \thresh^k.
\end{equation}
\begin{equation}\label{eq:normsR}
\Vert R^{(\ell)}_k \Vert \leqslant  \frac{cd} n \ln(n)^{23} L^{\ell}.
\end{equation}
\end{prop}

The proof of this proposition relies on a high-trace method. It is postponed to Section \ref{sec:high-trace}.  As a corollary, we obtain the following proposition.

\begin{prop}\label{prop:norm} Assume $0 \leq \kappa \leq 1/4$. There exists a universal constant $c >0$ such that, with probability greater than $1- c / \sqrt n$, one has 
$$
\Vert A^\ell \proj_{H^\perp}\Vert \leq  c  b^{20} r  \ln(n)^{12}    \thresh^\ell.
$$
\end{prop}
\begin{proof} By Lemma \ref{lem:decomposition_tangle_free} and Proposition \ref{prop:norms}, 
with probability at least $1-c / \sqrt n$,  we have 
\begin{align*}
\Vert A^\ell \proj_{H^\perp}\Vert &\leqslant \| \uA^{(\ell)} \| + c  b^4 \ell  r \sum_{k=1}^\ell \| \uA^{(k-1)} \| \thresh^{\ell-k} + \sum_{k=1}^{\ell} \Vert  R_k^{(\ell)} \Vert \nonumber\\
&\leqslant  C_0 \ln(n)^{10} \sum_{k=0}^\ell  \thresh ^\ell + \frac{4d \ell \ln(n)^{24}}{n}L^\ell \nonumber \\
&\leqslant  C_0  \ln(n)^{10}\ell  \thresh ^\ell + cd \ell \ln(n)^{24} \thresh_1^\ell \frac{d^\ell}{n} \nonumber \\
&\leqslant C \ln(n)^{12}    \thresh^\ell \label{nn1}
\end{align*}
with $C_0 \defeq   c' b^{20}  \ell r$, for some universal constant $c'$, in the second line (recall that $K\leq b^4$ from \eqref{eq:L_is_bounded} and that $b$ is bounded by \eqref{eq:defb}). To get the last line, we have observed t $d^{\ell+1} /n \leqslant \thresh_1^\ell n^{\kappa/2 +1/8 -1}$ where we have used $d \leq n^{1/8}$ from \eqref{eq:boundd} and $\kappa \leq 1/4$. The constant $C$ in this last line is taken to be $ C= c  b^{20} r $ with $c$ some absolute constant. 
\end{proof}

\subsection{Proof of Theorem \ref{thm:algebra}}

We gather the events and bounds from the last propositions, working out the error terms and presenting them in a way which keeps track of dependencies with the parameters.

As in \eqref{eq:defH}, we define the vector space $H' = \mathrm{vect}(u_1, \dotsc, u_{r_0})=\mathrm{im}(U)$. Since $A$ and $A^*$ have the same distribution, Proposition \ref{prop:norm} holds for $(A^*)^\ell \proj_{{H'}^\perp}$ in place of  $A^\ell \proj_{{H}^\perp}$. Then, we set $\kappa = 1/4$ so that $\ell = \lfloor (\kappa /2) \log_{2d} (n)\rfloor$ is as in Theorem \ref{thm:algebra}. We apply Propositions \ref{prop:b1}-\ref{prop:b2}. We also consider the event of Proposition \ref{prop:b1} for $\kappa$  equal to $\kappa/3$.
 Then, the intersection of the events in Propositions \ref{prop:b1}-\ref{prop:b2} and Proposition \ref{prop:norm} for $A^\ell \proj_{H^\perp}$ and  $(A^*)^\ell \proj_{{H'}^\perp}$ has probability greater than $1- c n^{3\kappa-1} = 1 - c n^{-\kappa}$.

On this good event, we should check that the error terms \eqref{eq:PhiAlPhi}-\eqref{norm_on_orthogonal} are as claimed in the statement of Theorem \ref{thm:algebra}. The claims  \eqref{norm_on_orthogonal}-\eqref{norm_on_orthogonal2} are contained in Proposition \ref{prop:norm}. The proofs of \eqref{eq:PhiAlPhi}-\eqref{eq:V*V} are essentially the same. They all rely on the basic inequality for $T \in \mathscr{M}_{r_0}(\mathbb{R})$ $$\| T \| \leq r_0 \max_{i,j} |T_{i,j}| $$ and the use of the events in Propositions \ref{prop:b1}-\ref{prop:b2} to control the individual entries of the matrices. 

The $(i,j)$-entry of the matrix $V^* A^\ell U = D^{-\ell} \Phi^* A^{3\ell} \Phi D^{-\ell}$ is precisely $(\mu_i \mu_j)^{-\ell} \langle \varphi_i, A^{3\ell} \varphi_j\rangle$. Hence on the event of Proposition \ref{prop:b1}, we find
\[\Vert D^{-\ell} \Phi^* A^{3\ell} \Phi D^{-\ell} - \Sigma^\ell\Vert \leqslant  c r_0b^2  \ln(n)^{5/2} n^{3\kappa - 1/2}  \frac{ \thresh^{3\ell}}{|\mu_i|^\ell |\mu_j|^\ell}.\]
Our choice of $\kappa = 1/10$ implies that $3\kappa - 1/2 = -\kappa$. We use that $\thresh \leq \tau_0 |\mu_i|$ for all $i \in [r_0]$. Adjusting the constant $c >0$, this gives the requested bound in \eqref{eq:PhiAlPhi}.

We now prove \eqref{eq:U*V}. The $(i,j)$ entry of $U^* V$ is  equal to $$(\mu_i \mu_j)^{-\ell} \langle   A^\ell \varphi_i , (A^*)^\ell \varphi_j\rangle = (\mu_i \mu_j)^{-\ell} \langle   A^{2\ell} \varphi_i ,  \varphi_j\rangle.$$
Hence, on the event of Proposition \ref{prop:b1}, we find
\[\Vert U^* V - I_{r_0} \Vert \leqslant  c r_0 b^2  \ln(n)^{5/2} n^{3\kappa-1/2}  \tau_0 ^{2\ell} .\]
This implies  \eqref{eq:U*V}. The bound \eqref{eq:PhiAlPhi0}  is proven similarly using the event of Proposition \ref{prop:b1} for $\kappa$ equal to $\kappa/3$. It gives 
\[\Vert \Phi^* V - I_{r_0} \Vert \leqslant  c r_0 b^2  \ln(n)^{5/2} n^{\kappa - 1/2 }  \tau_0 ^{\ell} .\]
It remains to use \eqref{eq:taudn} and our choice of $\kappa$.

We now prove \eqref{eq:U*U}-\eqref{eq:V*V}.  The $(i,j)$ entry of $U^* U$ is  equal to $$(\mu_i \mu_j)^{-\ell} \langle   A^\ell \varphi_i , A^\ell \varphi_j\rangle.$$
On the event of Proposition  \ref{prop:b2}, we find
\[\Vert U^* U - \Gamma^{(\ell)} \Vert \leqslant  c r_0 b^2  \ln(n)^{7/2} n^{2\kappa-1/2}  \tau_0 ^{2\ell}.\]follows
This implies  \eqref{eq:U*U}-\eqref{eq:V*V} and concludes the proof of Theorem \ref{thm:algebra}.

\section{Eigenwaves on Galton-Watson trees: proof of Theorem \ref{thm:tree_computations}}\label{sec:eigenwaves}

This section carries out the details of the expectation and variance computation in Theorem \ref{thm:tree_computations}.  Let us recall the notation, and especially the definition \eqref{def:functionals} of the functionals: if $\psi, \phi$ are two vectors in $\mathbb{R}^n$ and $t$ is an integer, then
\begin{equation*}
f_{\phi, \psi, t}(g,o) = \left( \frac{n}{d}\right)^t \phi(\imath(o))\sum_{\mathcal{P}_g(o,t)} P_{\imath(o), \imath(x_1)}  \dotsb  P_{\imath(x_{t-1}), \imath(x_t)} \times \psi(\imath(x_t)).
\end{equation*}

\subsection{An elementary computation on Poisson sums.}

Let $N$ be a $\POI(d)$ random variable, and let $(X_i), (Y_i)$ two i.i.d. sequences of random variables, both being independent from $N$ ; we suppose that $X_i$ is independent of $Y_j$ for $i \neq j$, but there might be a nontrivial dependence between $X_i$ and $Y_i$. Let us note 
\[A=\sum_{i=1}^N X_i \qquad B = \sum_{i=1}^N Y_i. \]
The following (classical) identity will be crucial in the next sections. For convenience, we provide a proof.

\begin{equation}\label{eq:covariance_poisson_sums} \COV(A,B)=d\EE[XY]. \end{equation}

\begin{proof}[Proof of \eqref{eq:covariance_poisson_sums}]

Primary computations shows that $\EE[A]=d\EE[X]$ and  $\EE[B]=d\EE[Y]$, hence $\COV(A,B) = \EE[AB] - \EE[A]\EE[B] =  \EE[AB] - d^2\EE[X]\EE[Y]$. The first term $\EE[AB]$ is thus equal to

\begin{align*}
\EE[AB] &=  \EE\left[\sum_{i=1}^N \sum_{i=1}^N X_i Y_j \right]\\
&= \sum_{k=0}^\infty \frac{e^{-d}d^k}{k!} \EE\left[\sum_{i=1}^k \sum_{i=1}^k X_i Y_j \right]\\
&= \sum_{k=0}^\infty \frac{e^{-d}d^k}{k!}\left( k \EE\left[ X Y \right]+k(k-1) \EE[X]\EE[Y] \right)\\
&= \EE[XY]\EE[N]+ \EE[X]\EE[Y]\EE[N(N-1)].
\end{align*}
We have $\EE[N(N-1)]=d^2$, hence \eqref{eq:covariance_poisson_sums} holds true.
\end{proof}

Identity \eqref{eq:covariance_poisson_sums} will be used many times in the following context. Fix one vertex $x \in [n]$ and suppose that $X=P_{x,U}\varphi_i(U)$ and $Y=P_{x,U}\varphi_j(U)$ with $U \sim \mathsf{Unif}[n]$. By the eigenvector equation $P\varphi_k = \mu_k \varphi_k$, we have $\EE[X]=(\mu_i d \varphi_i(x))/n$ and $\EE[Y]=(\mu_j d \varphi_j(x))/n$, hence in this case
\begin{equation}
\label{eq:EXY}
\EE[XY]= \frac{1}{n}\sum_{y \in [n]}P_{x,y}^2 \varphi_i(y)\varphi_j(y) = \frac{1}{n^2}(Q\varphi^{i,j})(x).
\end{equation}

These identities will be used later in variance computations.

\subsection{Proof of a martingale property}\label{sec:martingale_property}

Let $x$ be a fixed element in $[n]$ and let $(T, x)$ be the random rooted marked tree described in Section \eqref{sec:tree} and let $\cF_t$ be the sigma-algebra generated by $(T, x)_t$; from now on we will use the filtration $\cF=(\cF_t)_{t \geqslant 0}$. The key observation for this whole section is that the process $t \mapsto \mu_k^{-t} f_{\phi, \varphi_k, t}(T,x)$ is indeed an $\cF$-martingale.

\begin{lem}
Let $\phi$ be any vector and let $\varphi_i$ be an eigenvector of $P$ associated with the nonzero eigenvalue $\mu_i$. Then, the discrete-time stochastic process
\[Z_t \defeq  \frac{1}{\mu_i^t} f_{\phi, \varphi_i, t}(T, x)\]
is an $\cF$-martingale.
\end{lem}

From now on, the conditional expectation with respect to the sigma-algebra $\cF_t$ will be noted $\EE_t $ instead of $ \EE[\cdot |\cF_t]$. 

\begin{proof}
It is clear that $f_{\phi, \varphi_i, t}(T, x) =  \phi(x)f_{\mathbf{1}, \varphi_i, t}(T, x)$, hence it is sufficient to prove the martingale property for $Z_t = \mu_i^{-t} f_{\mathbf{1}, \varphi_i, t}(T, x)$. 

Let us fix an integer $t$. Then, upon factorizing up to depth $t$ we have 
\begin{multline*}Z_{t+1}-Z_t= \\ \left( \frac{n}{d\mu_i}\right)^{t+1}\sum_{d^+(o,x)=t}\prod_{s=1}^t P_{\imath(x_{s-1}), \imath(x_s)} \left(\sum_{x_t \to y}P_{\imath(x_t), \imath(y)}\varphi_i(\imath(y)) - \frac{d\mu_i}{n}\varphi_i(\imath(x_t)) \right).\end{multline*}
Let us note $\Delta_t = Z_{t+1}-Z_t$ the martingale increment. Then, 
\begin{multline*}\EE_{t}[\Delta_t]=\\
\left( \frac{n}{d\mu_i}\right)^{t+1}\sum_{d^+(o,x)=t}\prod_{s=1}^t P_{\imath(x_{s-1}), \imath(x_s)}\EE_{t}\left[ \sum_{x_t \to y}P_{\imath(x_t), \imath(y)}\varphi_i(\imath(y)) - \frac{d\mu_i}{n}\varphi_i(\imath(x_t)) \right]. \end{multline*}
Let $X_1, X_2, \dots$ be i.i.d. random variables with the following distribution (conditionally on $\cF_t$):
\[\PP_t(X = P_{\imath(x_t), z}\varphi_i(z)) = \frac{1}{n}\quad \text{ for each } z \in [n]. \]
In other words, conditionally on $\cF_t$, the rv's $X_s$ are i.i.d. samples with distribution $P_{\imath(x_t), U}\varphi_i(U)$ with $U \sim \mathsf{Unif}[n]$, just as in the end of the preceding paragraph.

It is clear that for every children $y$ of $x_t$, the random variable $P_{\imath(x_t), \imath(y)}\varphi_i(\imath(y))$ has this distribution, and as already noted,  
\[\EE[X]=\frac{1}{n}\sum_{z \in [n]} P_{\imath(x_t), z}\varphi_i(z) = \frac{1}{n}(P\varphi_i)(\imath(x_t))=\frac{\mu_i}{n}\varphi_i(\imath(x_t)). \]
The number $N$ of children of $x_t$ has a $\POI(d)$ distribution, and is independent of $\cF_t$, hence 
\begin{align*}\EE_{t}\left[ \sum_{x_t \to y}P_{\imath(x_t), \imath(y)}\varphi_i(\imath(y)) - \frac{d\mu_i}{n}\varphi_i(\imath(x_t)) \right] &= \EE_t\left[\sum_{s=1}^N X_s \right] - \frac{d\mu_i}{n}\varphi_i(\imath(x_t))\\
&= d\EE[X] - \frac{d\mu_i}{n}\varphi_i(\imath(x_t))\\
&=0.
\end{align*}
We have $\EE_{t}[\Delta_t]=0$ and the martingale property for $Z_t$ is true.

\end{proof}

\subsection{Proof of \eqref{TC1}} 

The proof of these two identities is straightforward. Indeed, the martingale property for $Z_t=\mu_i^{-t}f_{\psi, \varphi_i, t}(T, x)$ shows that 
\[\frac{\EE[f_{\psi, \varphi_i, t}(T, x)]}{\mu_i^t}=\EE[Z_t] = \EE[Z_0]= \psi(x)\varphi_i(x) \]
which is exactly \eqref{TC1}. 

\subsection{Proof of \eqref{TC3}-\eqref{TC5}}

We fix $i,j$ in $[r]$  and $x$ in $[n]$ for the rest of the proof. Clearly, it is enough to do the computations with $\psi=\mathbf{1}$. We set 
\[Z^i_t = \frac{f_{\mathbf{1}, \varphi_i, t}(T, x)}{\mu_i^t} \]
and 
\[\Delta_t = \EE_t[(Z^i_{t+1}-Z^i_t)(Z^j_{t+1} -Z^j_t)]. \]
The $Z^i$ are martingales, hence 
\[\EE[Z^i_tZ^j_t] = \EE[\Delta_0+\dotsb+\Delta_{t-1}]. =\EE[\Delta_0]+\dotsb+\EE[\Delta_{t-1}].\]
Our goal is to compute those $\Delta_s$. First, we have 
\begin{align}\label{Etdelta}
\Delta_t &= \left(\frac{n^2}{\mu_i \mu_j d^2} \right)^{t+1}\sum_{\substack{d^+(o,x)=t \\d^+(o,x')=t}}\prod_{s=0}^{t-1} P_{\imath(x_s), \imath(x_{s+1})}P_{\imath(x'_s), \imath(x'_{s+1})} \times E(x_t, x'_t)
\end{align}
where the sum runs over all the couples of paths of length $t$ started at the root: $o=x_0 \to x_1 \to \dotsb x_t$ and $x'_0=o \to x'_1 \to \dotsb \to x'_t$, and where $E(x_t, x'_t)$ is given by
\begin{multline*}\mathbf{E}_t\left[ \left(\sum_{x_t \to y}P_{\imath(x_t), \imath(y)} \varphi_j(\imath(y)) - \frac{d\mu_j\varphi_j(\imath(x_t))}{n} \right) \times \right.\\ \left. \left(\sum_{{x_t}' \to y'}P_{\imath(x'_t), \imath(y')} \varphi_k(\imath(y')) - \frac{d\mu_k}{n}\varphi_k(\imath(x'_t)) \right)\right].
\end{multline*}
We have already computed those expectations in \eqref{eq:covariance_poisson_sums}. More precisely, when $x_t \neq x_t'$, the content of the two parentheses inside the expectation are totally independent and centered, hence the only contributions to \eqref{Etdelta} correspond to the summands where $x_t=x'_t$. In this case, \eqref{eq:covariance_poisson_sums} and \eqref{eq:EXY} yields
\[E(x_t, x_t) = \frac{d}{n^2}Q\varphi^{i,j}(\imath(x_t)).\]
We thus have
\begin{align}
\Delta_t &= \frac{d}{n^2}\left(\frac{n^2}{\mu_i \mu_j d^2} \right)^{t+1}\sum_{d^+(o,x)=t }\prod_{s=0}^{t-1} P_{\imath(x_s), \imath(x_{s+1})}^2 Q\varphi^{i,j}(\imath(x_t)).
\end{align}
Since our goal is to compute $\EE[\Delta_s]$ for any $s$, we now apply repeated conditioning: $\EE[\Delta_t]=\EE[\EE_0[\EE_1[...\EE_s[\Delta_s]\dots ]$. By the computations done earlier, it is easy to see that 
\[\EE[\Delta_s] = \left(\frac{d}{n^2}\right)^{s+1} \left(\frac{n^2}{\mu_i \mu_j d^2} \right)^{s+1} Q^s \varphi^{i,j}(\imath(x)) = \frac{Q^s\varphi^{i,j}(\imath(x))}{(\mu_i \mu_j d)^s}. \]
This directly gives identity \eqref{TC5}. Identity \eqref{TC3} also readily follows:
\begin{align*}
\EE[f_{\phi, \varphi_i, t}(T, x)f_{\phi, \varphi_j, t}(T, x)] &= \mu_i^t \mu_j^t \phi(x)^2 \EE[Z_t^i Z_t^j] \\
&= \mu_i^t \mu_j^t \phi(x)^2 \sum_{s=0}^t \frac{(Q^s \varphi^{i,j})(\imath(x))}{(\mu_i \mu_j d)^s}.
\end{align*}
As requested.

\section{Proof of Proposition \ref{prop:norms}}\label{sec:high-trace}

In this section, we prove Proposition \ref{prop:norms}. The proof relies on the expected high trace method introduced in random matrix theory by F\"uredi and Koml\`os \cite{MR637828} and on techniques developed in \cite{bordenave_lelarge_massoulie} for sparse random matrices. 

\subsection{Norm of $\uA^{(k)}$}

In this subsection, we prove the following lemma. 
\begin{lem}\label{le:nA}
There exists a universal constant $c \geq 3$ such that for all integers $1 \leq k \leq \ln(n)$ and $n \geq c K^{42}$,
$$
\PAR{ \EE \BRA{ \Vert \uA^{(k)} \Vert ^ {2m}}}^{\frac 1{2 m}} \leq \ln(n)^{9} K^4 \thresh^k,
$$
where $m =  \ln(n / K^6 ) / ( 12 \ln( \ln (n)) )$.
\end{lem}

From Markov inequality, Lemma \ref{le:nA} implies Equation \eqref{eq:normsA}. We start the proof of Lemma \ref{le:nA} by the norm identities
$$
\Vert \uA^{(k)} \Vert ^ {2m} = \Vert \uA^{(k)} {\uA^{(k)}}^*  \Vert ^ {m}  = \left\Vert \PAR{ \uA^{(k)} {\uA^{(k)}}^* }^m \right\Vert. 
$$
From the trace formula, we get
\begin{eqnarray*}
\Vert \uA^{(k)} \Vert ^ {2m} & \leq & \tr\BRA{ \PAR{ \uA^{(k)} {\uA^{(k)}}^* }^m } \\
& = & \sum_{(x_1,\ldots,x_{2m})} \prod_{t=1}^{m} (\uA^{(k)})_{x_{2t-1} x_{2t} } (\uA^{(k)})_{x_{2t+1} x_{2t}},
\end{eqnarray*}
where the product if over all $(x_1,\ldots,x_{2m})$ in $[n]^{2m}$ and we have set $x_{2m+1} = x_1$. From the definition of $\uA^{(k)}$ in \eqref{eq:defuA}, taking expectation, we get
\begin{eqnarray}\label{eq:nA1}
\EE \Vert \uA^{(k)} \Vert ^ {2m} & \leq & \PAR{\frac n d }^{2k m} \sum_{\gamma} \EE \prod_{i=1}^{2m} \prod_{t=1}^ k P_{\gamma_{i,t-1}\gamma_{i,t}} \uM_{\gamma_{i,t-1}\gamma_{i,t}},
\end{eqnarray}
where the sum is over all $\gamma = (\gamma_1, \ldots, \gamma_{2m} )$ with $\gamma_i = (\gamma_{i,0},\ldots, \gamma_{i,k}) \in F^k$ and the boundary conditions: for all $i \in [m]$,
$$
\gamma_{2i,0} = \gamma_{2i+1,0} \ANDalt \gamma_{2i-1,k} = \gamma_{2i,k}
$$
with $\gamma_{2m+1} = \gamma_1$.

We associate to an element $\gamma$ as above, a directed graph $G_\gamma = (V_\gamma,E_\gamma)$ with vertices $V_\gamma = \{\gamma_{i,t} : 1 \leq i \leq 2m , 0 \leq t \leq k \}$ and edge set $E_\gamma = \{(\gamma_{i,t-1},\gamma_{i,t}) : 1 \leq i  \leq 2m, 1 \leq t \leq k \}$. This graph may have loops (edges of $E_{\gamma}$ of the form $(x,x)$) and inverse edges (pair of edges $(x,y)$ and $(y,x)$ in $E_{\gamma}$). From the above boundary conditions, the graph $G_\gamma$ is simply connected. In particular, the genus of $G_\gamma$ is non-negative:
\begin{equation}\label{eq:genus}
|E_\gamma| - |V_\gamma| + 1 \geq 0. 
\end{equation}
Each oriented edge $e \in E_\gamma$ has a multiplicity $m_e$ defined as the number of times it is visited by $\gamma$: 
$$m_e = \sum_{(i,t) \in [2m] \times [k-1]} \IND_{(\gamma_{i,t},\gamma_{i,t+1}) = e}.$$
By construction, 
\begin{equation}\label{eq:summe}
\sum_{e \in E_\gamma} m_e = 2 k m.
\end{equation}

We may now estimate the expectation on the right-hand side of \eqref{eq:nA1}. Recall that the random variables $\uM_{xy} = (M_{xy} - d/n )$, $x,y$, are iid, centered, bounded by $1$ and with variance $(d/n) - (d/n)^2$. It follows that for any $p \geq 1$, $|\EE[ \uM_{xy} ^ p]| \leq d/n$. We deduce that
$$
\EE \prod_{i=1}^{2m} \prod_{t=1}^ k  P_{\gamma_{i,t-1}\gamma_{i,t}} \uM_{\gamma_{i,t-1}\gamma_{i,t}} =  \prod_{e \in E_\gamma} P_e^{m_e} \EE \SBRA{ \uM_{11}^{m_e }  }  \leq  \PAR{ \frac d n }^{|E_\gamma|} \prod_{e \in E_\gamma} |P_e|^{m_e} .
$$
Moreover, the above expectation is zero unless all edges have multiplicity at least $2$. From \eqref{eq:nA1}, we thus obtain that 
\begin{eqnarray}\label{eq:nA2}
\EE \Vert \uA^{(k)} \Vert ^ {2m} & \leq & \sum_{\gamma \in W_{k,m}}  \PAR{\frac n d }^{2k m - |E_\gamma|} \prod_{e \in E_\gamma} |P_e|^{m_e},
\end{eqnarray}
where $W_{k,m}$ is the set of paths $\gamma$ as above such that each edge of $E_\gamma$ is visited at least twice.

We now organize the sum \eqref{eq:nA2} in terms of the topological properties of the paths. We introduce the equivalence class in $W_{k,m}$, we write $\gamma \sim \gamma'$ if there exists a permutation $\sigma \in S_n$ such that $\gamma ' = \sigma \circ \gamma $, where $\sigma$ acts on $\gamma$ by mapping $\gamma_{i,t}$ to $\sigma(\gamma_{i,t})$. We denote by $\cW_{k,m}$ the set of equivalence classes. Obviously, $|V_\gamma|$ and $|E_\gamma|$ are invariant in each equivalence class. For $a,s$ integers, we denote by $\cW_{k,m} (s,a)$ the equivalence classes such that  $|V_\gamma| = s $ and $|E_\gamma| = a$. From \eqref{eq:genus},  $\cW_{k,m} (s,a)$ is empty unless $a - s + 1 \geq 0$.  Our first lemma is a rough estimate on $\cW_{k,m} (s,a)$.
\begin{lem}\label{le:numberiso}
Let $a,s \geq 1$ be integers such that $a - s + 1 \geq 0$. We have 
$$
|\cW_{k,m}(s,a)| \leq  (2km)^{6m(a-s+1) + 2m}.
$$
\end{lem}
\begin{proof}
This lemma is contained in the proof of \cite[Lemma 17]{bordenave_lelarge_massoulie}. We reproduce the proof for the reader convenience. Let $\gamma = (\gamma_1, \cdots, \gamma_{2m}) \in W_{k,m}$. We order the set $T  = \{ (i,t) : 1 \leq i \leq 2m, 0 \leq t \leq k-1\}$ with the lexicographic order.   We think of $T$ as time. For $0 \leq t \leq k-1$ and $i$ odd, we define $e_{i,t} = (\gamma_{i,t}, \gamma_{i,t+1})$, $y_{i,t} = \gamma_{i,t+1}$, while for $i$ even, we set $e_{i,t} = (\gamma_{i,k-t-1},\gamma_{i,k-t})$, $y_{i,t} = \gamma_{i,k-t-1}$ (in words: we reverse $\gamma_i$ for even $i$). A vertex $x \in V_\gamma \backslash \{\gamma_{1,1}\}$ is visited for the first time at $\tau \in T$ if $y_{\tau} = x$ and for all smaller $\sigma \in T$, $y_{\sigma} \neq x$.

We pick a distinguished path in each equivalence class by saying that $\gamma \in W_{k,m}$ is {\em canonical} if $V_\gamma = \{1, \ldots, |V_\gamma|\}$, $\gamma_{1,1} = 1$ and vertices are first visited in order. There exactly one canonical path in each equivalence class. We thus aim for an upper bound on the number of canonical paths in $W_{k,m}$ with $|V_\gamma| = s$ and $|E_\gamma| = a$ by designing an injective map  (or encoding) on such canonical paths.

Our goal is to retrieve unambiguously the values of $y_{\tau},\tau \in T,$ from minimal information. For $\tau \in T$, we say that $\tau$ is a {\em first time}, if $y_{\tau}$ has not been seen before.  If $\tau$ is a first time the edge $e_\tau$ is called a {\em tree edge}. By construction, the set of tree edges is a sub-graph of $G_\gamma$ with no weak cycle (without orientation) and vertex set $V_\gamma$. We call the other edges of $G_\gamma$ the {\em excess  edges}.  Any vertex different from $1$ has its associated tree edge.  It follows that the number of excess edges is $$g=a-s+1.$$ 
If $e_{\tau}$ is an excess edge, we say that $\tau$ is an {\em important time}. Other times are {\em tree times} (visit of a tree edge which has been seen before).

The set $T_i  = \{(i,t) : 0 \leq t \leq k-1\}$ is composed by the successive repetitions of  $(i)$ a sequence of the tree times (possibly empty), $(ii)$ a sequence of first times (possibly empty), $(iii)$ an important time.

We build a first encoding of canonical paths.  We mark the important times $(i,t)$ by the vector $(y_{i,t},y_{\tau-1})$, where $\tau \in T_i\cup \{ (i,k) \}$ is the next time that $e_{\tau}$ will not be a tree edge (by convention $\tau = (i,k)$  if $\gamma_i$ remains on the tree after $(i,t)$).  We can reconstruct a canonical path $\gamma \in W_{k,m}$, from the positions of the important times and their marks. Indeed, this follows from two observations $(1)$ there is at most one path between two vertices in an oriented tree, and $(2)$ if $v$ vertices has been seen so far and $\tau$ is a first time then $y_{\tau} = v+1$. It is our first encoding.

We refine this encoding by using the assumption that for each $i$, $\gamma_i$ is tangle-free. We partition important times into three categories, {\em short cycling}, {\em long cycling} and {\em superfluous times} as follows. Assume that $\gamma_i$ contains a cycle. Consider the smallest time $(i,t_1)$ such that $y_{i,t_1} \in \{y_{i,-1},...,y_{i,t_1-1}\}$, where $y_{i,-1}  = \gamma_{i,0}$ for odd $i$ and $y_{i,-1} = \gamma_{i,k}$ for even $i$. Let $-1\leq t_0 \leq t_1$ be such that $y_{i,t_1} = y_{i,t_0}$.  By the assumption of $\gamma_i$ being tangle-free, $C = (y_{i,t_0},\cdots,y_{i,t_1})$ is the only directed cycle visited by $\gamma_i$. The last important time, say  $(i,t_i)$, before $(i,t_1)$ is called the short cycling time. We denote by $t_2$ the next time after $(i,t_1)$ that is not an edge of $C$ ($\gamma_i$ circles around $C$ between times $(i,t_0)$ and $(i,t_2)$). We modify the mark of the short cycling time as $(y_{i,t_i},y_{\tau-1},t_3)$ where $\tau = (i,t_3) \in T_i\cup \{ (i,k) \}$ is the next time after $(i,t_2)$ that $e_{\tau}$ will not be a tree edge (by convention $\tau = (i,k)$  if $\gamma_i$ remains on the tree). Note that that this $(i,t_i)$ is the last important time, all steps to close the cycle are on tree edges. It follows that the pair $(y_{i,t_i},y_{\tau-1})$ determines $y_{i,t_0}$. Important times $(i,t)$ with $0 \leq t < t_i$ or $t_2 \leq t \leq k-1$ are called long cycling times.  The other important times are superfluous. The key observation is that for each $1 \leq i \leq 2m$, the number of long cycling times $(i,t)$ is bounded by $g-1$ (since there is at most one cycle, no edge of $\gamma_i$ can be seen twice outside those of $C$, $-1$ coming from the fact that the short cycling time is an important time). Now consider the case where the $i$-th path does not contain a cycle, then all important times are called long cycling times and their number is bounded by $g$.

We can reconstruct a canonical path $\gamma \in W_{k,m}$, from the sole positions of the short and long cycling times and their marks. This our second encoding. For each $i$, there are at most $(k+1)^{g}$ ways to position the short and long cycling times of $T_i$, $s^2$ possibilities for the mark of a long cycling time and $s^2 k$ possibilities for the mark of a short cycling time. We deduce that 
$$
|\cW_{k,m}(s,a)| \leq (k+1)^{2mg} (s^2)^{2m(g-1)} (s^2 k)^{2m}.
$$
Since $s \leq 2km$, the conclusion follows.
\end{proof}

Our second lemma bounds the contributions of paths in each equivalence class. This lemma is the new main technical difference of this section with \cite{bordenave_lelarge_massoulie}.

\begin{lem}\label{le:sumiso}
Let $\gamma \in W_{k,m}$ such that $|V_\gamma| = s$ and  $|E_\gamma| = a$. We have 
$$
\sum_{\gamma' : \gamma' \sim \gamma} \prod_{e \in E_{\gamma'}} |P_e|^{m_e}  \leq n^{-2km + s} K^{k m - a} K^{6(a-s) + 8m} \rho^{km} .
$$
\end{lem}
\begin{proof}
We first express the product of entries of $P$ in terms of the matrix $Q$:
$$
 \prod_{e \in E_\gamma} |P_e|^{m_e} =  \prod_{e \in E_\gamma} |P_e|^{m_e-2} P_e^ 2 \leq \PAR{\frac {\sqrt{\rho K}}n}^{2km -2a}   \prod_{e \in E_\gamma} \frac{Q_e}{n},
$$
where we have used \eqref{eq:summe}, $m_e \geq 2$ and $\max |P_{xy}| = \sqrt{\rho K}/n$.

The statement of the lemma immediately follows from the claim: 
\begin{equation}\label{eq:koko}
\sum_{\gamma' : \gamma' \sim \gamma}\prod_{e \in E_{\gamma'}} Q_e \leq  \rho^{a} n^{s-a} K^{6(a-s) + 8m}
\end{equation}
Indeed, let us check \eqref{eq:koko}. Let us define the degree of a vertex $x$ 
in $V_\gamma$ as the sum of in-degrees and out-degrees: $\sum_{i,t} (\IND_{\gamma_{i,t-1} = x} +   \IND_{\gamma_{i,t} = x })$.  Let $s_k$ and $s_{\geq k}$ be the set of vertices of degree $k$ and at least $k$. We have
$$
s_1 + s_2 + s_{\geq 3}  = s\ANDalt s_1 + 2 s_2 + 3 s_{\geq 3} \leq \sum_k k s_k = 2 a. 
$$
Subtracting the right-hand side to twice the left-hand side, we find 
$$
 s_{\geq 3} \leq 2 ( a- s ) + s_1  \leq 2 ( a -s )  + 2 m.
$$
The bound $s_1 \leq 2m$ comes from the fact that only the vertices $\gamma_{i,0}$ and $\gamma_{i,k}$, with $i \in [2m]$, can be of  degree $1$. Indeed, other vertices are of degree at least $2$: for $1 \leq t \leq k$,  $\gamma_{i,t}$ has in-degree at least $1$ and out-degree at least $1$.

Consider the subset of vertices $\hat V_\gamma \subset V_\gamma$ which are of degree at least $3$ or are among the extremes vertices $ \gamma_{i,0}$, $\gamma_{i,k}$, $i \in [2m]$. In particular, $V_\gamma \backslash \hat V_\gamma$ contains only vertices of degree $2$. From what precedes 
\begin{equation}\label{eq:hats}
\hat s = | \hat V_\gamma| \leq 2 (a-s) + 4m.
\end{equation}
We may partition the edges of $E_\gamma$ into $\hat a$ sequences of edges of the form, for $1 \leq j \leq \hat a$,  $\hat e_j = (e_{j,1}, \ldots,e_{j,q_j})$,  with $e_{j,t} = (x_{j,t-1},x_{j,t}) \in E_{\gamma}$, $x_{j,0}, x_{j,q}$ in $\hat V_\gamma$  and $x_{j,t} \notin \hat V_\gamma$ for $1 \leq t \leq q_j-1$. By construction
\begin{equation}\label{eq:hatqj}
\sum_{j=1}^{\hat a} q_j = a.
\end{equation}
We consider the directed graph $\hat G_\gamma$ on the vertex set $\hat V_\gamma$ whose  $\hat a$ edges are,  for $1 \leq j \leq \hat a$,  $(x_{j,0},x_{j,q_j})$ (this is a multi-graph: if two sequences $\hat e_j$ and $\hat e_i$, $i \ne j$, have the same extreme vertices, it creates two edges).   It is straightforward to check that  this operation preserves the genus: 
\begin{equation}\label{eq:hatgenus}
a - s = \hat a - \hat s.
\end{equation}

For ease of notation, let $y_1, \cdots, y_{\hat s}$ be the elements of $\hat V_\gamma$. Let $a_j$ and $b_j$ the indices such that $x_{j,0} = y_{a_j}$ and $x_{j,q_j} = y_{b_j}$. Summing over all possible vertices, we get 
$$
 \sum_{\gamma' : \gamma' \sim \gamma} \prod_{e \in E_{\gamma'}} Q_e \leq \sum_{(y_1,\cdots,y_{\hat s}) \in [n]^{\hat s}} \prod_{j=1}^{\hat a}  Q^{q_j} _{y_{a_j}y_{b_j}},
$$
where we have used that $$\sum_{ (x_{j,1}, \cdots, x_{j,q_{j} -1})} \prod_{t=1}^{q_j} Q_{x_{j,t-1} x_{j,t}} = Q^{q_j} _{x_{j,0}, x_{j,q_j}}.
$$
We apply \eqref{eq:boundQt} and find
$$
 \sum_{\gamma' : \gamma' \sim \gamma} \prod_{e \in E_{\gamma'}} Q_e \leq \sum_{(y_1,\cdots,y_{\hat s}) \in [n]^{\hat s}} \prod_{j=1}^{\hat a}  \PAR{ \frac{K^2 \rho^{q_j}}{n}}.
$$
Using \eqref{eq:hats}-\eqref{eq:hatgenus}, we have $\hat a \leq 3(a-s) + 4m$ and, from \eqref{eq:hatqj}, Equation \eqref{eq:koko} follows. 
\end{proof}

We are ready for the proof of Lemma \ref{le:nA}.

\begin{proof}[Proof of Lemma \ref{le:nA}]
Note that $\cW_{k,m}(s,a)$ is empty unless $0 \leq s -1 \leq a \leq km$ (since each edge has multiplicity at least $2$, we have $2 |E_\gamma|  \leq 2km$ from \eqref{eq:summe}) From \eqref{eq:nA2}, we get
\begin{eqnarray*}
\EE \Vert \uA^{(k)} \Vert ^ {2m} & \leq & \sum_{a= 1}^{km } \sum_{s = 1}^{a+1} \PAR{ \frac{n}{d} }^ {2km -a} | \cW_{k,m} (s,a) | \max_{\gamma \in \cW_{k,m}(s,a)}  \sum_{\gamma' : \gamma' \sim \gamma} \prod_{e \in E_{\gamma'}} |P_e|^{m_e}.
\end{eqnarray*}
Using Lemma \ref{le:numberiso} and Lemma \ref{le:sumiso}, we arrive at
\begin{eqnarray*}
\EE \Vert \uA^{(k)} \Vert ^ {2m} & \leq & n  \sum_{a= 1}^{km } \sum_{g = 0}^{\infty} d^{a - 2km}(2km)^{6mg +2m}  n^{-g} K^{km -a} K^{6 g + 8m -6} \rho^{km} \\
& = & n\thresh_2^{2km} (2km)^{2m} K^{8m-6} \sum_{a= 1}^{km } \PAR{\frac{K}{d}}^{km-a} \sum_{g = 0}^{\infty}   \PAR{\frac{K^ 6 (2km)^{6m}}{n} }^{g},
\end{eqnarray*}
where we have performed the change of $s \to g = a - s +1$ and used $\thresh_2 = \sqrt{\rho/d}$.

Recall $k \leq \ln(n)$. We take $ m = \lceil \ln (n / K^6) / (12 \ln (\ln(n))) \rceil$. If $n \geq c K^6$ for some universal constant $c$, we find that 
$$
\frac{K^6 (2km)^{6m}}{n} \leq \frac 1 2.
$$
We deduce that
$$\EE \Vert \uA^{(k)} \Vert ^ {2m}  \leq  n\thresh_2^{2km} \PAR{1 \vee  \frac K d} ^{km} (2km)^{2m} K^{8m} (2km).
$$
For our choice of $m$, $2km \leq \ln (n)^2/ \ln (\ln (n)) $ and, if $n \geq  K^{42}$, then $n^{1/(2m)} \leq \ln (n)^{7}$. Since $\theta_2 \sqrt{K/d} = \thresh_1$, the conclusion follows easily. \end{proof}

\subsection{Norm of $R_{k}^{(\ell)}$}

In this subsection, we prove \eqref{eq:normsR}. 

\begin{lem}\label{le:nR}
There exists a universal constant $c\geq 3$ such that for all integers $1 \leq k  \leq \ell \leq \ln(n)$ and $n \geq c$,
$$
\PAR{ \EE \BRA{ \Vert R^{(\ell)}_k \Vert ^ {2m}}}^{\frac 1{2 m}} \leq  \frac{d}{n} \ln(n)^{23}  L^\ell,
$$
where $m =  \ln(n ) / ( 24 \ln( \ln (n)) )$.
\end{lem}

The proof follows from the same line than the proof of Lemma \ref{le:nA}. It is also essentially contained in \cite{bordenave_lelarge_massoulie}. To avoid repetitions, we only focus on the main differences with the proof of Lemma \ref{le:nA}. The computation leading to \eqref{eq:nA1} gives
\begin{equation}
\label{eq:nR1}
\EE \Vert R^{(\ell)}_k \Vert ^ {2m}  \leq  \PAR{\frac n d }^{2(\ell -1) m} \hspace{-6pt} \sum_{\gamma \in W'_{\ell,m}}\hspace{-3pt} \EE \prod_{i=1}^{2m} \prod_{t=1}^{\ell}   P_{\gamma_{i,t-1}\gamma_{i,t}} \PAR{ \IND_{ t <  k} \uM_{\gamma_{i,t-1}\gamma_{i,t}} +  \IND_{ t = k} + \IND_{ t >  k} M_{\gamma_{i,t-1}\gamma_{i,t}} },
\end{equation}
where $W'_{\ell,m}$ is the set of  $\gamma = (\gamma_1, \ldots, \gamma_{2m} )$ with $\gamma_i = (\gamma_{i,0},\ldots, \gamma_{i,\ell}) \notin F^\ell$, $(\gamma_{i,0}, \ldots, \gamma_{i,k-1}) \in F^{k-1}$, $(\gamma_{i,k+1}, \ldots, \gamma_{i,\ell}) \in F^{\ell -k}$ and the boundary conditions: for all $i \in [m]$,
$$
\gamma_{2i,0} = \gamma_{2i+1,0} \ANDalt \gamma_{2i-1,\ell} = \gamma_{2i,\ell}
$$
with $\gamma_{2m+1} = \gamma_1$.

We associate to an element $\gamma \in W'_{\ell,m}$ the directed graph $G'_\gamma = (V'_\gamma,E'_\gamma)$ with vertices $V'_\gamma = \{\gamma_{i,t} : 1 \leq i \leq 2m , 0 \leq t \leq \ell\}$ and edge set $E'_\gamma = \{(\gamma_{i,t-1},\gamma_{i,t}) : 1 \leq i  \leq 2m, 1 \leq t \leq \ell, t \ne k \}$. The graph $G_\gamma$ is not necessarily weakly connected (since $E'_\gamma$ does not contain the edges $(\gamma_{i,k-1},\gamma_{i,k})$). However, we have the following observation.
\begin{lem}\label{le:genusR}
If $\gamma$ is as above then each connected component of of $G'_\gamma$ contains a cycle. In particular, $|E'_\gamma| \geq |V'_\gamma|$.
\end{lem}
\begin{proof}
By recursion, it is then enough to check that each connected component of $G'_{\gamma_i}$ contains a cycle. By assumption, $\gamma_i = (\gamma_{i,0},\ldots, \gamma_{i,\ell}) \notin F^\ell$ contains two distinct cycles. Up to recomposing a new cycle, we may assume without loss of generality that the edge $(\gamma_{i,k-1},\gamma_{i,k})$ is in zero or one of the two cycles. If it is in one of them, then the graph $G'_{\gamma_i}$ is weakly connected and it contains the other cycle. Assume now that $(\gamma_{i,k-1},\gamma_{i,k})$ is in none of the two cycles. If $G'_{\gamma_i}$ is weakly connected, there is nothing to prove. If $G'_{\gamma_i}$ is not weakly connected, then the two connected components are the vertices of $ (\gamma_{i,0},\ldots, \gamma_{i,k-1})$
 and $ (\gamma_{i,k},\ldots, \gamma_{i,\ell})$. Since these two paths are tangle-free, each must contain exactly one of the two cycles  and the statement follows.
\end{proof}

Using the independence of the entries of $M$ and $|P_{xy}| \leq L /n$, 
$$
\EE \prod_{i=1}^{2m} \prod_{t=1}^{\ell}   P_{\gamma_{i,t-1}\gamma_{i,t}}\PAR{ \IND_{ t <  k} \uM_{\gamma_{i,t-1}\gamma_{i,t}} +  \IND_{ t = k} + \IND_{ t >  k} M_{\gamma_{i,t-1}\gamma_{i,t}} } \leq  \PAR{ \frac d n }^{|E'_\gamma|} \PAR{\frac{L}{n}}^{2 \ell m} .
$$

We thus obtain that 
\begin{eqnarray}\label{eq:nR2}
\EE \Vert R^{(\ell)}_k \Vert ^ {2m} & \leq & \PAR{ \frac{d}n }^{2m} \PAR{\frac L d  }^{2\ell m}\sum_{\gamma \in W'_{\ell,m}}  \PAR{\frac d n }^{|E'_\gamma|},
\end{eqnarray}

 We introduce the equivalence class in $W'_{\ell,m}$, we write $\gamma \sim \gamma'$ if there exists a permutation $\sigma \in S_n$ such that $\gamma ' = \sigma \circ \gamma $, where $\sigma$ acts on $\gamma$ by mapping $\gamma_{i,t}$ to $\sigma(\gamma_{i,t})$. We denote by $\cW'_{\ell,m} (s,a)$ the set of equivalence classes such that  $|V'_\gamma| = s $ and $|E'_\gamma| = a$. From Lemma \ref{le:genusR},  $\cW_{k,m} (s,a)$ is empty unless $a \geq s$. We have the following estimate on $\cW'_{\ell,m} (s,a)$.
\begin{lem}\label{le:numberisoR}
Let $a \geq s \geq 1$ be integers. We have 
$$
|\cW'_{\ell,m}(s,a)| \leq  (2\ell m)^{12 m(a-s) + 22 m}.
$$
\end{lem}
\begin{proof}
A proof is contained in \cite[Lemma 18]{bordenave_lelarge_massoulie}. We give a proof for the reader convenience. We order the sets $T'  = \{ (i,t) : 1 \leq i \leq 2m, 1 \leq t \leq \ell-1 , t \ne k\}$ and $T  = \{ (i,t) : 1 \leq i \leq 2m, 1 \leq t \leq \ell-1 \}$ with the lexicographic order. We think of $T$ and $T'$ as times. If $\tau \in T'$, we denote $\tau^-$ the largest element in $T'$ smaller than $\tau$. By convention $(1,0)^- = (1,-1)$. For $\tau \in T$, we define $e_{\tau}$ and $y_\tau$ as in Lemma \ref{le:numberiso} and  we say $\gamma \in W'_{\ell,m}$ is {\em canonical} if $V'_\gamma = \{1, \ldots, |V'_\gamma|\}$, $\gamma_{1,1} = 1$ and vertices are first visited in order. We aim for an upper bound on the number of canonical paths in $W'_{\ell,m}$ with $|V'_\gamma| = s$ and $|E'_\gamma| = a$ by designing an injective map.

We define a sequence of growing sub-forests $(F_\tau)_{\tau \in T'}$ of $G'_\gamma$ as follows. We start with $F_{(1,-1)}$, the trivial graph with no edge and a $\gamma_{1,1} = 1$ as unique vertex. For $\tau \in T'$, we say that $\tau$ is a {\em first time}, if adding $e_{\tau}$ to $F_{\tau^-}$ does not create a weak cycle.  If $\tau$ is a first time the edge $e_\tau$ is called a {\em tree edge} and we define $F_\tau$ as the union of $e_\tau$ and $F_{\tau^-}$. Otherwise, $F_{\tau} = F_{\tau^-}$. We set $F = F_{2m,\ell}$.  By construction, the set of tree edges is a sub-graph of $G'_\gamma$ with no weak cycle and vertex set $V'_\gamma$. Moreover, the weak connected components of $G'_\gamma$ and $F$ are equal. We call the other edges of $G'_\gamma$ the {\em excess  edges}. In each weak connected component of $G'_\gamma$ there are at most  $g=a-s+1$ excess edges. Indeed,  if $a'$, $s'$ are the numbers of directed edges and vertices of a connected component, then there are $a' - s' +1 $ excess edges in this connected component. However by Lemma \ref{le:genusR}, $a' - s' \leq a - s$.  If $e_{\tau}$ is an excess edge, we say that $\tau \in T'$ is an {\em important time}. Other times in $T'$ are {\em tree times} (visit of a tree edge which has been seen before).

Let $k_i  = k$ for odd $i$ and $k_i = k-\ell+1$ for even $i$. We define the sets $T^1_i  = \{(i,t) : 0 \leq t \leq k_i -1\}$ and $T^2_i  = \{(i,t) :k_i \leq t \leq \ell\}$. 
For each $i$, there could be a special first time $(i,t) \in T_i^2$, called the {\em merging time}, such that a connected component of $F_{(i,t)^-}$ merges into a connected component of $F_{i,k_i-1}$ by the addition of $e_{i,t}$.

The sets $T^\veps_i$  are composed by the successive repetitions of  $(i)$ a sequence of the tree times (possibly empty), $(ii)$ a sequence of first times (possibly empty), $(iii)$ an important time or the merging time. We mark the important and merging times (for $\veps = 2$) $(i,t) \in T^\veps_i$ by the vector $(y_{i,t},y_{\tau-1})$, where $\tau \in T^\veps_i \cup \{ (i,k_i)\}$ is the next time that $e_{\tau} $ will not be a tree edge (by convention $\tau = (i,k_i)$  if $T^\veps_i$ only contains tree times after $(i,t)$). We can reconstruct a canonical path $\gamma \in W'_{\ell,m}$, from the positions of the merging and important times and their marks.

We refine this encoding by partitioning important times into three categories, {\em short cycling}, {\em long cycling} and {\em superfluous times} exactly as done in Lemma \ref{le:numberiso}, except that there are short and long cycling times for each $i$ and $\veps \in \{1,2\}$ in the sequence $T_i^\veps$. There are either $0$ short cycling times and at most $g$ long cycling times, or $1$ short cycling time and at most $g-1$ long cycling time (because in each connected component of $G'_\gamma$ there are at most $g$ excess edges).

We can reconstruct a canonical path $\gamma \in W'_{\ell,m}$, from the positions of the merging, short and long cycling times and their marks. There are at most $\ell^{2m}$ ways to position the merging times. For each $i,\veps$, there are at most $\ell^{g}$ ways to position the short and long cycling times of $T^\veps_i$, $s^2$ possibilities for the marks of a merging or long cycling time and $s^2k$ possibilities for the marks of a short cycling time. We deduce that 
$$
|\cW'_{\ell,m}(s,a)| \leq \ell^{4mg +2m} (s^2)^{4m(g-1) +2m} (s^2 k )^{4m}.
$$
Since $s \leq 2\ell m$, the conclusion follows.
\end{proof}

We are ready for the proof of Lemma \ref{le:nR}.

\begin{proof}[Proof of Lemma \ref{le:nR}]
There are $n (n-1) \cdots (n-s+1)$ elements of $W'_{\ell,m}$ in an equivalence class in $\cW'_{\ell,m}(s,a)$. From \eqref{eq:nR2} and Lemma \ref{le:numberisoR} we get
\begin{eqnarray*}
\EE \Vert R_k^{(\ell)} \Vert ^ {2m} & \leq &  \PAR{ \frac{d}n }^{2m}  \PAR{\frac L d  }^{2\ell m} \sum_{a= 1}^{2 \ell m } \sum_{s = 1}^{a} n^s  (2\ell m)^{12 m(a-s) + 22 m} \PAR{\frac d n }^{a},
\end{eqnarray*}
We perform the change of variable $s \to p = a - s$:
\begin{eqnarray*}
\EE \Vert R_k^{(\ell)} \Vert ^ {2m}  & \leq &  \PAR{ \frac{d}n }^{2m}  \PAR{\frac L d  }^{2\ell m} (2\ell m)^{22m} \sum_{a= 1}^{ 2 \ell m } d^a \sum_{p = 0}^{\infty} \PAR{\frac{(2\ell m)^{12 m}}{n}}^p.
\end{eqnarray*}

Recall $\ell \leq \ln(n)$ and $d \geq 1$. We take $ m = \lceil \ln (n) / (24 \ln (\ln(n))) \rceil$. If $n \geq c_1$ for some universal constant $c_1$, we find that 
$$
\frac{ (2km)^{12 m}}{n} \leq \frac 1 2.
$$
We deduce that
$$\EE \Vert R_k^{(\ell)} \Vert ^ {2m}  \leq \PAR{ \frac{d}n }^{2m}  L ^{2\ell m} (2\ell m)^{22m} (2 \ell m).
$$
For our choice of $m$, $2\ell m \leq \ln (n)^2$. The conclusion follows easily. \end{proof}

\section{Proof of Theorem \ref{thm:1nb} and Corollary \ref{cor:1nb}}
\label{sec:nb}

The proofs of Theorem \ref{thm:1nb} follows the same line than the proof of Theorem \ref{thm:1nb}. In this section we explain the differences.

\subsection{Main technical result}

We will assume without loss of generality that 
$ \mu_1 = 1$.  For ease of notation, we set in this section $d =\bar d$, $\thresh = \bar \thresh$, $M  = \bar M$ and $D = \bar d \vee 1.01$. Note that since $d \geq 1$, we have $D \leq 2d$.

We will assume in the sequel without loss of generality that 
\begin{equation}\label{eq:bounddnb}
\log (n) \geq 8 \log (d).
\end{equation}
Also, from \eqref{lower_bound_on_rho} $\rho \geq \mu^2_1 = 1$. In particular 
$
\thresh \geq 1 / \sqrt d$ and $\tau^{-1}_0 \leq \sqrt d$.  It follows that in the statement of  Theorem \ref{thm:1nb}, we have
\begin{equation}\label{eq:taudnnb}
\tau^{-2\ell}_0 \leq d^{\ell} \leq n^{1/8}.
\end{equation}

We set $E_n = [n]^2 = \{ (x,y) : x,y \in [n] \}$. With a slight abuse of notation, we can identify the non-backtracking matrix with its natural extension on $E_n$ defined as follows: for all $e = (x,y) \in E_n$, $f = (a,b) \in E_n$, 
$$
B_{e,f} = \frac{n}{d} M_e M_f \IND ( y = a )\IND ( b \ne x) P_{f}.
$$
The advantage of this new equivalent definition of $B$ is that it is defined on the deterministic set $E_n$ rather than on the random set $E$. Similarly, if $\varphi$ is a vector in $\dR$, we set for all $e  = (x,y) \in E_n$,
\begin{equation}\label{eq:liftpm}
\varphi^{+}(e) = \frac {\IND_{e \in E}} {\sqrt{d}} \varphi(y) \quad \hbox{ and } \quad \varphi^{-}(e) = \frac {\IND_{e \in E}} {\sqrt{d}} \varphi(x) 
\end{equation}

If $r_0 \geq 1$, we define $\Phi^{\pm} = ( \varphi^\pm_1, \dotsc, \varphi^\pm_{r_0})$ and $\Sigma = \mathrm{diag}(\mu_1, \dotsc, \mu_{r_0})$.  
The candidate left and right eigenvectors $(u_1, \ldots, u_{r_0})$ and $(v_1,\ldots,v_{r_0})$ are the columns of the matrices:
\begin{equation*}
U = B^\ell \Phi^+ \Sigma^{-\ell} \qquad \text{ and } \qquad V = (B^*)^\ell \Delta \Phi^- \Sigma^{-\ell-1},
\end{equation*}
where $\Delta \in \mathscr{M}_{E_n} (\dR)$ is the diagonal matrix defined for all $e \in E_n$ by $$\Delta_{e,e} =  n P_{e}.$$
We set $S = U\Sigma^\ell V^*$ and 
introduce the vector spaces 
$H = \mathrm{vect}(v_1, \dotsc, v_{r_0}) 
$ and $H' = \mathrm{vect}(u_1, \dotsc, u_{r_0}).$

Finally, if $r_0 =0$, then $S$ is simply set to be the zero matrix, and $H,H'$ are the trivial vector spaces.

Then, we define the matrix $\hat \Gamma^{(t)} = d(\Gamma^{(t+1)} - I_{r_0}) \in \mathscr{M}_{r}( \mathbb{R} )$: that is,
\begin{equation}\label{def:Hgammat}
\hat \Gamma^{(t)}_{i,j} = d \sum_{s=1}^{t+1} \frac{\langle \mathbf{1}, Q^s \varphi^{i,j}\rangle}{(\mu_i \mu_j d)^s}.
\end{equation}
The proof of Lemma \ref{lem:Gammat} implies the following.

\begin{lem}[Properties of $\hat \Gamma^{(t)}$]\label{lem:HGammat}For any $t$, the matrix $\hat \Gamma^{(t)}$ is a semi-definite positive matrix with eigenvalues greater  than $\sigma$, and with
\[\sigma \leq \Vert \hat \Gamma^{(t)} \Vert \leqslant  d r_0 b^8  \frac{\tau_0^2  - \tau_0^{2(t+2)} }{1 - \tau_0^2}.\]
\end{lem}

The following result is the analog of Theorem \ref{thm:algebra} for the non-backtracking operator $B$.

\begin{theorem}[Algebraic structure of $U$ and $V$]\label{thm:algebranb}
There are a universal constant $c >0$ and  an event with probability greater than $1-c n^{-1/4}$ such that the following holds:
\begin{equation}\label{eq:PhiAlPhinb}
\Vert V^* B^\ell U - \Sigma^\ell \Vert \leqslant C_1  n^{-1/4} \tau_0^{2\ell}\thresh^{\ell}
\end{equation}
\begin{equation}\label{eq:U*Vnb}
\Vert U^* V - I_{r_0} \Vert \leqslant C_1 n^{-1/4} \tau_0^{2 \ell}
\end{equation}
\begin{equation}\label{eq:PhiAlPhi0nb}
\Vert (\Phi^+)^* U - I_{r_0} \Vert \leqslant C_1  n^{-1/4} \tau_0^{2\ell}
\end{equation}
\begin{equation}\label{eq:PhiAlPhi00nb}
\Vert (\Phi^+)^* V - I_{r_0} \Vert \leqslant C_1  n^{-1/4} \tau_0^{2\ell}
\end{equation}
\begin{equation}\label{eq:U*Unb}
 \Vert U^* U -   \Gamma^{(\ell)} \Vert \leqslant C_1 n^{-1/4} \tau_0^{2 \ell}
\end{equation}
\begin{equation}\label{eq:V*Vnb}
\Vert V^* V -   \hat \Gamma^{(\ell)} \Vert  \leqslant C_1 n^{-1/4} \tau_0^{2 \ell}
\end{equation}
\begin{equation}\label{norm_on_orthogonalnb}
\Vert B^\ell \proj_{H^\perp} \Vert \leqslant C_2 \thresh ^\ell 
\end{equation}
\begin{equation}\label{norm_on_orthogonal2nb}
\Vert  \proj_{{H'}^\perp}  B^\ell \Vert \leqslant C_2 \thresh^{\ell},
\end{equation}
where $C_1 =  c r_0 b^6\ln(n)^{7/2}$ and $C_2 = c   d b^{16} \bar r \ln(n)^{10} $.
\end{theorem}

The same statements holds, with the same constants, if we replace $\ell$ by $\ell' = \ell+1$. Repeating the proof of Theorem \ref{thm:1} in Subsection \ref{subsec:proofth1} and Corollary \ref{cor:thm1}, we obtain Theorem \ref{thm:1nb}.

We now check that Corollary \ref{cor:1nb} holds. We start by a comment. We first observe the matrix identity
$$
B^* \Delta J =  \Delta J B
$$
where $J$ is the matrix such that for all $(x,y) \in E$, we have $J\delta_{(x,y)} = \delta_{(y,x)}$. In particular, we find that $\Delta J \psi_i$ and $\psi'_i$ are proportional vectors: there exists $c >0 $ such that for all $(x,y) \in E$, $ c \psi'_i(x,y) =  n P_{x,y} \psi_i (y,x)$. if $\mathrm{deg}(i) = 1$, we have that $\hat \psi_i(y) = 0$. If $\deg(y) \geq 2$, using the eigenvalue equation, we deduce that 
$$
\lambda_i \hat \psi_i (y) = \frac{1}{d(\mathrm{deg}(y) -1)} \sum_{x : (x,y) \in E } \lambda_i \psi_i((x,y)) = \frac{1}{d} \sum_{x : (x,y) \in E } n P_{y,x} \psi_i((y,x)) = c \check \psi_i '(y).
$$
The estimators $\check \varphi_i$ and $\hat \varphi_i$ are thus very close (note however that if $\mathrm{deg}(i) = 1$, we have that $\hat \psi_i(y) = 0$ but $\check \psi'_i(y)$ could be different from $0$).

Let us now check that the conclusion of Corollary \ref{cor:1nb} holds. Arguing as in the proof of Corollary \ref{cor:thm1}, it suffices to check the conclusion with our approximate eigenvectors $u_i$ and $v_i$ in place of the actual eigenvectors $\psi_i$ and $\psi'_i$. It then follows from a slight modification of Theorem \ref{thm:algebranb} that, with probability at least $1- c n^{-1/4}$,  $|\langle \hat u_i , \hat u_j \rangle - \delta_{i,j} \Gamma_{i,j}^{(\ell)}| \leq C_1 n^{-1/4} \tau_0^{2\ell}$ and $|\langle \hat u_i , \varphi_j \rangle - \delta_{i,j}| \leq C_1 n^{-1/4} \tau_0^{2\ell}$. The same holds with the vectors $\check v_i$. The conclusion follows.

\subsection{Near eigenvectors: claims \eqref{eq:PhiAlPhinb}-\eqref{eq:V*Vnb}}

We denote by $G = ([n], E)$ the undirected random graph whose adjacency matrix is $M$: this is an \erd random graph with parameter $d/n$ (with loops).

The first step is to extend the notion of graph neighborhood $(G,x)_t$ to directed edge neighborhood. If $e = (y,x) \in E_n$, we define $(G,e)_t$ has the rooted graph $(G^e,x)_t$ where $G^e$ is the graph obtained from $G$ by removing the edge $\{x,y\}$ if it is in $E$ otherwise $G^e = G$.  As explained in \cite{bordenave_lelarge_massoulie}, Lemma \ref{prop:tangle-free} and Proposition \ref{prop:DTV} also hold for $(G,e)$ in place of $(G,x)$ and $(T,x) = (T,e)$ is now the marked Poisson Galton-Watson tree with parameter $d$ with root mark $x$ if $e = (y,x)$.

If $(g,o) \in \mathscr{G}_*$, we write $o \sim u$ if $u$ and $o$ share an edge and $(g,(o,u))$ is defined as above as the rooted graph rooted at $u$ where the edge with $o$ has been removed. The version of Theorem \ref{thm:concentration} that we will need is the following: 
\begin{theorem}\label{thm:concentration2}
Let $0 \leq \kappa \leq 0.49$ and $h$ an integer as in \eqref{eq:defhrad}. Let $f : \mathscr{G}_* \to \mathbb{R}$ be a $(h+1)$-local function  such that $|f(g,o)|\leqslant \alpha  \sum_{u\sim o} |(g,(o,u))_h|^\beta$ for some $\alpha ,\beta >0$. Then, for some universal constant $c >0$, for any $s  \geq 1$, with probability greater than $1-n^{-s} $, we have
\begin{equation*}
\left| \sum_{x \in [n]} f(G,x) - \EE  \sum_{x \in [n]} f(T,x)  \right| \leqslant c e^\beta \alpha s^{5/2 + \beta} \ln(n)^{5/2 + \beta}  D n^{\kappa(1+ \beta) +1/2}.
\end{equation*}
\end{theorem}
The proof is the same than the proof of Theorem \ref{thm:concentration}. The only difference is at the first lines: we set $\bar f(g,o) = \alpha \sum_{u\sim o} |(g,(o,u))_h|^\beta / d$. Note that if an integer random variable $N$ is independent of $(Z_k)$ i.i.d. non-negative, then for any $q \geq 1$,
$$
\EE \left[ \left( \sum_{k=1}^N Z_k \right)^{q} \right] \leq \EE \left[N^ {q-1} \sum_{k=1}^N Z^{q}_k  \right] = \EE [N^{q}] \EE[ Z_1^{q}].
$$
Hence, by \eqref{growth:expectation_balls_gr} (applied to $t=h$ and $t=1$), we have for all $p \geq 1$ and $x \in [n]$, for some universal constant $c>0$,
$$
\EE [ \bar f(G,x)^{2p}] \leq  \alpha^{2p} \left(cp\right)^{2 p (\beta+1)}  D^{2p (\beta h +1)} .
$$ 
The rest is identical.

%

Theorem \ref{thm:concentration2} and Theorem \ref{thm:tree_computations} are used
to prove claims \eqref{eq:PhiAlPhinb}-\eqref{eq:V*Vnb}. As an illustration, we first check that \eqref{eq:U*Unb} holds.  The entry $(i,j)$ of $U^* U$ is equal to 
$$
(\mu_i \mu_j)^{-\ell} \langle B^\ell  \varphi^+_i  , B^\ell \varphi^+_j \rangle = \sum_{x ,y \in [n]} (\mu_i \mu_j)^{-\ell}   \IND_{\{x,y\} \in E}   (B^\ell \varphi^+_i ) (y,x)  (B^\ell \varphi^+_j )(y,x).   
$$
On the event that $G$ is $\ell$-tangle free, we have 
$$
 (B^\ell \varphi^+_i ) (y,x)  (B^\ell  \varphi^+_j )(y,x) = f(G, (y,x)) / d,
$$
where 
\[f(g,o) =  \mathbf{1}_{(g,o)_\ell \text{ is tangle free }} f_{\mathbf{1}, \varphi_i, \ell}(g,o)f_{\mathbf{1}, \varphi_j, \ell}(g,o) .
\]
We thus have on the event that $G$ is $\ell$-tangle free, 
$$
(\mu_i \mu_j)^{-\ell} \langle B^\ell \varphi^+_i  , B^\ell \varphi^+_j \rangle = \sum_{x \in [n]} \tilde f (G,x),   
$$
with 
$$
\tilde f(g,o) = (\mu_i \mu_j)^{-\ell}  \frac 1 d  \sum_{u \sim o}   f(g, (o,u))  , 
$$
and the sum is over all neighbors of $o$ (end vertex of an edge attached to $o$). The function $\tilde f(g,o)$ is $(\ell+1)$-local and, from Lemma \ref{lemme:locality}, it is bounded by $$\frac{2 b^2}{nd} |\mu_i \mu_j|^{-\ell} \thresh^{2\ell} \sum_{u \sim o} |(g,(o,u))_{\ell}|^2.$$ Then claim \eqref{eq:U*Unb} follows by applying Theorem \ref{thm:concentration2} and Theorem \ref{thm:tree_computations} - Equation \eqref{TC3}.

We now check that \eqref{eq:V*Vnb} holds. The entry $(i,j)$ of $V^* V$ is equal to $(\mu_i \mu_j)^{-\ell-1} $ times
$$
\langle (B^*)^\ell  \Delta \varphi^-_i  , (B^*)^\ell \Delta \varphi^-_j \rangle = \sum_{x , y \in [n]} \IND_{\{x,y\} \in E}  ( (B^*)^\ell \Delta \varphi^-_i ) (x,y)  ((B^*)^ \ell \Delta\varphi^-_j )(x,y).   
$$
Using the symmetry of $E$ and $P$, we find that on the event that $G$ is $\ell$-tangle free, 
$$
 ((B^*)^\ell \Delta \varphi^-_i ) (x,y)  ((B^*)^\ell  \Delta \varphi^-_j )(x,y) = n^2 P^2_{x,y}f(G, (y,x)) / d = (n /d ) Q_{x,y}f(G, (y,x)) ,
$$
where as above
\[f(g,o) =  \mathbf{1}_{(g,o)_\ell \text{ is tangle free }} f_{\mathbf{1}, \varphi_i, \ell}(g,o)f_{\mathbf{1}, \varphi_j, \ell}(g,o) .
\]
We thus have on the event that $G$ is $\ell$-tangle free, 
$$
(\mu_i \mu_j)^{-\ell-1} \langle (B^*)^\ell  \Delta \varphi^-_i  , (B^*)^\ell \Delta \varphi^-_j \rangle  = \sum_{x \in [n]} \tilde f (G,x),   
$$
with 
$$
\tilde f(g,o) = (\mu_i \mu_j)^{-\ell-1} \frac{n}{d} \sum_{u \sim o}   Q_{\iota(o),\iota(u)}  f(g, (o,u)) , 
$$
and  $\iota$ is the mark function. The function $\tilde f$ is $(\ell+1)$-local and it is bounded by 
$$
\frac{2 b^6}{nd} |\mu_i \mu_j|^{-\ell} \thresh^{2\ell} \sum_{u \sim o} |(g,(o,u))_{\ell}|^2,
$$
where we have used that $\mu_1 = 1$ and $ K = \max_{x,y} n Q_{x,y} \leq b^4$. 
Moreover, by Theorem \ref{thm:tree_computations} - Equation \eqref{TC3}, we have
$$
\EE \tilde f (T,x) =   \sum_{y} \frac{Q_{x,y}}{ \mu_i \mu_j } \sum_{s=0}^\ell \frac{ Q^s \varphi^{i,j}(y) }{(\mu_i \mu_j d)^s} = d \sum_{s=1}^{\ell+1} \frac{ Q^s \varphi^{i,j}(x) }{(\mu_i \mu_j d)^s}.
$$
It follows that claim \eqref{eq:U*Unb} follows by applying Theorem \ref{thm:concentration}.

We now check that \eqref{eq:U*Vnb} holds. The entry $(i,j)$ of $U^* V$ is equal to $\mu_i^{-\ell} \mu_j^{-\ell-1}$ times
$$
 \langle B^\ell  \varphi^+_i  , (B^*)^\ell \Delta \varphi^-_j \rangle = \sum_{x ,y \in [n]}  \IND_{\{x,y\} \in E}   (B^{2\ell} \varphi^+_i ) (y,x) ( \Delta \varphi^-_j )(y,x).   
$$
On the event that $G$ is $\ell$-tangle free, we have 
$$
(B^{2\ell} \varphi^+_i ) (y,x) ( \Delta \varphi^-_j )(y,x) = (n/d) P_{x,y} \varphi_j(y) f(G, (y,x)),
$$
where 
\[f(g,o) =  \mathbf{1}_{(g,o)_\ell \text{ is tangle free }} f_{\mathbf{1}, \varphi_i, 2\ell}(g,o) .
\]
We thus have on the event that $G$ is $\ell$-tangle free, 
$$
\mu_i^{-\ell} \mu_j^{-\ell-1} \langle B^\ell  \varphi^+_i  , (B^*)^\ell \Delta \varphi^-_j \rangle= \sum_{x \in [n]} \tilde f (G,x),   
$$
with 
$$
\tilde f(g,o) = \mu_i^{-\ell} \mu_j^{-\ell-1}  \frac{n}{d} \sum_{u \sim o}  P_{\iota(u) \iota(o)} \varphi_j(\iota(o)) f(g,(o,u)) , 
$$
and  $\iota$ is the mark function.  The function $\tilde f(g,o)$ is $(\ell+1)$-local and, from Lemma \ref{lemme:locality}, it is bounded by $(2 b^4)/(nd) |\mu_i \mu_j|^{-\ell} \thresh^{2\ell} \sum_{u \sim o} |(g,(o,u))_{\ell}|$. Moreover by  Theorem \ref{thm:tree_computations} - Equation \eqref{TC1}, we have 
$$
\EE \tilde f (T,x) = \mu_i^{\ell} \mu_j^{-\ell-1} \sum_{y} P_{y,x} \varphi_j(y) \varphi_i(x) = \mu_i^{\ell} \mu_j^{-\ell} \varphi_j(x) \varphi_i(x).
$$
where at the second step, we use that $\varphi_j$ is a left eigenvector of $P$. 
 Then claim \eqref{eq:U*Unb} follows by applying Theorem \ref{thm:concentration2} and using the orthogonality relation $\langle \varphi_i, \varphi_j \rangle = \delta_{i,j}$.

We leave the remaining inequalities \eqref{eq:PhiAlPhinb}, \eqref{eq:PhiAlPhi0nb} and \eqref{eq:PhiAlPhi00nb} to the reader.

\subsection{Norm of the restricted matrix: claims \eqref{norm_on_orthogonalnb}-\eqref{norm_on_orthogonal2nb}}

We may perform essentially the same  tangle-free decomposition as performed for the matrix $A$. We follow \cite{bordenave_lelarge_massoulie} where a tangle-free decomposition for non-backtracking matrices was previously performed. The following argument introduces a new simplification.

%

We define {\em non-backtracking path} of length $k$ as a sequence $(x_0, \ldots, x_k)$ in $[n]$ such that $x_{t-1} \ne x_{t+1}$ for all $t \in [k-1]$. For example, if $x \ne y$, $(x,x,y,y)$ is a non-backtracking path of length $3$. For $e,f \in E_n = [n]^2$,  the $F_{e,f}^{\ell+1}$ is the set of all non-backtracking paths $(x_0, \cdots , x_{\ell+1})$ such that $(x_0,x_1)= e$, $(x_\ell,x_{\ell+1} ) = f$ and the graph of the path is tangle-free (see definition in Subsection \ref{sec:tree} on page \pageref{sec:tree}). More generally, $F^t$ denotes the set of all tangle-free non-backtracking paths of length $t$, whatever their endpoints.

If the graph $G$ is $(\ell+1)$-tangle free then $B^{\ell} = B^{(\ell)}$ where, for $e = (x,y)$ and $f = (a,b)$ in $E_n = [n]^2$,
$$
 B^{(\ell)}_{e,f}  =\PAR{\frac{n}{d} }^\ell \sum_{F^{\ell+1}_{e,f}}  M_e  \prod_{t = 1}^{\ell} P_{x_{t} x_{t+1}} M_{x_{t} x_{t+1}}.
$$
We also define the matrices $\uM$ 
and $\uB^{(\ell)}$, $\tilde \uB^{(\ell)}$ by $B^{(0)}_{ef} =   \tilde \uB^{(0)}_{ef} =\IND_{ e= f \in E}$, and 
\begin{eqnarray}
    \uM_{x,y} &= &M_{x,y} - \frac d n \nonumber \\
 \uB^{(\ell)}_{e,f} &=& \PAR{\frac{n}{d} }^\ell \sum_{F^{\ell+1}_{e,f}} M_{e} \left(\prod_{t = 1}^{\ell-1} P_{x_{t} x_{t+1}} \uM_{x_{t} x_{t+1}} \right) P_{f} M_{f}  \nonumber \\
  \tilde \uB^{(\ell)}_{e,f} &=& \PAR{\frac{n}{d} }^\ell \sum_{F^{\ell+1}_{e,f}} M_{e} \left(\prod_{t = 1}^{\ell-1} P_{x_{t} x_{t+1}} \uM_{x_{t} x_{t+1}} \right) P_{f} \uM_{f} .
 \label{eq:defuB}
\end{eqnarray}
Note that the difference between $\uB^{(\ell)}$ and $\tilde \uB^{(\ell)}$ lies at the weight attached to the last edge.

We use the convention that the product over an emptyset is $1$. For each $(x_0, \ldots, x_{\ell+1})$, we write the telescopic sum \eqref{eq:telesc} for the indices $t\in [\ell-1]$, we get
\begin{equation} \label{eq:teleBl}
B^{(\ell)}_{e,f} = \uB^{(\ell)}_{e,f} + \sum_{k=1}^{\ell-1} \PAR{\frac{n}{d} }^\ell \sum_{F^{\ell+1}_{e,f}} M_e \prod_{t = 1}^{k-1} P_{x_{t} x_{t+1}}  \uM_{x_{t} x_{t+1}} \PAR{ \frac{d}{n} P_{x_{k} x_{k+1}}  }\prod_{t = k+1}^{\ell} P_{x_{t} x_{t+1}} M_{x_{t} x_{t+1}} .
\end{equation}

We introduce the matrix in $\mathscr{M}_{E_n}(\mathbb{R})$: for $e = (x,y)$, $f=(a,b)$, 
$$
H_{e,f} =  \IND_{y = a , x \ne b} 
$$
The matrix $H$ is the unweighted non-backtracking matrix on $[n]^2$. 

The above identity \eqref{eq:teleBl} can then be rewritten as the following identity in $\mathscr{M}_{E_n}(\mathbb{R})$: 
\begin{equation}\label{eq:bldec}
   B^{(\ell)} = \uB^{(\ell)}  + \frac{1}{nd} \sum_{k=1}^{\ell-1} \tilde \uB^{(k-1)}(H \Delta) ^2 B^{(\ell - k-1)}  - \sum_{k=1}^{\ell-1}  R_k^{(\ell)} ,
\end{equation}
where 
\[
    (R_k ^ {(\ell)}) _{e,f}  =\PAR{\frac{n}{d} }^{\ell-1} \sum_{T_{e,f}^{k,\ell}}M_e \prod_{t = 1}^{k-1} P_{x_{t} x_{t+1}}  \uM_{x_{t} x_{t+1}} \PAR{   P_{x_{k} x_{k+1}}  }\prod_{t = k+1}^{\ell} P_{x_{t} x_{t+1}} M_{x_{t} x_{t+1}} 
\]
where the sum is over all sequences $(x_0, \ldots, x_{\ell+1})$  such that $(x_0,x_1) = e$, $(x_\ell,x_{\ell+1}) = f$, $(x_0, \ldots, x_{k}) \in F^{k}$, $(x_{k+1}, \ldots, x_{\ell+1}) \in F^{\ell - k}$ but $(x_0, \ldots, x_{\ell+1})$ is not in $F^{\ell+1}$.

We define two new lifts in $\dR^{E_n}$ of a vector $\varphi \in \dR^n$:  
\begin{equation}\label{eq:liftpm2}
\bar \varphi^{+}(e) = \frac {\varphi(y)} {\sqrt{d}}  \quad \hbox{ and } \quad \bar \varphi^{-}(e) = \frac { \varphi(x)} {\sqrt{d}}.
\end{equation}
The only difference with \eqref{eq:liftpm} is that we do not require that $e\in E$. Now, as in \cite{bordenave_lelarge_massoulie}, we write that for $e =(x,y)$, $f = (a,b)$
\begin{eqnarray*}
\frac{1}{n} (H \Delta H )_{e,f} &  = &  \IND_{a \ne x, y \ne b} P_{y a} \\
& = &   P_{y a} - \IND_{ \{ a = x \} \cup \{ y = b\} } P_{y a} \\
& = & d \sum_{j=1}^n  \mu_j \bar \varphi_j^+(e)\bar \varphi_j^-(f) -  d S_{e,f},
\end{eqnarray*}
where we have used the spectral decomposition $P=\mu_1 \varphi_1 \varphi_1^*+\dotsb+\mu_n\varphi_n\varphi_n^* $, used notation \eqref{eq:liftpm} and defined $S \in  \mathscr{M}_{E_n}(\mathbb{R})$ as
$$
S_{e,f} =  \IND_{ \{ a = x \} \cup \{ y = b\} } \frac{P_{y a}}{d}.
$$

For any unit vector $w$, we thus have 
\[
 \frac{1}{nd} \tilde \uB^{(k-1)}(H \Delta) ^2 B^{(\ell - k-1)} w = \sum_{j=1} ^{n}  \mu_j  \tilde  \uB^{(k-1)}  \bar \varphi^+_j \langle  \bar \varphi^-_j, \Delta B^{(\ell - k-1)} w \rangle - \tilde  \uB^{(k-1)} S  \Delta  B^{(\ell - k-1)} w .
\]
We observe that on the image of $B$, $\varphi^-$ and $\bar \varphi^-$ coincide.
Hence, in the above expression, we may replace $\Delta \bar \varphi^{-}_j$ by $\Delta\varphi^{-}_j$. From the triangle inequality,
\begin{eqnarray}
\left| \frac{1}{nd} \uB^{(k-1)}(H \Delta) ^2 B^{(\ell - k-1)} w \right| &\leq & \sum_{j=1} ^{n}  |\mu_j|  | \tilde \uB^{(k-1)}  \bar \varphi^+_j | | \langle \Delta  \varphi^-_j, B^{(\ell - k-1)} w \rangle | \nonumber \\
& & \quad \quad + \;\| \tilde \uB^{(k-1)}   S \Delta  B^{(\ell - k-1)} \| \label{remainder_section:eq2}.
\end{eqnarray}

From the direct analog of Proposition \ref{prop:iv} for  $\langle (B^*)^t \Delta \varphi_i^-,w \rangle$,  with probability at least $1 - c n^{2\kappa-1}$, the following holds for any $t\leq \ell$ and $i \in [r_0]$ and $w \in H^\perp$: 
\[|\langle \Delta \varphi^-_i, B^t w \rangle | \leqslant c b^4 \ell  \thresh^t.\] 

From the direct analog of Proposition \ref{prop:b2} for $(B^*)^{t} \Delta \varphi^-_i$, with probability at least $1 - c n^{2\kappa-1} $, for all $i \in [n] \backslash [r_0]$ and $t \leq \ell$, we have 
$$|(B^*)^{t} \Delta \varphi^-_i|^2 \leqslant \mu_i^{2t}\hat \Gamma_{i,i}^{(t)} + c b^6 (\ln n)^{9/2} n^{3\kappa/2  -1/2} \thresh^{2t}.$$ 
However, from Equations \eqref{eq:gammrr0}-\eqref{eq:L_is_bounded}, we have $\mu_i^{2t}\hat \Gamma_{i,i}^{(t)} \leqslant   b^8 (t+1)  \thresh^{2t}$. As a consequence, for some universal constant $c >0$, for all $i \in [n]\backslash [r_0]$ and $t \leq \ell$,
\[    \langle \Delta  \varphi^-_j, B^{t} w \rangle | \leqslant |w||(B^*)^{t} \Delta \varphi^-_i| \leqslant c b^4 \sqrt \ell \thresh^t.
\]

On the union of the two events events, in \eqref{remainder_section:eq2} we find, for all $w \in H^\perp$: 
\begin{align*}
\sum_{j=1} ^{n}  |\mu_j|  | \tilde \uB^{(k-1)}  \bar \varphi^+_j | | \langle \Delta  \varphi^-_j, B^{(\ell - k-1)} w \rangle |  &\leqslant  c b^4 \ell   \sum_{j=1}^{n}| \mu_j| \thresh^{\ell-k-1}  | \uB^{(k-1)}  \bar \varphi^+_j |.
\end{align*}

Hence, we find from \eqref{eq:bldec}-\eqref{remainder_section:eq2} that, for some $c >0$,  with probability at least $1 - c n^{2\kappa-1}$

\begin{align}
\Vert B^{\ell} \proj_{H^\perp}\Vert \leq&  \| \uB^{(\ell)} \| + c b^4 \ell \sum_{k=1}^{\ell-1}   \sum_{j=1}^{n}| \mu_j| \thresh^{\ell-k-1}  | \uB^{(k-1)}  \bar \varphi^+_j | \nonumber \\
& \quad +  \sum_{k=1}^{\ell-1}  \| \tilde \uB^{(k-1)}   S \Delta  B^{(\ell - k-1)} \|  + \sum_{k=1}^{\ell-1}  \| R_k^{(\ell)} \|. \label{remainder_section:eq3} 
\end{align}

The next proposition gathers the necessary norm bounds. 

\begin{prop}\label{prop:normsnb}  There exists a universal constant $c>0$ such that if $ n \geq cK^{42}$, with probability at least $1 - 1/\sqrt{n}$, the following holds for any $k  \in  [\ell]$ and $j \in [n]$:
\begin{equation} \label{eq:normsAnb}
\Vert \uB^{(k)} \Vert \leqslant  c  \sqrt d \ln(n)^{10} K^4 \thresh^k ,
\end{equation}
\begin{equation}\label{eq:normsRnb}
\Vert R^{(\ell)}_k \Vert \leqslant  \frac{cd^{3/2}} n \ln(n)^{23} L^{\ell},
\end{equation}
\begin{equation}\label{eq:normsBphij}
| \tilde\uB^{(k)} \bar \varphi_j^+ | \leqslant  c b \ln(n)^{5} K  \thresh^k,
\end{equation}
\begin{equation}\label{eq:normsHDB}
\Vert \tilde\uB^{(k-1)} S \Delta B^{(\ell-k-1)}  \Vert \leqslant c \ln(n)^{19} K^4 \thresh^{\ell} \frac{d^{\ell}}  { \sqrt{n}}.
\end{equation}
\end{prop}

Before checking this proposition in the next subsection, let us check  that it implies Claim \eqref{norm_on_orthogonalnb}. It suffices to use Proposition \ref{prop:normsnb} in \eqref{remainder_section:eq3}. We recall that $L = d\thresh_1 \leq d \thresh$, $\theta \geq 1/\sqrt d$ and $d^\ell \leq n^{1/8}$ from \eqref{eq:taudnnb} and we arrive at Claim \eqref{norm_on_orthogonalnb}.

The proof of Claim \eqref{norm_on_orthogonal2nb} follows from the same argument by considering the transpose of  \eqref{eq:bldec}. We omit the proof.

\subsection{Proof of Proposition \ref{prop:normsnb}}

In this section, we prove Proposition \ref{prop:normsnb}. It relies on similar ideas than the proof of Proposition \ref{prop:norms}.

\subsubsection*{Claim \eqref{eq:normsAnb}} From Markov inequality, the following lemma implies \eqref{eq:normsAnb}. 
\begin{lem}\label{le:nB}
There exists a universal constant $c \geq 3$ such that for all integers $1 \leq k \leq \ln(n)$ and $n \geq c K^{42}$,
$$
\PAR{ \EE \BRA{ \Vert \uB^{(k)} \Vert ^ {2m}}}^{\frac 1{2 m}} \leq \ln(n)^{9} \sqrt{d} K^4 \thresh^k,
$$
where $m =  \ln(n / K^6 ) / ( 12 \ln( \ln (n)) )$.
\end{lem}

The proof of this lemma is almost the same than the proof of Lemma \ref{le:nA}. The computation leading to \eqref{eq:nA1} gives
\begin{eqnarray}\label{eq:nA1nb}
\EE \Vert \uB^{(k)} \Vert ^ {2m} & \leq & \PAR{\frac n d }^{2k m} \sum_{\gamma} \EE \prod_{i=1}^{2m} M_{\gamma_{i,0}\gamma_{i,1}} \left( \prod_{t=1}^ {k-1} P_{\gamma_{i,t}\gamma_{i,t+1}} \uM_{\gamma_{i,t}\gamma_{i,t+1}}\right) P_{\gamma_{i,k}\gamma_{i,k+1}} M_{\gamma_{i,k}\gamma_{i,k+1}},
\end{eqnarray}
where the sum is over all $\gamma = (\gamma_0, \ldots, \gamma_{2m} )$ with $\gamma_i = (\gamma_{i,0},\ldots, \gamma_{i,k+1}) \in F^{k+1}$ and the boundary conditions: for all $i \in [m]$,
$$
(\gamma_{2i,0},\gamma_{2i,1}) = (\gamma_{2i+1,0},\gamma_{2i+1,1}) \qquad \text{ and } \qquad  (\gamma_{2i-1,k} , \gamma_{2i-1,k+1}) = (\gamma_{2i,k},\gamma_{2i,k+1})
$$
with $\gamma_{2m+1} = \gamma_1$.

We associate to an element $\gamma$ as above, an undirected graph $G_\gamma = (V_\gamma,E_\gamma)$ with vertices $V_\gamma = \{\gamma_{i,t} : 1 \leq i \leq 2m , 0 \leq t \leq k+1 \}$ and edge set $E_\gamma = \{\{\gamma_{i,t},\gamma_{i,t+1}\} : 1 \leq i  \leq 2m, 0 \leq t \leq k \}$.  From the above boundary conditions, the graph $G_\gamma$ is simply connected. In particular, $|E_\gamma| - |V_\gamma| + 1 \geq 0$.
Each  edge $e \in E_\gamma$ has a multiplicity $m_e$ defined as the number of times it is visited by $\gamma$. By construction, 
\begin{equation}\label{eq:summenb}
\sum_{e \in E_\gamma} m_e = 2 (k+1) m.
\end{equation}
We also define $\tilde m_e \leq m_e$ as the number of visits of $e \in E$ for some $\{ \gamma_{i,t},\gamma_{i,t+1} \}$ with $t \in  [k]$ (we exclude $t=0$).

We may now estimate the expectation on the right-hand side of \eqref{eq:nA1nb}. Recall that the random variables $\uM_{xy} = (M_{xy} - d/n )$, $x,y$, are i.i.d., centered, bounded by $1$ and with variance $(d/n) - (d/n)^2$. It follows that for any $p \geq 1$, $|\EE[ \uM_{xy} ^ p]| \leq d/n$. Note also that for all $p \geq 1$, $\EE M_{xy}^p = d/n$. We deduce that
$$
\EE \prod_{i=1}^{2m} M_{\gamma_{i,0}\gamma_{i,1}}\left( \prod_{t=1}^ {k-1} P_{\gamma_{i,t}\gamma_{i,t+1}} \uM_{\gamma_{i,t}\gamma_{i,t+1}} \right) M_{\gamma_{i,k}\gamma_{i,k+1}} \leq \PAR{ \frac d n }^{|E_\gamma|} \prod_{e \in E_\gamma} |P_e|^{\tilde m_e} .
$$
Moreover, the above expectation is zero unless all edges have multiplicity at least $2$ (note that  $M_{x,y}$ is not centered but from the boundary condition, the edges $\{\gamma_{i,0},\gamma_{i,1}\}$ and  $\{\gamma_{i,k},\gamma_{i,k+1}\}$  have multiplicity at least $2$). From \eqref{eq:nA1}, we thus obtain that 
\begin{eqnarray}\label{eq:nA2nb}
\EE \Vert \uB^{(k)} \Vert ^ {2m} & \leq & \sum_{\gamma \in \bar W_{k+1,m}}  \PAR{\frac n d }^{2 k m - |E_\gamma|} \prod_{e \in E_\gamma} |P_e|^{\tilde m_e} ,
\end{eqnarray}
where $\bar W_{k+1,m}$ is the set of paths $\gamma$ as above such that each edge of $E_\gamma$ is visited at least twice.

The right-hand side of Equation \eqref{eq:nA2nb} is very close to the right-hand side of  Equation \eqref{eq:nA2} for $k' = k+1$. It can be analyzed with the same method. We define  $\bar \cW_{k,m}(s,a)$ as the set of equivalence classes of paths in $\bar W_{k,m}$ with $|V_\gamma| = s $ and $|E_\gamma| = a$. We have that $\bar \cW_{k,m} (s,a)$ is empty unless $a - s + 1 \geq 0$.  The conclusion of Lemma \ref{le:numberiso} also holds for $\bar \cW_{k,m} (s,a)$.
\begin{lem}\label{le:numberisonb}
Let $a,s \geq 1$ be integers such that $a - s + 1 \geq 0$. We have 
$$
|\bar \cW_{k,m}(s,a)| \leq  (2km)^{6m(a-s+1) + 2m}.
$$
\end{lem}
\begin{proof}
We may repeat the proof Lemma \ref{le:numberisonb}. The new property that we use that there exists a unique non-backtracking path between two vertices in an (undirected) tree. \end{proof}

Our second lemma bounds the contributions of paths in each equivalence class. Due to the presence of the multiplicities $\tilde m_e$ instead of standard multiplicities $m_e$, there is an extra factor in the analog of lemma \ref{le:sumiso}.

\begin{lem}\label{le:sumisonb}
Let $\gamma \in \bar W_{k+1,m}$ such that $|V_\gamma| = s$ and  $|E_\gamma| = a$. We have 
$$
\sum_{\gamma' : \gamma' \sim \gamma} \prod_{e \in E_{\gamma'}} |P_e|^{\tilde m_e}  \leq n^{-2 k m + s} K^{k m - a} K^{6(a-s) + 8m}  \rho^{km}.
$$
\end{lem}
\begin{proof}
We let $\tilde \gamma = (\tilde \gamma_1, \ldots, \tilde \gamma_{2m})$ be obtained from $\gamma$ by setting $\tilde \gamma_i = (\tilde \gamma_{i,1}, \ldots, \tilde \gamma_{i,k+1})$ (we remove the first step for each $i$). Due to the boundary condition, the associated graph $G_{\tilde \gamma}$ is connected and 
\begin{equation}\label{eq:a-s'}
\tilde a - \tilde s \leq a- s, 
\end{equation}
with  $\tilde a = |E_{\tilde \gamma}|$ and  $\tilde s = |V_{\tilde \gamma}|$. Note also $\tilde m_e$ is the number of visits of $e$ in $\tilde \gamma$. We have that 
$$
\sum_{e \in E_\gamma} \tilde m_e = \sum_{e \in E_{\tilde \gamma}}\tilde m_e = 2km.
$$
Let $H \subset E_{\gamma'}$ be the subset of edges $e \in E_\gamma$ such that $\tilde m_e = 1$. We set $|H| =h$. 
We have
\begin{equation}\label{eq:kokonb} 
 \prod_{e \in E_\gamma} |P_e|^{\tilde m_e} =  \prod_{e \in E_{\tilde \gamma}} |P_e|^{\tilde m_e-2} P_e^ 2 \leq \PAR{\frac {\sqrt{\rho K}}n}^{2km -2\tilde a}  \PAR{\frac {\sqrt{\rho K}}n}^{2h}  \prod_{e \in E_{\tilde \gamma} \backslash H} \frac{Q_e}{n},
\end{equation}
where we have used $\sum_{e \notin H} \tilde m_e = 2km -h$ and $\max |P_{xy}| = \sqrt{\rho K}/n$.

As in the proof of Lemma \ref{le:sumiso}, we consider the graph $\hat G_{\tilde \gamma}$ obtained from $G_{\tilde \gamma}$ by gluing vertices of degree $2$ while keeping the extreme vertices $\gamma_{i,1}$, $\gamma_{i,k+1}$, $i \in [2m]$.  
Note that due the boundary conditions, the edges in $H$ are not modified by this operation: they remain edges in $\hat G'_{\tilde \gamma}$. The proof of \eqref{eq:koko} applied to the graph $\hat G_{\tilde \gamma}$ gives  
\begin{equation*}
\sum_{\gamma' : \gamma' \sim \gamma}\prod_{e \in E_{\tilde \gamma} \backslash H} Q_e \leq  \rho^{\tilde a-h} n^{s-\tilde a+ h} K^{6(\tilde a - \tilde s) + 8m}.
\end{equation*}
Therefore, we get from \eqref{eq:kokonb} and \eqref{eq:a-s'} that 
$$
\sum_{\gamma' : \gamma' \sim \gamma} \prod_{e \in E_{\tilde \gamma}} |P_e|^{\tilde m_e}  \leq n^{-2 k m + s} K^{k m - a} K^{6(a-s) + 8m}  \rho^{km} K^{h}.
$$
Since  $K \geq 1$, this concludes the proof. 
\end{proof}

We are ready for the proof of Lemma \ref{le:nB}.

\begin{proof}[Proof of Lemma \ref{le:nB}]
Note that $\bar \cW_{k+1,m}(s,a)$ is empty unless $0 \leq s -1 \leq a \leq (k+1)m$ (since each edge has multiplicity at least $2$). From \eqref{eq:nA2nb}, we get
\begin{eqnarray*}
\EE \Vert \uB^{(k)} \Vert ^ {2m} & \leq & \sum_{a= 1}^{(k+1)m } \sum_{s = 1}^{a+1} \PAR{ \frac{n}{d} }^ {2km -a} | \bar \cW_{k+1,m} (s,a) | \max_{\gamma \in \bar \cW_{k+1,m}(s,a)}  \sum_{\gamma' : \gamma' \sim \gamma} \prod_{e \in E_{\gamma'}} |P_e|^{\tilde m_e}.
\end{eqnarray*}
We use Lemma \ref{le:numberisonb}-Lemma \ref{le:sumisonb} and follow the computation of Lemma \ref{le:nA}. We find 
$$
\EE\Vert \uB^{(k)} \Vert ^ {2m}  \leq  n (2km)^{2m} K^{8m-6} \thresh_2^{2km} \sum_{a= 1}^{(k+1)m } \left(\frac{K}{d} \right)^{km-a} \sum_{g= 1 - m}^{\infty}  \PAR{\frac{K^ 6 (2(k+1)m)^{6m}}{n} }^{g}.
$$ The conclusion follows easily.The conclusion follows easily. \end{proof}

\subsubsection*{Claim \eqref{eq:normsRnb}} 
 Claim \eqref{eq:normsRnb} follows from this lemma.

\begin{lem}\label{le:nRb}
There exists a universal constant $c\geq 3$ such that for all integers $1 \leq k  \leq \ell \leq \ln(n)$ and $n \geq c $,
$$
\PAR{ \EE \BRA{ \Vert R^{(\ell)}_k \Vert ^ {2m}}}^{\frac 1{2 m}} \leq  \frac{d^{3/2}}{n} \ln(n)^{23}  L^\ell,
$$
where $m =  \ln(n ) / ( 24 \ln( \ln (n)) )$.
\end{lem}

The proof is exactly the same than the proof of Lemma \ref{le:nR}. The extra factor $\sqrt{d}$ comes from the same reason than in the proof of Lemma \ref{le:nB} (the scaling by $(n/d)^\ell$ while the paths are of length $\ell+1$). We omit the proof.

\subsubsection*{Claim \eqref{eq:normsBphij}} 
We have $|\phi_j^+|_{\infty} = |\phi_j |_{\infty} / \sqrt d \leq b/\sqrt{nd}$. Claim \eqref{eq:normsBphij} follows from this lemma  and the union bound.

\begin{lem}\label{le:nBpi}
There exists a universal constant $c \geq 3$ such that for all integers $1 \leq k \leq \ln(n)$, $n \geq c K^{42}$ and vectors $\psi \in \dR^ {E_n}$,
$$
\PAR{ \EE \BRA{ | B^{(k)} \psi | ^ {2m}}}^{\frac 1{2 m}} \leq  \sqrt {nd} |\psi |_{\infty} \ln(n)^{5} K \thresh^k,
$$
where $m =  \ln(n ) / ( 12 \ln( \ln (n)) )$.
\end{lem}

Note the lemma improves  on Proposition \ref{prop:normsnb} and the bound $| B^{(k)}\psi | \leq \| B^{(k)} \| | \psi | \leq  n \| B^{(k)} \|  |\psi |_{\infty}$ by a crucial factor $\sqrt n$ (up to logarithmic factors).

The proof is again a variation around the proofs of Lemma \ref{le:nA} and Lemma \ref{le:nR}. We write
$$
| B^{(k)} \psi |^2 = \sum_{e} (B^{k} \psi)(e) ^2 = \sum_{e,f,f'} (B^{k})_{e f} \psi (f)  (B^{k})_{e f'} \psi (f').  
$$
We raise the above expression to a power $m$. The computation leading to \eqref{eq:nA1nb} gives
\begin{eqnarray}\label{eq:nA1nbpsi}
\EE | B^{(k)} \psi |^ {2m} & \leq & \PAR{\frac n d }^{2k m} \sum_{\gamma} \EE \prod_{i=1}^{2m} M_{\gamma_{i,0}\gamma_{i,1}} \prod_{t=1}^ k P_{\gamma_{i,t}\gamma_{i,t+1}} \uM_{\gamma_{i,t}\gamma_{i,t+1}} \psi(\gamma_{i,k+1}),
\end{eqnarray}
where the sum is over all $\gamma = (\gamma_0, \ldots, \gamma_{2m} )$ with $\gamma_i = (\gamma_{i,0},\ldots, \gamma_{i,k+1}) \in F^{k+1}$ and the boundary conditions: for all $i \in [m]$,
$$
(\gamma_{2i-1,0},\gamma_{2i-1,1}) = (\gamma_{2i,0},\gamma_{2i,1}).
$$

We repeat the definitions in the proof of Lemma \ref{le:nB}. We associate to an element $\gamma$ as above, an undirected graph $G_\gamma = (V_\gamma,E_\gamma)$ with vertices $V_\gamma = \{\gamma_{i,t} : 1 \leq i \leq 2m , 0 \leq t \leq k+1 \}$ and edge set $E_\gamma = \{\{\gamma_{i,t},\gamma_{i,t+1}\} : 1 \leq i  \leq 2m, 0 \leq t \leq k \}$.  From the above boundary conditions, the graph $G_\gamma$ has at most $m$ connected component. In particular, 
\begin{equation}\label{eq:genuspsi}
|E_\gamma| - |V_\gamma| + m \geq 0.
\end{equation}
Each  edge $e \in E_\gamma$ has a multiplicity $m_e$ defined as the number of times it is visited by $\gamma$. By construction, \eqref{eq:summenb} holds.  We again define $\tilde m_e \leq m_e$ as the number of visits of $e \in E$ for some $\{ \gamma_{i,t},\gamma_{i,t+1} \}$ with $t \in  [k]$.

The computation leading to \eqref{eq:nA2nb} gives 
\begin{eqnarray}\label{eq:nA2psinb}
\EE| B^{(k)} \psi |^ {2m} & \leq & \sum_{\gamma \in \tilde W_{k+1,m}}  \PAR{\frac n d }^{2 k m - |E_\gamma|} \prod_{e \in E_\gamma} |P_e|^{\tilde m_e} |\psi |_{\infty}^{2m},
\end{eqnarray}
where $\tilde W_{k+1,m}$ is the set of paths $\gamma$ as above such that each edge of $E_\gamma$ is visited at least twice.

We define  $\tilde \cW_{k+1,m}(s,a)$ as the set of equivalence classes of paths in $\tilde W_{k+1,m}$ with $|V_\gamma| = s $ and $|E_\gamma| = a$. 

\begin{lem}\label{le:numberisonbpsi}
Let $a,s \geq 1$ be integers such that $a - s + m \geq 0$. We have 
$$
|\tilde \cW_{k,m}(s,a)| \leq  (2km)^{6m(a-s+m) + 5m}.
$$
\end{lem}
\begin{proof}
The proof is a variant of the proof of Lemma \ref{le:numberisoR}. Let $T = \{ (i,t) : 1 \leq i \leq 2m , 0 \leq t \leq k-1\} $ ordered with the lexicographic order. For $\tau = (i,t) \in T$, we set $e_\tau = \{ \gamma_{i,t},\gamma_{i,t+1}\}$ and $y_{\tau} =\gamma_{i,t+1}$. We build the same growing subforest $(F_{\tau})_{\tau \in T}$ which is a spanning forest of the graph spanned by edges $(e_s)_{s \leq \tau}$ seen so far.

For each $i \in [2m]$, we use the same long and short cycling important times. Finally, there are merging times such that two connected components of the graph seen so far merge. Since there are most $m$ connected components, there are at most $m-1$ merging times. Also, there are most $g  = a - s +m$ excess edges in any connected component of $G_\gamma$. It follows that, for each $i$, there are either at most $g-1$ long cycling time and $1$ short cycling time, or $g$ long cycling times and $0$ short cycling time. 

There are at most $2km$ ways to position those important and merging times. There are $s^2$ possibilities for the mark of a long cycling time or merging time and at most $s^2 k$ for a short cycling time. We deduce that    
$$
|\tilde \cW_{k,m}(s,a)| \leq (2km)^{2m g + m-1} (s^2)^{2m (g-1) + m-1} (s^2 k)^{2m}.
$$
We finally use $s \leq 2 k m$. \end{proof}

The proof of lemma \ref{le:sumisonb} applies verbatim to this case as well.
\begin{lem}\label{le:sumisonbpsi}
Let $\gamma \in \tilde W_{k+1,m}$ such that $|V_\gamma| = s$ and  $|E_\gamma| = a$. We have 
$$
\sum_{\gamma' : \gamma' \sim \gamma} \prod_{e \in E_{\gamma'}} |P_e|^{\tilde m_e}  \leq n^{-2 k m + s} K^{k m - a} K^{6(a-s) + 8m}  \rho^{km}.
$$
\end{lem}

We are ready for the proof of Lemma \ref{le:nBpi}.

\begin{proof}[Proof of Lemma \ref{le:nBpi}]
Note that $\bar \cW_{k+1,m}(s,a)$ is empty unless $0 \leq s -m \leq a \leq (k+1)m$ (since each edge has multiplicity at least $2$ and since \eqref{eq:genuspsi} holds). From \eqref{eq:nA2psinb}, we get
\begin{eqnarray*}
\EE| B^{(k)} \psi |^ {2m}  & \leq & \sum_{a= 1}^{(k+1)m } \sum_{s = 1}^{a+m} \PAR{ \frac{n}{d} }^ {2km -a} | \tilde \cW_{k+1,m} (s,a) | \max_{\gamma \in \tilde \cW_{k+1,m}(s,a)}  \sum_{\gamma' : \gamma' \sim \gamma} \prod_{e \in E_{\gamma'}} |P_e|^{\tilde m_e}.
\end{eqnarray*}
We use Lemma \ref{le:numberisonbpsi}-Lemma \ref{le:sumisonbpsi} and follow the computation of Lemma \ref{le:nA}. Replacing the summation over $s$ as a summation over $g = a - s + m$, we find 
$$
\EE| B^{(k)} \psi |^ {2m}  \leq  n^m (2km)^{5m} K^{2m} \thresh_2^{2km} \sum_{a= 1}^{(k+1)m } \left(\frac{K}{d} \right)^{km-a} \sum_{g= 0}^{\infty}  \PAR{\frac{K^ 6 (2(k+1)m)^{6m}}{n} }^{g}.
$$ The conclusion follows easily. \end{proof}

\subsubsection*{Claim \eqref{eq:normsHDB}}
The proof of Claim \eqref{eq:normsHDB} follows again from the same arguments. 

We first define a new matrix $\tilde S \in  \mathscr{M}_{E_n}(\mathbb{R})$ defined for $e = (x,y)$ and $f = (a,b)$, by 
$$
\tilde S_{ef} = S_{e f} M_f =  \IND_{ \{ a = x \} \cup \{ y = b\} } \frac{P_{y a}}{d} M_{ab}.
$$
Since $M_e^2 = M_e$, we find 
$$
 \tilde\uB^{(k-1)} S \Delta B^{(\ell-k-1)}   =  \tilde\uB^{(k-1)} \tilde S \Delta B^{(\ell-k-1)}.
$$
We then write 
\begin{equation}\label{eq:bsdb}
\|  \tilde\uB^{(k-1)} S \Delta B^{(\ell-k-1)} \|  \leq  \|  \tilde\uB^{(k-1)} \| \|  \tilde S \|  \| \Delta B^{(\ell-k-1)}\|.
\end{equation}
We may estimate the operators norms on right-hand side of \eqref{eq:bsdb} separately.

The statement of Lemma \ref{le:nB} applies unchanged to $ \tilde\uB^{(k-1)}$. We thus find 
\begin{equation}\label{eq:ntildeuBk}
\PAR{ \EE \BRA{ \Vert \tilde \uB^{(k)} \Vert ^ {2m}}}^{\frac 1{2 m}} \leq \ln(n)^{9} \sqrt{d} K^4 \thresh^{k},
\end{equation}
where $m =  \ln(n / K^6 ) / ( 12 \ln( \ln (n)) )$.

The computation leading to \eqref{eq:nA2nb} applied to  $B^{(k)}$ gives 
\begin{eqnarray}\label{eq:nBB2nb}
\EE \Vert \Delta B^{(k)} \Vert ^ {2m} & \leq & d^{2m}  \sum_{\gamma \in  W_{k+1,m}}  \PAR{\frac n d }^{2 (k+1) m - |E_\gamma|} \prod_{e \in E_\gamma} |P_e|^{m_e} ,
\end{eqnarray}
where $\bar W'_{k+1,m}$ is the set of paths $\gamma$ as below \eqref{eq:nA1nb}.  The only difference between $\bar W'_{k+1,m}$ and $\bar W_{k+1,m} \subset \bar W'_{k+1,m}$ is that there is no constraint on edge multiplicities. We note that Lemma \ref{le:numberisonb} holds also for the equivalence classes of $\bar W'_{k+1,m}$  since the proof does not use the edge multiplicities. Due to the boundary conditions, there at most $2km$ edges visited by a path $\gamma \in \bar W'_{k+1,m}$. We also use the rough bound $|P_e| \leq L/n$. We deduce from \eqref{eq:nBB2nb} that 
$$
\EE \Vert \Delta B^{(k)} \Vert ^ {2m}  \leq d^{2m} \sum_{a = 1}^{2 k m} \sum_{s = 1} ^{a+1}  \PAR{\frac n d }^{2 (k+1) m - a}  n^{s} (2km)^{6m(a-s+1) + 2m} \PAR{\frac L n }^{2 (k+1) m }.
$$
We deduce that 
\begin{equation}\label{eq:nDBk}
\PAR{ \EE \BRA{ \Vert \Delta B^{(k)} \Vert ^ {2m}}}^{\frac 1{2 m}} \leq \ln(n)^{9} L^{k+1},
\end{equation}
where $m =  \ln(n ) / ( 12 \ln( \ln (n)) )$.

Finally, we have for any $ e = (x,y) \in E_n$,
$$
\sum _{f \in E_n} | (\tilde S \tilde S^*)_{ef} | \leq \sum_{f,g \in E_n} |S_{e g } S_{f g}| M_ g \leq \frac{L^2}{n d^2}  ( 2 \deg(x) + 2 \deg(y)) \leq \frac{4 L^2}{n d^2} \max_{o \in [n]} \deg(o), 
$$
where $\deg(x) = \sum_{y} M_{xy}$ is the degree of vertex $x \in [n]$ in the random graph $G$. We deduce from \eqref{eq:GxtD} that with probability $1 - 1/n$, for some some universal constant $c>0$, we have for all $ e \in E_n$
$$
\| (\tilde S \tilde S^* \| \leq  \max_{e \in E_n} \sum _{f \in E_n} |(\tilde S \tilde S^*)_{ef} | \leq \frac{c L^2}{nd} \ln(n).
$$ 
This implies that 
$$
\| S \| \leq \sqrt{ \frac{c L^2}{nd} \ln(n)}.
$$
We thus obtain from this last bound and \eqref{eq:ntildeuBk}-\eqref{eq:nBB2nb}-\eqref{eq:nDBk} that for all $k \in [\ell-1]$ with probability at least $1 - 3/n$, we have, for some $c>0$,
$$
\sqrt{n} \|  \tilde\uB^{(k-1)} S \Delta B^{(\ell-k-1)} \|  \leq c \ln(n)^{19} K^4 \thresh^{k-1} L^{\ell - k +1} \leq c \ln(n)^{19} K^4 \thresh^{\ell} d^{\ell-k +1}. 
$$
where we have used that $L = \thresh_1 d$. We obtain claim \eqref{eq:normsHDB}.

\section{Proofs for the rectangular case}\label{sec:proofs:rect}

\subsection{Proof for Lemma \ref{lem:Ztilde}}\label{proof:Ztilde}

\begin{proof}
We recall that two Bernoulli random variables $B_1,B_2$ are independent with parameters $b_1,b_2$ if and only if $\PP(B_1=1, B_2=1)=b_1b_2$. 

For the proof we note $\delta=d/n$ and $p=(1-q)/2$. Clearly, entries above the diagonal and on the diagonal of $\MM$ are independent and their distribution is Bernoulli with parameter $\delta p+\delta q = q/(p+q)=2q/(1+q)$, where we used \eqref{pq_choice}.  

We must check that the entries $\MM_{x,y}$ and $\MM_{y,x}$ are themselves independent when $x \neq y$. We have $\PP(\MM_{x,y}=1, \MM_{y,x}=1)=\delta q$. Consequently, the entries of $\MM$ are independent if and only if $\delta q = (\delta p+\delta q)^2$, or equivalently if $q$ satisfies \eqref{pq_choice}. 
\end{proof}

\subsection{Link between a matrix and its hermitization}\label{subsec:Girko}

The link between a matrix $P$ and its so-called `Girko hermitization'
\[\begin{pmatrix}
0 & P \\ P^* & 0
\end{pmatrix}\]
is well-known in the litterature on random matrices. We let the proofs to the reader --- they are mainly verifications. 

 \begin{lem}[structure of $\PPP$]Let us write the singular value decomposition of $P$ in the following way:
\[P = \sum_{i=1}^{\mathrm{rank}(P)} \sigma_i \zeta_i \xi_i^* \]
where \begin{itemize}\item $\sigma_1\geqslant \dotsb \geqslant \sigma_{\mathrm{rank}(P)}>0$ are the singular values, 
\item $\zeta_i \in \mathbb{C}^m$ is an orthonormal family of left singular vectors, 
\item $\xi_i \in \mathbb{C}^n$ is an orthonormal family of right singular vectors.
\end{itemize}

Then a spectral decomposition of the Hermitian matrix $\PPP$ is 
\begin{equation}
\PPP = \sum_{i=1}^{\mathrm{rank}(P)} \sigma_i \varphi^+_i (\varphi^+_i)^*  - \sum_{i=1}^{\mathrm{rank}(P)} \sigma_i \varphi^-_i (\varphi_i^-)^*
\end{equation} 
where the $\varphi_i^\pm \in \mathbb{C}^{m+n}$ are the orthonormal vectors defined by
 \begin{equation}\label{def:girkovec}
 \varphi^+_i = \frac{1}{\sqrt{2}}\begin{pmatrix}
 \zeta_i \\ \xi_i
 \end{pmatrix},  \qquad \varphi_{i}^- = \frac{1}{\sqrt{2}}\begin{pmatrix}
- \zeta_i \\ \xi_i
 \end{pmatrix}.
 \end{equation}
 \end{lem}

\subsection{Proof of Theorem \ref{thm:smallsquare}}We keep the notations and setting of the theorem and we place ourselves on the event of Theorem \ref{thm:1-rectangular}. Let us decompose the unit right-eigenvectors $\psi^+_i$ associated with the eigenvalue $+\lambda_i$ into their first $m$ components and their last $n$ ones: for $i \in [\rr_0]$:
\begin{equation}\label{def:chik}
\psi^+_i =: \begin{pmatrix}
\psi_{i,1} \\ \psi_{i,2}
\end{pmatrix}
\end{equation}
where $\psi_{i,1} \in \mathbb{C}^m$ and $\psi_{i,2} \in \mathbb{C}^n$ and $|\psi_{i,1}|^2+|\psi_{i,2}|^2=|\psi_i|^2=1$. We saw in Lemma \ref{lem:AAAsymmetry} that the eigenvectors $\psi^-_{i}$ are linked with these ones by
\begin{equation}
\psi^-_{i} = \begin{pmatrix}
-\psi_{i,1} \\
\psi_{i,2}
\end{pmatrix}.
\end{equation}
In the course of the proof, we are going to need an important proposition.

\begin{prop}\label{prop:auxnorm}
On the event of Theorem \ref{thm:1-rectangular}, we have
\begin{align}\label{rect:norm}
&\left| |\psi_{i,1}|^2 - \frac{\gamma_{i}^\triangle}{2\gamma_i} \right| \leqslant \frac{C_0 \ttau_0^{\ell}}{1 - \ttau_{i,\ell}} &&\left| |\psi_{i,2}|^2 - \frac{\gamma_{i}^\triangleu}{2\gamma_i} \right| \leqslant \frac{C_0 \ttau_0^{\ell}}{1 - \ttau_{i,\ell}} .  
\end{align}
\end{prop}

\begin{proof}[Proof of Proposition \ref{prop:auxnorm}]

The eigenvectors $\psi_i^\pm$ have unit norms, hence $\langle \psi_i^+, \psi_i^+ \rangle = 1$, and thanks to Theorem \ref{thm:1-rectangular} we also have $\langle \psi_i^+, \psi_i^-\rangle \approx \Gamma^{+,-}_{i,i}/\gamma_i$. Both equations can be written 
\begin{align*}
&|\psi_{i,1}|^2+|\psi_{i,2}|^2 = 1 &&-|\psi_{i,1}|^2+|\psi_{i,2}|^2 \approx \frac{\Gamma^{+,-}_{i,i}}{\gamma_i}.
\end{align*}
Adding/subtracting both equations yields
\begin{align}
|\psi_{i,1}|^2 \approx \frac{\gamma_i- \Gamma^{+,-}_{i,i}}{2\gamma_i} = \frac{\gamma_i^\triangle}{2\gamma_i}&& |\psi_{i,2}|^2\approx \frac{\gamma_i+ \Gamma^{+,-}_{i,i}}{2\gamma_i} = \frac{\gamma^\triangle_i}{2\gamma_i}.
\end{align}
\end{proof}

We can now turn to the proof of the theorem.

\begin{proof}[Proof of Theorem \ref{thm:smallsquare}]
Let us fix $i,j \in [\rr_0]$. By \eqref{eigenvector_errorbound+-}, we
 have $\langle \psi_i^+, \varphi_j^+ \rangle \approx \delta_{i,j}\sqrt{1/\gamma_i}$, and $\langle \psi_i^+, \varphi_j^- \rangle \approx 0$, which translates into
\begin{align}
& \langle \psi_{i,1},  \zeta_j \rangle + \langle \psi_{i,2},  \xi_j \rangle \approx \delta_{i,j}\sqrt{\frac{2}{\gamma_i} }&& -\langle \psi_{i,1},  \zeta_j \rangle + \langle \psi_{i,2},  \xi_j \rangle \approx 0.
\end{align}
Adding or subtracting both approximations yields
\begin{align}\label{235}
& \langle \psi_{i,1},  \zeta_j \rangle  \approx \delta_{i,j}\sqrt{\frac{1}{2\gamma_i} }&& \langle   \psi_{i,2},  \xi_i \rangle \approx \delta_{i,j} \sqrt{\frac{1}{2\gamma_i} }.
\end{align}
Similarly, by \eqref{eq:corr_signs}, we have $\langle \psi^+_i, \psi^+_j\rangle \approx \Gamma^{+,+}_{i,j}/\sqrt{\gamma_i \gamma_j}$ and $\langle \psi^+_i, \psi^-_j\rangle \approx \Gamma^{+,-}_{i,j}/\sqrt{\gamma_i \gamma_j}$, which translates into
\begin{align}
& \langle \psi_{i,1},  \psi_{j,1} \rangle + \langle \psi_{i,2},  \psi_{j,2} \rangle \approx \frac{\Gamma^{+,+}_{i,j}}{\sqrt{\gamma_i \gamma_j}}&&- \langle \psi_{i,1},  \psi_{j,1} \rangle + \langle \psi_{i,2},  \psi_{j,2} \rangle \approx \frac{\Gamma^{+,-}_{i,j}}{\sqrt{\gamma_i \gamma_j}}.
\end{align}
Here again, we can subtract or add both identities and we obtain
\begin{align}\label{237}
& \langle \psi_{i,1},  \psi_{j,1} \rangle  \approx \frac{\Gamma^{+,+}_{i,j} - \Gamma^{+,-}_{i,j}}{2\sqrt{\gamma_i \gamma_j}} = \frac{\Gamma^\triangle_{i,j}}{2\sqrt{\gamma_i\gamma_j}}\\
& \langle \psi_{i,2},  \psi_{j,2} \rangle \approx \frac{\Gamma^{+,-}_{i,j}+\Gamma^{+,+}_{i,j}}{2\sqrt{\gamma_i \gamma_j}}= \frac{\Gamma^\triangleu_{i,j}}{2\sqrt{\gamma_i\gamma_j}}.\label{237bis}
\end{align}
Similarly, by \eqref{227}, we have $\langle \psi_i^+, \psi_{j,{\rm left}}^+ \rangle \approx \delta_{i,j}/\gamma_i$ and  $\langle \psi_i^+, \psi_{j,{\rm left}}^- \rangle \approx 0$, which translates into 
\begin{align*}
&\langle \psi_{i,1}, \psi_{j,1,{\rm left}}\rangle + \langle \psi_{i,2}, \psi_{j,2,{\rm left}}\rangle \approx \frac{\delta_{i,j}}{\gamma_i} &&-\langle \psi_{i,1}, \psi_{j,1,{\rm left}}\rangle + \langle \psi_{i,2}, \psi_{j,2,{\rm left}}\rangle \approx 0.
\end{align*}
Adding/subtracting yields 
\begin{align}\label{240}
&\langle \psi_{i,1}, \psi_{j,1,{\rm left}}\rangle  \approx \frac{\delta_{i,j}}{2\gamma_i} && \langle \psi_{i,2}, \psi_{j,2,{\rm left}}\rangle \approx \frac{\delta_{i,j}}{2\gamma_i}.
\end{align}
Thanks to Lemma \ref{lem:AAAsymmetry} and \eqref{rect:norm}, we know that up to some sign, 
\begin{align*}
&\chi_i = \frac{\psi_{i,1}}{|\psi_{i,1}|} \approx \psi_{i,1}\sqrt{\frac{2\gamma_i}{\gamma_{i}^\triangle}}&&\pi_i = \frac{\psi_{i,2}}{|\psi_{i,2}|}\approx \psi_{i,2}\sqrt{\frac{2\gamma_i}{\gamma_{i}^\triangleu}}.
\end{align*}
Combining this with \eqref{235}-\eqref{237}-\eqref{237bis}-\eqref{240} closes the proof of all the claims contained in Theorem \ref{thm:smallsquare}.
\end{proof}

\subsection{Proof of Theorem \ref{thm:stats}}\label{sec:proof:rect:stats}

The first two identities directly follow from \eqref{2231}. For the third identity we have 
\begin{align*}
\langle \zavg_i, \zeta_j\rangle = \frac{\langle \chi_i, \zeta_j\rangle+\langle \chi'_i, \zeta_j\rangle }{\sqrt{2(1+\langle \chi_i, \chi'_i\rangle)}} &\approx \frac{\sqrt{1/\gamma_{i}^\triangle}+\sqrt{1/\gamma_{i}^\triangle}}{\sqrt{2\left( 1+\frac{1}{\gamma_{i}^\triangle}\right)}} \\
&\approx \sqrt{\frac{2}{\gamma_i^\triangle +1}}
\end{align*}
and the proof for $\xavg_i$ is the same. We now turn to the correlations between these estimators. Clearly, 
\[\langle \zsim_i, \zsim_j \rangle = \langle\chi_i, \chi_j \rangle = \frac{\Gamma^\triangle_{i,j}}{\sqrt{\gamma_i^\triangle \gamma_j^\triangle}}.\]
For the averaged estimators, we have 
\begin{align*}
\langle \zavg_i, \zavg_j \rangle &= \frac{\langle \chi_i, \chi_j\rangle + \langle \chi_i, \chi'_j\rangle + \langle \chi'_i, \chi_j \rangle + \langle \chi'_i, \chi'_j \rangle}{\sqrt{4(1+\langle \chi_i, \chi'_i\rangle)(1+\langle \chi_j, \chi'_j\rangle)}} \\
&\approx \frac{2\frac{\Gamma^\triangle_{i,j}}{\sqrt{\gamma_i\gamma_j}}+2\frac{\delta_{i,j}}{\gamma^\triangle_i}}{\sqrt{4(1+1/\gamma^\triangle_i)(1+1/\gamma^\triangle_j)}}.
\end{align*}
When $i=j$ this is obviously equal to $1$, and when $i \neq j$ it is equal to 
\begin{align*}\frac{2\frac{\Gamma^\triangle_{i,j}}{\sqrt{\gamma_i\gamma_j}}}{\sqrt{4(1+1/\gamma^\triangle_i)(1+1/\gamma^\triangle_j)}} = \Gamma^\triangle_{i,j}\sqrt{\frac{1}{(\gamma_i^\triangle+1)(\gamma_j^\triangle+1)}} .
\end{align*}

\subsection{Proof of Proposition \ref{prop:rect:rank1}}\label{proof:rect:rank1}
We recall the setting: here, $P$ has rank one and can be written $P = \zeta \xi^*$ where $\zeta, \xi$ are unit vectors. We recall that
\[\QQ = (n+m) \begin{pmatrix}
0 &(\zeta \xi^*) \odot (\zeta \xi^*)\\ [(\zeta \xi^*) \odot (\zeta \xi^*)]^* & 0
\end{pmatrix}. \]
The SVD of $(\zeta \xi^*) \odot (\zeta \xi^*)$ will be written $\kappa \check{\zeta}\check{\xi}^* $ where $\kappa \defeq |\zeta|_4^2|\xi|_4^2 $ and $\check{\zeta}, \check{\xi}$ are the unit singular vectors:
\begin{align*}
&\check{\zeta} = \frac{\zeta^2 }{|\zeta|_4^2} &\check{\xi} = \frac{\xi^2 }{|\xi|_4^2}
\end{align*}
where we used $\zeta^2$ for the entry-wise product $\zeta \odot \zeta$, same thing for $\xi^2$. The operator norm $\rhoo=\Vert \QQ \Vert$ is equal to 
\begin{equation}
\rhoo=(n+m)\kappa = n(1+\alpha)|\zeta|_4^2|\xi|_4^2.
\end{equation}
The spectral decomposition of $\tilde{Q}$ is 
\begin{equation}
\tilde{Q} = \rhoo  \phi^+ (\phi^+)^* -\rhoo  \phi^- (\phi^-)^* ,
\end{equation}
where $\phi^\pm$ are the unit eigenvectors of $\QQ$ which are given by
\begin{align*}
\phi^\pm = \frac{1}{\sqrt{2}}\begin{pmatrix}
\pm \check{\zeta}\\ \check{\xi}
\end{pmatrix}.
\end{align*}
In particular, the powers of $\QQ$ are given by $\QQ^s = \rhoo^s  \phi^+ (\phi^+)^* +(-1)^s \rhoo^s  \phi^- (\phi^-)^* $. Our detection threshold is
\[ \max\{\threshh_1, \threshh_2 \}=\threshh_1 =\sqrt{\frac{\rhoo}{\dd}}=  \sqrt{\frac{2n(1+\alpha)(|\xi|_4|\zeta|_4)^2}{(1+\alpha)d}}= \sqrt{\frac{2n|\zeta|_4^2|\xi|_4^2}{d}}.\]
We now compute $\gamma_i^\triangle=\Gamma^\triangle_{i,i}$, which is defined in \eqref{def:gammatri} by
\begin{align}\label{proof:rect_gamma}
\gamma_i^\triangle &= \sum_{s=0}^\ell \frac{\langle \QQ^s \mathbf{1}, \zeta_i \triangle \zeta_i \rangle}{\dd^s}.
\end{align}
To do the computation in \eqref{proof:rect_gamma} we are going to need a few steps. First, we have
\begin{equation*}\label{rect:upsilon}
\langle \phi^\pm, \mathbf{1}\rangle = \frac{|\zeta|_4^2 \pm |\xi|_4^2}{\sqrt{2}|\xi|_4^2|\zeta|_4^2} =: \Upsilon_{\pm}.
\end{equation*}
Second, we have
\begin{align*}\langle \phi^\pm, \zeta \triangle \zeta \rangle &= \frac{1}{\sqrt{2}}\left(\pm \langle \check{\zeta}, \zeta^2\rangle \right)\\
&=\frac{\pm |\zeta|^2_4}{\sqrt{2}}\\&=: \Xi_\pm.
\end{align*}
Third, using the spectral decomposition of $\QQ$, we have
\begin{align*}
\langle \QQ^s \mathbf{1}, \zeta \triangle \zeta \rangle &= \rhoo^s \langle\phi^+,\zeta \triangle \zeta \rangle \langle \phi^+, \mathbf{1}\rangle  + (-1)^s \rhoo^s \langle\phi^-, \zeta \triangle \zeta \rangle \langle \phi^-, \mathbf{1}\rangle \\
&= \rhoo^s \Xi_+ \Upsilon_+  + (-1)^s \rhoo^s \Xi_-\Upsilon_-. 
\end{align*}
When we sum over $s$ in \eqref{proof:rect_gamma} and we gather the preceding identities, we obtain
\begin{align*}
\gamma^\triangle_{1}&= \Xi_+ \Upsilon_+ \sum_{s=0}^\ell \frac{\rhoo^s}{\dd^s}+  \Xi_- \Upsilon_- \sum_{s=0}^\ell \frac{(-1)^s\rhoo^s}{\dd^s} \\
&= \Xi_+ \Upsilon_+ \sum_{s=0}^\ell \threshh_1^{2s}+  \Xi_- \Upsilon_- \sum_{s=0}^\ell (-1)^s \threshh_1^{2s} \\
&= \Xi_+ \Upsilon_+ \frac{1-\threshh_1^{2(\ell+1)}}{1-\threshh_1^2}+  \Xi_- \Upsilon_- \frac{1-(-1)^{\ell+1}\threshh_1^{2(\ell+1)}}{1+\threshh_1^2}.
\end{align*}
When computing the asymptotic value of $\gamma_+$ as $n$ is large, we can neglect  $\threshh_1^{2\ell}$. Moreover, simple manipulations show that
\begin{align*}&\Xi_+\Upsilon_+ = \frac{|\zeta|_4^2+|\xi|_4^2}{2|\xi|_4^2}&&\Xi_-\Upsilon_- = \frac{|\xi|_4^2-|\zeta|_4^2}{2|\xi|_4^2}.
\end{align*}
We find
\begin{align*}
\gamma^\triangle_i &\sim \frac{\Xi_+\Upsilon_+}{1-\threshh^2} + \frac{\Xi_-\Upsilon^1_-}{1+\threshh_1^2}\\
&=  \frac{1}{2(1-\threshh_1^2)}\frac{|\zeta|_4^2+|\xi|_4^2}{|\xi|_4^2} +\frac{1}{2(1+\threshh_1^2)}\frac{|\xi|_4^2-|\zeta|_4^2}{|\xi|_4^2} \\
 &= \frac{1}{2|\xi|_4^2} \left(  \frac{|\zeta|_4^2+|\xi|_4^2}{1-\frac{2n|\zeta|_4^2|\xi|_4^2}{d}}+\frac{|\xi|_4^2-|\zeta|_4^2}{1+\frac{2n|\zeta|_4^2|\xi|_4^2}{d}} \right),
\end{align*}
as requested in \eqref{rect:gamma12+1}. The other identity \eqref{rect:gamma12+2} is proved in the same way.  

\subsection{Proof of Proposition \ref{prop51} on mean square errors}\label{proof:prop51}

\begin{proof}
\newcommand{\Uh}{\hat{U}}
\newcommand{\Dh}{\hat{\Delta}}
\newcommand{\Vh}{\hat{V}}
For the sake of this proof only, we will note $P_0 = U\Delta V^*$ and $\Pest = \Uh \Dh \Vh$, where the columns of $U$ (resp $\Uh$) are $\zeta_i$ (resp $\hat{\zeta}^\#_i$), where $\Delta = \mathrm{diag}(\sigma_i)$ and $\Dh= \mathrm{diag}(w^\#_i)$. We have
\begin{align*}
\MSE_\star &= \Vert P_0 - \Pest \Vert^2_F \\
&= \tr[(U\Delta V^* - \Uh\Dh \Vh^*)(V\Delta U^* - \Vh \Dh \Uh^*)]\\
&= \tr[U\Delta V^* V \Delta U^* - U\Delta V^* \Vh \Dh \Uh^* - \Uh \Dh \Vh^* V \Delta U^* + \Uh \Dh \Vh^* \Vh \Dh \Uh^*].
\end{align*}
By unitary invariance, the first of these four terms is equal to $\tr(\Delta^2)=\sum \sigma_i^2$. For the other terms, we will use the approximations of Theorem \ref{thm:stats}. We will abuse the symbol $\approx$ to get rid of the error terms (they can rigorously be neglected). Thanks to Theorem \ref{thm:stats}, we know that $V^* \Vh \approx \nabla$ and $U^* \Uh \approx \nabla'$ where $\nabla,\nabla'$ are the diagonal matrices whose $i$-th entry is equal to $c^\#_{1,i}$ or $c^\#_{2,i}$. 

Consequently, the second and third terms in the trace above are close to $-\tr(\Delta \nabla \Dh \nabla')$, which is itself equal to
\[-\sum_{i=1}^{\rr_0} \sigma_i w^\#_i c_{1,i}^\#c_{2,i}^\#  \approx -\sum_{i=1}^{\rr_0} (\sigma_i c^\#_{1,i}c^\#_{2,i})^2.\]
Moreover, from the second part of Theorem \ref{thm:stats} we have 
\[\Vh^* \Vh \approx \mathfrak{C}_2^\# \ANDalt \Uh^* \Uh \approx \mathfrak{C}_1^\#, \]
consequently $\tr(\Uh \Dh \Vh^* \Vh \Dh \Uh^*) \approx \tr(\Dh \mathfrak{C}_2 \Dh \mathfrak{C}_1)$. This can be written as 
\begin{align*}
\sum_{i,j} \sigma_i \sigma_j c^\#_{1,i}c^\#_{2,j}(\mathfrak{C}_1)_{i,j}(\mathfrak{C}_2)_{i,j}.
\end{align*}
As requested. \end{proof}

\section{Some possible improvement of the quantitative bounds}
\label{sec:techdiscuss}

\subsection{On the threshold $\thresh_1$}
\label{subsec:theta1}
We have already discussed that the threshold $\thresh_2$ is an intrinsic limitation of the problem as it is the Kesten-Stigum bound.  The threshold $\thresh_1$ however is more artificial and, in most cases, it can be reduced to a smaller value $\thresh'_1 = L' / d$ with $L' \leq L$ at the cost of more technicalities. In this subsection, we discuss a few ways to improve on this new parameter $L'$ by some recipes. The way to formalize them would depend on the structure of the matrix $P$. We discuss the case of the square matrix $A$ but the same comments extend to the non-backtracking matrix $B$.

A first possibility is the following: we consider the deterministic set  $S \subset [n]^2$ of the $o(n/d)$ largest entries  of the matrix $P$ in modulus. Then with high probability, the revealed entries (non-zeros entries of $M$) will not intersect $S$. In particular, on this good event, the conclusions of Theorem \ref{thm:1} remain valid if we replace $L$ by $L' = \max_{ (x,y)\notin S} n |P_{x,y}|$.

In the proofs, we have many times bounded $P_{x,y}$ by $L/n$ and $Q_{x,y}$ by $K / (n \rho)$. A finer probabilistic and combinatorial analysis could be performed by partitioning the entries of $P$ into two  or more sets depending on the value of $n |P_{xy}|$. An analysis of this kind was done in \cite{BQZ} (see the parameter $\delta$ there).

Another possibility is to define, for a given $L' \leq L$, the matrix $P'$ obtained from $P$ by putting to $0$ all entries of $P$ larger than $nL'$. We could then apply Theorem \ref{thm:1} directly to $P'$ and then use spectral perturbation theory (such as Hoffman-Wielandt inequality) to guarantee that the spectra of $P$ and $P'$ are close provided that $L - L'$ is not too large.

\subsection{On the parameter $C_0$}

The conclusions of Theorem \ref{thm:1} and Theorem \ref{thm:1nb} are controlled by a parameter $C_0$ and $\bar C_0$. Again, depending on the structure of $P$ there are ways reduce the value of this constant at the cost of some extra technicalities.

First of all, in the proof, we have always bounded $L = \max_{x,y} n |P_{xy}|$ by $b^2\mu_1$ and $K = \max_{x,y} n  Q_{xy} / \rho$ by $b^4$. This can be a very rough bound, for example, if $nP$ is the adjacency matrix of a graph then $L = 1$. In the proof of Theorem \ref{thm:1} in Subsection \ref{subsec:proofth1}, there is a factor $N^6 C_2$ in $C_0$ which contains a factor $K^{10}$ (bounded by $b^{40}$ in the proof).

Also, the same remarks than in Subsection \ref{subsec:theta1} applies. Every time that we have bounded $n P_{xy}$ by $L$ or $n Q_{xy}$ by $K \rho$, there is room for improvement.

Finally, the incoherence parameter $b$ is important for the eigenvectors $i \in [r_0]$. Since we are interested in an average quantity such as the scalar product $\langle \psi_i , \varphi_i \rangle$, it is possible to relax the notion of incoherence by partitioning the entries $x \in [n]$ depending on the value of $\sqrt n |\varphi_i(x)|$. This makes the application of the general concentration bound 
Theorem \ref{thm:concentration} more tedious. For the other eigenvectors, $i \in [n] \backslash [r_0]$, their importance is weighted by their corresponding eigenvalue and there is also room for improvement here.

\end{document}